\documentclass[10pt,a4paper]{amsart}
\usepackage[T1]{fontenc}

\usepackage{tikz}
\usetikzlibrary{cd}
\usepackage[all,cmtip]{xy}
\usepackage{amsmath,amsthm,amssymb,mathrsfs}
\usepackage{mathtools,bbm,latexsym}
\usepackage{color,enumitem}
\usepackage{verbatim} 
\usepackage{hyperref}
\usepackage{booktabs}
\input diagxy

\DeclareMathOperator{\sub}{sub}

\DeclareMathOperator{\ub}{ub}
\DeclareMathOperator{\lb}{lb}
\DeclareMathOperator{\colim}{colim}
 
\DeclareMathOperator{\with}{\&}
\DeclareMathOperator{\idl}{Idl}

\DeclareMathOperator{\Da}{\Downarrow}

\DeclareMathOperator{\wayb}{{\rotatebox[origin=c]{-90}{$\twoheadrightarrow$}}}

\newcommand{\bbT}{\mathbb{T}}
\newcommand{\bbB}{\mathbb{B}}
\newcommand{\bbC}{\mathbb{C}}
\newcommand{\bbI}{\mathbb{I}}
\newcommand{\bbF}{\mathbb{F}}
\newcommand{\bbN}{\mathbb{N}}
\newcommand{\bbP}{\mathbb{P}}

\newcommand{\CC}{\mathcal{C}}
\newcommand{\CF}{\mathcal{F}}

\newcommand{\CB}{\mathcal{B}}
\newcommand{\CN}{\mathcal{N}}

\newcommand{\CT}{\mathcal{T}}

\newcommand{\CP}{\mathcal{P}}
\newcommand{\CPd}{\mathcal{P}^\dag}

\newcommand{\CI}{\mathcal{I}}
\newcommand{\CJ}{\mathcal{J}}
\newcommand{\CO}{\mathcal{O}}
\newcommand{\lam}{\lambda}
\newcommand{\bv}{\sup}
\newcommand{\bw}{\inf}
\newcommand{\lra}{\longrightarrow}

\newcommand{\ra}{\rightarrow}

\newcommand{\da}{\downarrow}
\newcommand{\cpt}{\circ}
\newcommand{\sy}{{\sf y}}
\newcommand{{\sfm}}{{\sf m}}

\newcommand{\syd}{{\sf y}^\dag}
\newcommand{\ua}{\uparrow} 
\newcommand{\sV}{{\sf V}}
\newcommand{\sQ}{[0,1]}
\newcommand{\id}{{\rm id}}
\newcommand{\QOrd}{[0,1]\text{-}{\bf Cat}}
\newcommand{\QCL}{[0,1]\text{-}{\bf ConLat}}
\newcommand{\TCL}{\bbT\text{-}{\bf CCAlg}}
\newcommand{\QCD}{[0,1]\text{-}{\bf CDL}}
\newcommand{\QSup}{[0,1]\text{-}{\bf Sup}} 
\newcommand{\QMod}{[0,1]\text{-}{\bf Mod}}

\newcommand{\TAlg}{\bbT\text{-}{\bf Alg}}
\newcommand{\PAlg}{\bbP\text{-}{\bf Alg}}
\newcommand{\PdAlg}{\bbP^\dag\text{-}{\bf Alg}}
\newcommand{\IAlg}{\bbI\text{-}{\bf Alg}}
\newcommand{\CAlg}{\bbC\text{-}{\bf Alg}}

\newcommand{\rto}{{\lra\hspace*{-2.9ex}{\mapstochar}\hspace*{3.1ex}}}
\newcommand{\oto}{{\lra\hspace*{-3.1ex}{\circ}\hspace*{1.7ex}}}

\theoremstyle{plain}
\newtheorem{thm}{Theorem}[section]
\newtheorem{lem}[thm]{Lemma}
\newtheorem{prop}[thm]{Proposition}
\newtheorem{cor}[thm]{Corollary}
\theoremstyle{definition}
\newtheorem{defn}[thm]{Definition}
\newtheorem{ques}[thm]{Question}
\newtheorem{exmp}[thm]{Example}
\newtheorem{rem}[thm]{Remark}

\newtheorem*{con} {Convention}

\allowdisplaybreaks 

\begin{document}
	
	\title{Introductory notes on real-enriched categories}
	
	\author{Dexue Zhang} \address{School of Mathematics, Sichuan University, Chengdu, China} \email{dxzhang@scu.edu.cn}
	
	\thanks{The author gratefully acknowledges the support of the National Natural Science Foundation of China (No. 12371463).}
	
	\subjclass[2020]{18B35, 18C15, 18C20, 18D15, 18D20}

	\keywords{Real-enriched category, functor, distributor, Cauchy completeness, Yoneda completeness, class of weights, monad, real-enriched domain}
	
	\begin{abstract} Real-enriched categories are  categories with real numbers as enrichment. Precisely, a real-enriched category is a category enriched over the commutative and unital quantale composed of the unit interval and a continuous t-norm. These notes present a brief introduction to such categories, focusing on the presheaf monad and its submonads in the category of  real-enriched categories.   \end{abstract} 
	
	\maketitle
	
	\tableofcontents

	\section{Introduction} 
	Real-enriched categories are a very special kind of enriched categories, namely, those categories enriched over real numbers. Precisely, a real-enriched category is a category enriched over the unital quantale composed of the the interval  $[0,1]$ and  a continuous t-norm. To require   enrichment be real numbers limits drastically  the scope of the theory, but the benefit is also apparent. First, real numbers posses much richer structures than a general quantale or more generally, a symmetric and monoidal closed category, one expects more can be said about such categories because of the richer structures of the enrichment. Second, the study of such categories is closely related to many-valued logic. In 1973 Lawvere pointed out in his influential paper \cite{Lawvere1973}  a deep connection between enriched categories and \emph{generalized pure logic}.  For real-enriched categories, this ``generalized pure logic'' is the \emph{Basic Logic}  developed by  H\'{a}jek \cite{Ha98}. As shown in H\'{a}jek \cite{Ha98} and  Cignoli,   Esteva,   Godo and Torrens \cite{CEGT},  the \emph{Basic Logic} is a \emph{logic of continuous t-norms}. So, real-enriched categories may be viewed as many-valued ordered sets with the unit interval as truth value set, we'll emphasize this viewpoint in these notes. 
	We would like to note here that around 1970 Goguen \cite{Goguen67,Goguen69,Goguen74} has argued the potential application of unital quantales as sets of truth values in the disposition of inexact concepts. 
	
	Though being a narrow class, real-enriched categories are sufficiently general to provide a common framework for ordered sets and quasi-metric spaces,  they are natural and core objects for Quantitative Domain Theory, see  e.g.  America and Rutten \cite{AR89}, Antoniuk and Waszkiewicz \cite{AW2011}, Bonsangue,   van Breugel and Rutten \cite{BBR98}, Flagg \cite{Flagg1997}, Flagg and Kopperman \cite{FK1997}, Gutierres and Hofmann \cite{GH}, Hofmann and Waskiewicz \cite{HW2011,HW2012}, Kostanek and Waszkiewicz \cite{KW2011}, K\"{u}nzi and Schellekens \cite{KS2002}.
	
	There exist many texts on enriched categories, e.g. Eilenberg and Kelly \cite{EK1966},  Lawvere \cite{Lawvere1973}, the book of Kelly \cite{Kelly}, Chapter 6 in the Handbook of Borceux \cite{Borceux1994b}, and the  works of Stubbe \cite{Stubbe2005,Stubbe2006,Stubbe2007} on quantaloid-enriched categories. The aim of these notes is to explain basic ideas  about real-enriched categories,  focusing on the presheaf monad and its submonads. Special attention is paid to  how the logic structures of the truth value set affect mathematical results, see e.g. Proposition \ref{subalgebra of a free one}, Theorem \ref{Vop is CD}, Theorem \ref{CD implies CL}, and Lemma \ref{conical filter form a monad}. Many topics are not covered here.   Mac\thinspace Neille completion of real-enriched categories, Formal Concept Analysis  of fuzzy contexts  (see B\v{e}lohl\'{a}vek \cite{Belo04}) are not touched;  topology of real-enriched categories, except the notion of open ball topology, is not touched either, interested readers are referred to Goubault-Larrecq \cite{Goubault}, Hofmann \cite{Hof2007},   Lowen \cite{Lowen1997,Lowen2015}, and the book \cite{Monoidal top} edited by Hofmann, Seal and Tholen.

	These notes are arranged into four parts. The first part consists of Section \ref{galois connections} and Section \ref{continuous t-norm}. This part recalls, for convenience of the reader,  some basic ideas in theories of ordered sets and continuous t-norms. Many notions for real-enriched categories  find their roots in the theory of ordered sets; the interval $[0,1]$ together with a continuous t-norm is our set of truth values. The second part consists of Section \ref{real-enriched cats} -- Section \ref{tensor and cotensor}. This part deals with basic structures of real-enriched categories, including   (adjoint) functors, (adjoint) distributors, Yoneda lemma, weighted limits and weighted colimits.
	The third part consists of Section \ref{cauchy complete} -- Section \ref{Smyth complete}, focusing on special kinds of colimits and limits. 
	The last 7 sections constitute the fourth part,  dealing with the presheaf monad in the category of real-enriched categories,  its submonads and algebras.

	
	
	\section{Galois connection and domains} \label{galois connections} 
	This   section recalls some basic ideas in order theory, and fixes some notations. The materials can be  found in many places, e.g. Davey and Priestley \cite{Davey2002}, Gierz,   Hofmann,   Keimel et al \cite{Gierz2003}, and  Goubault-Larrecq \cite{Goubault}.  
	
	\begin{defn} Let $X$ be a set. An order (or a preorder) on $X$ is a binary relation $\leq$ on $X$ such that, for all $x, y, z \in X$,
		\begin{enumerate}[label={\rm(\roman*)}] 
			\item $x \leq x$,
			\item $x\leq y$ and $y \leq z$ imply $x \leq z$.
		\end{enumerate}
		These conditions are referred to, respectively, as reflexivity and transitivity. A set $X$ equipped with an order relation $\leq$ is called an ordered set (or a preordered set).
		Usually we  say simply ``$X$ is an ordered set''. When it is necessary to specify the order relation we write $(X,\leq)$. If $x\leq y$ and $y\leq x$, we say that $x$ and $y$ are equivalent and write $x\cong y$. \end{defn}
	
	A partial order on a set $X$ is an order that satisfies additionally the axiom of antisymmetry:
	\begin{enumerate}\item[\rm(iii)] $x \cong y$ implies $x = y$. \end{enumerate}
	
	\begin{exmp}\begin{enumerate}[label={\rm(\roman*)}]  \item The identity relation  on a set $X$ is a partial order, called the discrete order on $X$.
			
			\item For each set $X$, the inclusion relation $\subseteq$ is a partial order on the powerset $2^X\coloneqq\{A\mid A\subseteq X\}$ of $X$.
			
			\item The smaller than or equal to relation $\leq$ is a partial order on the set $\mathbb{R}$ of real numbers. But, the strict smaller than relation $<$ on $\mathbb{R}$ is not an order since it is not reflexive.
			
			\item The divisibility relation $\sqsubseteq$,  defined by  $m\sqsubseteq n$ if $m=kn$ for some natural number $k$, is a partial order on the set $\mathbb{N}$ of natural numbers.
	\end{enumerate}\end{exmp}
	
	If $\leq$ is an order on a set $X$, then the opposite relation $\leq^{\rm op}$, given by $x\leq^{\rm op}y$ if $y\leq x$, is also an order on $X$, called the \emph{opposite} of $\leq$.
	
	Suppose $X$ is an ordered set, $A$ is a subset and $x$ is an element of $X$.
	We say that $x$  is  an \emph{upper bound} of   subset $A$ if $a\leq x$ for all $a\in A$. We say that $x$ is a \emph{join}, or  a \emph{supremum}, of $A$ if $x$ is the least upper bound of $A$; that means, $x$ is an upper bound of $A$ and $x\leq y$ whenever  $y$ is an upper bound of $A$. It is clear that any two joins of a subset $A$ are equivalent, hence every subset of an ordered set has  at most one join, up to isomorphism. So, we will speak of \emph{the join} of a subset.  Dually, $x$ is a \emph{lower bound} of $A$ if $x\leq a$ for all $a\in A$. The  greatest lower bound of $A$, if exists, is called the \emph{meet}, or the \emph{infimum}, of $A$.
	
	\begin{defn}Suppose $X$ is an ordered set and $A\subseteq X$. We say that $A$ is  \begin{enumerate}[label={\rm(\roman*)}] 
			\item a lower set if $x\in A$ and $y\leq x$ imply $y\in A$; \item an upper set if $x\in A$ and $x\leq y$ imply $y\in A$. \end{enumerate}\end{defn}
	
	For a subset $A$ of an ordered set $X$, let \[\ua A=\{x\in X\mid \exists\thinspace a\in A, a\leq x\}; \quad \da A=\{x\in X\mid \exists\thinspace a\in A, x\leq a\}.\] Then $\ua A$ and $\da A$ are, respectively, the smallest upper set and the smallest lower set containing $A$. So,   $\da A$ is the lower set and $\ua A$ the upper set generated by $A$, respectively. When $A$ is a singleton set $\{a\}$, we write $\ua a$ and $\da a$. A lower set of the form $\da a$ is called a \emph{principal lower set}. Likewise, an upper set of the form $\ua a$ is called a \emph{principal upper set}.
	
	\begin{exmp}For each topological space $(X,\tau)$, define a relation  $\leq_\tau$ on the set $X$ by \[x\leq_\tau y\iff x\in\overline{\{y\}}.\] Then $\leq_\tau$ is an order on $X$, called the \emph{specialization order} of the space $(X,\tau)$. The specialization  order $\leq_\tau$ is antisymmetric if and only if $(X,\tau)$ is   $T_0$; $\leq_\tau$ is the discrete order  if and only if $(X,\tau)$ is   $T_1$. Every open set is an upper set  and every closed set is a lower set with respect to the specialization order. For each $x\in X$, the principal lower set $\da x$ is the closure of $\{x\}$, i.e. $\da x=\overline{\{x\}}$. \end{exmp}
	
	Suppose $X$ is an ordered set and $A\subseteq X$. Let  \[\ub A=\{x\in X\mid \forall a\in A, a\leq x\}; \quad \lb A=\{x\in X\mid \forall a\in A, x\leq a\}.\] In other words, $\ub A $ and $\lb A $ are the set of all upper bounds and the set of all lower bounds of $A$, respectively. Then, $\ub A$ is an upper set and $\lb A$ is a lower set;    $A$ has a join  if and  only if $\ub A$ is a principal upper set;   $A$ has a meet if and  only if $\lb A$ is a principal lower set.
	
	\begin{defn}Suppose $X$ is an ordered set. \begin{enumerate}[label={\rm(\roman*)}] 
			\item A subset $D$  of  $X$ is   directed if $D$ is not empty and for all $x,y\in D$, there is some $z\in D$ such that $x\leq z$ and $y\leq z$. A  directed lower set of $X$ is called an ideal of $X$. 
			
			\item A subset $F$ of $X$ is  filtered if $F$ is not empty and for all $x,y\in F$, there is some $z\in F$ such that $z\leq x$ and $z\leq y$. A filtered  upper set of $X$ is called a filter of $X$. \end{enumerate}\end{defn}
	
	In particular, an ordered set $X$ is \emph{directed} if it is a directed subset of itself. This means that $X$ is not empty and every pair of elements in $X$ has an upper bound.
	
	\begin{exmp}Suppose $X$ is an ordered set. \begin{enumerate}[label={\rm(\roman*)}] 
			\item For each $x\in X$, the principal lower set $\da x$ is   an ideal. 
			\item For each $x\in X$, the principal upper set $\ua x$ is   a filter.  
			\item Let $X$ be a topological space. Then for each $x\in X$, the family $\CN_x=\{U\subseteq X\mid x\in U^\circ\}$ of all neighborhoods of $x$ is a filter of  $(2^X,\subseteq)$. 
			\item For each set $X$, the family of finite subsets of $X$ is an ideal of $(2^X,\subseteq)$. \end{enumerate}\end{exmp}
	
	\begin{defn}Suppose $X$ is an ordered set. We say that $X$ is     
		\begin{enumerate}[label={\rm(\roman*)}] 
			\item bounded if it has a least and a greatest element. 
			\item a join semilattice if it is antisymmetric and every pair of elements $x,y$ in $X$ has a join, denoted by $x\vee y$. 
			\item  a meet semilattice if it is antisymmetric and every pair of elements $x,y$ in $X$ has a meet, denoted by $x\wedge y$. 
			\item  a  lattice if it is at the same time a join semilattice and a meet semilattice.   \end{enumerate}\end{defn}

	If $X$ is a lattice, then so is the opposite $X^{\rm op}$ of $X$. The ordered set $(\mathbb{R},\leq)$ is a lattice, but not bounded.

	\begin{lem}\label{join vs meet}For each ordered set $X$, the following   are equivalent: \begin{enumerate}[label={\rm(\arabic*)}] 
			\item  Every subset of $X$, including the empty one, has a join. 
			\item Every subset of $X$, including the empty one,  has a meet.  \end{enumerate} \end{lem}
	
	\begin{proof}$(1)\Rightarrow(2)$ We show that each subset $A$ of $X$  has a meet. For this consider the set $\lb A$ of lower bounds of $A$. By assumption,  $\lb A$ has a join, say $x$. If we could show that $x$ is a lower bound of $A$, then it would be the greatest lower bound, hence the meet of $A$. Since every element $a\in A$ is an upper bound of $\lb A$,  it follows that $A$ is a set of   upper bounds of $\lb A$. Since $x$, being the join of $\lb A$, is the least element of $\ub (\lb A)$, then $x\leq a$ for all $a\in A$, hence $x$ is a lower bound of $A$.  
		
		$(2)\Rightarrow(1)$ Similar. \end{proof}
	
	\begin{defn} Suppose $X$ is an ordered set. We say that $X$ is \begin{enumerate}[label={\rm(\roman*)}] 
			\item complete if it satisfies the equivalent conditions in Lemma \ref{join vs meet}. \item a complete lattice if it is complete and antisymmetric. \end{enumerate} \end{defn}
	
	The empty subset of an ordered set $X$  has a join (meet, respectively)  if and only if $X$ has a least (greatest, respectively) element. So, every complete ordered set is bounded.
	
	For each set $X$, $(2^X,\subseteq)$ is a complete lattice. For any $a\leq b$ in $\mathbb{R}$, the partially ordered set $([a,b],\leq)$ is a complete lattice. For each ordered set $X$, the set $\CP X$ of all lower sets of $X$ ordered by inclusion   is a complete lattice. In particular, if $(X,\leq)$ is discrete then $(\CP X,\subseteq)$ coincides with the complete lattice $(2^X,\subseteq)$.

	\begin{defn} Suppose $X,Y$ are ordered sets and  $f\colon X\lra Y$ is a map. We say that \begin{enumerate}[label={\rm(\roman*)}] 
			\item $f$  preserves order  if $x_1\leq x_2$ implies $f(x_1)\leq f(x_2).$ 
			\item  $f$  reflects order  if $f(x_1)\leq f(x_2)$  implies $x_1\leq x_2 $. 
			\item  $f$  preserves joins  (meets, resp.) if for each subset $A$ of $X$ and each join (meet, resp.) $a$ of $A$, $f(a)$ is a join (meet, resp.) of $\{f(x)\mid x\in A\}$. 
	\end{enumerate}\end{defn}
	
	The identity map on any ordered set is order-preserving; the composite of  order-preserving maps is order-preserving. So with ordered sets and order-preserving maps we have a category {\bf Ord}.
	
	It is clear that an order isomorphism (i.e. an isomorphism in the category {\bf Ord}) preserves both joins and meets.
	A map $f\colon X\lra Y$ is called an \emph{order-reversing isomorphism} if $f\colon X^{\rm op}\lra Y$ is an order isomorphism, or equivalently, $f\colon X\lra Y^{\rm op}$ is an order isomorphism. Every order-reversing isomorphism transforms joins  to meets  and  meets to joins.
	
	Suppose $\{X_i\}_{i\in J}$ is a family of ordered sets. Define a binary relation on the product set $\prod_{i\in J}X_i$ by \[(x_i)_i\leq (y_i)_i ~\quad~ \text{if $x_i\leq  y_i$  for each $i$}.\]  Then $\leq$ is an order on $\prod_iX_i$ and for each $j$, the projection $p_j\colon  \prod_iX_i \lra X_j $ preserves order. 
	
	\begin{exmp}\label{exmp of order-preserving maps} Suppose $X$ is an ordered set. \begin{enumerate}[label={\rm(\roman*)}] 
			\item The diagonal map
			\[\Delta\colon X\lra X\times X, \quad \Delta(x)=(x,x)\] preserves and reflects order. 
			
			\item The map \[\sy\colon X\lra\CP X, \quad x \mapsto\thinspace\da x\] preserves and reflects order, where $\CP X$ is the set of all lower sets of $X$ ordered by  inclusion. This is a special case of   \emph{Yoneda embedding} in category theory.
			
			\item The map \[\syd\colon X\lra\CPd X, \quad x \mapsto\thinspace\ua x\]  preserves and reflects order, where $\CPd X$ is the set of all upper sets of $X$ ordered by converse inclusion. This is a special case of  \emph{coYoneda embedding} in category theory.
	\end{enumerate} \end{exmp}

	\vskip 5pt \noindent{\bf Galois connection}
	
	\begin{defn}Suppose $X,Y$ are ordered sets; $f\colon X\lra Y$ and $g\colon Y\lra X$ are a pair of maps. We say that $f$ is left adjoint to $g$, or $g$ is right adjoint to $f$, and write $f\dashv g$,  if for all $x\in X$ and $y\in Y$ it holds that \[f(x)\leq y\iff x\leq g(y).\]   In this case,  the pair $(f,g)$  is called a Galois connection, or an adjunction, with $f$ being the left adjoint and $g$ the right adjoint. \end{defn}
	
	If $f\colon X\lra Y$ is an order isomorphism, then $f$ is both  left and right adjoint to its inverse $f^{-1}$.  
	
	\begin{lem}If  both $f_1\colon X\lra Y$ and $f_2\colon X\lra Y$ are left adjoint to $g\colon Y\lra X$, then $f_1(x)\cong f_2(x)$ for all $x\in X$. \end{lem}

	Therefore, the left adjoints of a map, if exist, are essentially  unique; likewise for right adjoints.  So, we shall speak of \emph{the left adjoint} and \emph{the right adjoint}.
	
	\begin{lem}Left adjoints  and  right adjoints  preserve  order. \end{lem}
	
	\begin{proof}Suppose   $f\colon X\lra Y$ is left adjoint to $g\colon Y\lra X$. We show that both $f$ and $g$ preserve order.   Assume that $x_1\leq x_2$. Since $f(x_2)\leq f(x_2)$, then $x_2\leq gf(x_2)$, so $x_1\leq gf(x_2)$, hence $f(x_1)\leq f(x_2)$. This shows that $f$ preserves order. Likewise, $g$ preserves order. \end{proof}
	
	\begin{lem}\label{composite of left adjoints} Each composite of left adjoints is a left adjoint.  Dually, each composite of  right  adjoints is a  right  adjoint. \end{lem}

	\begin{exmp}Suppose $X$ is a topological space, $\CO(X)$ is the set of open sets and $\mathcal{C}(X)$ is the set of closed sets. Then the interior operator $( -)^\circ\colon 2^X \lra  \CO(X) $ is right adjoint to the inclusion   $  \CO(X) \lra 2^X ;$  the closure operator $(-)^-\colon 2^X \lra  \mathcal{C}(X) $ is  left adjoint to the inclusion  $  \mathcal{C}(X) \lra 2^X .$  \end{exmp}
	
	\begin{exmp}Suppose $X$ is an ordered set; $\CP X$ and $\CPd X$ are the ordered sets of lower sets and upper sets of $X$, respectively, as in Example \ref{exmp of order-preserving maps}.  For all $A\in\CP X$ and $B\in\CPd X$, it holds that \[\ub A\supseteq B\iff \forall x\in A, \forall y\in B,~ x\leq y\iff A\subseteq \lb B,\]  so $\ub\colon \CP X\lra\CPd X$ is left adjoint to $\lb\colon \CPd X\lra \CP X.$ \end{exmp}
	
	\begin{exmp} \label{image is left adjoint to preimage} Suppose $X,Y$ are sets and $f\colon X\lra Y$ is a map. Let \[f_\exists\colon2^X \lra 2^Y\] be the map that sends each subset $A$ of $X$ to its image $ \{f(a)\mid a\in A\}$; let \[f^{-1}\colon2^Y \lra 2^X\] be the map that sends each subset $B$ of $Y$ to its preimage $\{x\in X\mid f(x)\in B\}$.  
		Then $f_\exists$ is left adjoint to $f^{-1}$. The map $f^{-1}$ also has a right adjoint, given by \[f_\forall \colon 2^X \lra  2^Y ,\quad f_\forall (A)=\{y\in Y\mid f^{-1}(y)\subseteq A\}. \]   Thus, we have a string  of adjunctions: \[f_\exists\dashv f^{-1}\dashv f_\forall\colon 2^X\lra 2^Y .\]   \end{exmp}
	
	\begin{exmp} Suppose $X,Y$ are ordered sets and $f\colon X\lra Y$  preserves order. For each lower set $B\subseteq Y$, the inverse image $f^{-1}(B)$ is a lower set of $X$, so we have an order-preserving map \[f^{-1}\colon \CP Y  \lra \CP X .\] For each lower set $A\subseteq X$,  let \[f_\exists(A)=\{y\in Y\mid y\leq f(a) ~\text{for some}~ a\in A\}.\]  
		Then  \[f_\exists\colon \CP X  \lra  \CP Y \] is order-preserving and  it is left adjoint to $f^{-1}$. The map $f^{-1}$ also has a right adjoint   \[f_\forall \colon \CP X  \lra  \CP Y \] given by \[f_\forall (A)=\{y\in Y\mid f^{-1}(\da y)\subseteq A\}.\] Thus, we  have a string of adjunctions: \[f_\exists\dashv f^{-1}\dashv f_\forall\colon \CP X\lra\CP Y .\]
		If both $X$ and $Y$ are discrete ordered sets, then this string of adjunctions reduces to that in the above example. \end{exmp}
	
	\begin{thm}\label{Galois connection via unit}Suppose $X,Y$ are ordered sets; $f\colon X\lra Y$ and $g\colon Y\lra X$  preserve order. The following   are equivalent: 
		\begin{enumerate}[label={\rm(\arabic*)}] 
			\item $f\dashv g$.
			\item $\id_X\leq gf$ and $fg\leq\id_Y$.
			\item For each $x\in X$, $f(x)$ is a least element in $g^{-1}(\ua x)$. 
			\item For each $y\in Y$, $g(y)$ is a greatest element in $f^{-1}(\da y)$.  
	\end{enumerate} \end{thm}
	
	\begin{proof} $(1)\Rightarrow(2)$ 
		For each $x\in X$, since $f(x)\leq f(x)$, then  $x\leq gf(x)$, so  $\id_X\leq gf$.  For each $y\in Y$, since $g(y)\leq g(y)$, then  $fg(y)\leq y$, so $fg\leq\id_Y$.  
		
		$(2)\Rightarrow(3)$ First, since $x\leq gf(x)$, then $f(x)\in g^{-1}(\ua x)$. Second, for each $y\in g^{-1}(\ua x)$, since $x\leq g(y)$, then  $f(x) \leq fg(y) \leq y$.  So  $f(x)$ is the least element in $g^{-1}(\ua x)$.
		
		$(3)\Rightarrow(1)$ Suppose   $f(x)\leq y$. Since $f(x)\in g^{-1}(\ua x)$, then  $x\leq gf(x) \leq g(y)$. Conversely, suppose   $x\leq g(y)$. Since $fg(y)$ is the least element in $g^{-1}(\ua g(y))$ and $y\in g^{-1}(\ua g(y))$, then $f(x)\leq fg(y)\leq y$.   Therefore, $f$ is left adjoint to $g$.
		
		That $(2)\Rightarrow(4) \Rightarrow(1)$ can be proved in a similar way.  \end{proof} 
	
	\begin{prop} Suppose   $X, Y$ are partially ordered sets;  $f\colon X\lra Y$ is left adjoint to $g\colon Y\lra X$. Then $f=fgf$ and $g=gfg$, hence both $gf\colon X\lra X$ and $fg\colon Y\lra Y$ are idempotent.  
	\end{prop}
	
	\begin{prop}\label{surjective right adjoint} Suppose  $X, Y$ are partially ordered sets,   $f\colon X\lra Y$ is left adjoint to $g\colon Y\lra X$. The following conditions are equivalent:
		\begin{enumerate}[label={\rm(\arabic*)}] 
			\item   $g$ is surjective.
			\item $f(x)=\min g^{-1}(x)$ for all $x\in X$.
			\item $gf=1_X$.
			\item $f$ is injective.
		\end{enumerate} Likewise, the following conditions are equivalent:
		\begin{enumerate}[label={\rm(\arabic**)}] 
			\item   $g$ is injective.
			\item $g(y)=\max f^{-1}(y)$ for all $y\in Y$.
			\item $fg=1_Y$.
			\item $f$ is surjective.
	\end{enumerate}\end{prop}
	
	\begin{thm}\label{lapj} Every left adjoint  $f\colon  X\lra Y$  preserves joins. Conversely, if $X$ is complete and $f\colon X\lra Y$ preserves joins, then $f$ is a left adjoint. \end{thm}
	
	\begin{proof} Suppose   $f\colon  X\lra Y$ has a right adjoint, say $g\colon Y\lra X$. We show that for any subset $A$ of $X$, $f(\bv A)$ is a join of $\{f(a)\mid a\in A\}$ whenever $\bv A$ exists. Since $f$ preserves order, then $f(\bv A)$ is an upper bound of $\{f(a)\mid a\in A\}$. If $y$ is an upper bound of $\{f(a)\mid a\in A\}$, then $a\leq g(y)$ for all $a\in A$, so  $\bv A\leq g(y)$, hence $f(\bv A)\leq y$.
		
		Conversely, suppose $X$ is complete and $f\colon X\lra Y$ preserves joins.  Since $f$ preserves joins, for each $y\in Y$  the join of   $f^{-1}(\da y)$ is also a member of $f^{-1}(\da y)$. Assigning to  $y$ the join of $f^{-1}(\da y)$ yields a right adjoint of $f$. \end{proof}
	
	Dually, every right adjoint preserves meets. Conversely, if $f\colon X\lra Y$ preserves meets and $X$ is complete then $f$ is a right adjoint.
	
	Galois connections    are of fundamental importance in order theory. Many concepts in order theory are postulated, or can be characterized, in terms of Galois connections. 
	
	\begin{prop}\label{complete lattice via left adjoint}For each ordered set $X$, the following statements are equivalent: 
		\begin{enumerate}[label={\rm(\arabic*)}] 
			\item $X$ is   complete. 
			\item The Yoneda embedding $\sy\colon X\lra\CP X$ has a left adjoint. 
			\item The coYoneda embedding $\syd\colon X\lra\CPd X$ has a right adjoint.
	\end{enumerate} \end{prop}
	
	\begin{proof} $(1)\Rightarrow(2)$ Suppose   $X$ is complete. Let $s\colon\CP X\lra X$ be a map such that $s(A)$ is a join of $A$ for each $A\in\CP X$. Then $s$ is a left adjoint of $\sy$. 
		
		$(2)\Rightarrow(1)$ If  $s\colon\CP X\lra X$ is left adjoint to $\sy$, then for all $x\in X$ and $A\in\CP X$, \[A\subseteq\thinspace\da x\iff s(A)\leq x,\] which says that $x$ is an upper bound of $A$ if and only if $x\geq s(A)$.  So  $s(A)$ is a least upper bound of $A$.
		
		Likewise for $(1)\Leftrightarrow(3)$.\end{proof}
	
	\begin{prop}Suppose $X$ is a partially ordered set.  \begin{enumerate}[label={\rm(\roman*)}] 
			\item $X$ is a meet semilattice if and only if the diagonal map $\Delta\colon X\lra X\times X$ has a right adjoint. 
			\item  $X$ is a join semilattice if and only if the diagonal map $\Delta\colon X\lra X\times X$ has a left adjoint. 
			\item $X$ is a lattice if and only if the  diagonal map $\Delta\colon X\lra X\times X$ has, at the same time, a left and a right adjoint; that is, there is a string of adjunctions: $\vee\dashv\Delta\dashv\wedge\colon X\times X\lra X.$ 
	\end{enumerate} \end{prop}
	
	Quantales provide another kind of interesting examples.  Here we only include the definition of commutative and unital  quantales, standard reference for  quantales is Rosenthal \cite{Rosenthal1990}. A \emph{commutative and unital quantale}  \[{\sf Q}=({\sf Q},\with,k)\]
	is a commutative monoid with $k$ being the unit, such that the underlying set ${\sf Q}$ is a complete lattice  and for each $x\in{\sf Q}$, the map $x\with-\colon{\sf Q}\lra{\sf Q}$ has a right adjoint $x\ra-\colon{\sf Q}\lra{\sf Q}$. This is equivalent to that   $\with$ distributes over  joins.  In particular, a \emph{frame} is a complete lattice $L$ such that $(L, \wedge,1)$ is a quantale; or equivalently,  for each $a\in L$ the meet operator $a\wedge-\colon L\lra L$ distributes over joins.

	\vskip 5pt \noindent{\bf Completely distributive lattices}
	
	Suppose $X$ is a partially ordered set. By Proposition \ref{complete lattice via left adjoint}, $X$ is a complete lattice if and only if the Yoneda embedding $\sy\colon X\lra\CP X$ has a left adjoint $\sup\colon \CP X\lra X$, which sends each subset $A$ of $X$ to its join $\bv A$.
	
	\begin{defn}A complete lattice $X$ is  completely distributive if the left adjoint of the Yoneda embedding $\sy\colon X\lra\CP X$ also has a left adjoint; that means, there is a string of adjunctions \[\Downarrow\thinspace\dashv\sup\dashv\sy\colon X\lra\CP X.\]\end{defn}
	
	Lattices satisfying the requirement in the above definition are also said to be \emph{constructively completely distributive}, see e.g. Wood \cite{Wood}, because, compared with the complete distributivity law (see Theorem \ref{CompDist}), it does not resort to the Axiom of Choice.
	
	To characterize completely distributive lattices, we need the notion of totally below relation. Suppose $x,y$ are elements of a complete lattice $X$. We say that $x$ is \emph{totally below} $y$, in symbols $x\lhd y$, if for each subset $A$ of $X$, $$ y\leq\sup A \implies   x\leq a~\text{for some}~  a\in A.$$
	
	The following  facts are easy to check: \begin{itemize} 
		\item If $x\lhd y$ then $x\leq y$.  
		\item $x\lhd y$ if and only if   $x\in \bigcap\{A\in \CP X\mid y\leq \sup A\}$. 
		\item If $a\leq x\lhd y\leq b$  then $a\lhd b$. \item If $x\lhd \bv_{i\in I}y_i$  then $x\lhd y_i$  for some $i\in I$.
	\end{itemize}
	
	\begin{thm}\label{CompDist} For each complete lattice $X$, the following   are equivalent: \begin{enumerate}[label={\rm(\arabic*)}] 
			\item $X$ is completely distributive.
			\item For each family $\{A_i\}_{i\in I}$ of lower sets of $X$, \[\sup\bigcap_{i\in I}A_i=\bw_{i\in I}\sup A_i.\]
			\item Every element $x$ of $X$ is the join of the elements   totally below it, i.e. $x=\bv\{z\in X\mid z\lhd x\}$.
			\item $X$ satisfies the complete distributivity law: for any family $\{x_{i,j}\mid i\in I, j\in J_i\}$ of elements of $X$, \begin{equation*}  \label{CD law}\bw_{i\in I}\bv_{j\in J_i}x_{i,j} = \bv_{f\in\prod_{i\in I}J_i}\bw_{i\in I}x_{i,f(i)}.\end{equation*} \end{enumerate}  \end{thm}
	
	\begin{proof}$(1)\Rightarrow(2)$ This follows from that $\sup\colon \CP X\lra X$ is a right adjoint.
		
		$(2)\Rightarrow(3)$ This follows from that an element $z$  is totally below $x$ if and only if  \[z\in \bigcap\{A\in \CP X\mid x\leq \sup A\}.\]  
		
		$(3)\Rightarrow(4)$ Let \[x=\bw_{i\in I}\bv_{j\in J_i}x_{i,j}\] and suppose $z\lhd x$. By definition of the totally below relation, for each $i\in I$, there is some element in $J_i$, say $g(i)$, such that $z\leq x_{i,g(i)}$. Thus \[z\leq \bv_{f\in\prod_{i\in I}J_i}\bw_{i\in I}x_{i,f(i)}\] and consequently, \[\bw_{i\in I}\bv_{j\in J_i}x_{i,j} \leq \bv_{f\in\prod_{i\in I}J_i}\bw_{i\in I}x_{i,f(i)}\] by arbitrariness of $z$. The converse inequality is obvious.
		
		$(4)\Rightarrow(1)$ For each $x\in X$, let \[\Da x=\bigcap\{A\in\CP X\mid x\leq\sup A\}.\] We show that for all $x\in X$ and $A\in\CP X$, \[\Da x\subseteq A\iff x\leq\sup A,\]  which entails that \[\Downarrow\colon X\lra \CP X\] is left adjoint to \[\sup\colon \CP X\lra X.\] To this end, we check that   each $x\in X$  is the join of $\Da x$, i.e. $\sup\Da x=x$.
		
		Index the family $\{A\in\CP X\mid x\leq\sup A\}$ as $\{A_i\}_{i\in I}$; for each $i$, index the elements of $A_i$ as $\{x_{i,j}\mid j\in J_i\}$. Since each $A_i$ is a lower set, it follows that for each $f\in\prod_{i\in I}J_i$,   \[\bw_{i\in I}x_{i,f(i)}\in\bigcap_{i\in I}A_i.\] Then,  from the complete distributivity law it follows that $x$ is the join of $\Da x$ which is, by definition, the intersection $\bigcap_{i\in I}A_i$.   \end{proof}
	
	\begin{cor}The totally below relation in a completely distributive lattice $X$ is interpolative in the sense that if $x\lhd y$ then there is some $z\in X$ such that $x\lhd z\lhd y$.  \end{cor}
	
	\begin{proof}This follows immediately from the fact that the map $\Downarrow\colon X\lra\CP X$, being a left adjoint, preserves joins.\end{proof}
	
	For each ordered set $X$, the set $\CP X$ of all lower sets of $X$ ordered by inclusion  is  a completely distributive lattice. That $(\CP X,\subseteq)$ is a complete lattice is trivial. For each $A\in \CP X$, it is clear that  for every $x\in A$,  the lower set $\da x$ is totally below $A$,  and $A$ is the join of $\{\da x\mid x\in A\}$, so, $\CP X$ is a completely distributive lattice.  
	
	\begin{prop}\label{CD is self dual}
		The complete distributivity law is self dual in the sense that if $X$ is a completely distributive lattice, then so is $X^{\rm op}$. \end{prop}
	
	\begin{proof} By Theorem \ref{CompDist}, it suffices to check that for each  family $\{A_i\}_{i\in I}$ of upper sets in $X$, \[\inf\bigcap_{i\in I}A_i=\bv_{i\in I}\inf A_i.\]
		
		We only need to check that the left side is smaller than or equal to the right side. To this end, we check that if   $x$ is totally below $\inf\bigcap_{i\in I}A_i$, then $x\leq \inf A_i$ for some $i\in I$. Suppose on the contrary that $x\not\leq \inf A_i$ for any $i\in I$. Take some $z$ such that $x\lhd z\lhd\inf\bigcap_{i\in I}A_i$;   for each $i$  take some $a_i\in A_i$ with $x\not\leq a_i$. Since each $A_i$ is an upper set, it follows that \[\bv_{i\in I}a_i\in \bigcap_{i\in I}A_i.\] Thus, $z\not\leq \bv_{i\in I}a_i$, contradicting that $z\leq \inf \bigcap_{i\in I}A_i$. \end{proof}
	
	\vskip 5pt \noindent{\bf Domains}
	
	Suppose $X$ is an ordered set and $\idl X$ is the set of  ideals (= directed lowers sets) of $X$ ordered by inclusion. Since for each $x\in X$, the principal lower set $\da x$ is an ideal, the Yoneda embedding $\sy\colon X\lra\CP X$ factors through $\idl X$.  We also use the symbol $\sy$ to denote the Yoneda embedding with codomain restricted to   $\idl X$. This will cause no confusion,  it is easily detected from the context which one is meant.
	
	Let $x,y$ be elements of a partially ordered $X$. We say that $x$ is \emph{way below} $y$, in symbols $x\ll y$, if   for each directed set $D$ of $X$ with a join,  \[y\leq\sup D \implies   x\leq d~\text{for some}~  d\in D.\]
	
	The following  facts are easily verified: \begin{itemize}
		\item If $x\ll y$ then $x\leq y$. 
		\item $x\ll y$ if and only if   $x\in \bigcap\{I\in \idl X\mid \text{$I$ has a join and}~y\leq \sup I\}$. 
		\item If $a\leq x\ll y\leq b$  then $a\ll b$.  
	\end{itemize}
	
	\begin{defn}Suppose $X$ is a partially ordered set. We say that  
		\begin{enumerate}[label={\rm(\roman*)}] 
			\item $X$ is a directed complete partially ordered set, a dcpo for short, if every directed subset of $X$ has a join. 
			\item $X$ is a continuous partially ordered set if for each $x\in X$, the set $\{y\in X\mid y\ll x\}$ is directed and has $x$ as a join. 
			\item $X$ is a domain if $X$ is a continuous dcpo.  
	\end{enumerate} \end{defn}

	We say that a map $f\colon X\lra Y$ between ordered sets is \emph{Scott continuous} if it preserves directed joins. With dcpos and Scott continuous maps we have a category $${\bf dcpo}.$$  
	
	\begin{thm}\label{domain} Suppose $X$ is a partially ordered set.   
		\begin{enumerate}[label={\rm(\roman*)}] 
			\item $X$ is a dcpo if and only if  the Yoneda embedding $\sy\colon X\lra\idl X$ has a left adjoint, which is denoted by  $\sup\colon \idl X\lra X$.  
			\item $X$ is continuous if and only if for each $x\in X$, there is a directed set $D$ consisting of elements that are way below $x$ such that $x=\sup D$.
			
			\item $X$ is a domain if and only if   the left adjoint $\sup\colon \idl X\lra X$ of  $\sy\colon X\lra\idl X$  has a left adjoint, which is denoted by  $\wayb \colon X\lra\idl X$.
			In this case, $\wayb x=\{z\in X\mid z\ll x\}$.
	\end{enumerate}  \end{thm}
	
	Therefore, a partially ordered set $X$ is a domain if there is a string of adjunctions \[\wayb \dashv \sup\dashv\sy\colon X\lra\idl X.\]
	
	\begin{cor}\label{way below interpolates} The way below relation in a continuous partially ordered set $X$ is interpolative in the sense that   $x\ll y$ implies  that $x\ll z\ll y$ for some $z\in X$.  \end{cor}
	
	\begin{exmp}For each ordered set $X$, the set $\idl X$ of all ideals of $X$ ordered by inclusion is a domain. For each directed family $\{I_d\}_{d\in D}$ of  ideals in $X$,  the union $\bigcup_{d\in D}I_d$ is an ideal in $X$, and it is the join of $\{I_d\}_{d\in D}$  in $(\idl X,\subseteq)$, hence $(\idl X,\subseteq)$ is directed complete. For each ideal $I$   of $X$, it is clear that $I$ is the join of the directed family $\{\da x\mid x\in I\}$ of ideals. Since for each $x\in I$, the principal ideal $\da x$ is   way below $I$ in $(\idl X,\subseteq)$, it follows that $\idl X$ is a domain. \end{exmp}
	
	\begin{defn} {\rm (Scott \cite{Scott72})} A partially ordered set is   a continuous lattice if it is  at the same time  a domain and a complete lattice. \end{defn}
	
	It is clear that a subset $A$ of a join semilattice $X$ is an ideal if and only if $A$ is a lower set and closed under binary joins. So, the intersection of a family of ideals in a bounded lattice   is also an ideal, though the (set-theoretic) intersection of a family of ideals in a partially ordered may fail to be so.
	
	For a bounded lattice $X$, the set of ideals $\idl X$ is closed under meets in $\CP X$, so it is a complete lattice and the inclusion $\idl X\lra \CP X$ has a left adjoint. In particular, if $X$ is a completely distributive lattice, then the composite of $\Downarrow\colon X\lra\CP X$ and the left adjoint of the inclusion \(\idl X\lra \CP X\) is   left adjoint to $\sup\colon \idl X\lra X$. This shows that every completely distributive lattice is continuous.
	
	\begin{thm}\label{DD law} For each complete lattice $X$, the following   are equivalent: \begin{enumerate}[label={\rm(\arabic*)}] 
			\item $X$ is continuous.
			\item The map \(\sup\colon \idl X\lra X\) is a right adjoint.   
			
			\item $X$ satisfies the directed distributivity law: for any  family $\{x_{t,s}\mid t\in T, s\in S_t\}$ of elements of $X$ such that  $\{x_{t,s}\mid s\in S_t\}$ is directed for each $t\in T$, it holds that \begin{equation*}  \label{DDlaw2}\bw_{t\in T}\bv_{s\in S_t}x_{t,s} = \bv_{f\in\prod_{t\in T}S_t}\bw_{t\in T}x_{t,f(t)}.\end{equation*}
	\end{enumerate} \end{thm}
	
	\begin{defn}Suppose $X$ is a lattice and $x\in X$. We say that 
		\begin{enumerate}[label={\rm(\roman*)}]  
			\item $x$ is meet-irreducible if,   $ a\wedge b=x\implies x=a$ or $x=b$.
			\item $x$ is  prime if,   $a\wedge b\leq x  \implies a\leq x$ or $b\leq x$.
			\item $x$ is join-irreducible if,   $x= a\vee b\implies x=a$ or $x=b$.
			\item $x$ is  coprime if,   $x\leq a\vee b\implies x\leq a$ or $x\leq b$.     
	\end{enumerate} \end{defn}
	
	It is clear that $x$ is   meet-irreducible in $X$ if and only if $x$ is join-irreducible in $X^{\rm op}$;   $x$ is prime in $X$ if and only if $x$ is coprime in $X^{\rm op}$. Every prime element is   meet-irreducible,  but the converse is not   true.
	
	\begin{prop}Every meet-irreducible element in a distributive lattice is prime. \end{prop}
	
	\begin{prop} \label{CL has enough irr} Suppose   $x$ and $y$ are elements of a continuous lattice $X$ with $x\not\leq y$. Then there is a meet-irreducible element $p$ with $y\leq p$ and $x\not\leq p$. \end{prop}
	
	\begin{proof}Since $X$ is a continuous lattice and $x\not\leq y$, there is some $z\ll x$ such that $z\not\leq y$.  By interpolation property of the way below relation, there is a sequence of elements $\{z_n\}_n$ with \[z\ll\cdots\ll z_n\ll z_{n-1}\ll\cdots\ll z_1=x.\]  Let $U=\bigcup_{n=1}^\infty \ua z_n$. Then  $U$ is an upper set and is closed under finite meets.
		Pick a maximal chain $C$ in $X\setminus U$ that contains $y$, and let $p=\sup C$. We claim that $p$ is not in $U$, hence maximal in $X\setminus U$. Otherwise, there is some $n$ such that $z_n\leq p$. Since $z_{n+1}\ll z_n$ and $C$ is a chain, then $z_{n+1}\leq c$ for some $c\in C$, contradicting that $c\notin U$. It is clear that $y\leq p$ and  $x\not\leq p$, thus it remains to check that $p$ is meet-irreducible. Suppose   $p=a\wedge b$. Since $U$ is closed under finite meets,   either $a\notin U$ or $b\notin U$.  Since $p$ is maximal in $X\setminus U$, it follows that $p=a$ or $p=b$. \end{proof}

	The above conclusion shows  that each element in a continuous lattice is the meet of a family of meet-irreducible elements. In other words, every continuous lattice has enough meet-irreducible elements.
	
	We say that a complete lattice $X$ has enough coprimes if every element of $X$ is the join of a set of coprimes.

	\begin{thm}\label{CD vs DD} For each complete lattice $X$, the following are equivalent: \begin{enumerate}[label={\rm(\arabic*)}]  
			\item $X$ is a completely distributive.
			\item $X$ is distributive and both $X$ and $X^{\rm op}$ are   continuous lattices.
			\item $X$ is continuous and has enough coprimes. \end{enumerate} \end{thm}
	
	\begin{proof} $(1)\Rightarrow(2)$ First,  it follows from the complete distributivity law that $X$ is distributive. Second, since complete distributivity is self-dual and   every completely distributive lattice is continuous, it follows that both $X$ and $X^{\rm op}$ are   continuous lattices.
		
		$(2)\Rightarrow(3)$ Apply Proposition \ref{CL has enough irr} to $X^{\rm op}$.
		
		$(3)\Rightarrow(1)$ Since $X$ is continuous and has enough coprimes,   for each $x\in X$, \[x=\sup\{z\in X\mid z\ll x\}=\sup\{z\in X\mid z~\text{is   coprime and }  z\ll x\}.\] Since for each  coprime $z$,  $z$  is way below $x$ implies $z$ is totally below $x$,    the conclusion thus follows. \end{proof}

	\section{Continuous t-norm} \label{continuous t-norm}
	
	The interval $[0,1]$ together with a continuous t-norm is our truth-value set. So, the structure of continuous t-norms is of fundamental importance in the development of real-enriched categories. For sake of self-containment and convenience of the reader,   proofs of the well-known  characterization of continuous Archimedean t-norms (Theorem \ref{representation}) and the ordinal sum decomposition of continuous t-norms (Theorem \ref{ordinal sum decomposition}) are included.  
	Standard references for continuous t-norms are  the monographs: Alsina,  Frank,     Schweizer \cite{AFS},   and Klement,   Mesiar, Pap \cite{Klement2000}.
	
	\begin{defn} A triangular norm (a t-norm  for short) on an interval $[a,b]$, where $a$ is allowed to be $-\infty$ and $b$ is allowed to be $\infty$, is a binary operator $\&$ such that for all $x,y,z\in[a,b]$,   
		\begin{enumerate}[label={\rm(T\arabic*)}] 
			\item $x\with y=y\with x$; 
			\item $(x\with y)\with z=x\with(y\with z)$; 
			\item $x\with y\leq x\with z$ whenever $y\leq z$; \item $x\with b =x$.\end{enumerate}  \end{defn}
	
	From (iii) and (iv) we immediately obtain that $a\with x=a$ for all $x\in [a,b]$.
	
	In the vocabulary of lattice theory, a t-norm on $[a,b]$ is a binary operator $\with$ such that $([a,b],\with,b)$ is a commutative lattice-ordered monoid, see e.g. Birkhoff \cite{Birkhoff}. In the language of category theory, a t-norm is a binary operator $\with$ such that $([a,b],\with,b)$ is a commutative monoid in the category of ordered sets and order-preserving maps, see e.g. Mac\thinspace Lane \cite{MacLane1998}.

	Suppose $\with_1$ is a t-norm on $[a_1,b_1]$ and $\with_2$ is a t-norm on $[a_2,b_2]$. We say that $\with_1$ is \emph{isomorphic} to $\with_2$ if there is an order-preserving bijection $f\colon [a_1,b_1]\lra[a_2,b_2]$ such that for all $x,y\in[a_1,b_1]$, \[f(x\with_1 y)=f(x)\with_2 f(y).\]  
	
	Given a  t-norm $\&$ on $[c,d]$ and an order isomorphism $f\colon [a,b]\lra [c,d]$, the binary operator \[\with_f\colon [a,b]\times[a,b]\lra[a,b],\quad  x\with_f y=f^{-1}(f(x)\with f(y)) \] is a  t-norm on $[a,b]$ and   isomorphic to $\with$.  This fact shows that we only need to consider t-norms on the unit interval $[0,1]$.
	
	\begin{exmp} Some basic   t-norms on $[0,1]$:
		\begin{enumerate}[label={\rm(\roman*)}] 
			\item The G\"{o}del t-norm: $x\with y= \min\{x,y\}$.
			
			\item The product t-norm: $x\with y=x\cdot y$.
			
			\item The {\L}ukasiewicz t-norm:
			$x\with y=\max\{x+y-1,0\}$.
			
			\item The nilpotent minimum t-norm: $$x\with  y= \begin{cases} 0  & x+y\leq 1,\\ \min\{x,y\}  & x+y>1.\end{cases}$$
			
			\item The drastic product t-norm: $$x\with  y= \begin{cases} 0  & x,y<1,\\ \min\{x,y\}  & {\rm otherwise}.\end{cases}$$
	\end{enumerate}\end{exmp}
	
	It is clear that   every  t-norm on $[0,1]$ lies between the drastic product  and the G\"{o}del t-norm.
	
	\begin{exmp}The addition $+$ is a t-norm on the interval $[-\infty,0]$. The correspondence $x\mapsto \mathrm{e}^x$ defines an isomorphism  $([-\infty,0],+,0)\lra([0,1],\cdot,1)$. \end{exmp}
	
	\begin{exmp}[Truncated addition] For each $a\in(-\infty,0)$, the truncated addition \[+_a\colon[a,0]\times[a,0]\lra[a,0],\quad x+_a y=\max\{a,x+y\}\] is a t-norm on $[a,0]$. \end{exmp}
	
	\begin{exmp}[Truncated product] For each $b\in(0,1)$, the truncated product \[\cdot_b\colon[b,1]\times[b,1]\lra[b,1],\quad x\cdot_b y=\max\{b,x\cdot y\}\] is a t-norm on $[b,1]$. \end{exmp}
	
	\begin{prop}For all $a\in(-\infty,0)$ and  $b\in(0,1)$, the following t-norms are isomorphic to each other: \begin{enumerate}[label={\rm(\roman*)}] 
			\item The truncated addition on $[a,0]$. \item The truncated product on $[b,1]$. \item The {\L}ukasiewicz t-norm on $[0,1]$. \end{enumerate}  \end{prop}
	\begin{proof} It is clear that the truncated addition on $[a,0]$ is isomorphic to the truncated addition on $[-2,0]$  and the truncated product on $[b,1]$  is isomorphic to the truncated product on $[1/4,1]$, so it suffices to show that both the truncated addition on $[-2,0]$ and the truncated product on $[1/4,1]$  are isomorphic to  the {\L}ukasiewicz t-norm on $[0,1]$.
		
		The map $$f\colon [0,1]\lra[1/4,1],\quad f(x)=2^{2(x-1)}$$ defines an isomorphism between the {\L}ukasiewicz t-norm on $[0,1]$ and the truncated product on $[1/4,1]$;  the map $$g\colon [0,1]\lra [-2,0],\quad g(x)=2(x-1)$$ defines an isomorphism between the {\L}ukasiewicz t-norm on $[0,1]$ and the truncated addition on $[-2,0]$.   The proof is thus completed. \end{proof}
	
	Given a  t-norm $\&$ on $[0,1]$ and $x\in[0,1]$, the powers (with respect to $\&$) is defined recursively by \[x^1=x \quad {\rm and} \quad x^{n+1}=x^n\with x \] for all positive integers $n$. For all $x$ and $n$, $x^{n+1}\leq x^n$, so $\{x^n\}_{n\in\mathbb{N}}$ is a decreasing sequence. By associativity of $\&$ we have \[x^{m+n}=x^m\with x^n=x^n\with x^m \quad {\rm and} \quad x^{mn}=(x^m)^n=(x^n)^m\] for all $m,n\geq 1$.
	
	\begin{defn}Suppose $\&$ is a  t-norm on $[0,1]$ and $x\in[0,1]$. We say that 
		\begin{enumerate}[label={\rm(\roman*)}]   
			\item $x$ is idempotent if $x\with x=x$;
			\item $x$ is nilpotent if $x^n=0$ for some $n\geq 1$.
	\end{enumerate}  \end{defn}
	
	The G\"{o}del t-norm and the product t-norm have only one nilpotent element. Every element in $[0,1)$ is nilpotent for the {\L}ukasiewicz t-norm and the drastic t-norm. For the nilpotent minimum t-norm, every $x\in[0,1/2]$ is   nilpotent.

	\begin{defn}A  t-norm $\&$ on $[0,1]$ is Archimedean  if for all $x,y\in(0,1)$, there is some $n$ such that $x^n<y$. \end{defn}
	
	Both  the product t-norm and the {\L}ukasiewicz t-norm are Archimedean.
	
	\begin{prop}Suppose $\&$ is an  Archimedean t-norm. If $\&$ has a nonzero nilpotent element, then all $x\in[0,1)$ are nilpotent.  \end{prop}
	
	\begin{proof}Suppose that $a>0$ is a nilpotent element of $\&$. Then $a^m=0$ for some $m\geq 1$. For each $x<1$, since $\&$ is Archimedean, there is some $n$ such that $x^n<a$. Then $x^{nm}\leq a^m=0$. \end{proof}

	A t-norm $\&$ on $[0,1]$ is \emph{left continuous on the first variable} if  
	for each $y$ and each increasing sequence $\{x_n\}_{n\in\mathbb{N}}$ in $[0,1]$,   \[\sup\limits_{n\in\mathbb{N}}(x_n\with y)=(\sup\limits_{n\in\mathbb{N}}x_n)\with y.\] 
	
	By commutativity of $\&$, left continuity on the first variable implies left continuity on the second variable.
	
	\begin{defn}\label{left continuous} Let $\&$ be  a  t-norm on $[0,1]$. We say that 
		\begin{enumerate}[label={\rm(\roman*)}]   
			\item $\&$ is   left continuous   if it is left continuous on each variable;
			\item    $\&$ is  continuous   if it is a continuous map $[0,1]\times[0,1]\lra[0,1]$ with respect to the usual topology. 
	\end{enumerate}   \end{defn}
	
	Said differently, a left continuous t-norm is a binary operator $\&$ on the unit interval such that  $([0,1],\&,1)$ is a commutative and unital qutantale. Every continuous t-norm is left continuous. The G\"{o}del t-norm, the product t-norm and the {\L}ukasiewicz t-norm are all continuous; the nilpotent minimum t-norm  is left continuous but not continuous;  the drastic product t-norm is not left continuous.
	
	Suppose $\&$ is a left continuous t-norm on $[0,1]$. For each $x\in[0,1]$, the map $$x\with-\colon [0,1]\lra[0,1]$$ has a right adjoint  $$x\ra-\colon [0,1]\lra[0,1].$$  For all $x,y,z\in[0,1]$ it holds that \[x\with y\leq z\iff y\leq x\ra z.\]  Because of this property, in fuzzy logic left continuous t-norms  and their right adjoints are   employed to model the logic connectives ``conjunction'' and   ``implication'', respectively. So, the binary operator \[\ra\colon [0,1]\times[0,1]\lra[0,1],\quad (x,y)\mapsto x\ra y\] is called the \emph{implication} of the t-norm $\&$.
	
	\begin{exmp}  
		\begin{enumerate}[label={\rm(\roman*)}] 
			\item The implication operator of the G\"{o}del t-norm: \[x\ra y= \begin{cases} 1  & x\leq y,\\ y & x>y.\end{cases}\]
			It is continuous on $[0,1]^2\setminus\{(x,x)\mid x<1\}$.
			
			\item The implication operator of the product t-norm: $$x\ra y= \begin{cases}  1  & x\leq y,\\ y/x  & x>y.\end{cases}$$ It is continuous except at $(0,0)$.
			
			\item The implication operator of the  {\L}ukasiewicz t-norm: $$x\ra y= \min\{1, 1-x+y\}. $$ It is continuous on $[0,1]^2$. 
			
			\item The implication operator of the nilpotent minimum t-norm is given by $$x\ra y= \begin{cases} 1  & x\leq y,\\ \max\{1-x,y\}  & x>y.\end{cases}$$ It is continuous on $[0,1]^2\setminus\{(x,x)\mid 0<x<1\}$. 
	\end{enumerate} \end{exmp}

	\begin{prop}{\rm(H\'{a}jek \cite{Ha98})} For each left continuous t-norm $\&$ on $[0,1]$, the following are equivalent: 
		\begin{enumerate}[label={\rm(\arabic*)}] 
			\item $\&$ is continuous.  
			\item $\&$ is divisible in the sense that if $x\leq y$, then $x=y\with z$ for some $z\in[0,1]$.  
			\item  For all $x,y\in[0,1]$, $x\with(x\ra y)=\min\{x,y\}$.
	\end{enumerate}    \end{prop}
	
	\begin{proof}$(1)\Rightarrow(2)$ Suppose that $x\leq y$. Since $y\with 0=0$ and $y\with 1=y$, by continuity of $\&$, the function $y\with-$ maps $[0,1]$ onto $[0,y]$,  then  there is some $z\in[0,1]$ such that $x=y\with z$.
		
		$(2)\Rightarrow(3)$  The equality is trivial if $x\leq y$. Assume that $x>y$. Since $\&$ is divisible, $y=x\with z$ for some $z\in[0,1]$. Then, by definition of $\ra$,    $z\leq x\ra y$, hence  $$y= x\with z\leq x\with(x\ra y)\leq y,$$ and consequently, $x\with(x\ra y)=\min\{x,y\}$.
		
		$(3)\Rightarrow(1)$ It suffices to show that $\&$ is separately continuous. Suppose on the contrary that there is some $y$ such that $y\with -\colon [0,1]\lra[0,1]$ is not continuous. Since $y\with-$ preserves joins, there must exist some $x$ such that \[y\with x<\inf_{z>x}(y\with z).\]  Pick some $b$ with $$y\with x<b<\inf_{z>x}(y\with z).$$  Then $b<y$ and   $y\with z\not=b$ for all $z\in[0,1]$,   contradicting $y\with(y\ra b)=b$.\end{proof}
	
	\begin{cor} \label{idempotent between} Suppose $\&$ is a continuous  t-norm on $[0,1]$ and   $b$ is an idempotent element. Then 
		\begin{enumerate}[label={\rm(\roman*)}] 
			\item  $y\with x=\min\{x,y\}=x$ whenever $x\leq b\leq y$. 
			\item $y\ra x=x$ whenever $x<b\leq y$. \end{enumerate} \end{cor}
	
	\begin{proof}(i) Since $b\with b=b$, then \[x=b\with(b\ra x)=b\with b\with(b\ra x)\leq b\with x\leq x, \] hence \[ x=b\with x\leq y\with x \leq x,\] so $x=y\with x$.
		
		(ii) Since $\&$ is continuous, then $y\with(y\ra x)= x<b$, so we must have $y\ra x<b$. Thus, by (i) we have  \begin{align*}x&= y\with(y\ra x)=y\ra x.\qedhere \end{align*} \end{proof}
	
	\begin{cor}\label{restriction of CT}
		Let $\&$ be a continuous t-norm on $[0,1]$.  Then for any idempotent elements  $a$ and $b$ with $a<b$, the restriction of $\&$ on $[a,b]$ is a continuous t-norm on the interval $[a,b]$. \end{cor}
	
	\vskip 5pt \noindent{\bf Continuous   Archimedean  t-norm}
	
	\begin{prop}A continuous  t-norm $\&$ is Archimedean if and only if it has no idempotent element other than $0$ and $1$. \end{prop}
	
	\begin{proof} Necessity is clear. For sufficiency, we show that $\lim\limits_{n\ra\infty}x^n=0$ for each $x<1$. Since   $\{x^n\}_{n\in\mathbb{N}}$ is a decreasing sequence,  the limit $\lim\limits_{n\ra\infty}x^n$ exists. Let $b=\lim\limits_{n\ra\infty}x^n$. From continuity of $\&$ one infers that $b\with b=b$, so $b$ is idempotent and $b=0$. \end{proof}
	
	\begin{prop}Suppose $\&$ is a continuous  Archimedean  t-norm on $[0,1]$. 
		\begin{enumerate}[label={\rm(\roman*)}]  
			\item If $x^n>0$ then $x^n>x^{n+1}$.
			\item If $x^n>0$ then $x^n>y^n$ whenever $x>y$.
	\end{enumerate} \end{prop}
	
	\begin{proof}(i) Suppose on the contrary that $x^n=x^n\with x$. By induction we have $x^n=x^n\with x^m$ for all $m\geq 1$. So, \[x^n=\lim\limits_{m\ra\infty}(x^n\with x^m)= x^n\with\lim\limits_{m\ra\infty}x^m = x^n\with 0=0, \]   contradicting that $x^n>0$.
		
		(ii) Suppose on the contrary that $x^n\leq y^n$. Since $\&$ is divisible, $y=x\with z$ for some $z<1$. Then \[x^n\leq y^n=y^{n-1}\with x\with z\leq x^n\with z. \] Consequently, \[x^n\leq x^n\with\lim\limits_{m\ra\infty}z^m = x^n\with 0=0,\]   contradicting that $x^n>0$. \end{proof}
	
	Suppose $t\colon [0,1]\lra[-\infty,0]$ is  continuous and strictly increasing and $t(1)=0$. Since $t$ preserves meets, it has a left adjoint $$t^\dashv\colon  [-\infty,0]\lra [0,1] , \quad  t^\dashv (x) =\begin{cases}  t^{-1}(x)  & x\in[t(0),0], \\
		0  & x< t(0).\end{cases}$$ 
	
	\begin{prop} The left adjoint $t^\dashv \colon  [-\infty,0]\lra [0,1] $ satisfies:   
		\begin{enumerate}[label={\rm(\roman*)}] 
			\item $t^\dashv \circ t(x)=x$ for all $x\in[0,1]$. \item $t\circ t^\dashv (x)=x$ for all $x\in[t(0),0]$. 
			\item $t^\dashv $ is continuous and is strictly increasing on $[t(0),0]$. 
	\end{enumerate} \end{prop}
	
	\begin{prop}\label{generation of Archimedean t-norm} Suppose $t\colon [0,1]\lra[-\infty,0]$ is continuous and strictly increasing and $t(1)=0$. Define a binary operator $\&$ on $[0,1]$ by \begin{equation*} \label{formula 1} x\with y=t^\dashv (t(x)+t(y)),\end{equation*} where $t^\dashv $ is the left adjoint of $t$. Then 
		\begin{enumerate}[label={\rm(\roman*)}] 
			\item $\&$ is a continuous Archimedean t-norm.
			\item $\&$ has a nonzero nilpotent element if and only if $t(0)>-\infty$.
			\item $\&$ is isomorphic to the truncated addition on $[t(0),0]$.
		\end{enumerate}
		The function $t\colon [0,1]\lra[-\infty,0]$ is called an additive generator of the t-norm $\&$. \end{prop}
	
	\begin{proof} (i) The verification of that $\&$ is a continuous t-norm is routine. To see that $\&$ is Archimedean, suppose   $0<x<y<1$. Since $t$ is strictly increasing, then $0>t(y)>t(x)>-\infty$. Pick some $n\geq 1$ with $nt(y)< t(x)$, then \[y^n=t^\dashv (nt(y))< t^\dashv (t(x))=x.\]
		
		(ii) Suppose   $t(0)>-\infty$. For each $x<1$, since $t(x)<0$, there exists $n\geq 1$ such that $nt(x)\leq t(0)$, then $x^n=t^\dashv (nt(x))=0$, hence every $x<1$ is a nilpotent element of $\&$. Now suppose  $t(0)=-\infty$. For each $x\in(0,1)$, since $t(x)>-\infty$ and $t^\dashv $ is strictly increasing on $[-\infty,0]$, it follows that for all $n\geq 1$, \[x^n=t^\dashv (nt(x))>t^\dashv (-\infty)=0,\] so $\&$ has no nilpotent element other than $0$.

		(iii) By definition. \end{proof}
	
	\begin{exmp}\begin{enumerate}[label={\rm(\roman*)}] 
			\item    The  function $$t\colon [0,1]\lra[-\infty,0],\quad t(x)=\ln x$$   is an additive generator of the product t-norm.
			
			\item  The  function $$t\colon [0,1]\lra[-\infty,0],\quad t(x)=x-1$$   is an additive generator of the {\L}ukasiewicz t-norm.
			
			\item For each $p>0$,  the function $$t\colon [0,1]\lra[-\infty,0],\quad t(x)= x^{p}-1$$ is continuous and strictly increasing and  $t(1)=0$.   The  t-norm $\with_p$ with $t$ as  additive generator is given by $$x\with_p y= (\max\{x^p+y^p-1,0\})^{\frac{1}{p}}.$$ The G\"{o}del t-norm is the limit of $\with_p$ when $p$ tends to $0$; the drastic product is the limit of $\with_p$ when $p$ tends to infinity. 
	\end{enumerate}\end{exmp}
	
	\begin{thm}\label{representation} Every continuous  Archimedean t-norm $\&$  on $[0,1]$ has an additive generator. That means, there exists a continuous and strictly increasing function $t\colon [0,1]\lra[-\infty,0]$ with $t(1)=0$   such that for all $x,y\in[0,1]$, $$x\with y =t^\dashv (t(x)+t(y)),$$ where $t^\dashv $ is the left adjoint of $t$. \end{thm}
	
	We make some preparations first. Under the assumption of Theorem \ref{representation}, define for each natural number $n\geq 1$ a function  \[f_n\colon [0,1]\lra[0,1]\] by $$f_n(x)=x^n.$$ Then $f_n(0)=0$, $f_n(1)=1$,   $f_n$ is continuous and is strictly increasing whenever it is positive. In particular,  $f_n\colon [0,1]\lra[0,1]$ preserves joins, so it has a right adjoint \[g_n\colon [0,1]\lra[0,1],\] which sends each $x>0$ to the  unique $y$ with $y^n=x$.
	
	\begin{lem}\label{f_n,g_n} For each $n\geq1$, let $$a_n=g_n(0)=\sup\{x\in[0,1]\mid f_n(x)=0\}.$$  Then for all $n,m\geq 1$, we have: 
		\begin{enumerate}[label={\rm(\roman*)}] 
			\item $f_n(a_n)=0$ and $f_n$ is continuous and strictly increasing on $[a_n,1]$.
			\item $g_n(0)=a_n$ and $g_n$ is continuous and strictly increasing on $[0,1]$.
			\item For all $x\in[0,1]$, $f_n\circ g_n(x)=x$; for all $x\in[a_n,1]$, $g_n\circ f_n(x)=x$.
			\item $f_n(x)\geq f_{n+1}(x)$ and $g_n(x)\geq g_{n+1}(x)$.
			\item $f_n\circ f_m(x)=f_{nm}(x)=f_m\circ f_n(x)$.
			\item $g_n\circ g_m(x)=g_{nm}(x)=g_m\circ g_n(x)$.
			\item For each natural number $k\geq 1$, $f_{kn}\circ g_{km}=f_n\circ g_m$
	\end{enumerate} \end{lem}
	
	\begin{proof} (i)-(v)   are left to the reader. (vi) follows from the fact that all of the functions  $g_n\circ g_m$, $g_{nm}$ and $g_m\circ g_n$ are right adjoint to $f_{nm}$. As for (vii),
		since $f_{kn}\circ g_{km}=f_n\circ f_k\circ g_k\circ g_m$ by (v) and (vi), the conclusion follows from (iii) immediately. \end{proof}
	
	For each rational number $r=-n/m<0$, define $$f_r\colon [0,1]\lra [0,1]$$ by \[f_r=f_n\circ g_m.\]
	Lemma \ref{f_n,g_n}\thinspace(vii) ensures that $f_r$ is well-defined.
	
	\begin{lem}\label{def of h}  Fix some $c\in(0,1)$. Define a function \[ h\colon \mathbb{Q}\cap(-\infty,0]\lra[0,1]\]  by $$h(0)=1,\quad \text{$h(r)=f_r(c)$ for all $r<0$}.$$ Then
		\begin{enumerate}[label={\rm(\roman*)}] 
			\item If $r<s$ and $h(s)>0$, then $h(r)<h(s)$.
			\item For all $r,s\in (-\infty,0]$, $h(r+s)=h(r)\with h(s)$.
			\item $\lim\limits_{r\ra-\infty}h(r)=0$.
			\item $h$ is uniformly continuous.
	\end{enumerate}\end{lem}
	
	\begin{proof}(i) Assume that $r=-m/k, s=-n/k$ and $m>n$. Then \[h(s) =f_n\circ g_k(c)=g_k(c)^n >g_k(c)^{n+1}\geq g_k(c)^m=f_m\circ g_k(c) =h(r).\]
		
		(ii) Assume that $r=-m/k, s=-n/k$. Then \[h(r)\with h(s)=(f_m\circ g_k(c))\with  (f_n\circ g_k(c)) =g_k(c)^{m+n}= f_{m+n}\circ g_k(c)=h(r+s).\]
		
		(iii) This follows from (i) and that $\lim\limits_{n\ra\infty}h(-n)= \lim\limits_{n\ra\infty}f_n(c) =\lim\limits_{n\ra\infty}c^n=0$.
		
		(iv) It suffices to show that for each $\epsilon>0$, there is a rational number $l>0$ such that for all  $r\in\mathbb{Q}\cap(-\infty,0]$, \[ h(r+l)-\epsilon\leq h(r)\leq h(r-l)+\epsilon,\] where we agree  that $h(r+l)=1$ if $r+l>0$.
		
		Since $\&$ is continuous and $[0,1]^2$ is compact, there exists   $\delta>0$ such that \[|x\with y-u\with v|<\epsilon\] whenever $|x-u|<\delta$ and $|y-v|<\delta$. Since $g_n(c)$ tends to $1$ when $n$ tends to infinity, there is a natural number $N$ such that $1-g_N(c)<\delta$. Let   $l=1/N$. Then  $$1-h(-l)= 1-f_1\circ g_N(c)= 1-g_N(c)<\delta.$$ We claim that $l$ satisfies the requirement.
		
		Since \[h(r)-h(r-l)=|h(r)\with 1-h(r)\with h(-l)|\leq \epsilon,\] it follows that $h(r)\leq h(r-l)+\epsilon$.
		As for the inequality $h(r+l)-\epsilon\leq h(r)$, we proceed with two cases. If $r+l\leq0$, then \[h(r+l)-h(r)= |h(r+l)\with 1-h(r+l)\with h(-l)|\leq \epsilon.\] If $r+l>0$, then \[h(r+l)-h(r)= 1- h(r) \leq 1-h(-l)\leq \epsilon.\] In either case we have $h(r+l)-\epsilon\leq h(r)$.
	\end{proof}

	\begin{proof}[Proof of Theorem \ref{representation}]
		Consider the function $$h\colon \mathbb{Q}\cap(-\infty,0]\lra[0,1]$$ given in Lemma \ref{def of h}. Since $h$ is uniformly continuous and $\lim\limits_{r\ra-\infty}h(r)=0$, then $h$  can be extended to a continuous function $$\overline{h}\colon   [-\infty,0]\lra[0,1].$$  By continuity of $\&$ and Lemma \ref{def of h}\thinspace(ii), for all $u,v\in[-\infty,0]$  we have \[\overline{h}(u+v)=\overline{h}(u)\with\overline{h}(v). \]  
		Moreover, if $u<v$ and $\overline{h}(v)>0$, then $\overline{h}(u)<\overline{h}(v)$ by Lemma \ref{def of h}\thinspace(i).
		
		Since $\overline{h}\colon[-\infty,0] \lra[0,1]$ preserves joins, it has a right adjoint  $$t\colon [0,1]\lra[-\infty,0].$$  We claim that $t$ satisfies the requirement. 
		
		By definition $\overline{h}$ is the left adjoint of $t$, i.e. $t^\dashv =\overline{h}$.  It is readily verified that $t(1)=0$ and that $t$ is continuous and strictly increasing, in particular, $t$ is injective. Then by Proposition \ref{surjective right adjoint}, $t^\dashv (t(x))=x$ for all $x\in [0,1]$. Therefore,  for all $x,y\in[0,1]$ we have \begin{align*}x\with y&=t^\dashv (t(x))\with t^\dashv (t(y)) \\ 
			&=\overline{h}(t(x))\with\overline{h}(t(y))\\ &=\overline{h}(t(x)+t(y))\\ 
			&=t^\dashv (t(x)+t(y)),\end{align*}  which completes the proof. \end{proof}
	
	An immediate consequence of Theorem \ref{representation} is the following conclusion which asserts that there are essentially only two  continuous Archimedean t-norms: the product  and the {\L}ukasiewicz.
	
	\begin{cor}\label{CA=L or P} Suppose $\&$ is a continuous Archimedean t-norm  on $[0,1]$.  \begin{enumerate}[label={\rm(\roman*)}] 
			\item If $\&$ has a nonzero nilpotent element, then $\&$ is isomorphic to the {\L}ukasiewicz t-norm.
			\item If $\&$ has no nonzero nilpotent element, then $\&$ is isomorphic to the product t-norm.
	\end{enumerate}\end{cor}
	
	\begin{proof} 
		By Theorem \ref{representation} there is a continuous and strictly increasing function $t\colon [0,1]\lra[-\infty,0]$ with $t(1)=0$ such that $$x\with y=t^\dashv (t(x)+t(y))$$ for all $x,y\in[0,1]$.
		
		(i) If $\&$ has a nonzero nilpotent element, we must have $t(0)>-\infty$. Then $\&$ is isomorphic to the truncated addition on $[t(0),0]$, hence to the {\L}ukasiewicz t-norm on $[0,1]$.
		
		(ii) If $\&$ has no nonzero nilpotent element, we must have $t(0)=-\infty$. Then $\&$ is isomorphic to the addition on $[-\infty,0]$, hence to the product t-norm on $[0,1]$. \end{proof}

	Suppose $\&$ is a left continuous t-norm on $[0,1]$. We say that $([0,1],\&,1)$ (or simply $\&$) satisfies the \emph{law of double negation} if \[(x\ra0)\ra0 =x\] for all $x\in [0,1]$.
	
	Both the {\L}ukasiewicz t-norm and the nilpotent minimum satisfy the law of double negation.
	
	\begin{prop}\label{MV=Luka} Let $\&$ be a continuous t-norm on $[0,1]$.  Then $\&$ satisfies the law of double negation if and only if $\&$ is isomorphic to the {\L}ukasiewicz t-norm.  \end{prop}
	
	\begin{proof}First we show that $\&$ is Archimedean. Otherwise it has a nontrivial idempotent element, say $b$. Then by Corollary \ref{idempotent between}, $b\ra0=0$, hence $(b\ra0)\ra0 =1$, contradicting that $\&$ satisfies the law of double negation. Next, since the product does not satisfy the law of double negation, then by Corollary \ref{CA=L or P}, $\&$ must be isomorphic to  the {\L}ukasiewicz t-norm. \end{proof}
	
	\vskip 5pt \noindent{\bf Ordinal sum decomposition theorem}
	
	Suppose   $\{(a_i,b_i)\}_{i\in I}$ is a countable set of pairwise disjoint open intervals in $[0,1]$ and  for each  $i\in I$,   $\&_{i}$ is a t-norm on the closed interval $[a_{i},b_{i}]$.  It is readily verified that  $$\with \colon [0,1]\times[0,1]\lra[0,1],\quad  x\with y=\begin{cases}
		x\with_i y  &  (x,y)\in [a_{i},b_{i}]^{2},     \\
		\min\{x,y\}  &   {\rm otherwise}
	\end{cases}$$  is  a t-norm on $[0,1]$, called the \emph{ordinal sum} of $\{\&_i\}_{i\in I}$. Each $\with_i$ is called a \emph{summand}  of $\&$. The ordinal sum is (left) continuous if and only if so is each of its summands. In the case that the ordinal sum $\&$ is left continuous, its implication operator is given by \begin{equation*}
		x\ra y=\begin{cases} 1 & x\leq y, \\ x\ra_i y  & a_i\leq y<  x\leq b_i~\text{for some }i,\\ y &{\rm otherwise}, \end{cases} \end{equation*} where $\ra_i$ is the implication operator of the t-norm $\&_i$ on $[a_i,b_i]$.
	
	A fundamental result says that every continuous t-norm is an ordinal sum of continuous Archimedean t-norms. To see this, let $\&$ be a continuous t-norm on $[0,1]$ and let $E$ be the set of idempotent elements of $\&$. Then $E$ is a closed set by continuity of $\&$. The open set $[0,1]\setminus E$ is a countable union of disjoint open intervals $\bigcup_{i\in I}(a_i,b_i)$. By Corollary \ref{restriction of CT},   the restriction of $\&$ on each $[a_i,b_i]$ is a continuous t-norm without nontrivial idempotent elements, hence a continuous Archimedean t-norm. By Corollary \ref{idempotent between}, $x\with y=\min\{x,y\}$ whenever $(x,y)\notin\bigcup_{i\in I}[a_i,b_i]^2$. This proves the following \emph{ordinal sum decomposition theorem} for continuous t-norms.
	
	\begin{thm}\label{ordinal sum decomposition} {\rm (Mostert and Shields \cite[Theorem B]{Mostert1957})} Every continuous  t-norm $\&$ on $[0,1]$ is an ordinal sum of continuous Archimedean t-norms. This means, for each continuous t-norm $\&$ on $[0,1]$, there exists a countable set of disjoint open intervals $\{(a_i,b_i)\}_{i\in I}$ in $ [0,1]$ such that
		\begin{enumerate}[label={\rm(\roman*)}] 
			\item  for each $i\in I$, both $ a_i$ and $ b_i$ are idempotent and the restriction of $\&$ on the closed interval $ [a_i,b_i] $ is Archimedean; 
			\item  $x\with y= \min\{x,y\}$ whenever $(x,y)\notin\bigcup_{i\in I}[a_i,b_i]^2$.
		\end{enumerate}
		
		Each $[a_i,b_i]$  with the restriction of $\&$ is called an Archimedean block of $\&$. \end{thm}
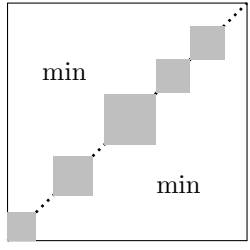
	\begin{center}\begin{tikzpicture}[scale=0.75]  
			\draw (0,0) rectangle (4.2,4.2);
			\draw[line width=1pt,dotted] (0,0) -- (4.2,4.2); 
			\fill [lightgray] (3.2,3.2) rectangle (3.8,3.8); 
			\fill [lightgray] (0,0) rectangle (0.5,0.5); 
			\fill [lightgray] (0.8,0.8) rectangle (1.5,1.5); 
			\fill [lightgray] (1.7,1.7) rectangle (2.6,2.6); 
			\fill [lightgray] (2.61,2.61) rectangle (3.2,3.2); 
			\node at (3,1) {$\min$};  \node at (1,3) {$\min$};
	\end{tikzpicture}\end{center} 
	
	\begin{prop}\label{Luka-free} Suppose $\&$ is a continuous t-norm on $[0,1]$. For each $p\in[0,1]$, let $p^-$ be the greatest idempotent element in $[0,p]$ and let $p^+$ be the least idempotent element in $[p,1]$. Then the following  are equivalent:
		\begin{enumerate}[label=\rm(\arabic*)] 
			\item For each non-idempotent element $p\in[0,1]$, the restriction of $\&$ on $[p^-,p^+]$ is isomorphic to the product t-norm on $[0,1]$ whenever $p^->0$.
			\item The implication  operator $\ra\colon [0,1]^2\lra[0,1]$ is  continuous  at every point off the diagonal $\{(x,x)\mid x\in[0,1]\}$.
			\item For each $p\in(0,1]$, the function $p\ra -\colon[0,1]\lra[0,1]$ is continuous  on the interval $[0,p)$.
		\end{enumerate}
	\end{prop}
	
	\begin{proof} $(1)\Rightarrow(2)$  It suffices to show that $\ra$ is continuous at each point $(x,y)$ with $x>y$. If there is some idempotent element $b$ satisfying $y<b<x$, then $(b,1]\times[0,b)$ is a neighborhood of $(x,y)$  such that $x'\ra y' =y'$ for all $(x',y')\in (b,1]\times[0,b)$, hence $\ra$ is continuous at $(x,y)$. If there is no idempotent element between $x$ and $y$, then $x$ and $y$ belong to the same block by the ordinal sum decomposition theorem. That means, there exists a non-idempotent $p$ such that $p^-\leq y<x\leq p^+$ and that the restriction of $\&$ on $[p^-,p^+]$ is a continuous Archimedean t-norm.  Now we proceed with two cases.
		
		Case 1. $p^-=0$. Then the conclusion follows from the fact that the implication operator of a continuous Archimedean t-norm on $[0,1]$ is continuous except possibly at $(0,0)$.
		
		Case 2. $p^-\not=0$. In this case, the restriction of $\&$ on $[p^-,p^+]$ is isomorphic to the product t-norm on $[0,1]$. 
		We only need to show that for all $x>p^-$, the function $x\ra-\colon[0,1]\lra[0,1]$ is continuous at $p^-$. Since $ x\ra z=z$ for each $z<p^-$, the right limit of $ x\ra -$ at $p^-$ is equal to $p^-$, then $\lim\limits_{z\ra p^-}(x\ra z)= p^-= x\ra p^-$,  hence $x\ra-$ is continuous at $p^-$.

		$(2)\Rightarrow(3)$ Obvious.
		
		$(3)\Rightarrow(1)$ Suppose on the contrary that there is some non-idempotent element $p$ for which $p^->0$ and the restriction of $\&$ on $[p^-,p^+]$ is isomorphic to the {\L}ukasiewicz t-norm. Then the function $p\ra-\colon[0,1]\lra[0,1]$ is not continuous at $p^-$, a contradiction.\end{proof}

	Thus, for a continuous t-norm $\&$, the implication  operator $$\ra\colon [0,1]^2\lra[0,1]$$ is  continuous  at every point off the diagonal $\{(x,x)\mid x\in[0,1]\}$ if and only if it is \emph{almost \L ukasiewicz free} in the sense that $\&$ has at most one Archimedean block that is isomorphic to the  \L ukasiewicz t-norm and that block  should start at $0$. Such t-norms are first studied by Morsi  \cite{Morsi1995a,Morsi1995b}.
	

	\begin{con}From now on, we always assume that $\with$ is a continuous t-norm on the interval $[0,1]$. \end{con}

	\section {Real-enriched categories} \label{real-enriched cats} 
	
	Suppose $X,Y$ are sets. A \emph{$[0,1]$-relation}   $r\colon X\rto Y$ is a map $r\colon X\times Y\lra [0,1]$, the value $r(x,y)$ is interpreted as the degree that $x$ is  related to $y$. For a $[0,1]$-relation $r\colon X\rto Y$, the $[0,1]$-relation $$r^{\rm op}\colon Y\rto X,\quad r^{\rm op}(x,y)=r(y,x) $$ is called the  opposite of $r$.
	
	$[0,1]$-relations can be composed.
	Given $[0,1]$-relations $r\colon X\rto Y$  and $s\colon Y\rto Z$, the composite $s\circ r$ is the $[0,1]$-relation $X\rto Z$   given by  \[s\circ r(x,z)=\sup_{y\in Y}s(y,z)\with r(x,y).\]
	
	For each set $X$, write $\id_X\colon X\rto X$ for the $[0,1]$-relation given by \[\id_X(x,y)=\begin{cases} 1 & x=y,\\
		0 & x\not=y.\end{cases}\] Then $\id_X$  is a unit  for the composition of $[0,1]$-relations; that is, $r\circ \id_X=r=\id_Y\circ r$ for each $[0,1]$-relation $r\colon X\rto Y$. So, we call $\id_X$ the identity relation on $X$.
	
	The composition of $[0,1]$-relations is associative. 
	Therefore, with sets as objects and $[0,1]$-relations as arrows we obtain a category \[[0,1]\text{-}{\sf Rel}.\] The category $[0,1]\text{-}{\sf Rel}$ has an important property: for all objects $X,Y$,  the  hom-set $$[0,1]\text{-}{\sf Rel}(X,Y)$$ is a complete lattice under the pointwise order,  the composition of $[0,1]$-relations preserves joins in both variables.
	
	\begin{prop}For all sets $X,Y,Z$, the composition \[\circ\colon [0,1]\text{-}{\sf Rel}(Y,Z)\times[0,1]\text{-}{\sf Rel}(X,Y)\lra [0,1]\text{-}{\sf Rel}(X,Z)\] preserves joins on each variable.   
	\end{prop}
	
	Therefore, $[0,1]\text{-}{\sf Rel}$ is a quantaloid in the sense of \cite{Rosenthal1996,Stubbe2005}.
	This fact has far-reaching consequences, it plays a crucial role in the calculus of $[0,1]$-relations, as we see now.
	
	For each $[0,1]$-relation   $r\colon X\rto Y$, since $\circ$ preserves joins in each variable, the map \[-\circ r\colon [0,1]\text{-}{\sf Rel}(Y,Z)\lra [0,1]\text{-}{\sf Rel}(X,Z)\] has a right adjoint \[-\swarrow r\colon [0,1]\text{-}{\sf Rel}(X,Z)\lra[0,1]\text{-}{\sf Rel}(Y,Z).\] Explicitly, for each   $t\colon X\rto Z$, the $[0,1]$-relation $t\swarrow r\colon  Y\rto Z$ is computed as follows:  \[(t\swarrow r)(y,z)=\inf_{x\in X}(r(x,y)\ra t(x,z)).\]
	
	$$\bfig \ptriangle(0,0)/>`>`>/<520,500>[X`Y`Z;r`t`t\swarrow r]
	\place(260,500)[\rotatebox{0}{$\mapstochar$}] \place(260,255)[\rotatebox{45}{$\mapstochar$}] \place(5,255)[\rotatebox{90}{$\mapstochar$}] \place(150,350)[\geq]
	\efig$$
	
	Likewise, for each $[0,1]$-relation   $s\colon Y\rto Z$, the map \[s\circ -\colon [0,1]\text{-}{\sf Rel}(X,Y)\lra [0,1]\text{-}{\sf Rel}(X,Z)\] has a right adjoint \[s\searrow -\colon [0,1]\text{-}{\sf Rel}(X,Z)\lra[0,1]\text{-}{\sf Rel}(X,Y) \] given by    \[(s\searrow t)(x,y)=\inf_{z\in Z}(s(y,z)\ra t(x,z)) \] for all   $t\colon X\rto Z$ and $x\in X$ and $y\in Y$.
	$$\bfig \qtriangle(0,0)/>`<-`<-/<520,500>[Y`Z`X;s`s\searrow t`t]
	\place(260,500)[\rotatebox{0}{$\mapstochar$}] \place(255,255)[\rotatebox{-45}{$\mapstochar$}] \place(525,255)[\rotatebox{90}{$\mapstochar$}] \place(350,350)[\leq]
	\efig$$

	Some basic properties of the $[0,1]$-relation calculus are listed below:
	\begin{enumerate}[label={\rm(\roman*)}]  
		\item For all $[0,1]$-relations $r\colon X\rto Y$, $s\colon Y\rto Z$ and $t\colon X\rto Z$, \[s\leq t\swarrow r\iff s\circ r\leq t\iff r\leq s\searrow t. \]
		
		\item  For all   $r\colon X\rto Y$ and $s\colon X\rto Y$, \[\id_X\leq r\searrow s\iff r\leq s\iff \id_Y\leq s\swarrow r.\] 
		
		\item  For all $r\colon X\rto Y$, $s\colon Y\rto Z$ and $t\colon X\rto W$, \[t\swarrow (s\circ r)= (t\swarrow r)\swarrow s.\] 
		
		\item For all $r\colon X\rto Y$, $s\colon Y\rto Z$ and $t\colon W\rto Z$, \[(s\circ r)\searrow t= r\searrow(s\searrow t).\]  
		
		\item For all $r\colon X\rto Y$, $s\colon W\rto Z$ and $t\colon X\rto Z$, \[(s\searrow t)\swarrow r= s\searrow(t\swarrow r).\] 
	\end{enumerate} 
	
	\begin{defn}
		A real-enriched category is a pair $(X, \alpha)$, where $X$ is a set and   $\alpha\colon X\times X\lra[0,1]$ is a function  such that \begin{enumerate}[label=\rm(\roman*)]  \item   $\alpha(x,x)= 1$ for all $x\in X$;     \item   $ \alpha(y,z)\with \alpha(x,y)\leq \alpha(x,z)$ for all $x,y,z\in X$.  
	\end{enumerate}  \end{defn}
	
	If we interpret  the value $\alpha(x,y)$  as the truth degree that $x$ is smaller than or equal to $y$, then  (i) is reflexivity and (ii) is transitivity. So, real-enriched categories can be viewed as many-valued ordered sets.
	
	To ease notations, for a real-enriched category $(X, \alpha)$,   we often omit the symbol $\alpha$ and write $X(x,y)$ for $\alpha(x,y)$.

	Two elements $x$ and $y$ of a real-enriched category $X$ are \emph{isomorphic} if $X(x,y)=1=X(y,x)$. A real-enriched category $X$ is   \emph{separated} if its isomorphic elements are identical.
	
	\begin{exmp}
		\begin{enumerate}[label=\rm(\roman*)]  
			\item The pair $([0,1],\alpha_L)$ is a separated real-enriched category, where   for all $x,y\in [0,1]$, $\alpha_L (x,y)= x\ra y$.  
			The opposite  of $([0,1],\alpha_L)$ is  denoted by $([0,1],\alpha_R)$; that means, $\alpha_R (x,y)=y\ra x.$ In the language of enriched category theory,  $\alpha_L$ is the \emph{internal hom} of (the truth-value set) $([0,1],\with,1)$. In the sequel we  write $\sV$ for   $([0,1],\alpha_L)$, hence $\sV^{\rm op}$ for $([0,1],\alpha_R)$. Both $\sV$ and $\sV^{\rm op}$ play  an important role in the theory of real-enriched categories. 
			
			\item For each set $X$, $([0,1]^X,\sub_X)$ is a separated real-enriched category, where   for all $\lam,\mu\in [0,1]^X$,
			$$\sub_X(\lam,\mu) =\inf_{x\in X}\lam(x)\ra \mu(x). $$  
			If we, following Zadeh \cite{Zadeh65}, view $\lam$ and $\mu$ as fuzzy subsets of $X$,  then the value $\sub_X(\lam,\mu)$ measures the truth degree that $\lam$ is contained in $\mu$, so the category $([0,1]^X,\sub_X)$ is also known as the \emph{enriched powerset} of $X$.  
		\end{enumerate}   
	\end{exmp}  
	
	Suppose $X, Y$ are real-enriched categories. A functor   $f\colon X\lra Y$  is a map such that  $X(x,y)\leq Y(f(x),f(y))$ for all $x,y\in X$. The category of real-enriched categories and functors is denoted by \[[0,1]\text{-}{\bf Cat}.\]  
	
	Given an ordered set $(X,\leq)$, define a map $\omega(\leq)\colon X\times X\lra[0,1]$ by $\omega(\leq)(x,y)=1$ if $x\leq y$, and $\omega(\leq)(x,y)=0$ otherwise. Then $(X,\omega(\leq))$ is a real-enriched category. In this way we obtain a full and faithful  functor $$\omega\colon{\bf Ord}\lra\QOrd.$$
	
	The functor $\omega$ has both a left and a right adjoint. We spell out the right adjoint here. For each real-enriched category $X$, define a binary relation $\sqsubseteq$ on $X$ by $x\sqsubseteq y$ if $X(x,y)=1.$ Then $\sqsubseteq$ is  a reflexive and transitive relation, hence an order   on $X$, called the \emph{underlying order} of  $X$. We write $X_0$ for the ordered set $(X,\sqsubseteq)$. The assignment $X\mapsto X_0$ defines a functor $$(\text{-})_0\colon\QOrd\lra {\bf Ord},$$ a right adjoint of $\omega$.
	
	\begin{defn}Suppose $X,Y$ are real-enriched categories, $f\colon X\lra Y$ and $g\colon Y\lra X$ are functors.  We say that $f$ is left adjoint to $g$, or $g$ is right adjoint to $f$, and write $f\dashv g$, if
		for all $x\in X$ and $y\in Y$, 
		$$Y(f(x),y)=X(x,g(y)).$$ The pair $(f,g)$ is  called an adjunction, or a(n enriched) Galois connection.\end{defn} This is an extension of Galois connections between ordered sets; and a special case of adjunction in the theory of enriched categories.

	\begin{thm}\label{Characterization of adjoints} For each pair of maps $f\colon X\lra Y$ and $g\colon Y\lra X$   between real-enriched categories, the following are equivalent:  
		\begin{enumerate}[label={\rm(\arabic*)}]  
			\item Both  $f$ and $g$ are  functors, and $f$ is left adjoint to $g$.  
			\item Both $f$ and $g$ are  functors, and 
			$f\colon X_0\lra Y_0$ is left adjoint to $g\colon Y_0\lra X_0$.  
			\item For all $x\in X$ and $y\in Y$, $Y(f(x),y)=X(x,g(y))$. 
	\end{enumerate}\end{thm}
	
	\begin{proof} $(1)\Rightarrow(2)$ Obvious. 
		
		$(2)\Rightarrow(3)$  
		Since $f(x)\sqsubseteq f(x)$, then $x\sqsubseteq gf(x)$, hence $1\leq X(x,gf(x))$ for all $x\in X$. Likewise, $1\leq Y(fg(y),y)$ for all $y\in Y$. So  we have 
		\[X(x,g(y))\leq   Y(fg(y),y)\with Y(f(x),fg(y))\leq
		Y(f(x),y)\] and \[Y(f(x),y)\leq  
		X(gf(x),g(y))\with X(x,gf(x))\leq X(x,g(y)).\] Therefore 
		$X(x,g(y))=Y(f(x),y)$. 
		
		$(3)\Rightarrow(1)$ It suffices to check that  $f$ and $g$ are  functors. For  all $x_1,x_2\in X$, \begin{align*}X(x_1,x_2)&\leq Y(f(x_2),f(x_2))\with X(x_1,x_2) \\ &= X(x_2,gf(x_2))\with X(x_1,x_2)\\ &\leq X(x_1,gf(x_2))  \\ &=Y(f(x_1),f(x_2)),\end{align*} it follows that $f$ is a  functor. Likewise, $g$  is a  functor. \end{proof}
	
	\begin{exmp}\label{image vs preimage} Suppose $X,Y$ are set,  $f\colon X\lra Y$ is a map. Then,   $$f^{-1} \colon ([0,1]^Y,\sub_Y)\lra ([0,1]^X,\sub_X), \quad  f^{-1}(\mu)=\mu\circ f$$ is a functor, it has both a left adjoint and a right adjoint. The left adjoint $$f_\exists\colon ([0,1]^X,\sub_X)\lra ([0,1]^Y,\sub_Y) $$ is given by  $$ f_\exists(\lam)(y)=\sup\{\lam(x)\mid f(x)=y\};$$ the right  adjoint $$f_\forall\colon ([0,1]^X,\sub_X)\lra ([0,1]^Y,\sub_Y) $$ is given by  $$ f_\forall(\lam)(y)=\inf\{\lam(x)\mid f(x)=y\}.$$ Therefore, we have a string of adjunctions   $$ f_\exists\dashv f^{-1}\dashv f_\forall.$$ \end{exmp}   
	
	\section{The Yoneda lemma}
	
	\begin{defn}
		Suppose $X, Y$ are real-enriched categories. A distributor $\phi$ from $X$ to $Y$, written $\phi\colon X\oto Y$, is a map $\phi\colon X\times Y\lra [0,1]$ such that 
		\begin{enumerate}[label={\rm(\roman*)}] 
			\item for all $x_1,x_2\in X$ and $y\in Y$, $\phi(x_2,y)\with X(x_1,x_2)\leq \phi(x_1,y)$;  \item for all $x\in X$ and $y_1,y_2\in Y$,   $Y(y_1,y_2)\with\phi(x,y_1)\leq \phi(x,y_2)$.  \end{enumerate}      \end{defn} 
	
	If   $\phi$ is a distributor from $X$ to $Y$, then  $\phi^{\rm op}(y,x)\coloneqq\phi(x,y)$ is a distributor from   $Y^{\rm op}$ to  $X^{\rm op}$.  As we have agreed, for each real-enriched category $X$,  the symbol ``$X$'' also denotes  the real-enriched category structure  on $X$. So,  a distributor $\phi\colon X\oto Y$ is  a $[0,1]$-relation $\phi\colon X\rto Y$ such that $\phi\circ X\leq \phi$ and $Y\circ\phi\leq\phi$, or equivalently, $\phi\circ X=\phi=Y\circ\phi$. 
	
	Let $\phi\colon X\oto Y$ be a distributor. Fixing the first argument $x$ gives a functor $$\phi(x,-)\colon Y\lra \sV, \quad y\mapsto \phi(x,y),$$   where $\sV=([0,1],\alpha_L)$; fixing the second argument $y$ gives a functor $$\phi(-,y)\colon X^{\rm op}\lra \sV, \quad x\mapsto\phi(x,y).$$  
	
	\begin{exmp}
		Every functor $f\colon X\lra Y$ induces a pair of distributors: \[f_*\colon X\oto Y, \quad f_*(x,y)=Y(f(x),y)\] and \[f^*\colon Y\oto X, \quad f^*(y,x)=Y(y,f(x)).\] The  distributors $f_* $  and $f^* $ are called,  respectively, the  graph  and the  cograph  of $f$.  \end{exmp}
	
	A  functor $f\colon X\lra Y$ is   \emph{fully faithful} if  for all $x_1,x_2\in X$, \[X(x_1,x_2)=Y(f(x_1),f(x_2)).\]  We leave it to the reader to check that a functor  $f\colon X\lra Y$ is fully faithful   if and only if $f^*\circ f_*=X$.

	For every pair of real-enriched categories $X$ and $Y$, the set of distributors from $X$ to $Y$  is a subset of the   $[0,1]$-relations from $X$ to $Y$. The following proposition says that this subset is closed with respect to various operations on $[0,1]$-relations.
	
	\begin{prop}\label{Distcalcu} Suppose $X,Y,Z$ are real-enriched categories. 
		\begin{enumerate}[label={\rm(\roman*)}] 
			\item For any distributors $\phi\colon X\oto Y$ and $\psi\colon Y\oto Z$, the composite $\psi\circ\phi$ is a distributor $X\oto Z$.
			\item For each family $\{\phi_i\}_i$ of distributors $X\oto Y$, both the pointwise join $\sup_i\phi_i$ and the pointwise meet $\inf_i\phi_i$ are distributor  $X\oto Y$.
			\item For any distributors $\phi\colon X\oto Y$, $\psi\colon Y\oto Z$ and $\xi\colon X\oto Z$, both of the $[0,1]$-relations $\xi\swarrow\phi$ and $\psi\searrow\xi$ are distributors. 
	\end{enumerate}\end{prop}
	
	With real-enriched categories as objects and distributors as morphisms we have a category, indeed a quantaloid \[[0,1]\text{-}{\bf Dist},\] called the category of  distributors.
	
	\begin{defn}Let $\phi\colon Y\oto X$ and $\psi\colon X\oto Y$ be a pair of distributors. We say   $\phi$ is right adjoint to   $\psi$,  or $\psi$ is left adjoint to $\phi$, and write $\psi\dashv\phi$, if $X\leq \phi\circ\psi$ and $\psi\circ\phi\leq Y$.  \end{defn}

	\begin{exmp}\label{graph adjoint to cograph} 
		For each  functor $f\colon X\lra Y$, the graph $f_*\colon X\oto Y$ is left adjoint to the cograph $f^*\colon Y\oto X$. 
	\end{exmp}
	
	The following lemma, contained in Heymans \cite[Proposition 2.3.4]{Heymans}, is very useful in the calculus of distributors.
	
	\begin{lem} \label{adjoint_arrow_calculation} 
		Suppose   $\psi\colon X\oto Y$ is left adjoint to $\phi\colon Y\oto X$.  
		\begin{enumerate}[label={\rm(\roman*)}] 
			\item For each distributor $\xi\colon Y\oto Z$, $\xi\circ \psi=\xi\swarrow \phi$. In particular, $\psi=Y\swarrow\phi$.
			
			\item  For each distributor $\lambda\colon W\oto Y$, $\phi\circ \lambda =\psi\searrow \lambda$. In particular, $\phi=\psi\searrow Y$. 
	\end{enumerate}  \end{lem}

	\begin{proof}  We prove (i)  for example. On the one hand, since $\xi\circ\psi\circ\phi \leq \xi\circ Y= \xi$, then $\xi\circ\psi\leq\xi\swarrow \phi$.  On the other hand, since  $(\xi\swarrow \phi)\circ\phi\circ\psi\leq\xi\circ\psi$, then $\xi\swarrow \phi\leq \xi\circ\psi$. 
	\end{proof}
	
	In particular, a distributor has at most one right (left, resp.) adjoint. So, we shall speak of \emph{the right (left, resp.) adjoint} of a distributor.
	
	\begin{prop}Suppose $f\colon X\lra Y$ and $g\colon Y\lra X$ are functors. The following  are equivalent:
		\begin{enumerate}[label={\rm(\arabic*)}] 
			\item  $f$ is left adjoint to $g$.
			\item   $f_*=g^*$.
			\item   As distributors, $g_*$ is left adjoint to  $f_*$.
			\item   As distributors, $g^*$ is left adjoint to   $f^*$.
	\end{enumerate} \end{prop}
	
	Suppose $X$ is a real-enriched category. A \emph{weight} of $X$  is defined to be a distributor $\phi\colon X\oto \star$ from $X$ to the terminal real-enriched category $\star$. In other words, a weight of $X$ is a functor $\phi\colon X^{\rm op}\lra \sV$, where $\sV=([0,1],\alpha_L)$.  Weights of $X$ constitute a real-enriched category $\mathcal{P}X$ with $$\mathcal{P}X(\phi_1,\phi_2) =\phi_2\swarrow\phi_1=\sub_X(\phi_1,\phi_2) .
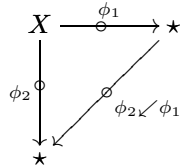
$$ 
	
	$$\bfig \ptriangle(0,0)/>`>`>/<500,500>[X`\star`\star;\phi_1`\phi_2`\phi_2\swarrow \phi_1]
	\place(230,500)[\circ] \place(250,250)[\circ] \place(0,275)[\circ] 
	\efig$$
	
	Dually, a distributor of the form $\star\oto X$ is called a \emph{coweight} of $X$. In other words, a coweight of $X$ is  a  functor $\psi\colon X\lra \sV$, where $\sV=([0,1],\alpha_L)$. 
	All coweights of $X$ constitute a real-enriched category $\CPd X$ with $$\CPd X(\psi_1,\psi_2) =\psi_2\searrow\psi_1=\sub_X(\psi_2,\psi_1) .$$     $$\bfig \qtriangle(0,0)/>`<-`<-/<500,500>[\star`X`\star;\psi_2`\psi_2\searrow \psi_1`\psi_1]
	\place(250,500)[\circ] \place(250,250)[\circ] \place(500,265)[\circ]  
	\efig$$
	
	\begin{con} Every functor $f\colon X\lra \sV$ determines a distributor  $\star\oto X, ~(\star,x)\mapsto f(x)$; conversely, every distributor  $\psi\colon\star\oto X$ determines a functor $X\lra\sV,~x\mapsto\psi(\star,x)$.  So, sometimes we'll view a functor $X\lra \sV$ as a distributor $\star\oto X$, and vice versa. To avoid confusion of notations, we make an agreement here. \begin{enumerate}[label={\rm(\roman*)}] 
			\item For each functor $f\colon X \lra\sV$,   $f^\sharp$ denotes the   distributor $\star\oto X, ~(\star,x)\mapsto f(x)$. For each distributor $\psi\colon \star\oto X$,   $\psi^\sharp$   denotes the corresponding functor $X \lra\sV,~x\mapsto\psi(\star,x)$. 
			\item For each functor $g\colon X^{\rm op}\lra\sV$,   $g^\sharp$ denotes the   distributor $X\oto\star,~(x,\star)\mapsto g(x)$. For each distributor $\phi\colon X\oto\star$,   $\phi^\sharp$   denotes the  corresponding functor $X^{\rm op}\lra\sV,x\mapsto\phi(x,\star) $.   \end{enumerate} \end{con} 
	
	\begin{rem}For each real-enriched category $X$, the underlying order of $\CP X$ coincides with the (pointwise) order on $[0,1]$-${\bf Dist}(X,\star)$, while the underlying order of  $\CPd  X$ is the reverse order  on  $[0,1]$-${\bf Dist}(\star,X)$. That is to say, for any weights $\phi_1,\phi_2\colon X\oto\star $,  \[\phi_1\sqsubseteq\phi_2~~ {\rm in}~~ (\CP X)_0\iff  \forall x\in X, ~\phi_1(x) \leq\phi_2 (x);  
		\] while for any coweights $\psi_1,\psi_2\colon \star\oto X$,  \[\psi_1\sqsubseteq\psi_2~~ {\rm in}~~ (\CPd  X)_0\iff \forall x\in X,~ \psi_1(x) \geq\psi_2(x). 
		\]  
	\end{rem}
	
	Suppose $X,Y$ are real-enriched categories, $f\colon X\lra Y$ is a functor. The map \[f^{-1}\colon\CP  Y\lra \CP X,\quad f^{-1}(\gamma)=\gamma\circ f_*\] is   readily verified to be a functor. The functor $f^{-1}$ has at the same time a left adjoint $f_\exists$ and a right adjoint $f_\forall$. The left adjoint is given by   \[f_\exists\colon\CP  X\lra\CP Y, \quad f_\exists(\phi)=\phi\circ f^*;  \] the  right adjoint is given by $$f_\forall\colon \CP X\lra \CP Y, \quad f_\forall(\phi)=  \phi\swarrow f_*.$$  In the vocabulary of category theory,  $f_\exists$ and  $f_\forall$ are called the left and the right Kan extension of $f$, respectively.
	
	Dually, the functor \[f^{-1}\colon\CPd  Y\lra \CPd X,\quad f^{-1}(\mu)=f^*\circ\mu \] also has  a left adjoint $f^\dag_\forall$  and a right adjoint $f^\dag_\exists$. Before spelling out the left adjoint and the right adjoint, we would like to remind the reader that the underlying order of $\CPd X$ is opposite to the pointwise order of coweights. The left adjoint of $f^{-1}$ is     \[f^\dag_\forall\colon\CPd  X\lra\CPd Y, \quad f^\dag_\forall(\psi)=f^*\searrow\psi;  \] the  right adjoint is   $$f^\dag_\exists\colon \CPd X\lra \CPd Y, \quad f^\dag_\exists(\psi)= f_*\circ \psi.$$    
	
	\begin{prop}\label{left and right kan} For each functor $f\colon X\lra Y$,  we have $$ f_\exists\dashv f^{-1}\dashv f_\forall\colon \CP X\lra\CP Y$$ and $$f^\dag_\forall\dashv f^{-1}\dashv f^\dag_\exists\colon \CPd X\lra \CPd Y.$$  \end{prop}
	
	\begin{prop} \label{left and right kan extension} Suppose $f\colon X\lra Y$ is a fully faithful functor.  
		\begin{enumerate}[label={\rm(\roman*)}] 
			\item For all $\phi\in\CP X$, $f^{-1}\circ f_\exists(\phi) =\phi =f^{-1}\circ f_\forall(\phi)$.  \item For all $\psi\in\CPd X$, $f^{-1}\circ f^\dag_\exists(\psi) =\psi =f^{-1}\circ f^\dag_\forall(\psi)$. 
	\end{enumerate}\end{prop} 
	
	\begin{proof}  
		(i) Since $f$ is fully faithful,  $f^*\circ f_*=X$, then  for all $\phi\in\CP X$, $$f^{-1}\circ f_\exists(\phi)=\phi\circ f^*\circ f_*=\phi$$ and \begin{align*} f^{-1}\circ f_\forall(\phi)&=(\phi\swarrow f_*)\circ f_*\\ &= (\phi\swarrow f_*)\swarrow f^* &(f_*\dashv f^*)\\ &= \phi\swarrow(f^*\circ f_*)\\ &=\phi.  \end{align*} 
		
		(ii) Similar. \end{proof} 
	
	Assigning to each $f\colon X\lra Y$ the functor  $$\CP f\coloneqq f_\exists\colon \CP X\lra\CP Y$$  defines a functor $$\CP\colon\QOrd\lra\QOrd, $$ called the \emph{presheaf functor}.\footnote{The terminology comes from the fact that,  in category theory, a contravariant functor from a category   to the category of sets is called a presheaf of that category.}   This functor plays a central role in the theory of real-enriched categories. 
	Dually,  we have a \emph{copresheaf functor} $$\CPd\colon\QOrd\lra\QOrd $$ that maps $f\colon X\lra Y$ to   \[\CPd f\coloneqq f_\exists^\dag\colon\CPd  X\lra\CPd  Y,  \quad \psi\mapsto f_*\circ\psi.  \]   
	
	For an element $a$ of a real-enriched category $X$, write $\sy(a)$ for the weight  \[  X\oto\star, \quad x\mapsto X(x,a); \]   write $\syd(a)$ for the coweight \[ \star\oto X, \quad x\mapsto X(a,x).\] Weights of the form $\sy(a)$  and coweights of the form $\syd(a)$ are said to be  \emph{representable}.  
	As distributors, the representable coweight $\syd(a)\colon \star\oto X$ is left adjoint to the representable weight $\sy(a)\colon X\oto\star$.  
	
	The following lemma, a special case of the Yoneda lemma in the theory of enriched categories, implies that for each real-enriched category $X$, both of   \[\sy\colon X\lra\CP X, \quad a\mapsto \sy(a)\] and \[\syd\colon X\lra\CPd  X, \quad a\mapsto \syd(a)\] are fully faithful functors.  
	
	\begin{lem}[Yoneda lemma] Suppose $X$ is a real-enriched category and $a\in X$. 
		\begin{enumerate}[label={\rm(\roman*)}] 
			\item For each   $\phi\in\CP X$, $\CP X(\sy(a),\phi)=\phi(a)$.  
			\item   For each   $\psi\in\CPd  X$, $\CPd  X(\psi,\syd(a))=\psi(a)$. 
	\end{enumerate}   \end{lem}
	
	\begin{proof} (i) Since \[\phi(a)=\sup_{x\in X}\phi(x)\with X(a,x)=\phi\circ\syd(a)\] and $\sy(a)$ is right adjoint to $\syd(a)$,    by Lemma \ref{adjoint_arrow_calculation} we have \[\phi(a)=\phi\circ\syd(a) =\phi\swarrow \sy(a)=\CP X(\sy(a),\phi).\]
		
		(ii) Since $\psi(a)= \sy(a)\circ \psi$ and $\syd(a)$ is left adjoint to $\sy(a)$,   by Lemma \ref{adjoint_arrow_calculation} we have \begin{align*}\psi(a)&= \sy(a)\circ \psi=\syd(a)\searrow\psi =\CPd  X(\psi,\syd(a)).\qedhere\end{align*}
	\end{proof}
	
	From the Yoneda lemma one infers that for each real-enriched category $X$, the functor $\sy\colon X\lra \CP X$ is fully faithful, so it is called   the \emph{Yoneda embedding}.
	Dually, the fully faithful functor  $\syd\colon X\lra \CPd X$ is called the \emph{coYoneda embedding}. 
	
	\begin{prop}\label{kan for yoneda} Suppose $X$ is a real-enriched category. 
		\begin{enumerate}[label={\rm(\roman*)}] 
			\item  $(\sy_X)_\forall=\sy_{\CP X}$, i.e. $(\sy_X)_\forall(\phi)=\CP X(-,\phi)$   for all $\phi\in\CP X$. 
			\item  $(\syd_X)^\dag_\forall=\syd_{\CPd X}$, i.e. $(\syd_X)^\dag_\forall(\psi)=\CPd X(\psi,-)$   for all $\psi\in\CPd X$. 
	\end{enumerate} \end{prop}
	
	\begin{proof} We prove (i) for example. For all $\gamma\in \CP X$, \begin{align*}\sy_\forall(\phi)(\gamma) &= (\phi\swarrow\sy_*)(\gamma)\\ &= \bw_{x\in X} (\CP X(\sy(x),\gamma)\ra\phi(x))\\ &= \bw_{x\in X}(\gamma(x)\ra\phi(x))  
			\\ &=\CP  X(\gamma,\phi). \qedhere \end{align*}
	\end{proof}
	
	Yoneda embeddings constitute a natural transformation $$\sy\colon{\rm id}\lra\CP$$ from the identity functor   to the presheaf functor $\CP$. Dually, coYoneda embeddings constitute a natural transformation $$\syd\colon{\rm id}\lra\CPd$$ from the identity functor to the copresheaf functor $\CPd$.

	There exist natural connections between  $\CP X$ and $\CPd X$. We present two examples here. The first is the Isbell adjunction. 
	For   $\phi\in\CP X$ and $\psi\in\CPd  X$, let $$\ub\phi=X\swarrow\phi, \quad \lb\psi=\psi\searrow X.$$    $$\bfig \ptriangle(0,0)/>`>`>/<500,500>[X`\star`X;\phi`X`X\swarrow \phi] \place(250,500)[\circ] \place(250,250)[\circ] \place(0,250)[\circ]   
	\qtriangle(1000,0)/>`<-`<-/<500,500>[\star`X`X;\psi`\psi\searrow X`X] \place(1250,500)[\circ] \place(1250,250)[\circ] \place(1500,250)[\circ]  
	\efig
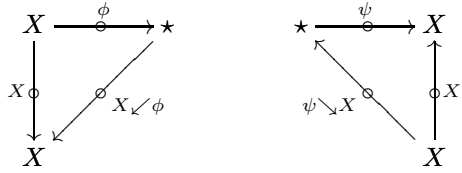
$$  
	
	\begin{prop}[Isbell adjunction]\label{prop.Isbell adjunction} The functor  $\ub\colon \CP X\lra\CPd  X $ is left adjoint to the functor $\lb\colon \CPd  X\lra\CP X$.   \end{prop} \begin{proof} Because \begin{align*}\CPd  X(\ub\phi,\psi)&=\psi\searrow(X\swarrow\phi)=(\psi\searrow X)\swarrow\phi=\CP X(\phi,\lb\psi).\qedhere \end{align*} \end{proof}
	
	The second is the distributor $$\cpt\colon \CPd X\oto\CP X $$ that assigns to each pair  $(\psi,\phi)\in\CPd X\times\CP X$ the real number $\phi\circ\psi$.    
	A close look at this distributor in a special case may help us understand its role in the theory of real-enriched categories. Let $X$ be a set, viewed as a discrete real-enriched category. Then, $\CP X$ is $([0,1]^X,\sub_X)$ and $\CPd X$ is $([0,1]^X,\sub_X^{\rm op})$.   For each $\phi\in[0,1]^X$, the functor $$\phi\cpt-\colon (\CPd X)^{\rm op}\lra\sV$$ obtained by fixing the second argument of $\cpt\colon \CPd X\oto\CP X$ turns out to be the functor $$\phi\cpt-\colon([0,1]^X,\sub_X)\lra ([0,1],\alpha_L), \quad \psi\mapsto  \sup_{x\in X}\phi(x)\with\psi(x).$$ Intuitively, the value $\phi\cpt\psi=\sup_{x\in X}\phi(x)\with\psi(x)$ 
	measures the truth degree that $\phi$ and $\psi$ have a common point.  
	
	Suppose $\phi$ is a weight and $\psi$ is a coweight of a real-enriched category $X$. For all $p\in[0,1]$,  write $\psi\ra p$ for the weight $\star_p\swarrow\psi$ and write $\phi\ra p$ for the coweight $\phi\searrow\star_p$ of $X$, where $\star_p$ denotes the distributor $\star\oto\star$ with value $p$. 
	
	\begin{lem}\label{sub vs tensor}
		Let $X$ be a real-enriched category.  \begin{enumerate}[label={\rm(\roman*)}] 
			\item For all $\phi\in\CP X$,  $\psi\in\CPd X$ and  $p\in[0,1]$,
			$\CP X(\phi,  \psi\ra p)= \phi\cpt\psi\ra p$. 
			\item For all $\phi\in\CP X$ and  $\psi\in\CPd X$, $\phi\cpt\psi=\inf_{p\in [0,1]}(\CP X(\phi,\psi\ra p)\ra p)$. 
			\item For all $\phi_1,\phi_2\in\CP X$, $\CP X(\phi_1,\phi_2)=\inf_{p\in [0,1]}((\phi_1\cpt(\phi_2\ra p)) \ra p)$. 
		\end{enumerate}
	\end{lem}

	\section{Colimit and limit}
	
	A \emph{colimit} of a weight $\phi\colon X\oto\star $ is an element $\colim\phi$ of $X$ such that for all $x\in X$, \[X(\colim\phi,x)=\CP X(\phi,\sy(x)).\]
	It is clear that each weight has, up to isomorphism, at most one colimit. So we shall speak of \emph{the} colimit of a weight.
	
	Since \[\CP X(\phi,\sy(x))=\sy_X(x)\swarrow\phi = (X\swarrow\phi)(x) \] for all $x\in X$, it follows that $\phi$ has a colimit if and only if the coweight $X\swarrow\phi$  
	is representable. A colimit of $\phi$ is then a representation of the coweight $X\swarrow\phi$, i.e. $X(\colim\phi,-)=X\swarrow\phi$. $$\bfig \ptriangle(0,0)/>`>`>/<500,500>[X`\star`X;\phi`X`X\swarrow \phi]
	\place(250,500)[\circ] \place(250,250)[\circ] \place(0,250)[\circ] \place(180,350)[\geq] \efig$$
	
	For each functor $f\colon K\lra X$  and each weight $\phi$  of $K$, the colimit of $\phi\circ f^*$, when exists, is called the \emph{colimit of $f$ weighted by $\phi$}. In particular, $\colim\phi$ is the colimit of the identity functor $X\lra X$ weighted by $\phi$.  
	
	We  write $\colim_\phi f$ for the colimit of $f$ weighted by $\phi$. By definition, $b$ is a colimit of $f\colon K\lra X$  weighted by $\phi\colon K\oto\star$ if and only if for all $x\in X$, $$X(b,x)= \CP X(\phi\circ f^*,X(-,x))=\CP K(\phi,X(f(-),x)). $$ 
	
	\begin{exmp} \label{cpt as colimit} Suppose $f\colon X\lra\sV$ is a functor,   $\sV=([0,1],\alpha_L)$. Then for each weight $\phi\colon X\oto\star$,  the colimit of $f$ weighted by $\phi$ exists and is equal to the composite $\phi\circ f^\sharp$, i.e. $\colim_\phi f =\phi\circ f^\sharp$, where  $f^\sharp\colon\star\oto X$ is the distributor corresponding to $f$.  To see this we  calculate:  for all $r\in[0,1]$, \begin{align*}\CP\sV(  \phi\circ f^*,\sV(-,r))& = \CP X(\phi, \sV(f(-),r))  \\   
			&= \inf_{x\in X}(\phi(x)\ra (f(x)\ra r)) \\ 
			&= \Big(\sup_{x\in X}\phi(x)\with f(x)\Big)\ra r\\ &= \sV(\phi\circ f^\sharp,r),  
		\end{align*} 
		hence $\phi\circ f^\sharp$ is a colimit of $f$ weighted by $\phi$. \end{exmp}
	
	\begin{defn}A real-enriched category $X$ is   cocomplete if every weight of $X$ has a colimit. \end{defn}
	
	It is readily seen that $X$ is cocomplete if and only if all weighted colimits of functors to $X$ exist; that means, for each functor $f\colon K\lra X$ and each weight $\phi$ of $K$, the colimit of $f$ weighted by $\phi$ exists.
	
	\begin{prop}\label{cocomplete via adj} A real-enriched category $X$ is cocomplete   if and only if the Yoneda embedding $\sy\colon X\lra\CP X$ has a left adjoint. \end{prop}
	
	\begin{proof} If $X$ is cocomplete,   the correspondence $\phi\mapsto\colim\phi$ defines a left adjoint of $\sy\colon X\lra\CP X$. Conversely,   any left adjoint of $\sy\colon X\lra\CP X$ must send  a weight $\phi$ of $X$ to a colimit of $\phi$. \end{proof}
	
	\begin{exmp}  
		Example \ref{cpt as colimit} shows that   $\sV=([0,1],\alpha_L)$ is cocomplete. The left adjoint of   $\sy\colon\sV\lra\CP\sV$ is given by   $$\CP \sV\lra\sV,\quad \phi\mapsto \phi(1).$$ Or equivalently, for each weight $\phi$ of $\sV$,  $\colim\phi=\phi(1)$.
		
		In fact, since  for all $z\in[0,1]$, $z=1\ra z  \leq   \phi(z)\ra\phi(1)$, then \[\phi(1)=\sup_{z\in [0,1]}\phi(z)\with z=\colim\phi,\]
		the second equality holds by Example \ref{cpt as colimit}. \end{exmp}
	
	\begin{exmp}\label{sup in PX}  
		For each real-enriched category $X$, $\CP X$ is cocomplete. For this consider the Yoneda embedding $\sy_X\colon X\lra \CP X$. Since $$ (\sy_X)_\exists \dashv \sy_X^{-1}\dashv (\sy_X)_\forall $$  and $(\sy_X)_\forall=\sy_{\CP X}$ (Proposition \ref{left and right kan} and Proposition  \ref{kan for yoneda}), it follows that $\sy_{\CP X}$ has a left adjoint, so $\CP X$ is cocomplete. 
		For each weight $\Phi\colon \CP X\oto\star $   of $\CP X$, the   colimit of $\Phi$ is $\sy_X^{-1}(\Phi)$, i.e. $\colim\Phi=\Phi\circ(\sy_X)_*$. $$\bfig \qtriangle(0,0)/>`>`>/<520,500>[X`\CP X`\star ;(\sy_X)_*`\colim\Phi`\Phi]
		\place(250,500)[\circ] \place(260,250)[\circ] \place(520,250)[\circ] \efig$$ 
		
		Since  $\CP \sy_X=(\sy_X)_\exists$,  $\sy_X^{-1}=\colim_{\CP X}$ and $(\sy_X)_\forall=\sy_{\CP X}$, we obtain a string of adjunctions  \[\CP \sy_X\dashv {\colim}_{\CP X}  \dashv \sy_{\CP X}\colon \CP X\lra\CP\CP X.\] \end{exmp}
	
	\begin{exmp}\label{sup in PdX} For each real-enriched category $X$, the real-enriched category $\CPd  X$ is cocomplete. We check that for each weight $\Phi$ of $\CPd  X$,  the coweight $(\syd)^*\swarrow\Phi$ of $X$ is a colimit of $\Phi$.  
		$$\bfig \ptriangle(0,0)/>`>`>/<520,500>[\CPd  X`\star`X;\Phi`(\syd)^\star`\colim\Phi]
		\place(275,500)[\circ] \place(260,250)[\circ] \place(0,250)[\circ] \place(160,330)[\geq] \efig$$  
		
		First we note that for all $\psi \in\CPd X$, $$\psi\searrow (\syd)^*  =\bw_{x\in X}(\psi(x)\ra \CPd X(-,\syd(x))) =\CPd X(-,\psi). $$ $$\bfig \qtriangle(0,0)/>`<-`<-/<520,500>[\star`X`\CPd X;\psi`\psi\searrow (\syd)^*`(\syd)^*]
		\place(230,500)[\circ] \place(260,250)[\circ] \place(520,275)[\circ] \place(360,360)[\leq] \efig$$ Thus for all $\psi\in\CPd X$, 
		\begin{align*} \CPd X((\syd)^*\swarrow\Phi, \psi) &= \psi\searrow ((\syd)^*\swarrow\Phi) \\ 
			& = (\psi\searrow (\syd)^*)\swarrow\Phi \\ 
			& = \CP(\CPd X)(\Phi,\CPd X(-,\psi)),\end{align*} showing that $(\syd)^*\swarrow\Phi$ is a colimit of $\Phi$.
		
		In particular, for each weight $\phi$ of $X$, the colimit of the coYoneda embedding $\syd\colon X\lra \CPd X$ weighted by $\phi$ is $X\swarrow\phi$, because 
		\begin{align*}{\colim}_\phi\thinspace\syd &=  \colim \phi\circ (\syd)^* \\ 
			&= (\syd)^*\swarrow(\phi\circ(\syd)^*)\\ 
			&= ((\syd)^*\swarrow(\syd)^*)\swarrow\phi \\ 
			&= ((\syd)^*\circ(\syd)_*)\swarrow\phi &((\syd)_*\dashv(\syd)^*) \\ 
			&= X\swarrow\phi. &(\text{$\syd$ is fully faithful})\end{align*}
	\end{exmp}
	
	The dual notion of colimit is limit. Precisely, a  \emph{limit} of a coweight $\psi\colon \star\oto X$ is an element $\lim\psi$ in $X$ such that for all $x\in X$, \[\CPd  X(\syd(x),\psi)=X(x,\lim\psi).\] It is obvious that each coweight has at most one limit up to isomorphism.
	
	Since \[\CPd  X(\syd(x),\psi)= \psi\searrow\syd(x) = (\psi\searrow X)(x),\] it follows that $\psi$ has a limit if and only if the weight $\psi\searrow X$  
	is representable. $$\bfig \qtriangle(0,0)/>`<-`<-/<520,500>[\star`X`X;\psi`\psi\searrow X`X]
	\place(230,500)[\circ] \place(260,250)[\circ] \place(520,275)[\circ] \place(360,360)[\leq] \efig$$
	
	\begin{rem}\label{limit vs colimit} Let $X$ be a real-enriched category and let $\psi\colon \star\oto X$ be a coweight of $X$. Then, $a\in X$ is a colimit of the weight $\psi^{\rm op}$ of   $X^{\rm op}$  if and only if for all $x\in X$, \[X^{\rm op}(a,x)=\CP(X^{\rm op})(\psi^{\rm op},X^{\rm op}(-,x)).\] Since \[\CP(X^{\rm op})(\psi^{\rm op},X^{\rm op}(-,x))= \inf_{z\in X}(\psi(z)\ra X(x,z)),\] it follows that $a$ is a colimit of the weight $\psi^{\rm op}$ of   $X^{\rm op}$ if and only if \[X(x,a)=\CPd X(\syd(x),\psi).\] Therefore, a limit of a coweight $\psi$ of   $X$ is precisely a colimit of the weight $\psi^{\rm op}$ of  $X^{\rm op}$.
	\end{rem}
	
	For each functor $f\colon K\lra X$  and each coweight $\psi$  of $K$, the limit of $f^\dag_\exists(\psi)$, when exists, is called the \emph{limit of $f$ weighted by $\psi$}. In particular, $\lim\psi$ is the limit of the identity functor $X\lra X$ weighted by $\psi$.  We  write $\lim_\psi f$ for the limit of $f$ weighted by $\psi$. 
	
	\begin{exmp} \label{sub as limit}  Suppose $X$ is a real-enriched category and  $f\colon X\lra\sV$ is a functor.  Then, for each coweight $\psi$  of  $X$, the limit of $f$ weighted by $\psi$ exists and is given by ${\lim}_\psi f=\CPd X(f^\sharp,\psi)$, where  $f^\sharp\colon\star\oto X$ is the distributor corresponding to $f$. In fact,  for all $r\in[0,1]$, 
		\begin{align*}    \CPd \sV(\syd_\sV(r), f_*\circ\psi) & =  \psi\searrow(f_*\searrow \sV(r,-))  \\   
			&= \inf_{x\in X}(\psi(x)\ra (r\ra f(x))) \\ 
			&= \inf_{x\in X}(r\ra(\psi(x)\ra f(x))) \\ 
			&= \sV(r, \CPd X(f^\sharp,\psi)), \end{align*} which shows $\CPd X(f^\sharp,\psi)$ is a limit of $f$ weighted by $\psi$. \end{exmp}
	
	\begin{defn}A real-enriched category $X$ is complete if each coweight of $X$ has a limit. \end{defn}
	
	It is readily seen that $X$ is complete if and only if for any functor $f\colon K\lra X$ and any coweight $\psi$ of $K$, the limit of $f$ weighted by $\psi$ exists. 
	Dual to Proposition \ref{cocomplete via adj}, we have the following:
	\begin{prop}A real-enriched category $X$ is complete  if and only if the coYoneda embedding $\syd\colon X\lra\CPd  X$ has a right adjoint. \end{prop}
	
	\begin{exmp} The real-enriched category $\sV=([0,1],\alpha_L)$ is  complete. For each coweight $\psi$ of $\sV$, by Example \ref{sub as limit},  $\inf_{z\in [0,1]}(\psi(z)\ra z)$ is a limit of $\psi$. \end{exmp}

	\begin{exmp}\label{inf in PdX} For each real-enriched category $X$,  $\CPd  X$ is complete. For this consider the coYoneda embedding $\syd_X\colon X\lra\CPd X$. Since $$ (\syd_X)^\dag_\forall\dashv (\syd_X)^{-1}\dashv (\syd_X)^\dag_\exists $$ and $(\syd_X)^\dag_\forall=\syd_{\CPd X}$ (Proposition \ref{left and right kan} and Proposition \ref{kan for yoneda}), it follows that   $\syd_{\CPd X}$ has a right adjoint, so $\CPd X$ is complete.  
		For each coweight $\Psi\colon \star\oto\CPd  X$,  the limit of $\Psi$ is $(\syd_X)^{-1}(\Psi)$, i.e.  $\lim\Psi=(\sy_X^\dag)^*\circ\Psi$. $$\bfig \qtriangle(0,0)/>`>`>/<520,500>[\star`\CPd  X`X;\Psi`\lim\Psi`(\sy_X^\dag)^*]
		\place(260,500)[\circ] \place(260,250)[\circ] \place(520,250)[\circ] \efig$$  
		Thus, we have a string of adjunctions $$ \syd_{\CPd X} \dashv {\lim}_{\CPd X}\dashv \CPd\syd_X\colon \CPd X\lra\CPd\CPd X . $$
	\end{exmp}
	
	\begin{exmp}\label{inf in PX} For each real-enriched category $X$,  the real-enriched category $\CP X$ is complete. For each coweight $\Psi$ of $\CP X$,   the weight $\Psi\searrow\sy_*$ of $X$ is a limit of $\Psi$. $$\bfig \qtriangle(0,0)/>`<-`<-/<520,500>[\star`\CP X`X;\Psi`\lim\Psi`\sy_*]
		\place(230,500)[\circ] \place(260,250)[\circ] \place(520,275)[\circ] \place(360,360)[\leq] \efig$$
		
		In particular, for each coweight $\psi$ of $X$, the limit of the Yoneda embedding $\sy\colon X\lra\CP X$ weighted by $\psi$ is $\psi\searrow X$, i.e. $\lim_\psi\sy=\psi\searrow X$. \end{exmp}

	\begin{thm} \label{complete is self-dual} {\rm (Stubbe \cite{Stubbe2005})} A real-enriched category is complete  if and only if it is cocomplete. \end{thm}
	
	\begin{proof}Let $X$ be a complete real-enriched category. We show that each weight $\phi$   of $X$ has a colimit, so $X$ is cocomplete. Consider the coweight $X\swarrow\phi$ of $X$. Since $X$ is complete, there is some $a\in X$ such that \[X(-,a)= (X\swarrow\phi)\searrow X.\] Then \[X(a,-)= X\swarrow X(-,a) = X\swarrow ((X\swarrow\phi)\searrow X) =X\swarrow\phi,\] showing that $\phi$ has a colimit.
		Likewise for the converse implication. \end{proof}
	
	In the following we examine the relation between functors, colimits of weights and limits of coweights.
	First of all,   for any  functor $f\colon X\lra Y$   and any weight $\phi$   of $X$, if both $\colim \phi$ and $\colim_\phi f   $ exist, then $ {\colim}_\phi f  \sqsubseteq f({\colim} \phi)$, because \begin{align*}1&= X({\colim}\phi,{\colim}\phi)\\ &=\CP X(\phi, X(-,{\colim} \phi))\\ &\leq \CP X(\phi, Y(f(-),f({\colim}\phi))) \\ &= \CP X(\phi, Y(-,f({\colim}\phi))\circ f_*) \\ & =\CP Y(\phi\circ f^*, Y(-,f({\colim}\phi)) &(-\circ f^*\dashv -\circ f_*)\\ &= Y({\colim}_\phi f  ,f({\colim}\phi)).\end{align*}
	
	Suppose $f\colon X\lra Y$ is a functor. We say that  $f$  \emph{preserves the colimit of a weight} $\phi$ of $X$, if $f(\colim\phi)$ is a colimit of $f$ weighted by $\phi$, i.e.  $f(\colim\phi)=\colim_\phi f $, whenever $\colim\phi$ exists. Likewise,   $f$ preserves the limit of a coweight $\psi$ of $X$ if $f(\lim\psi)$ is a limit of $f$ weighted by $\psi$ whenever $\lim\psi$ exists.
	
	We say that a functor  $f\colon X\lra Y$    \emph{preserves colimits} if it preserves all colimits that exist. Likewise,  $f$  \emph{preserves limits} if it preserves all  limits that exist.

	\begin{thm}Each left adjoint preserves colimits. Dually, each right adjoint preserves limits. \end{thm}
	
	\begin{proof} Suppose that $f\colon X\lra Y$ is left adjoint to $g\colon Y\lra X$. We show that $f$ preserves colimits and $g$ preserves limits.
		
		Assume that $\phi\colon X\oto\star $ has a colimit.   Since for all $y\in Y$, 
		\begin{align*} Y(f(\colim\phi),y) &= X(\colim\phi,g(y)) &(f\dashv g) \\ 
			&= \CP X(\phi,  X(-,g(y))) &(\text{definition of $\colim\phi$}) \\ 
			&= \CP X(\phi,Y(f(-),y)) &(f\dashv g)\\ 
			&= \CP X(\phi,  Y(-,y)\circ f_*)   \\   
			&= \CP Y(\phi\circ f^*, Y(-,y)), &(-\circ f^*\dashv -\circ f_*) 
		\end{align*} it follows that $f(\colim\phi)$ is a colimit of $f$ weighted by $\phi$.  
		
		By duality, $g$ preserves limits. \end{proof}

	\begin{thm}If $f\colon X\lra Y$ preserves colimits and $X$ is cocomplete, then $f$ is a left adjoint. \end{thm}
	
	\begin{proof} Since $X$ is cocomplete, the Yoneda embedding $\sy_X\colon X\lra\CP X$ has a left adjoint $\colim_X\colon \CP X\lra X$. Let $g$ be the composite \[Y\to^{\sy_Y}\CP Y\to^{f^{-1}}\CP X\to^{{\colim}_X}X;\] that is, $g(y) =\colim_X Y(f(-),y)$. We claim that $g$ is a right adjoint of $f$.
		
		By Theorem \ref{Characterization of adjoints}, we need to check that for all $x\in X$ and $y\in Y$, $f(x)\sqsubseteq y$  if and only if $x\sqsubseteq g(y)$.
		If $f(x)\sqsubseteq y$, then 
		\begin{align*}1&\leq Y(f(x),y)\\ &  = \CP X(X(-,x), Y(f(-),y)))& (\text{Yoneda lemma})\\ 
			&\leq X({\colim}_X X(-,x),{\colim}_X Y(f(-),y))) &(\text{${\colim}_X$ is a  functor})\\ 
			& = X(x,g(y)), &(\text{definition of $g$})
		\end{align*} showing that $x\sqsubseteq g(y)$.
		
		For the converse implication,   first we show  that   $1\leq Y(f(g(y)),y)$  for all $y\in Y$. Since $f$ preserves colimits and $g(y)=   {\colim}_X(Y(-,y)\circ f_*)$, then \[f(g(y))= {\colim}_Y(Y(-,y)\circ f_*\circ f^*)\leq {\colim}_Y  Y(-,y) \cong y,\] so $1\leq Y(f(g(y)),y)$.
		
		Now, suppose that $x\sqsubseteq g(y)$. Then \[1\leq X(x,g(y))\leq Y(f(g(y)),y)\with Y(f(x),f(g(y)))\leq Y(f(x),y),\] hence $f(x)\sqsubseteq y$.
	\end{proof}

	\begin{prop}The Yoneda embedding $\sy\colon X\lra\CP X$ preserves limits; the coYoneda embedding $\syd\colon X\lra\CPd  X$ preserves colimits.  \end{prop}
	
	\begin{proof}   For the first conclusion, we need to show that for each coweight $\psi$ of $X$, if $\lim\psi$ exists, then $\sy(\lim\psi)$ is the limit of $\sy_*\circ\psi$ in $\CP X$, i.e. $\sy(\lim\psi)={\lim}_{\CP X}(\sy_*\circ\psi)$. To see this, we calculate:
		\begin{align*}  \lim(\sy_*\circ\psi)& =(\sy_*\circ\psi)\searrow\sy_*  & (\text{Example \ref{inf in PX}})\\ 
			&=\psi\searrow(\sy_*\searrow \sy_*)\\ 
			& = \psi\searrow (\sy^*\circ \sy_*) &(\sy_*\dashv\sy^*)\\ 
			& = \psi\searrow X  &(\text{$\sy$ is fully faithful})\\ 
			&= X(-,\lim\psi) & (\text{definition of $\lim\psi$})\\ 
			&= \sy(\lim\psi).\end{align*}
		
		The second conclusion follows by duality. \end{proof}
	
	A functor  $f\colon X\lra Y$ is   \emph{colimit-dense} if for each $y\in Y$, there is a weight $\phi$ of $X$ such that $y$ is a colimit of $f$ weighted by $\phi$. Likewise, $f$ is  \emph{limit-dense}  if for each $y\in Y$, there is a coweight $\psi$ of $X$ such that $y$ is a limit of $f$ weighted by $\psi$.
	
	There is a nice characterization of colimit-dense and limit-dense functors.
	
	\begin{prop}\label{colimit-dense} {\rm (Lai and Shen \cite{LaiS2017})} Suppose $f\colon X\lra Y$ is a functor. \begin{enumerate}[label={\rm(\roman*)}] 
			\item $f$  is colimit-dense if and only if $f_*\swarrow f_*=Y$. \item $f$ is limit-dense if and only if $f^*\searrow f^* =Y$. \end{enumerate}  \end{prop}
	
	\begin{proof} We prove (i), leaving (ii) to the reader. Assume that $f$  is colimit-dense. We wish to show   $f_*\swarrow f_*=Y$. It suffices to show that for each $y$ of $Y$, $(f_*\swarrow f_*)(y,-)\leq Y(y,-)$. Since $f$ is colimit-dense, $y=\colim_\phi f$ for some  weight $\phi$ of $X$. So, \[Y(y,-)=Y\swarrow(\phi\circ f^*) = (Y\swarrow f^*)\swarrow\phi= f_*\swarrow\phi,\] hence   $\phi\leq   f_*(-,y)$. Therefore, \[(f_*\swarrow f_*)(y,-)= f_*\swarrow f_*(-,y) \leq f_*\swarrow\phi = Y(y,-).\] Conversely, suppose that $f_*\swarrow f_*=Y$. Then for each $y$ of $Y$, \[Y(y,-)=(f_*\swarrow f_*)(y,-) = (Y\swarrow f^*)\swarrow f_*(-,y)=Y\swarrow(f_*(-,y)\circ f^*),\] showing that $y$ is a colimit of $f$ weighted by $f_*(-,y)$, hence $f$  is colimit-dense.  \end{proof}
	
	\begin{prop}\label{y is colimit dense} For each real-enriched category, the Yoneda embedding is colimit-dense and the coYoneda embedding is limit-dense. \end{prop}
	
	\begin{proof}   For  each $\phi \in\CP X$,  by Example \ref{sup in PX} we have $${\colim}_\phi\thinspace\sy =\colim\phi\circ\sy^*= \phi\circ\sy^*\circ\sy_*=\phi,$$ the last equality holds by full  faithfulness of $\sy$, hence $\sy$ is colimit-dense.  \end{proof}
	
	\begin{thm}\label{free cocompletion} {\rm(Kelly \cite{Kelly})} For each functor $f\colon X\lra Y$ with $Y$ separated and cocomplete, there is a unique left adjoint $\overline{f}\colon \CP X\lra Y$ such that $f=\overline{f}\circ\sy_X$. \end{thm} 
	
	\begin{proof}   Let $\overline{f}$ be the composite of left adjoints  \[ \CP X\to^{f_\exists}\CP Y\to^{\colim_Y}Y,\] where $\colim_Y\colon\CP  Y\lra Y$ is the left adjoint of  $\sy_Y\colon Y\lra \CP Y$. Then $\overline{f}$ satisfies the requirement. This proves the existence.
		For uniqueness, suppose   $\overline{f}$ satisfies the conditions. For each $\phi\in\CP X$, since $\phi$ is the colimit of the Yoneda embedding $\sy_X\colon X\lra \CP X$ weighted by $\phi$ and $\overline{f}$ preserves colimits, then \[\overline{f}(\phi)= \overline{f}({\colim}_\phi\thinspace\sy_X)={\colim}_\phi(\overline{f}\circ\sy_X)  ={\colim}_\phi\thinspace f,\] showing that $\overline{f}$ is unique.
	\end{proof}

	\section{Tensor and cotensor} \label{tensor and cotensor}
	
	Cocompleteness of a real-enriched category $X$ requires that every weight of $X$ has a colimit. The aim of this section is to show that in order that all weights of a real-enriched category $X$ have   colimits, it suffices to require certain subclass of weights of $X$ have colimits.
	
	Let $X$ be a real-enriched category. For each $x\in X$ and  $r\in [0,1]$, the \emph{tensor of $r$ with $x$}, denoted by $r\otimes  x$, is an element of $X$ such that for all $y\in X$, \[X(r\otimes  x,y)= r\ra X(x,y).\] Tensors are a special kind of colimits. Actually, $r\otimes x$ is the colimit of the functor $s_x\colon\star\lra X$ weighted by $r_\star$, where, $s_x$ is the functor sending the unique object of $\star$ to $x$, and  $r_\star$ is the weight of   $\star$   with value $r$.  Said differently,  \[r\otimes  x=\colim(r\with\sy(x)).\]
	
	For each $y\in X$ and  $r\in [0,1]$, the \emph{cotensor of $r$ with} $y$, denoted by $r\multimap y$,  is an element of $X$ such that for all $x\in X$, \[X(x,r\multimap y)=r\ra X(x,y).\] The cotensor $r\multimap y$ is the limit of the functor $s_y\colon\star\lra X$ weighted by $r_\star$. Said differently,  \[r\multimap y=\lim(r\with\syd(y)).\]
	
	\begin{defn}A real-enriched category $X$ is   tensored if the tensor $r\otimes  x$ exists for all $x\in X$ and $r\in [0,1]$; $X$ is   cotensored if the cotensor $r\multimap y$ exists for all $y\in X$ and $r\in [0,1]$. \end{defn}
	
	It is clear that $X$ is tensored if and only if   its opposite $X^{\rm op}$  is cotensored.  Every complete (hence cocomplete) real-enriched category is   both tensored and  cotensored.
	
	If   $X$  is both tensored and cotensored, then for all $r\in[0,1]$ and $x,y\in X$, \begin{align*}  r\otimes x\sqsubseteq y~~\text{in}~~ X_0&\iff r\leq  X(x,y)\\ &\iff x\sqsubseteq r\multimap y ~~\text{in}~~ X_0.\end{align*} Therefore,   \[r\otimes -\colon X_0\lra X_0\] is  left adjoint to \[r\multimap -\colon X_0\lra X_0.\] This shows that in a both tensored and cotensored  real-enriched category, tensors and cotensors are determined by each other.

	\begin{exmp}The real-enriched category $\sV=([0,1],\alpha_L)$ is tensored and cotensored. For each $r\in [0,1]$ and each $x\in V$,\[r\otimes  x=r\with x, \quad r\multimap x=r\ra x.\]
		
		The real-enriched category $\sV^{\rm op}=([0,1],\alpha_R)$ is also tensored and cotensored. For each $r\in [0,1]$ and each $y\in \sV^{\rm op}$, \[r\otimes  y =r\ra y, \quad r\multimap y=y\with r.\] \end{exmp}
	
	\begin{exmp}\label{tensor and cotensor in PX} Let $X$ be a real-enriched category. In $\CP X$, the tensor and cotensor of $r\in [0,1]$ with $\phi\in\CP X$ are, respectively,   $r\with\phi$ and $r\ra\phi$. In $\CPd X$, the tensor and cotensor of $r\in [0,1]$ with $\psi\in\CPd X$ are  $r\ra  \psi$ and $r\with\psi $, respectively.\end{exmp}

	Let $\mathcal{B}X$ be the subset of $\CP X$ consisting of  weights of the form $r\with\sy(x)$; that is, \[\mathcal{B}X=\{r\with\sy(x)\mid x\in X, r\in [0,1]\}. \] Since for each functor $f\colon X\lra Y$, each $x\in X$ and each $r\in[0,1]$, $f_\exists(r\with\sy_X(x))=r\with\sy_Y(f(x))$, it follows that the assignment $X\mapsto \CB X$ defines a subfunctor of the presheaf functor $\CP$, through which the natural transformation $\sy\colon \id\lra\CP$ factors.  
	
	Write $\mathfrak{i}\colon\CB X\lra\CP X$ for the inclusion functor and ${\sf t}\colon X\lra  \mathcal{B}X$ for the Yoneda embedding with codomain restricted to $\mathcal{B}X$. It is clear that $\sy=\mathfrak{i}\circ{\sf t}$.
	
	\begin{prop}\label{property of tensor} For each  real-enriched category $X$, the following   are equivalent: 
		\begin{enumerate}[label={\rm(\arabic*)}] 
			\item $X$ is tensored. 
			\item The  functor ${\sf t}  \colon X\lra  \mathcal{B}X$ has a left adjoint. 
			\item For each $x\in X$, the functor $X(x,-)\colon X\lra \sV$  is a right adjoint, where $\sV=([0,1],\alpha_L)$.
		\end{enumerate} In this case, $X(x,-)\colon X\lra \sV$ is right adjoint to $-\otimes  x\colon \sV\lra X$.    \end{prop}
	
	\begin{proof} $(1)\Leftrightarrow(2)$ This follows from the fact that the tensor $r\otimes  x$ of $r$ with $x$, if exists, is precisely the colimit  of the weight $r\with\sy(x)$.
		
		$(1)\Leftrightarrow(3)$ A restatement of  the equality $X(r\otimes  x,y)= r\ra X(x,y)$. \end{proof}

	Dual conclusions hold for cotensored real-enriched categories. We spell them out for later use. Let $\mathcal{B}^\dag X$ be the subset of $\CPd  X$ consisting of  coweights of the form $r\with\syd(y) $, $y\in X, r\in[0,1]$. 
	
	Write $\mathfrak{i}^\dag\colon\CB^\dag X\lra\CPd X$ for the inclusion functor and ${\sf t}^\dag\colon X\lra  \mathcal{B}^\dag X$ for   the coYoneda embedding with codomain restricted to $\mathcal{B}^\dag X$.  
	
	\begin{prop}\label{chara of cotensored} For each  real-enriched category $X$, the following   are equivalent: 
		\begin{enumerate}[label={\rm(\arabic*)}]  
			\item $X$ is cotensored. 
			\item The map ${\sf t}^\dag\colon X\lra  \mathcal{B}^\dag X$ has a right adjoint. 
			\item For each $y\in X$, the  map $X(-,y)\colon X\lra \sV^{\rm op} $ is a left adjoint, where $\sV^{\rm op}=([0,1],\alpha_R)$. 
		\end{enumerate}  In this case, $X(-,y)\colon X\lra \sV^{\rm op}$ is left adjoint to $-\multimap  y\colon \sV^{\rm op}\lra X$.    \end{prop}

	\begin{prop}\label{functor via tensors} Suppose $f\colon X\lra Y$ is a map between real-enriched categories. 
		\begin{enumerate}[label={\rm(\roman*)}] 
			\item If both $X$ and $Y$ are tensored, then $f$ is a functor if and only if $f\colon X_0\lra Y_0$ preserves order and $r\otimes  f(x)\sqsubseteq f(r\otimes  x)$ for all $x\in X$ and $r\in [0,1]$.   \item If both $X$ and $Y$ are cotensored, then $f$ is a functor if and only if $f\colon X_0\lra Y_0$ is order-preserving and $f(r\multimap y)\sqsubseteq  r\multimap f(y)$ for all $y\in X$ and $r\in [0,1]$. \end{enumerate} \end{prop}
	
	\begin{proof} We prove (i) for example.
		Suppose  $f$ is a functor. First, it is trivial that $f\colon  X_0\lra Y_0$ preserves order. Next, let $r\in [0,1]$ and  $x\in X$. Since $X$ is tensored and $1\leq X(r\otimes  x, r\otimes  x)$, it follows that $r\leq X(x,r\otimes  x)$, then $r\leq Y(f(x),f(r\otimes  x))$ and consequently, $r\otimes  f(x)\leq f(r\otimes  x)$. This proves the necessity.
		
		Now we prove the sufficiency. Let  $x, y\in X$. Since $X(x, y)\leq X(x, y)$, it follows that  $X(x, y)\otimes  x\leq y$, then \[X(x, y)\otimes  f(x)\leq f(X(x, y)\otimes  x)\leq f(y),\] and consequently $X(x,y)\leq Y(f(x),f(y))$,  so  $f$ is a functor. \end{proof}

	A weight $\phi$ of $X$ is   \emph{conical} if there is a functor $f\colon K\lra X$ such that $\phi=\thinspace!_*\circ f^*$, where $!_*$ is the graph of the unique functor   $!\colon K\lra\star$. Dually, a coweight $\psi$ of $X $ is   \emph{conical} if there is a functor $f\colon K\lra X$ such that $\psi=  f_*\circ\thinspace !^*$, where $!^*$ is the cograph  the unique functor   $!\colon K\lra\star$. It is easily verified that a weight $\phi$ is conical if and only if \[\phi=\sup_{a\in A}\sy(a)\] for a subset $A$ of $X$; a coweight $\psi$ is conical if and only if   \[\psi=\sup_{b\in B}\syd(b)\] for a subset $B$ of $X$.  
	
	\begin{defn} {\rm (Kelly \cite{Kelly}, Stubbe \cite{Stubbe2006})} Suppose  $X$ is a real-enriched category. We say that  
		\begin{enumerate}[label={\rm(\roman*)}]  
			\item $X$ is  conically cocomplete  if every conical weight of $X$ has a colimit. 
			\item $X$ is  conically  complete  if every conical coweight of $X$ has a limit. 
			\item  $X$ is order-complete if the underlying ordered set   $X_0$ is complete. \end{enumerate} \end{defn}
	
	\begin{thm}\label{characterizations of complete Q-orders} {\rm (Kelly \cite{Kelly}, Stubbe \cite{Stubbe2006})} For each real-enriched category $X$, the following statements are equivalent: \begin{enumerate}[label={\rm(\arabic*)}]  
			\item  $X$ is cocomplete. 
			\item  $X$ is complete. 
			\item $X$ is tensored and  conically cocomplete. \item $X$ is cotensored and  conically  complete. 
			\item $X$ is order-complete, tensored and  cotensored.
	\end{enumerate} \end{thm}
	
	\begin{proof} By duality and Theorem \ref{complete is self-dual} we only need to check  that (1), (3) and (5) are equivalent.
		
		$(1)\Rightarrow(3)$ Obvious.
		
		$(3)\Rightarrow(5)$ First we show that $X$ is order-complete. Suppose $A$ is a subset of $X$.  We show that each colimit of the conical weight $\phi= \sup_{x\in A}\sy(x)$   is a join of $A$ in $X_0$.  For all $x\in A$,   we have \begin{align*}X(x,\colim\phi)&=   X(\colim\phi,-)\searrow X(x,-) & (\text{Yoneda lemma}) \\ 
			&= (X\swarrow\phi)\searrow X(x,-) & (\text{definition of $\colim\phi$}) \\ 
			&\geq (X\swarrow X(-,x))\searrow X(x,-) &(x\in A)\\ &= X(x,-)\searrow X(x,-) \\ 
			&=1,
		\end{align*} which shows that $\colim\phi$ is an upper bound of $A$ in $X_0$. Suppose that $y$ is another upper bound of $A$ in $X_0$. Then \begin{align*}X(\colim\phi,y) 
			&= X(-,y)\swarrow\phi \\ 
			&=\inf_{x\in A}(X(-,y)\swarrow X(-,x))\\ 
			& = \inf_{x\in A} X(x,y)\\ 
			& = 1, 
		\end{align*} hence $\colim\phi\sqsubseteq y$ in $X_0$. Therefore, $\colim\phi$ is a join of $A$ in $X_0$.
		
		Next we show that $X$ is cotensored. For each $r\in[0,1]$ and $y\in X$, let $$ A=\{a\in X\mid  r\leq X(a,y)\}.$$ We show that the colimit of the conical weight $\phi=\bv_{a\in A}\sy(a)$ is the cotensor of $r$ with $y$; that is, $X(x,\colim\phi)=r\ra X(x,y)$ for all $x\in X$. For all  $s\in[0,1]$, \begin{align*} s\leq r\ra X(x,y) &\iff r\leq s\ra X(x,y) \\ 
			&\iff r\leq X(s\otimes x, y) \\ 
			& \iff  s\otimes x\sqsubseteq\colim\phi &(\text{$\colim\phi$ is a join of $A$ in $X_0$})\\ &\iff s\leq X(x,\colim\phi), 
		\end{align*} hence $\colim\phi$ is the cotensor of $r$ with $y$.
		
		$(5)\Rightarrow(1)$ 
		For a weight $\phi$ of $X$, let $a$ be the join in $X_0$ of the set $A=\{\phi(x)\otimes  x\mid x\in X\}$. We   show that $a$ is a colimit of $\phi$, hence $X$ is cocomplete.   
		For all $y\in X$, since $a$ is a join of $A$ in $X_0$, then 
		\begin{align*}X(a,y)& =\inf_{z\in A}X(z,y) &(\text{Proposition \ref{chara of cotensored}\thinspace(3)})\\ 
			& = \inf_{z\in A}(X(-,y)\swarrow X(-,z)) &(\text{Yoneda lemma}) \\ 
			&= X(-,y)\swarrow\sup_{z\in A}X(-,z)\\  
			& =X(-,y)\swarrow\sup_{x\in X}X(-,\phi(x)\otimes  x)\\  
			& = \inf_{x\in X}(X(-,y)\swarrow X(-,\phi(x)\otimes  x))   \\ 
			& = \inf_{x\in X}   X(\phi(x)\otimes  x,y) &(\text{Yoneda lemma}) \\ 
			& = \inf_{x\in X} (\phi(x)\ra  X(x,y)) &(\text{definition of tensor})\\ &= \CP X (\phi,\sy(y)),
		\end{align*} hence $a$ is a colimit of $\phi$. 
	\end{proof}
	
	\begin{cor}\label{left adjoint via tensor and join} Let $X,Y$ be cocomplete real-enriched categories. Then  a functor $f\colon X\lra Y$ is a left adjoint if and only if 
		\begin{enumerate}[label={\rm(\roman*)}] 
			\item $f\colon X_0\lra Y_0$ preserves joins;     \item $f$ preserves tensors in the sense that $f(r\otimes  x)=r\otimes  f(x)$ for all $r\in [0,1]$ and $x\in X$. 
	\end{enumerate}   \end{cor}
	
	\begin{proof} Necessity is trivial since tensors are a special kind of colimits. For sufficiency, let $g\colon Y\lra X$ be the map such that $g\colon Y_0\lra X_0$ is right adjoint to $f\colon X_0\lra Y_0$. We claim that $g\colon Y\lra X$ is a right adjoint of $f\colon X\lra Y$. By Theorem \ref{Characterization of adjoints} and Proposition \ref{functor via tensors}, we only need to check that $r\otimes  g(y)\leq g(r\otimes  y)$ for all $r\in [0,1]$ and $y\in Y$. Since \[f(r\otimes  g(y))=r\otimes  f(g(y))\leq r\otimes  y,\] it follows that $r\otimes  g(y)\leq g(r\otimes  y)$, as desired.
	\end{proof}
	
	\begin{cor}\label{calculation of sup by tensors} Let $X$ be a complete real-enriched category. Then for each functor  $f\colon K\lra X$ and each weight $\phi$ of $ K$, the colimit of $f$ weighted by $\phi$  is given by  the join of  $\{\phi(z)\otimes  f(z)\mid z\in K\}$ in $X_0$. Dually,  for each coweight $\psi$ of $K$, the limit of $f$ weighted by $\psi$ is given by  the meet of  $\{\psi(z)\multimap f(z)\mid z\in K\}$ in $X_0$.  \end{cor}
	
	\begin{proof}  Since $\phi\circ f^*$ is the join  in   $(\CP X)_0$ of the subset $\{\phi(z)\with X(-,f(z))\mid z\in K\},$   the functor $\colim\colon\CP X\lra X$ preserves colimits, it follows that \[{\colim}_\phi f =\colim\phi\circ f^* = \sup_{z\in K}\colim(\phi(z)\with X(-,f(z))) = \sup_{z\in K} \phi(z)\otimes  f(z) . \]
		
		The second conclusion follows by duality. \end{proof} 
	
	\begin{cor}\label{calculation of colimit in PX} Suppose $X$ is a real-enriched category and $f\colon K\lra \CP X$ is a functor.  
		\begin{enumerate}[label={\rm(\roman*)}] 
			\item For each weight $\phi$ of $K$, $\colim_\phi f=\sup_{z\in K}\phi(z)\with f(z)$. 
			\item For each coweight $\psi$ of $K$, $\lim_\psi f=\inf_{z\in K}(\psi(z)\ra f(z))$. 
	\end{enumerate} \end{cor}
	
	Since $\CP\star=\sV$, Example \ref{cpt as colimit} and Example \ref{sub as limit} are then   special case of the above corollary.
	The following conclusion, proved in  Lai and Zhang \cite{LaiZ09}, extends the  Tarski fixed point theorem, which says that every order-preserving map from a complete lattice to itself has a fixed point, to the enriched context.
	
	\begin{thm}  If $X$ is a complete real-enriched category and $f\colon X\lra X$ is a functor, then the subcategory \[{\rm Fix}(f)=\{x\in X\mid f(x)\cong x\}\] of $X$ composed of fixed points of $f$  is   complete. \end{thm}
	
	\begin{proof} First  we show that the subcategory \[Y=\{x\in X\mid x\sqsubseteq f(x) \}\] of $X$ is complete. Let $i\colon Y\lra X$ denote the inclusion. Then for each weight $\phi$ of $Y$, the join  \[a=\bv\{\phi(x)\otimes  x\mid x\in Y\}\] in $X_0$ is a  colimit of the functor $i$ weighted by $\phi$. Since $f$ is a  functor,   \[f(a)\sqsupseteq \bv\{f(\phi(x)\otimes  x)\mid x\in Y\}\sqsupseteq \bv\{ \phi(x)\otimes  f(x)\mid x\in Y\}\sqsupseteq a,\] which shows that $a\in Y$. Therefore, $Y$ is closed in $X$ under formation of colimits. In particular, $Y$ is cocomplete, hence complete.
		
		Next, restricting the domain and codomain of $f$ to $Y$ we obtain a  functor $f\colon  Y\lra Y$. A similar argument shows that \[Z=\{y\in Y\mid f(y)\sqsubseteq y\}\] is closed in $Y$ under formation of limits, hence a complete real-enriched category.
		It is clear that ${\rm Fix}(f)=Z$, the conclusion thus follows. \end{proof}
	
	\begin{defn} Suppose $X$ is a partially ordered set. A (left) action of $[0,1]$  on $X$ is a map $\otimes\colon[0,1]\times X \lra X$ subject to the following conditions:  for all $x\in X$ and $r,s, \in [0,1]$, 
		\begin{enumerate}[label={\rm(\roman*)}] 
			\item $1\otimes  x=x$; 
			\item $s\otimes (r\otimes x) =(s\with r)\otimes  x$; 
			\item $r\otimes-\colon X\lra X$ is a left adjoint;   \item $-\otimes  x\colon [0,1]\lra X$ is a left adjoint.
		\end{enumerate}
	\end{defn}
	
	It is clear that $X$ has a bottom element $\bot$ and that $0\otimes x=\bot=x\otimes 0$ for all $x\in X$. By (iii),  the map $0\otimes-\colon X\lra X$ has a right adjoint, which is denoted by $0\multimap-\colon X\lra X$. Then, for each $x\in X$,  $0\multimap x$ is  a top element of $X$. So, the partially ordered set $X$ is bounded, i.e., it has both a bottom element and a top element.
	
	A \emph{lax morphism} $f\colon (X,\otimes )\lra(Y,\otimes )$ between   $[0,1]$-actions is an order-preserving map $f\colon X\lra Y$ such that \[r\otimes  f(x)\leq f(r\otimes  x)\] for all $r\in [0,1]$ and $x\in X$. \[\bfig\Square[{[0,1]}\times X`X`{[0,1]}\times Y`Y;\otimes `1\times f`f`\otimes ] \place(300,240)[\leq]\efig\]
	
	Let $X$ be   a separated, tensored and cotensored real-enriched category  $X$. Assigning to each pair $(r,x)$   the tensor $r\otimes  x$ defines a  $[0,1]$-action on the partially ordered set $X_0$. Proposition \ref{functor via tensors} says that   $X\mapsto(X_0,\otimes )$ is a functor from the full subcategory   of $\QOrd $ consisting of separated, tensored and cotensored real-enriched categories to the category of $[0,1]$-actions and lax morphisms. This functor   is indeed an isomorphism of categories.
	
	\begin{prop}Suppose $\otimes$ is an action of  $[0,1]$  on  a partially ordered set $(X,\leq)$. For all $x,y\in X$, let \[\alpha_{\otimes}(x,y)=\bv\{r\in [0,1]\mid r\otimes  x\leq y\}. \] Then $(X,\alpha_{\otimes})$ is a separated, tensored and cotensored real-enriched category with $\leq$ being its underlying order.   \end{prop}
	
	\begin{proof}First of all,    it is clear that for all $r\in [0,1]$ and $x,y\in X$, \[r\leq  \alpha_{\otimes}(x,y)\iff r\otimes  x\leq y.\]
		
		Since $1\otimes  x=x$, then $\alpha_{\otimes}(x,x)\geq 1$ for all $x\in X$. If $r\leq \alpha_{\otimes}(x,y)$ and $s\leq \alpha_{\otimes}(y,z)$, then \[(s\with r)\otimes  x=s\otimes (r\otimes  x)\leq s\otimes  y\leq z,\] which shows that $s\with r\leq \alpha_{\otimes}(x,z)$, hence $\alpha_{\otimes}(y,z)\with \alpha_{\otimes}(x,y)\leq \alpha_{\otimes}(x,z)$ and consequently, $(X,\alpha_{\otimes})$ is a real-enriched category. It remains to show that $(X,\alpha_{\otimes})$ is   tensored, cotensored,   separated,  and has $\leq$ as underlying order
		
		For all $r,s\in [0,1]$ and $x,y\in X$, since \begin{align*} s\leq \alpha_{\otimes}(r\otimes  x, y) &\iff s\otimes (r\otimes  x)\leq y\\ 
			&\iff s\with r\leq \alpha_{\otimes}(x, y) \\ 
			&\iff s\leq r\ra\alpha_{\otimes}(x, y)  , \end{align*} then $r\otimes  x$ is a tensor of $r$ with $x$ in $(X,\alpha_{\otimes})$, hence $(X,\alpha_{\otimes})$ is  tensored.
		
		For each $r\in [0,1]$, let $r\multimap-\colon X\lra X$ be the right adjoint of $r\otimes-\colon X\lra X$.  
		For all $x,y\in X$ and all $s\in [0,1]$, since  \begin{align*} s\leq X(x,r\multimap y)&\iff s\otimes  x\leq r\multimap y\\ 
			&\iff  r\otimes (s\otimes  x)\leq y \\ 
			&\iff (r\with s)\otimes  x\leq y\\ 
			&\iff r\with s \leq X(x,y)\\ 
			&\iff s\leq r\ra X(x,y),
		\end{align*} then    
		$r\multimap y$ is a cotensor of $r$ with $y$ in $(X,\alpha_{\otimes})$, hence $(X,\alpha_{\otimes})$ is  cotensored.
		
		Finally, since \[1\leq \alpha_{\otimes}(x, y)\iff 1\otimes  x\leq y\iff x\leq y,\] it follows that $(X,\alpha_{\otimes})$ is separated with $\leq $ being its underlying order. \end{proof}
	
	\begin{prop} The category of  $[0,1]$-actions and lax morphisms is isomorphic to the full subcategory of $\QOrd$ consisting of separated, tensored and cotensored real-enriched categories.
	\end{prop}
	
	\section{Cauchy completeness} \label{cauchy complete}  
	
	A weight $\phi\colon X\oto\star$ of a real-enriched category  $X$ is \emph{Cauchy} if it is a right adjoint, as a distributor. That means, there is a distributor $\psi\colon\star\oto X$ such that $\phi\circ\psi\geq1$ and $\psi(y)\with\phi(x)\leq X(x,y)$ for all $x,y\in X$.
	
	Let $f\colon X\lra Y$ be a functor between real-enriched categories and let $\phi$ be a weight of $X$.  If $\phi$ is a Cauchy weight,  as a composite of right adjoints, $\phi\circ f^*$ is a right adjoint, hence a Cauchy weight of $Y$. So,  assigning to each real-enriched category $X$ the subcategory $\CC X$ of $\CP X$ composed of Cauchy weights defines a subfunctor $$\CC\colon\QOrd\lra\QOrd$$ of the presheaf functor $\CP$. Since each representable  weight  is Cauchy, the natural transformation  from the identity functor to $\mathcal{P}$ with Yoneda embeddings as components  factors through the functor $\CC$.
	
	The following conclusion is a special case of the characterization of Cauchy weights of enriched categories in Street \cite{Street83}. 
	
	\begin{prop}\label{absolute sup} Suppose $X$ is a real-enriched category and $\phi$ is a weight of $X$. The following are equivalent: 
		\begin{enumerate}[label={\rm(\arabic*)}]   
			\item $\phi$ is Cauchy. 
			\item The colimit $\colim_\phi f$ of any functor $f\colon X\lra Y$ weighted by $\phi$, when exists, is preserved by every functor $g\colon Y\lra Z$, i.e.  $g(\colim_\phi f)=\colim_\phi(g\circ f)$. \end{enumerate} In particular, every Cauchy weight with a colimit  is representable.
	\end{prop}
	
	\begin{proof} $(1)\Rightarrow(2)$ Since $\colim_\phi f$, when exists, is the colimit of the Cauchy weight $\phi\circ f^*$ of $Y$, it suffices to show that if $b$ is a colimit of a Cauchy weight $\xi$ of $Y$, then for every functor $g\colon Y\lra Z$, $g(b)$ is a colimit of $g$ weighted by $\xi$, i.e. $g(b)=\colim_\xi g$.  By definition of colimit  we have  $Y\swarrow\xi=Y(b,-)$. Since $Y\swarrow\xi$ is   left adjoint to  $\xi$ by Lemma \ref{adjoint_arrow_calculation}, then $$\xi=Y(b,-)\searrow Y =Y(-,b).$$  This  shows  in particular that $\xi$ is representable. Since \[  \xi\circ g^*= Y(-,b)\circ g^* = Z(-,g(b)),\] it follows that $g(b)$ is a colimit of $\xi\circ g^*$, i.e. $g(b)=\colim_\xi g$. 
		
		$(2)\Rightarrow(1)$ We  show that the coweight $\psi\coloneqq X\swarrow\phi$ is left adjoint to $\phi$. Consider the coYoneda embedding $\syd\colon X\lra\CPd X$ and the functor $$g \colon \CPd X\lra\sV, \quad  \xi\mapsto\CPd X(\psi,\xi),$$ where $\sV=([0,1],\alpha_L)$. On the one hand, by Example \ref{sup in PdX} we have $${\colim}_\phi\thinspace \syd= X\swarrow\phi=\psi,$$    
		so $g(\colim_\phi\thinspace \syd)=1$. On the other hand, since for each $x\in X$, $$g\circ \syd(x) =\CPd X(\psi,X(x,-))=\psi(x),$$ i.e. $g\circ \syd=\psi^\sharp \colon X\lra\sV$, then by Example \ref{cpt as colimit} we have $$ {\colim}_\phi(g\circ \syd)={\colim}_\phi\thinspace\psi^\sharp =\phi\circ\psi.$$ Therefore, $\phi\circ\psi=1$. The inequality $\psi\circ\phi\leq X$ is trivial, so $\psi$ is left adjoint to $\phi$, as desired. \end{proof}
	
	The next characterization of Cauchy weights  is a special case of that   in  Kelly and  Schmitt \cite{KS2005}.
	
	\begin{prop}\label{Cauchy weight via sub}  
		Suppose $X$ is a real-enriched category and $\phi$ is a weight of $X$. The following are equivalent: \begin{enumerate}[label={\rm(\arabic*)}]   
			\item $\phi$ is Cauchy. 
			\item The functor $\sub_X(\phi,-)\colon \CP X \lra\sV$ preserves  colimits. 
			\item The functor $\phi\circ-\colon (\CPd X)^{\rm op} \lra\sV$ preserves limits. 
	\end{enumerate} \end{prop}
	
	\begin{proof}
		$(1)\Rightarrow(2)$ If $\psi$ is a left adjoint of $\phi$, then for all $\xi\in\CP X$, $$\sub_X(\phi,\xi)=\xi\swarrow\phi=\xi\circ\psi,$$ from which one deduces that $\sub_X(\phi,-)\colon \CP X \lra\sV$ preserves colimits. 
		
		$(2)\Rightarrow(1)$ Assume that the functor  $\sub_X(\phi,-)$ preserves colimits. We show that $\psi\coloneqq X\swarrow\phi$   is left adjoint to $\phi$, hence $\phi$ is Cauchy. It suffices to check that $\phi\circ\psi\geq1$.   Since $\sub_X(\phi,-)$ preserves colimits and  $$\phi={\colim}_\phi\thinspace\sy=\sup_{x\in X}(\phi(x)\otimes\sy(x)),$$ it follows that \begin{align*}\phi\circ\psi &=\sup_{x\in X}\phi(x)\with\psi(x)\\ 
			&=\sup_{x\in X} (\phi(x)\with \sub_X(\phi,  \sy(x))\\ 
			& = \sub_X(\phi,\colim_\phi\sy) \quad \quad ~(\text{$\sub_X(\phi,-)$ preserves colimits})\\ 
			& =\sub_X(\phi,\phi) \\ 
			& =1.\end{align*}   
		
		$(1)\Rightarrow(3)$ If $\psi$ is a left adjoint of $\phi$, then for all $\lam\in\CPd X$, $$ \phi\circ\lam =\psi\searrow\lam=\CPd X(\lam,\psi),$$ from which one infers that   $\phi\circ-\colon (\CPd X)^{\rm op} \lra\sV$ preserves limits.
		
		$(3)\Rightarrow(1)$ Assume that the functor $\phi\circ-\colon (\CPd X)^{\rm op} \lra\sV$ preserves limits. We show that $\psi\coloneqq X\swarrow\phi$   is left adjoint to $\phi$. It suffices to check that $\phi\circ\psi\geq1$. Since $$X\swarrow\phi =\bw_{y\in X}(\phi(y)\ra X(y,-)) $$ and $\phi\circ -\colon (\CPd X)^{\rm op} \lra\sV$ preserves limits, it follows that \begin{align*}\phi\circ\psi &= \bw_{y\in X}(\phi(y)\ra \phi\circ X(y,-))= \bw_{y\in X}(\phi(y)\ra  \phi(y))=1.  \end{align*} 
		
		The proof is completed. \end{proof}
	
	\begin{defn}A real-enriched category $X$ is  Cauchy complete  if it is separated and all of its  Cauchy weights are representable. \end{defn}
	
	In other words, $X$ is Cauchy complete if each of its Cauchy weight has a unique colimit. Suppose $\phi\colon X\oto\star$ is a Cauchy weight of $X$, and $\psi\colon\star\oto X$ is a left adjoint of $\phi$. Then $\psi^{\rm op}\colon X^{\rm op}\oto\star$ is a Cauchy weight of $X^{\rm op}$. Since all Cauchy weights of $X^{\rm op}$ arise in this way, it follows that the notion of Cauchy completeness is self-dual; that is, if $X$ is Cauchy complete, then so is its opposite $X^{\rm op}$. 
	
	Each  functor $f\colon X\lra Y$ generates an adjunction of distributors  $f_*\dashv f^*$. The following proposition shows that when $Y$ is Cauchy complete,   every such adjunction  arises in this way.
	
	\begin{prop}{\rm(Lawvere \cite{Lawvere1973})} Suppose $Y$ is a real-enriched category. The following are equivalent: 
		\begin{enumerate}[label={\rm(\arabic*)}] 
			\item $Y$ is Cauchy complete. 
			\item For any pair of distributors $\phi\colon X\oto Y$ and $\psi\colon Y\oto X$, if $\phi$ is left adjoint to $\psi$, then there is a unique functor  $f\colon X\lra Y$ such that $\phi=f_*$ and $\psi=f^*$.     
	\end{enumerate} \end{prop}
	
	\begin{proof} It suffices to observe that if $\phi\colon X\oto Y$ is left adjoint to $\psi\colon Y\oto X$, then for each $x\in X$,   the distributor $\phi(x,-)\colon\star\oto Y$   is left adjoint to  $\psi(-,x)\colon Y\oto\star$. \end{proof} 
	
	For a real-enriched category $X$, let $\CC X$ be the subcategory of $\CP X$ composed of Cauchy weights. Let $\mathfrak{i}\colon\CC X\lra\CP X$ be the inclusion functor, and let ${\sf c}\colon X\lra  \mathcal{C}X$ be the Yoneda embedding with codomain restricted to $\mathcal{C}X$. The composite $ \mathfrak{i}\circ{\sf c}$ is  then the Yoneda embedding $\sy\colon X\lra\CP X$.  
	
	\begin{lem}\label{X is isomorphic to CX} {\rm (Stubbe \cite{Stubbe2005})} Suppose $X$ is a real-enriched category. Then $${\sf c}_*\circ{\sf c}^*=\CC X\quad \text{and}\quad {\sf c}^*\circ {\sf c}_*=X.$$  Therefore, $X$ is isomorphic to $\CC X$ in the category of distributors.     \end{lem}    
	
	\begin{proof} The equality ${\sf c}^*\circ {\sf c}_*=X$ follows from that ${\sf c}   \colon X\lra \CC X$ is fully faithful. For the equality ${\sf c}_*\circ{\sf c}^*=\CC X$, let $\phi,\psi\in\CC X$. Since  $X\swarrow \phi$ is left adjoint to $\phi$ by Lemma \ref{adjoint_arrow_calculation}, then   \begin{align*}{\sf c}_*\circ{\sf c}^*(\phi,\psi)
			&= \sup_{x\in X} {\sf c}_*(x,\psi)\with{\sf c}^*(\phi,x) \\ 
			&= \sup_{x\in X}\CC X(\sy  (x),\psi)\with\CC X(\phi,\sy  (x))\\ 
			&=\sup_{x\in X}\psi(x)\with(X(-,x)\swarrow \phi) \\ &= \sup_{x\in X}\psi(x)\with(X\swarrow\phi)(x)   \\  &= \psi\circ(X\swarrow\phi)  \\ 
			&= \psi\swarrow\phi  \\ 
			& =\CC X(\phi,\psi).  \qedhere\end{align*} 
	\end{proof}
	
	\begin{thm}\label{univeral property of CX} {\rm (Kelly \cite{Kelly})} Suppose  $X$ is a real-enriched category. 
		\begin{enumerate}[label={\rm(\roman*)}] 
			\item   $\CC X$ is Cauchy complete. 
			\item For each functor $f\colon X\lra Y$ with $Y$ Cauchy complete, there is a unique functor $\overline{f}\colon\CC X\lra Y$ that extends $f$. Furthermore, if $f$ is fully faithful, then so is $\overline{f}$. 
	\end{enumerate} \end{thm}
	
	\begin{proof}  
		(i) We  show that for each Cauchy weight $\Phi$ of $\CC X$, the colimit $\phi$ of the inclusion functor  $\mathfrak{i}\colon \CC X\lra\CP X$ weighted by $\Phi$ is a Cauchy weight of $X$, which implies that $X$ is Cauchy complete. In other words, we show that $\CC X$ is closed in $\CP X$ under formation of colimits of Cauchy weights.   Since the composite $\mathfrak{i}\circ{\sf c}$ is the Yoneda embbeding, from Example \ref{sup in PX} it follows that $$\phi={\colim}_\Phi\mathfrak{i} =\colim \Phi\circ\mathfrak{i}^*= \Phi\circ \mathfrak{i}^*\circ\mathfrak{i}_*\circ {\sf c}_*= \Phi\circ{\sf c}_*.$$  
		Let $\psi$ be the coweight  \[ {\sf c}^*\circ \Psi\colon *\oto\CC X \oto X,\] where
		$\Psi\colon \star\oto\CC X$ is the left adjoint of $\Phi\colon\CC X\oto\star$. We claim that $\psi$ is a left adjoint of $\phi$, hence $\phi$ is Cauchy. 
		
		On the one hand, since ${\sf c}\colon X\lra\CC X$ is fully faithful, we have \[\psi\circ\phi = {\sf c}^* \circ\Psi\circ \Phi\circ   {\sf c}_*\leq {\sf c}^*\circ {\sf c}_*= X.  \] On the other hand, by  Lemma \ref{X is isomorphic to CX}, we have  \[\phi\circ\psi = \Phi\circ {\sf c}_*\circ {\sf c}^* \circ\Psi=   \Phi\circ  \Psi   \geq 1.\]  Therefore, $\psi$ is   left adjoint to $\phi$, as claimed.
		
		(ii) It is readily verified that the functor \[\overline{f}\colon \CC X\lra Y, \quad \phi\mapsto {\colim}_\phi f \]   extends $f$. Uniqueness follows from that  each Cauchy weight $\phi$ is the colimit of ${\sf c}\colon X\lra\CC X$ weighted by $\phi$ and that colimits of Cauchy weights are preserved by all functors.
		
		Finally, we check that if $f$ is fully faithful, then  so is $\overline{f}$; that is, $\CC X(\phi_1,\phi_2)= Y(\overline{f}(\phi_1), \overline{f}(\phi_2))$ for all Cauchy weights $\phi_1,\phi_2$ of $X$.  Since \begin{align*}\CP X(\phi_1,\phi_2) &\leq \CP Y(\phi_1\circ f^*,\phi_2\circ f^*)\\ 
			& \leq \CP X(\phi_1\circ f^*\circ f_*,\phi_2\circ f^*\circ f_*)\\ 
			&=\CP X(\phi_1,\phi_2), 
		\end{align*} it follows that $$\CC X (\phi_1,\phi_2) = \CP Y(\phi_1\circ f^*,\phi_2\circ f^*) = Y(\overline{f}(\phi_1), \overline{f}(\phi_2)),$$ where the last equality holds because for $i=1,2$, the weight $\phi_i\circ f^*$ is Cauchy, hence represented by $\overline{f}(\phi_i)$.
	\end{proof} 
	
	An immediate corollary of the above theorem is that the full subcategory of $\QOrd$ composed of Cauchy complete real-enriched categories is reflective, with reflection given by $\sy\colon X\lra\CC X$. 
	The real-enriched category $\CC X$ is called the  \emph{Cauchy completion}  of $X$.

	\begin{cor}Cauchy completion is idempotent in the sense that $\CC X=\CC(\CC X)$ for every real-enriched category $X$. \end{cor}
	
	\begin{proof}Since $\CC X$ is Cauchy complete,   all of its Cauchy weights are representable, then the conclusion follows. \end{proof}  
	
	The following characterization of Cauchy completion is taken from Hofmann \cite{Hof2013}, Hofmann and Tholen \cite{HT2010}.
	\begin{prop} \label{CC as equaliser} For each real-enriched category $X$, the Cauchy completion of $X$ is the equalizer in $\QOrd$  of the parallel pair  \begin{equation*}\bfig
			\morphism(0,0)|a|/@{->}@<3pt>/<500,0>[\CP X`\CP\CP X;\CP\sy_X]
			\morphism(0,0)|b|/@{->}@<-3pt>/<500,0>[\CP X`\CP \CP X; \sy_{\CP X} ] 
			\efig\end{equation*} \end{prop}
	
	\begin{proof} We show that a weight $\phi$ of $X$ is Cauchy if and only if $\CP\sy_X(\phi)=\sy_{\CP X}(\phi)$. 
		
		For sufficiency consider the coweight $\psi=X\swarrow\phi=\sy_X^*\swarrow\CP X(-,\phi)$.   We wish to show that $\phi$ is right adjoint to $\psi$, hence a Cauchy weight.  Since $\CP\sy_X(\phi)=\sy_{\CP X}(\phi)=\CP X(-,\phi)$, then 
		\begin{align*} \phi\circ\psi &=\phi\circ  (X\swarrow\phi)\\ 
			& =\phi\circ(\sy_X^*\swarrow\CP X(-,\phi))\\ 
			&= \phi\circ \sy_X^*\circ \CP X(\phi,-) &(\CP X(\phi,-)\dashv \CP X(-,\phi))\\ 
			&=  \CP\sy_X(\phi)\circ \CP X(\phi,-)  \\ 
			& = \CP X(-,\phi)\circ \CP X(\phi,-)  \\ 
			&=1. \end{align*} The inequality  $\psi\circ\phi\leq X$ is trivial. Thus, $\phi$ is right adjoint to $\psi$. 
		
		For necessity suppose $\phi$ is a Cauchy weight. Then $\psi\coloneqq X\swarrow\phi$ is left adjoint to $\phi$, hence $$ 1=\phi\circ\psi=\phi\circ \sy_X^*\circ\CP X(\phi,-)= \CP\sy_X(\phi)\circ \CP X(\phi,-),$$ therefore \begin{align*}\sy_{\CP X}(\phi) &= \CP\sy_X(\phi)\circ \CP X(\phi,-)\circ \CP X(-,\phi)= \CP\sy_X(\phi). \qedhere \end{align*}
	\end{proof}
	
	In the following we   characterize Cauchy completeness in terms of topological property.
	Let $X$ be a real-enriched category. For each $x\in X$ and $r<1$, the set $$B(x,r)\coloneqq\{y\in X\mid X(x,y)>r\}$$ is called the \emph{open ball} of $X$ with center $x$ and radius $r$. The collection $\{B(x,r)\mid x\in X, r<1\}$  is a base for a topology on $X$, the resulting topology is called the \emph{open ball topology} of $X$. 
	
	For each real-enriched category $(X,\alpha)$, the  \emph{symmetrization}  of  $(X,\alpha)$  refers to the symmetric (in an evident sense) real-enriched category $(X,S(\alpha))$, where $S(\alpha)(x,y)=\min\{\alpha(x,y),\alpha(y,x)\}$. 
	
	It is clear that the open ball topology of the symmetrization of $X$ is the least common refinement of the open ball topology of $X$ and that of $X^{\rm op}$.
	
	\begin{prop} \label{limit in open ball top} A net $\{x_i\}_{i\in D}$ of a real-enriched category $X$ converges to $x$ w.r.t the open ball topology if and only if $\bv_{i\in D}\inf_{j\geq i}X(x,x_j)=1$. \end{prop}

	Assigning open ball topology to real-enriched categories yields  a functor from   real-enriched categories to  topological spaces. 
	
	\begin{defn}
		Suppose $\{x_i\}_{i\in D}$ is a net and $a$ is an element of a real-enriched category $X$.  We say that 
		\begin{enumerate}[label=\rm(\roman*)]   
			\item $\{x_i\}_{i\in D}$   is   Cauchy (also called  biCauchy) if \[\sup_{i\in D}\inf_{j,k\geq i}X(x_j,x_k)= 1.\]  
			\item $a$ is a bilimit  of $\{x_i\}_{i\in D}$ if for all $x\in X$, \[\sup_{i\in D}\inf_{j\geq i}X(x,x_j)=X(x,a),\quad  \sup_{i\in D}\inf_{j\geq i}X(x_j,x)=X(a,x).  \]  
	\end{enumerate} \end{defn}

	It is clear that a net has at most one bilimit up to isomorphism.
	
	\begin{lem}\label{bilimit of Cauchy net} Let $\{x_i\}_{i\in D}$ be a Cauchy net and $a$ be an element of a real-enriched category $X$. The  following   are equivalent:  
		\begin{enumerate}[label={\rm(\arabic*)}]   
			\item $a$ is a bilimit of $\{x_i\}_{i\in D}$.  
			\item   $\sup_{i\in D}\inf_{j\geq i}X(a,x_j)= 1$ and $ \sup_{i\in D}\inf_{j\geq i}X(x_j,a)= 1$. 
			\item  $\{x_i\}_{i\in D}$ converges to $a$ in the open ball topology of the symmetrization of $X$. \end{enumerate} \end{lem} 
	
	\begin{proof} $(1)\Rightarrow(2)$   trivial.
		
		$(2)\Leftrightarrow(3)$ This follows from Proposition \ref{limit in open ball top} and  the fact  that the requirement in (2) is equivalent to that $$\sup_{i\in D}\inf_{j\geq i}\min\{X(a,x_j),X(x_j,a)\}= 1.$$
		
		$(2)\Rightarrow(1)$ For all $x\in X$, we calculate: \begin{align*} X(x,a)&= \Big(\sup_{i\in D}\inf_{j\geq i}X(a,x_j)\Big)\with X(x,a) \\ &\leq  \sup_{i\in D}\inf_{j\geq i} X(a,x_j) \with X(x,a)  \\ &\leq  \sup_{i\in D}\inf_{j\geq i}X(x,x_j)   \\ &=  \Big( \sup_{i\in D}\inf_{j\geq i}X(x_j,a)\Big) \with \Big(\sup_{i\in D}\inf_{j\geq i}X(x,x_j)\Big) \\ &= \sup_{i\in D}\inf_{j\geq i} X(x_j,a)\with X(x,x_j)    \\ &\leq X(x,a),\end{align*} which implies   $\sup_{i\in D}\inf_{j\geq i}X(x,x_j)=X(x,a)$.  The other equality is verified likewise. So  $a$ is a bilimit of $\{x_i\}_{i\in D}$.  \end{proof}

	Proposition \ref{limit in open ball top}  implies that if a net of $X$ converges in the open ball topology of the symmetrization of $X$, then it is a Cauchy net of $X$.

	\begin{exmp}Let $\with$ be the product t-norm. Consider the real-enriched category $\sV=([0,1],\alpha_L)$. Then, a sequence $\{x_n\}_{n\in\mathbb{N}}$  is Cauchy if and only if it is  eventually constant with value $0$ or converges in the usual sense to some point other than $0$. Likewise for   $\sV^{\rm op}=([0,1],\alpha_R)$. \end{exmp}

	\begin{prop} \label{colimit =bilimit for Cauchy net}  
		For each real-enriched category $X$, the  following are equivalent: \begin{enumerate}[label={\rm(\arabic*)}]   \item Every Cauchy weight of $X$ has a colimit. \item Every Cauchy net of $X$ has a bilimit. \item  Every Cauchy sequence of $X$ has a bilimit.  \end{enumerate} \end{prop}
	
	\begin{proof}$(1)\Rightarrow(2)$ Suppose   $\{x_i\}_{i\in D}$ is a Cauchy net of $X$. Then the weight \[\phi =\sup_{i\in D}\inf_{j\geq i}X(-,x_j)\] is Cauchy with a left adjoint given by the coweight  \[ \psi=\sup_{i\in D}\inf_{j\geq i}X(x_j,-). \]   Actually, a slightly stronger conclusion will be proved in Proposition \ref{Cauchy net implies Cauchy weight}. Let $a$ be a colimit of the Cauchy weight $\phi$. Then $\phi(x)=X(x,a)$ and $\psi(x)=X(a,x)$  for all $x\in X$, which implies that $a$ is a bilimit of $\{x_i\}_{i\in D}$.  
		
		$(2)\Rightarrow(3)$ Trivial.  
		
		$(3)\Rightarrow(1)$  
		Suppose $\phi\colon X\oto\star$ is a Cauchy weight, with a left adjoint  $\psi\colon\star\oto X$. Then, $\sup_{x\in X} \phi(x)\with \psi(x) \geq 1$ and $\psi(y)\with\phi(x)\leq X(x,y)$ for all $x,y\in X$. 
		
		For each $n\geq1$, pick some $x_n$ such that \[\phi(x_n)\with \psi(x_n) \geq  1-1/n.\] Since for all $n,m\geq1$ we have \[X(x_n,x_m)\geq \psi(x_m)\with\phi(x_n)\geq (1-1/m)\with (1-1/n),\]  then $\{x_n\}_{n\geq1}$  is a Cauchy sequence, hence has a bilimit, say $a$. This means that  for all $x\in X$,  \[\sup_{n\geq1}\inf_{m\geq n}X(x,x_m)=X(x,a), \quad  \sup_{n\geq1}\inf_{m\geq n}X(x_m,x)=X(a,x).  \]
		
		We claim that $a$ is a colimit of $\phi$. For this it suffices to show that \[\phi =\sup_{n\geq1}\inf_{m\geq n}X(-,x_m).\]  For all $x\in X$ and  $k\geq1$, since \[(1-1/k)\with \sup_{n\geq1}\inf_{m\geq n}X(x,x_m)\leq   \sup_{n\geq k}\inf_{m\geq n}\phi(x_m)\with X(x,x_m)\leq\phi(x),\] it follows that $$\sup_{n\geq1}\inf_{m\geq n}X(x,x_m)\leq\phi(x)$$ by arbitrariness of $k$.     On the other hand, since $$X(x,x_m)\geq \psi(x_m)\with\phi(x)\geq (1-1/m)\with\phi(x)$$  for all $m\geq1$, then \begin{align*} \sup_{n\geq1}\inf_{m\geq n}X(x,x_m)&\geq  \sup_{n\geq1}\inf_{m\geq n}(1-1/m)\with\phi(x)\geq \phi(x). \qedhere \end{align*} \end{proof}
	
	The following conclusion was first observed by Lawvere \cite{Lawvere1973} for metric spaces.
	
	\begin{thm} \label{complete metric}   A  real-enriched category $X$  is Cauchy complete if and only if  each of its Cauchy sequences converges uniquely in the open ball topology of its symmetrization. \end{thm}

	\section{Yoneda completeness}
	This section deals with an important notion in the theory of real-enriched categories, that of Yoneda completeness. Yoneda complete enriched categories may be viewed as analogue of directed partially ordered sets in the enriched context. 
	
	The following definition is taken from Bonsangue,   van Breugel and Rutten \cite{BBR98},  and Wagner \cite{Wagner97}.
	
	\begin{defn}  
		Suppose $\{x_i\}_{i\in D}$ is a net and $b$ is an element of  a real-enriched category $X$. We say that 
		\begin{enumerate}[label=\rm(\roman*)]  
			\item   $\{x_i\}_{i\in D}$ is   forward Cauchy  if \[\sup_{i\in D}\inf_{k\geq j \geq i}X(x_j,x_k)= 1.\]  
			\item   $b$   is a Yoneda limit  of   $\{x_i\}_{i\in D}$ if for all $y\in X$, \[X(b,y)=\sup_{i\in D}\inf_{i\leq j}X(x_j,y). \]  
	\end{enumerate} \end{defn} 
	
	Every Cauchy net is forward Cauchy; every bilimit is a Yoneda limit; and every net has at most one Yoneda limit up to isomorphism.   
	The following proposition says that for Cauchy nets,  Yoneda limits coincides with bilimits.
	
	\begin{prop} Suppose $\{x_i\}_{i\in D}$ is a Cauchy net and $b$ is an element of a real-enriched category $X$. Then $b$ of $X$  is a Yoneda limit of $\{x_i\}_{i\in D}$  if and only if it is a bilimit of $\{x_i\}_{i\in D}$. \end{prop}
	
	\begin{proof} Sufficiency is trivial, for necessity we  show  $\sup_{i\in D}\inf_{j\geq i}X(b,x_j)\geq 1.$ Since $b$ is a Yoneda limit of $\{x_i\}_{i\in D}$, then  \begin{align*}\sup_{i\in D}\inf_{j\geq i}X(b,x_j) &= \sup_{i\in D}\inf_{j\geq i} \sup_{h\in D}\inf_{k\geq h} X(x_k,x_j)\\ &\geq \sup_{i\in D}\inf_{k,j\geq i} X(x_k,x_j) \\ &\geq 1,\end{align*} which completes the proof.
	\end{proof}
	
	\begin{defn}A real-enriched category  is Yoneda complete  if  each of its forward Cauchy nets has a unique Yoneda limit. A functor between real-enriched categories (not necessarily Yoneda complete) is Yoneda continuous if it preserves Yoneda limits of all forward Cauchy nets. \end{defn} 
	
	It is clear that  $X$ is Yoneda complete if it is separated and every forward Cauchy net of $X$ has a  Yoneda limit.

	Suppose $(X,\leq)$ is an ordered set, $\{x_i\}_{i\in D}$ is a net and $b$ is an element of $X$. It is readily verified that \begin{enumerate}[label={\rm(\roman*)}]  
		\item $\{x_i\}_{i\in D}$ is forward Cauchy in  the real-enriched category $\omega(X,\leq)$ if and only if it is eventually monotone. 
		\item $b$ is a Yoneda limit of $\{x_i\}_{i\in D}$ if and only if it is the least eventual upper bound of $\{x_i\}_{i\in D}$. \end{enumerate} 
	Therefore, $\omega(X,\leq)$ is Yoneda complete whenever  $(X,\leq)$ is a dcpo. Furthermore, an order-preserving map  $f\colon(X,\leq)\lra(Y,\leq)$ is   Scott continuous if and only if   $f\colon \omega(X,\leq)\lra\omega(Y,\leq)$ is Yoneda continuous. 
	
	\begin{lem}\label{yoneda limit in Q} Suppose  $\{a_i\}_{i\in D}$ is a forward Cauchy net of $\sV=([0,1],\alpha_L)$.   \begin{enumerate}[label=\rm(\roman*)] 
			\item $\sup_{i\in D}\inf_{j\geq i}a_j$ is a Yoneda limit of $\{a_i\}_{i\in D}$, so $\sV$ is Yoneda complete. 
			\item  $\{a_i\}_{i\in D}$ is order convergent; that is, $\sup_{i\in D}\inf_{j\geq i}a_j=\inf_{i\in D}\sup_{j\geq i}a_j$.  \end{enumerate} \end{lem}
	
	\begin{proof}  The first conclusion is contained in Wagner \cite{Wagner97}, the second in Lai, D\thinspace\&{\thinspace}G Zhang \cite{LZZ2020}.
		
		(i)  We wish to show that for each $x\in [0,1]$, \[ \Big(\sup_{i\in D}\inf_{j\geq i}a_j\Big)\ra x  = \sup_{i\in D}\inf_{j\geq i}(a_j\ra x).\]
		
		Since $\{a_i\}_{i\in D}$ is forward Cauchy,  \begin{align*} 1 &= \sup_{i\in D}\inf_{j\geq i}\inf_{k\geq j}(a_j\ra a_k)\\ 
			& \leq \sup_{i\in D}\inf_{j\geq i}\sup_{l\in D}\inf_{k\geq l}(a_j\ra a_k) \\ 
			& \leq \sup_{i\in D}\inf_{j\geq i}\Big[a_j\ra\Big(\sup_{l\in D}\inf_{k\geq l}a_k\Big)\Big],
		\end{align*} then \begin{align*} \Big(\sup_{i\in D}\inf_{j\geq i}a_j\Big)\ra x  
			&\leq \Big[\Big(\sup_{l\in D}\inf_{k\geq l}a_k\Big)\ra x\Big]\with\sup_{i\in D}\inf_{j\geq i}\Big[a_j\ra\Big(\sup_{l\in D}\inf_{k\geq l}a_k\Big)\Big]\\ 
			&  \leq \sup_{i\in D}\inf_{j\geq i}(a_j\ra x).
		\end{align*}
		
		For the converse inequality, since the index set $D$ is directed, then $$\Big(\sup_{i\in D}\inf_{j\geq i}a_j\Big)\with\Big(\sup_{i\in D}\inf_{j\geq i}(a_j\ra x) \Big) \leq\sup_{i\in D}\inf_{j\geq i}a_j\with(a_j\ra x) \leq x,$$ 
		hence \[\sup_{i\in D}\inf_{j\geq i}(a_j\ra x)\leq \Big(\sup_{i\in D}\inf_{j\geq i}a_j\Big)\ra x .\]

		(ii) Since $\sup_{i\in D}\inf_{j\geq i}a_j$ is a Yoneda limit of $\{a_i\}_{i\in D}$, it follows that for all $x\in [0,1]$, \begin{align*} \Big(\sup_{i\in D}\inf_{j\geq i}a_j\Big)\ra x & = \sup_{i\in D}\inf_{j\geq i}(a_j\ra x)   \leq   \Big(\inf_{i\in D}\sup_{j\geq i}a_j\Big)\ra x. \end{align*} Putting $x= \sup_{i\in D}\inf_{j\geq i}a_j$ we obtain that  \[\sup_{i\in D}\inf_{j\geq i}a_j\geq \inf_{i\in D}\sup_{j\geq i}a_j.\]  The converse inequality is trivial.
	\end{proof}
	
	The following lemma, from Lai, D\thinspace\&{\thinspace}G Zhang \cite{LZZ2020}, concerns Yoneda limits in $\sV^{\rm op}$. We note that it is not the dual of Lemma \ref{yoneda limit in Q}. 
	
	\begin{lem}\label{order convergence}   Suppose  $\{a_i\}_{i\in D}$ is a forward Cauchy net of $\sV^{\rm op}=([0,1],\alpha_R)$.   
		\begin{enumerate}[label=\rm(\roman*)] 
			\item $\inf_{i\in D}\sup_{j\geq i}a_j$ is a Yoneda limit of $\{a_i\}_{i\in D}$, so $\sV^{\rm op}$ is Yoneda complete. 
			\item  $\{a_i\}_{i\in D}$ is order convergent; that is,  $\sup_{i\in D}\inf_{j\geq i}a_j=\inf_{i\in D}\sup_{j\geq i}a_j$.  \end{enumerate} 
	\end{lem}
	
	\begin{proof}
		(i) We wish to show that for all $x\in [0,1]$, \[x\ra  \inf_{i\in D}\sup_{j\geq i}a_j=\sup_{i\in D}\inf_{j\geq i}(x\ra a_j).\]
		
		Since $\{a_i\}_{i\in D}$ is a forward Cauchy net of $\sV^{\rm op}$,  \[1\leq \sup_{i\in D}\inf_{k\geq j\geq i}\alpha_R(a_j,a_k)=\sup_{i\in D}\inf_{k\geq j\geq i}(a_k\ra a_j),\] then
		\[\sup_{i\in D}\inf_{j\geq i}\Big[\sup_{h\in D}\inf_{k\geq h}(a_k\ra a_j)\Big]\geq 1,\] hence  \[\sup_{i\in D}\inf_{j\geq i}\Big[\Big(\inf_{h\in D}\sup_{k\geq h}a_k\Big)\ra a_j\Big]\geq 1,\] and consequently, \begin{align*}x\ra\inf_{i\in D}\sup_{j\geq i}a_j &\leq \Big[x\ra\inf_{h\in D}\sup_{k\geq h}a_k\Big]\with \sup_{i\in D}\inf_{j\geq i}\Big[\Big(\inf_{h\in D}\sup_{k\geq h}a_k\Big)\ra a_j\Big] \\ 
			&\leq  \sup_{i\in D}\inf_{j\geq i}(x\ra a_j). \end{align*}
		
		On the other hand, since for each $i\in D$ we  always have \[\sup_{h\in D}\inf_{k\geq h}(x\ra a_k)\leq x\ra \sup_{j\geq i}a_j,\] then \begin{align*}x\ra\inf_{i\in D}\sup_{j\geq i}a_j & =\inf_{i\in D}\Big(x\ra\sup_{j\geq i}a_j\Big)\\ & \geq \sup_{h\in D}\inf_{k\geq h}(x\ra a_k)\\ & =\sup_{i\in D}\inf_{j\geq i}(x\ra a_j). \end{align*}
		
		Therefore, $\inf_{i\in D}\sup_{j\geq i}a_j$ is a Yoneda limit of $\{a_i\}_{i\in D}$ in $\sV^{\rm op}$.

		(ii) Since $\inf_{i\in D}\sup_{j\geq i}a_j$ is a Yoneda limit of $\{a_i\}$ in $([0,1],\alpha_R)$, it follows that for all $x\in [0,1]$, \begin{align*} x\ra \inf_{i\in D}\sup_{j\geq i}a_j     &=\sup_{i\in D}\inf_{j\geq i}(x\ra a_j)  \leq x\ra \sup_{i\in D}\inf_{j\geq i}a_j. \end{align*} Putting  $x=1$  we obtain that \[\inf_{i\in D}\sup_{j\geq i}a_j\leq \sup_{i\in D}\inf_{j\geq i}a_j.\]
		The converse inequality is trivial. \end{proof}
	
	\begin{rem} Lemma \ref{yoneda limit in Q} implies that every forward Cauchy net of $\sV=([0,1],\alpha_L)$ is   order convergent;
		Lemma \ref{order convergence} implies that every forward Cauchy net of $\sV^{\rm op}=([0,1],\alpha_R)$ is order convergent too.   But,  $\sV$ and $\sV^{\rm op}$ may have different forward Cauchy nets. For example, in the case that $\with$ is the product t-norm, the sequence  $\{1/n\}_{n\geq1}$ is forward Cauchy in $\sV^{\rm op}$ but not in $\sV$. \end{rem}
	
	\begin{lem} \label{Yoneda limits in PX} {\rm (Wagner \cite{Wagner97})} Suppose $X$ is a real-enriched category and   $\{\phi_i\}_{i\in D}$ is a forward Cauchy net of $\CP X$. Then,  $\sup_{i\in D}\inf_{j\geq i}\phi_j$ is a Yoneda limit of $\{\phi_i\}_{i\in D}$. In particular, $\CP X$ is Yoneda complete. \end{lem}
	
	\begin{proof} We wish to show that for each  $\phi\in\CP X$, \[\CP X\Big(\sup_{i\in D}\inf_{j\geq i}\phi_j,\phi\Big)
		=\sup_{i\in D}\inf_{j\geq i}\CP X(\phi_j,\phi). \] 
		
		For each $x\in X$, it is clear that $\{\phi_i(x)\}_{i\in D}$ is a forward Cauchy net of  $([0,1],\alpha_L)$, hence \[  \Big(\sup_{i\in D}\inf_{j\geq i}\phi_j(x)\Big)\ra\phi(x) = \sup_{i\in D}\inf_{j\geq i}( \phi_j(x) \ra\phi(x)).\] Thus, we only need to check that \[\inf_{x\in X}\sup_{i\in D}\inf_{j\geq i}(  \phi_j(x) \ra\phi(x)) = \sup_{i\in D}\inf_{j\geq i}\CP X(\phi_j,\phi).\]
		
		Since $\{\phi_i\}_{i\in D}$ is forward Cauchy,  
		\begin{align*}&\quad \inf_{x\in X}\sup_{i\in D}\inf_{j\geq i}( \phi_j(x) \ra\phi(x))\\
			&\leq\Big[\inf_{x\in X}\sup_{h\in D}\inf_{k\geq h}(\phi_k(x)\ra\phi(x) )\Big]\with \Big[\sup_{i\in D}\inf_{j\geq i} \inf_{k\geq j}\CP X(\phi_j, \phi_k)\Big]\\
			&\leq \sup_{i\in D}\inf_{j\geq i}\inf_{x\in X}\Big[\Big(\sup_{h\in D}\inf_{k\geq h}(\phi_k(x)\ra\phi(x) )\Big)\with\Big(\inf_{k\geq j}\CP X(\phi_j, \phi_k)\Big)\Big]\\
			&\leq \sup_{i\in D}\inf_{j\geq i}\inf_{x\in X} \sup_{h\in D}\inf_{k\geq h,j}  ( \phi_k(x)\ra\phi(x) ) \with (\phi_j(x)\ra \phi_k(x))   \\
			&\leq \sup_{i\in D}\inf_{j\geq i}\inf_{x\in A}(  \phi_j(x)\ra\phi(x) )\\ &= \sup_{i\in D}\inf_{j\geq i}\CP X(\phi_j,\phi). \end{align*} The converse inequality is trivial.  \end{proof}

	\begin{prop}\label{Cauchy net implies Cauchy weight} A forward Cauchy net $\{x_i\}_{i\in D}$  of a real-enriched category $X$  is Cauchy if and only if  the weight $\phi\coloneqq\sup_{i\in D}\inf_{j\geq i}X(-,x_j)$ is   Cauchy. \end{prop}
	
	\begin{proof} The conclusion is contained in Hofmann and Reis \cite{HR2013}, Li and Zhang \cite{LiZ2018a}.  
		Suppose  $\{x_i\}_{i\in D}$ is a Cauchy net.  We show that $\phi$ is right adjoint to   \[ \psi\coloneqq\sup_{i\in D}\inf_{j\geq i}X(x_j,-); \] that is, $\phi\circ\psi\geq1$ and $\psi\circ\phi(x,y)\leq X(x,y)$ for all $x,y\in X$.
		For this we calculate:  
		\begin{align*}\phi\circ\psi & =\sup_{x\in X} \phi(x)\with\psi(x) \\ 
			&=\sup_{x\in X}\Big[\Big(\sup_{i\in D}\inf_{j\geq i}X(x,x_j)\Big) \with \Big(\sup_{i\in D}\inf_{j\geq i}X(x_j,x)\Big)\Big]\\ 
			&=\sup_{x\in X}\sup_{i\in D}\Big[\Big(\inf_{j\geq i}X(x,x_j)\Big) \with \Big(\inf_{j\geq i}X(x_j,x)\Big)\Big] &\text{($D$ is directed)}  \\ 
			&\geq \sup_{i\in D} \Big[\Big(\inf_{j\geq i}X(x_i,x_j)\Big) \with \Big(\inf_{j\geq i}X(x_j,x_i)\Big)\Big] &\text{(let $x=x_i$ for each $i$)} \\ 
			&= \Big(\sup_{i\in D}\inf_{j\geq i}X(x_i,x_j)\Big) \with \Big(\sup_{i\in D}\inf_{j\geq i}X(x_j,x_i)\Big) &\text{($D$ is directed)} \\ 
			&= 1; &\text{($\{x_i\}_{i\in D}$ is Cauchy)}
		\end{align*} 
		and  
		\begin{align*}\psi\circ\phi(x,y)&= \Big(\sup_{i\in D}\inf_{j\geq i}X(x_j,y)\Big) \with \Big(\sup_{i\in D}\inf_{j\geq i}X(x,x_j)\Big)\\ 
			&= \sup_{i\in D}\Big[\Big(\inf_{j\geq i}X(x_j,y)\Big) \with \Big(\inf_{j\geq i}X(x,x_j)\Big)\Big] &\text{($D$ is directed)} \\ 
			& \leq \sup_{i\in D} X(x_i,y)\with X(x,x_i) \\ 
			&\leq X(x,y).\end{align*}   
		
		Conversely, suppose  $\phi$ is a Cauchy weight. By Proposition \ref{adjoint_arrow_calculation}, 
		the left adjoint of $\phi$ is   $\psi\coloneqq X\swarrow\phi$.   By Lemma \ref{Yoneda limits in PX},  $\phi$ is a Yoneda limit of the forward Cauchy net $\{X(-,x_i)\}_{i\in D}$ of $\CP X$, then for all $x\in X$, 
		\begin{align*}
			\psi(x) &=\CP X(\phi,X(-,x)) \\
			&=\sup_{i\in D}\inf_{k\geq i}\CP X(X(-,x_k),X(-,x))\\
			&=\sup_{i\in D}\inf_{k\geq i}X(x_k,x).
		\end{align*}
		Therefore,
		\begin{align*}\sup_{i\in D}\inf_{k,j\geq i}X(x_k,x_j)
			&\geq \sup_{x\in X}\sup_{i\in D}\inf_{k,j\geq i}X(x,x_j)\with X(x_k,x)\\ 
			&\geq \sup_{x\in X}\Big[\Big(\sup_{i\in D}\inf_{j\geq i}X(x,x_j)\Big)\with \Big(\sup_{i\in D}\inf_{k\geq i}X(x_k,x)\Big)\Big]\\ 
			& =\sup_{x\in X} \phi(x)\with \psi(x)\\ & = 1,\end{align*} which shows that the net $\{x_i\}_{i\in D}$  is   Cauchy.
	\end{proof}
	
	In the following we show that Yoneda limits of forward Cauchy nets can be characterized as colimits of a special kind of weights --- ideals of real-enriched categories. 
	
	We say that a functor $f\colon X\lra Y$ between real-enriched categories \emph{preserves finite  colimits} if for any finite real-enriched category $K$ (that means, $K$ has a finite number of objects), any functor $h\colon K\lra X$, and any weight $\phi$ of $K$, $f(\colim_\phi h)$ is a colimit of $f\circ h$ weighted by $\phi$ whenever $\colim_\phi h$ exists; that is to say, $f(\colim_\phi h)=\colim_\phi (f\circ h)$.

	\begin{defn} \label{defn of ideal} Let $X$ be a real-enriched category. A weight  $\phi$ of $X$ is called an ideal if the functor $$\sub_X(\phi,-)\colon  \CP X \lra \sV$$ preserves   finite  colimits, where $\sV=([0,1],\alpha_L)$.  \end{defn}
	
	The term \emph{ideal} is chosen because of the fact that an ideal of an ordered set $P$  (i.e. a  directed lower set of $P$) is exactly a non-empty join-irreducible element in the set of lower sets of $P$.
	
	Every Cauchy weight is clearly an ideal; in particular, every representable weight  is an ideal. Since $\CP X$ is complete, it follows from Corollary \ref{calculation of sup by tensors} that  a weight $\phi$ of $X$ is an ideal if and only if it satisfies the following conditions:  
	\begin{enumerate}[label=\rm(I\arabic*)]  
		\item  $\sub_X(\phi,-)\colon  \CP X \lra \sV$ preserves tensors; that is, $$\sub_X(\phi,r\with\lam) =r\with\sub_X(\phi,\lam)$$ for all  $\lam\in\CP X$ and all $r\in[0,1]$. 
		\item   $\sub_X(\phi,-)\colon (\CP X)_0\lra[0,1]$ preserves finite joins; that is, $$\sub_X(\phi,\lam\vee\mu)=\sub_X(\phi,\lam)\vee\sub_X(\phi,\mu)$$  for all  $\lam,\mu\in\CP X$. \end{enumerate}    
	
	Suppose $\phi$ is an ideal of $X$. Since $\sub_X(\phi,1_X)=1$, it follows that $\sub_X(\phi,p_X)=p$ for all $p\in[0,1]$ and that $\phi$ is \emph{inhabited} in the sense that $\sup_{x\in X}\phi(x)=1$. 
	
	For each real-enriched category $X$, let $$\CI X  $$ be the subcategory of $\CP X$ composed of  ideals of $X$. Suppose $f\colon X\lra Y$ is a functor. By the equality $$\sub_Y(f_\exists(\phi),\mu) =\sub_X(\phi,f^{-1}(\mu))$$ for every weight $\phi$ of $X$ and every weight $\mu$ of $Y$, one readily verifies that if $\phi$ is an ideal of $X$, then   $f_\exists(\phi)$ is an ideal of $Y$. So the assignment $X\mapsto\CI X$ defines a  functor $$\CI\colon\QOrd\lra\QOrd,$$ which is a subfunctor of the presheaf functor $\CP\colon\QOrd\lra\QOrd$.  
	
	Proposition \ref{colimit =bilimit for Cauchy net} shows that for each real-enriched category $X$, every Cauchy weight of $X$ has a colimit if, and only if, every Cauchy net of $X$ has a bilimit.  A parallel result holds for Yoneda limits and colimits of ideals.
	
	\begin{thm}\label{Yoneda com via ideals} Let $X$ be a real-enriched category. Then, every forward Cauchy net of $X$ has a Yoneda limit  if and only if every ideal of $X$ has a colimit. Therefore, $X$ is Yoneda complete if and only if each of its ideals  has a unique colimit. \end{thm}
	
	\begin{proof} 
		This is a  consequence of  Lemma \ref{yoneda limit as colimits} and Proposition \ref{characterization of  ideal} below. \end{proof}
	
	The following lemma, from  Flagg,   S\"{u}nderhauf and  Wagner \cite{FSW}, relates Yoneda limits  of forward Cauchy nets to colimits of  weights.
	
	\begin{lem} 
		\label{yoneda limit as colimits} Suppose $\{x_i\}_{i\in D}$ is a forward Cauchy net of a real-enriched category   $X$.  Then, an element $b$ of $X$ is a Yoneda limit of $\{x_i\}_{i\in D}$ if and only if $b$ is a colimit of the  weight $\sup_{i\in D}\inf_{j\geq i}X(-,x_j)$. \end{lem}
	
	\begin{proof}  Since   $\sy \colon X\lra\CP X$ is fully faithful,  $\{X(-,x_i)\}_{i\in D}$ is a forward Cauchy net of
		$\CP X$. By  Lemma \ref{Yoneda limits in PX}, the weight $$\phi\coloneqq\sup_{i\in D}\inf_{j\geq i}X(-,x_j)$$ is a Yoneda limit of $\{X(-,x_i)\}_{i\in D}$, thus for each $x\in X$, \[\CP X(\phi, X(-,x))=\sup_{i\in D}\inf_{j\geq i}\CP X(X(-,x_j), X(-,x))= \sup_{i\in D}\inf_{j\geq i}X(x_j,x). \]
		Therefore,    
		\begin{align*} & \quad\quad\quad \text{$b$ is a Yoneda limit of $\{x_i\}_{i\in D}$}\\ 
			&\iff \forall x\in X,~ X(b, x)=\sup_{i\in D}\inf_{j\geq i}X(x_j,x) \\ 
			&\iff \forall x\in X,~   X(b, x)=\CP X(\phi, X(-,x))  \\ 
			&\iff \text{$b$ is a colimit of $\phi$}.\qedhere\end{align*}
	\end{proof} 
	
	For every dcpo $(X,\leq)$, the real-enriched category $\omega(X,\leq)$ is Yoneda complete. Conversely, as a consequence of the above lemma we show that the underlying order of a Yoneda complete real-enriched category is  directed complete.
	
	\begin{cor}\label{underlying order of Yoneda} The underlying ordered set of each Yoneda complete real-enriched category is a dcpo.  
	\end{cor}
	
	\begin{proof}Suppose $X$ is   Yoneda complete,  $D$ is a directed subset of $X_0$.  We view $D$ as a forward Cauchy net of $X$ and let $b$ be its Yoneda limit.  By Lemma \ref{yoneda limit as colimits}, $b$ is a colimit of the weight $\sup_{x\in D}\inf_{y\geq x}X(-,y)$. Since $\inf_{y\geq x}X(-,y)=X(-,x)$ for every $x\in D$, it follows that $b$ is a colimit of the conical weight $\sup_{x\in D}X(-,x)$, hence a join of $D$ in $X_0$ by the argument of Theorem \ref{characterizations of complete Q-orders}. \end{proof}

	To state and prove   Proposition \ref{characterization of  ideal}, we need the notion of formal balls of a real-enriched category.   Suppose  $X$ is a real-enriched category. A \emph{formal ball} of $X$ is a pair $(x,r)$ with $x\in X$ and $r\in[0,1]$, $x$ is called the center and $r$ the radius. For  formal balls  $(x,r)$ and $(y,s)$, let \[(x,r)\sqsubseteq  (y,s) \quad\text{if}~ r\leq s\with X(x,y).\] The relation $\sqsubseteq $ is reflexive and transitive, hence an order. We write $\mathrm{B}X$ for the set of  formal balls of $X$  ordered by $\sqsubseteq $. It is obvious that $(x,0)\sqsubseteq(y,s)$ for all $x,y\in X$ and $s\in[0,1]$, so any formal ball with radius $0$ is a bottom element of $\mathrm{B}X$. 
	
	The ordered set of formal balls is introduced by Weihrauch and Schreiber \cite{WS} for metric spaces. It has been extended to quasi-metric spaces and quantale-enriched categories in general, and has been investigated extensively,  see e.g. Edalat and Heckmann \cite{EH98},   Goubault-Larrecq and Ng \cite{Goubault-Ng}, Kostanek and Waszkiewicz \cite{KW2011},  Rutten \cite{Rutten98}.
	
	The ordered set $\mathrm{B}X$ of formal balls is closely related to $\CP X$. For all $x,y\in X$ and $r,s\in[0,1]$, since $$\CP X(r\with\sy(x),s\with\sy(y))=r\ra\CP X(\sy(x),s\with\sy(y))  =r\ra (s\with X(x,y)), $$  it follows that \[(x,r)\sqsubseteq(y,s)\iff 1\leq \CP X(r\with\sy(x),s\with\sy(y)). \] So, if we identify a  formal ball $(x,r)$ with the weight $r\with\sy(x)$,  then the order relation between formal balls is  the order relation inherited from $(\CP X)_0$.  
	
	The following proposition  is an improvement of related results in Lai, D\thinspace\&{\thinspace}G Zhang \cite{LZZ2020}. It shows in particular that ideals of a real-enriched category $X$ are precisely those weights generated by forward Cauchy nets.
	
	\begin{prop}   \label{characterization of  ideal} For each weight $\phi$ of a real-enriched category $X$, the following are equivalent: 
		\begin{enumerate}[label=\rm(\arabic*)]   
			\item  $\phi$ is an   ideal of $X$. 
			\item   $\sup_{x\in X}\phi(x)=1$ and $\phi$ is a coprime in $(\CP X)_0$ in the sense that  $$\phi\leq\psi_1\vee\psi_2\implies \phi\leq\psi_1\text{ or } \phi\leq\psi_2  $$ for all weights $\psi_1,\psi_2$ of $X$. \item   $\sup_{x\in X}\phi(x)=1$ and   $\mathrm{B}\phi\coloneqq\{(x,r)\in \mathrm{B}X \mid   \phi(x)>r\}$ is a directed subset of $\mathrm{B}X$. 
			\item  $\phi=\sup_{i\in D}\inf_{j\geq i} X(-,x_j)$ for some forward Cauchy net $\{x_i\}_{i\in D}$ of $X$.
	\end{enumerate} \end{prop}

	\begin{proof} $(1)\Rightarrow(2)$ Obvious. 
		
		$(2)\Rightarrow(3)$ First we show that for all weights $\psi_1,\psi_2$ of $X$, $$\sub_X(\phi,\psi_1\vee\psi_2)\leq \sub_X(\phi,\psi_1)\vee \sub_X(\phi,\psi_2).$$   Let $r=\sub_X(\phi,\psi_1\vee\psi_2)$. Since cotensors in $\CP X$ are computed pointwise (Example \ref{tensor and cotensor in PX}) and $[0,1]$ is linearly ordered, it follows that \begin{align*} \phi& \leq r\ra (\psi_1\vee\psi_2) = (r\ra\psi_1)\vee(r\ra\psi_2),
		\end{align*} then either $\phi\leq r\ra\psi_1$ or $\phi\leq r\ra\psi_2$, hence $r\leq \sub_X(\phi,\psi_1)\vee \sub_X(\phi,\psi_2)$. 
		
		Now we prove that $(\mathrm{B}\phi,\sqsubseteq)$ is  directed. For $(x,r)$ and $(y,s)$ of $\mathrm{B}\phi$, consider the weights  \[\phi_1=  X(x,-)\ra r \quad \text{and}\quad \phi_2 =  X(y,-)\ra s.\] Since   \begin{align*} \CP X(\phi,\phi_1\vee\phi_2) &\leq \CP X(\phi,\phi_1)\vee \CP X(\phi,\phi_2) =(\phi(x)\ra r)\vee(\phi(y)\ra s)<1, \end{align*}  there exists   $z\in X$ such that \[ \phi(z)>(X(x,z)\ra r)\vee(X(y,z)\ra s)  \] because $\sup_{x\in X}\phi(x)=1$. Pick   $t\in[0,1]$ such that \[ \phi(z)>t>(X(x,z)\ra r)\vee(X(y,z)\ra s).\] Then    $(z,t)\in \mathrm{B}\phi$. We assert that $(z,t)$ is an upper bound of $(x,r)$ and $ (y,s)$, hence $(\mathrm{B}\phi,\sqsubseteq)$  is directed. To see this, let $u= X(x,z)\ra r$. Then $u$ is by definition the largest element of $[0,1]$ with $u\with X(x,z)\leq r$. Since $t>u$, it follows that $r< t\with X(x,z)$, then $(x,r)\sqsubseteq(z,t)$. Likewise, $(y,s)\sqsubseteq(z,t)$.

		$(3)\Rightarrow(4)$ Index the directed set $(\mathrm{B}\phi,\sqsubseteq)$ by itself as $\{(x_i,r_i)\}_{i\in D}$; that means, $(x_i,r_i)\sqsubseteq (x_j,r_j)$ if and only if $i\leq j$.    We show in two steps that the net  $\{x_i\}_{i\in D}$ satisfies the requirement.
		
		{\bf Step 1}.  $\{x_i\}_{i\in D}$ is  forward Cauchy. 
		
		Let $r<1$. Pick   $(x_i,r_i)\in \mathrm{B}\phi$ with $r_i>r$. Then    $$r\leq r_i\leq r_j \leq r_k\with X(x_j,x_k)\leq X(x_j,x_k)$$ whenever $i\leq j\leq k$. Thus,  $\{x_i\}_{i\in D}$ is forward Cauchy  by arbitrariness of $r$.
		
		{\bf Step 2}. For all $x\in X$,  \[ \phi(x)=\sup_{i\in D} \inf_{j\geq i}X(x,x_j). \]
		
		If $r< \phi(x)$, then $(x,r)\in \mathrm{B}\phi$, so  $(x,r)=(x_k,r_k)$ for some   $k\in D$.  Therefore, \begin{align*}r&\leq\bw_{j\geq k}X(x_k,x_j) \leq \bv_{i\in D} \bw_{j\geq i}X(x_i,x_j),\end{align*}   and consequently, \[\phi(x)\leq \bv_{i\in D} \bw_{j\geq i}X(x,x_j) \] by arbitrariness of $r$.
		
		Conversely, if \[r<\bv_{i\in D} \bw_{j\geq i}X(x,x_j),\] then there exist $k\in D$ and $s\in[0,1]$ such that $$r< s\leq\bw_{j\geq k}X(x,x_j).$$ Since $\sup_{x\in X}\phi(x)=1$, there is $l\in D$ such that   $r\leq s\with r_l$.   Take an upper bound $(x_h,r_h)$   of $(x_k,r_k)$ and $(x_l,r_l)$. Then \begin{align*}r&\leq s\with r_l \leq s\with r_h \leq X(x,x_h)\with\phi(x_h) \leq\phi(x).\end{align*}    Therefore,  $$\phi(x)  \geq\bv_{i\in D} \bw_{j\geq i}X(x,x_j) $$ by arbitrariness of $r$.

		$(4)\Rightarrow(1)$ Suppose $\{x_i\}_{i\in D}$ is a forward Cauchy net of $X$.  We wish to show that $$\phi\coloneqq \sup_{i\in D}\inf_{j\geq i}X(-,x_j)$$ is an  ideal.

		{\bf Step 1}. We show that for each weight $\lam$ of $X$, \[\sub_X(\phi,\lam)= \inf_{i\in D}\sup_{j\geq i}\lam(x_j).\]
		
		Since $\lam$ is a weight,   $\{\lam(x_i)\}_{i\in D}$ is a forward Cauchy net of $\sV^{\rm op}$, 
		then \begin{align*}\sub_X(\phi,\lam)&= \sub_X\Big(\sup_{i\in D}\inf_{j\geq i} X(-,x_j),\lam\Big) \\ 
			&
			=\sup_{i\in D}\inf_{j\geq i}\CP X(X(-,x_j),\lam) & \text{(Lemma \ref{Yoneda limits in PX})}\\ 
			&= \sup_{i\in D}\inf_{j\geq i}\lam(x_j)& \text{(Yoneda lemma)}\\ 
			&= \inf_{i\in D}\sup_{j\geq i}\lam(x_j).& \text{(Lemma \ref{order convergence}\thinspace(ii))} \end{align*}
		
		{\bf Step 2}. For all  $\lam \in\CP X$ and all $r\in[0,1]$,    $\sub_X(\phi,r\with\lam)=r\with\sub_X(\phi,\lam)$. 
		
		By  Step 1, we have \begin{align*}\sub_X(\phi,r\with\lam) 
			&= \inf_{i\in D}\sup_{j\geq i}r\with\lam(x_j) \\ 
			&= r\with \inf_{i\in D}\sup_{j\geq i}\lam(x_j) &\text{($\with$ is continuous)} \\ 
			&= r\with\sub_X(\phi,\lam).\end{align*}
		
		{\bf Step 3}. For all  weights  $\lam,\mu$ of $X$, $\sub_X(\phi, \lam\vee\mu)=\sub_X(\phi, \lam)\vee\sub_X(\phi, \mu)$.
		
		For this we calculate:   \begin{align*}&~\quad  \sub_X(\phi, \lam)\vee\sub_X(\phi, \mu)\\  
			&= \inf_{i\in D}\sup_{j\geq i}\lam(x_j)\vee \inf_{i\in D}\sup_{j\geq i}\mu(x_j)& \text{(Step 1)} \\ 
			& = \inf_{i\in D}\sup_{j\geq i}(\lam(x_j)\vee\mu(x_j))  \\ 
			& = \sub_X(\phi, \lam\vee\mu).& \text{(Step 1)}\end{align*}
		
		The proof is completed.
	\end{proof}
	
	Suppose  $\{(x_i,r_i)\}_{i\in D}$ is a directed subset of the ordered set $\mathrm{B}X$ of formal balls of a real-enriched category $X$. If $r_i$ tends to $1$,  the net $\{x_i\}_{i\in D}$ is clearly forward Cauchy. Such   nets are said to be \emph{weightable} in Goubault-Larrecq \cite{Goubault}. 
	
	We list here some consequences of Lemma \ref{yoneda limit as colimits} and Proposition \ref{characterization of  ideal}. \begin{enumerate}[label=\rm(\roman*)]  
		\item  A functor $f\colon X\lra Y$ is Yoneda continuous if and only if it preserves colimits of all ideals.  
		\item  For each forward Cauchy net $\{x_i\}_{i\in D}$, there is a weightable forward Cauchy net $\{y_j\}_{j\in E}$ such that $\{x_i\}_{i\in D}$ and $\{y_j\}_{j\in E}$ have the same Yoneda limits. In particular, a real-enriched category is Yoneda complete  if and only if each of its weightable forward  Cauchy nets has a unique Yoneda limit.  
	\end{enumerate}
	
	
	
	Another consequence of Proposition \ref{characterization of  ideal} is that for real-enriched categories, ideals in the sense of Definition \ref{defn of ideal}  coincide with those in the sense of Flagg,  S\"{u}nderhauf and Wagner \cite[Definition 10]{FSW}.  
	
	\begin{cor} Suppose $\phi$ is a weight   of a real-enriched category  $X$. Then  $\phi$ is   an ideal of $X$ if and only if it  satisfies the following conditions:  
		\begin{enumerate}[label=\rm(\roman*)]  
			\item $\bv_{x\in X}\phi(x)=1$.
			\item If $r< 1$, $s_1<\phi(x_1)$ and $s_2<\phi(x_2)$, then  there exists  $x\in X$   such that $r<\phi(x)$, $s_1< X(x_1,x) $ and   $s_2< X(x_2,x) $. 
		\end{enumerate} 
	\end{cor} 
	
	
	\begin{cor} \label{ideals in [0,1]} 
		\begin{enumerate}[label={\rm(\roman*)}]  
			\item A weight $\phi$  of $\sV=([0,1],\alpha_L)$ is an   ideal if and only if either  $\phi(x)=x \ra a$ for some $a\in [0,1]$ or $\phi(x)=\sup_{b<a}(x\ra b)$ for some $ a>0$. 
			\item A weight $\psi$ of  $\sV^{\rm op}=([0,1],\alpha_R)$  is an   ideal if and only if either  $\psi(x)=a \ra x$ for some $a\in [0,1]$ or $\psi(x)=\sup_{b>a}(b\ra x)$ for some $ a<1$. \end{enumerate} \end{cor}
	
	\begin{proof}We prove (i) for example. Sufficiency is easy since the weight $\sup_{b<a}(x\ra b)$ is generated by the forward Cauchy sequence $\{a-1/n\}_{n\geq 1}$. For  necessity suppose  $\phi$ is an ideal of $([0,1],\alpha_L)$. Then there is a forward Cauchy net $\{x_i\}_{i\in D}$ of $([0,1],\alpha_L)$ such that \[\phi(x)=\sup_{i\in D}\inf_{j\geq i}(x\ra x_j).\] Let $a= \sup_{i\in D}\inf_{j\geq i}x_j$. Then \[\phi(x)=\sup_{i\in D}\inf_{j\geq i}(x\ra x_j)=\sup_{i\in D}\Big(x\ra \inf_{j\geq i}x_j\Big), \] so, either $\phi(x)=x \ra  a$   or $\phi(x)=\sup_{b<a}(x\ra b)$.  \end{proof}

	\begin{cor}\label{orde pres is limi cont}Every functor $f\colon X\lra Y$ with $X$ symmetric  is Yoneda continuous. \end{cor} 
	
	\begin{proof}Let $b$ be a Yoneda limit of a forward Cauchy net $\{x_i\}_{i\in D}$ of $X$. From symmetry of $X$ one infers that $\{x_i\}_{i\in D}$ is a Cauchy net, hence  $\phi=\sup_{i\in D}\inf_{j\geq i}X(-,x_j)$ is a Cauchy weight by Proposition \ref{Cauchy net implies Cauchy weight} and consequently, $f$ preserves the colimit of $\phi$ by Proposition \ref{absolute sup}. \end{proof}
	
	\begin{thm}\label{univeral property of IX}Let $X$ be a real-enriched category. Then 
		\begin{enumerate}[label={\rm(\roman*)}] 
			\item   $\CI X$ is Yoneda complete. 
			\item For each functor $f\colon X\lra Y$ with $Y$ Yoneda complete, there is a unique Yoneda continuous functor $\overline{f}\colon\CI X\lra Y$ that extends $f$.  \end{enumerate} \end{thm}
	
	\begin{proof} Write $\mathfrak{i}\colon\CI X\lra\CP X$ for the inclusion functor and write ${\sf t}\colon X\lra  \mathcal{I}X$ for the Yoneda embedding with codomain restricted to $\mathcal{I}X$. The composite $ \mathfrak{i}\circ{\sf t}$ is then the Yoneda embedding $\sy\colon X\lra\CP X$. 
		
		(i) We  show that $\CI X$ is closed in $\CP X$ under formation of colimits of ideals; or equivalently, $\CI X$ is closed in $\CP X$ under Yoneda limits of forward Cauchy nets. That means, for each ideal $\Lambda$ of $\CI X$, the colimit  of the inclusion functor  $\mathfrak{i}\colon \CI X\lra\CP X$ weighted by $\Lambda$ is an ideal of $X$.
		Since the composite $ \mathfrak{i}\circ{\sf t}$ is the Yoneda embedding $\sy\colon X\lra\CP X$, from Example \ref{sup in PX} it follows that $${\colim}_\Lambda\mathfrak{i} =\colim \Lambda\circ \mathfrak{i}^*= \Lambda\circ \mathfrak{i}^*\circ\mathfrak{i}_*\circ {\sf t}_*= \Lambda\circ{\sf t}_*,$$  hence for each  weight $\phi$ of $X$,   \begin{align*}\sub_X({\colim}_\Lambda\mathfrak{i},\phi)=\sub_X(\Lambda\circ{\sf t}_*,\phi)  = \sub_{\CI X}(\Lambda,\sub_X(-,\phi)). \end{align*} 
		Then, for each  weight $\phi$ of $X$ and each $r\in[0,1]$,  \begin{align*} 
			\sub_X({\colim}_\Lambda\mathfrak{i},  r\with\phi)  
			& = \sub_{\CI X}(\Lambda,\sub_X(-,r\with\phi))\\ 
			&=\sub_{\CI X}(\Lambda, r\with\sub_X (-, \phi) ) 
			\\
			&= r\with\sub_{\CI X}(\Lambda, \sub_X (-, \phi) ) 
			\\
			&=r\with \sub_X({\colim}_\Lambda\mathfrak{i},  \phi), 
		\end{align*}  showing that $\sub_X({\colim}_\Lambda\mathfrak{i}, -)$ preserves tensor. 
		In the same fashion one verifies that $\sub_X({\colim}_\Lambda\mathfrak{i}, -)$ preserves finite join. So ${\colim}_\Lambda\mathfrak{i}$ is an ideal of $X$, as desired.  
		
		(ii) The functor \[\overline{f}\colon \CI X\lra Y, \quad \phi\mapsto {\colim}_ \phi \] clearly  extends $f$. It remains to check that $\overline{f}$ preserves colimits of ideals. It is clear that $\overline{f}$ is equal to the composite $\colim\circ  f_\exists\colon\CI X\lra\CI Y\lra Y$. Since $\colim\colon\CI Y\lra Y$ is left adjoint to  ${\sf t}\colon Y\lra\CI Y$,  it preserves colimits. Since $\CI X$ and $\CI Y$ are, respectively,   closed in $\CP X$ and $\CP Y$  under formation of colimits of ideals, and $f_\exists\colon\CP X\lra\CP Y$ preserves all colimits, it follows that the composite $\overline{f}$ preserves colimits of all ideals.
		Uniqueness of $\overline{f}$ follows from that  each ideal $\phi$ of $X$ is the colimit of ${\sf t}\colon X\lra\CI X$ weighted by $\phi$.\end{proof}

	Proposition \ref{characterization of  ideal} shows that forward Cauchy nets are closely related to directed subsets of the set of formal balls. Actually, there exist interesting interactions between Yoneda completeness of a real-enriched category and directed completeness of its set of formal balls. 
	
	\begin{thm}\label{yoneda via formal ball} If the continuous t-norm $\&$ is Archimedean, then for each real-enriched category $X$, the following are equivalent: \begin{enumerate}[label=\rm(\arabic*)]\setlength{\itemsep}{0pt} \item  Every forward Cauchy net of  $X$ has a Yoneda limit. \item Every directed subset of the ordered set  $\mathrm{B}X $ of formal balls has a join.
	\end{enumerate} \end{thm}
	
	Theorem \ref{yoneda via formal ball} shows that, when $\&$ is a continuous Archimedean t-norm, Yoneda completeness of a real-enriched category is equivalent to directed completeness of its set of formal balls. This is first proved in Edalat and Heckmann \cite[Theorem 6]{EH98} for metric spaces,  then extended to categories enriched over a ``value quantale'' in Kostanek and Waszkiewicz \cite[Theorem 7.1]{KW2011}. The form stated here  appears in Yang and Zhang \cite{YangZ2023}. 
	
	To prove Theorem \ref{yoneda via formal ball} we need several lemmas. In the following, we agree by convention that  a directed subset $\{(x_i,r_i)\}_{i\in D}$ of $\mathrm{B}X$  is always indexed by itself; that means, $(x_i,r_i)\sqsubseteq(x_j,r_j)$ if and only if $i\leq j$.

	\begin{lem}\label{join of directed set in BX}  
		Suppose  $X$ is a real-enriched category and   $\{(x_i,r_i)\}_{i\in D}$ is a directed subset of $\mathrm{B}X $. If  $x$ is a Yoneda limit of  $\{x_i\}_{i\in D}$  and  $r=\bv_{i\in D}r_i$,  then $(x,r)$ is a join of  $\{(x_i,r_i)\}_{i\in D}$ in $\mathrm{B}X $. \end{lem}
	
	\begin{proof} The conclusion is  contained in Kostanek and Waszkiewicz \cite[Lemma 7.7]{KW2011}, also in Ali-Akbari,   Honari,  Pourmahdian and Rezaii \cite[Theorem 3.3]{Ali} and Goubault-Larrecq \cite[Lemma 7.4.25]{Goubault}. First we show that $(x,r)$ is an upper bound of $\{(x_i,r_i)\}_{i\in D}$; that is, $r_i\leq r\with X(x_i,x)$ for all $i\in D$.
		
		Since $x$ is a Yoneda limit of  $\{x_i\}_{i\in D}$, then \[1=X(x,x)= \bv_{k\in D}\bw_{j\geq k}X(x_j,x).\]
		For each $i\in D$ and each $\epsilon< r_i$, by continuity of $\&$ there is some $k\in D$ such that $\epsilon\leq r_i\with X(x_j,x)$ whenever $j\geq k$.  Thus, for all $j\geq i,k$  we have \begin{align*}\epsilon&\leq r_i\with X(x_j,x) \leq r_j\with X(x_i,x_j)\with X(x_j,x) \leq r\with X(x_i,x).\end{align*} By   arbitrariness of $\epsilon$  we obtain that $r_i\leq r\with X(x_i,x)$.
		
		Next  we show that $(x,r)\sqsubseteq (y,s)$  whenever $(y,s)$ is an  upper bound of $\{(x_i,r_i)\}_{i\in D}$.
		Since $(y,s)$ is an upper bound,
		then  $r_i\leq r_j\leq s\with X(x_j,y)$ whenever $i\leq j$, hence \begin{align*} r&=\bv_{i\in D}r_i  \leq \bv_{i\in D}\bw_{j\geq i}s\with X(x_j,y)  = s\with X(x,y),\end{align*} which shows that $(x,r)\sqsubseteq (y,s)$, as desired. \end{proof}
	
	\begin{lem}\label{directed set is forward Cauchy} Suppose  $X$ is a real-enriched category and   $\{(x_i,r_i)\}_{i\in D}$ is a directed subset of $\mathrm{B}X $ with  $r_i>0$ for some $i\in D$.  
		\begin{enumerate}[label=\rm(\roman*)] 
			\item  If $\bv_{i\in D}r_i=1$, then the net $\{x_i\}_{i\in D}$ is  forward Cauchy. 
			\item If $\&$ is Archimedean, then the net $\{x_i\}_{i\in D}$ is  forward Cauchy. 
	\end{enumerate}     \end{lem}
	
	\begin{proof} (i) Since $r_i \leq r_j\with X(x_i,x_j)$ whenever  $i\leq j$, it follows that \[1=\bv_{i\in D}r_i\leq\bv_{i\in D}\bw_{i\leq j}r_j \leq \bv_{i\in D}\bw_{i\leq j\leq k}X(x_j,x_k),\] then $\{x_i\}_{i\in D}$ is   forward Cauchy.  
		
		(ii) If $\&$ is  Archimedean, it is readily verified that for all $r,s,t\in[0,1]$, $$0<r\leq s\with t\implies t\geq s\rightarrow r.$$
		Without loss of generality we assume that   $r_i>0$ for all $i\in D$. Since $r_i\leq r_j\with X(x_i,x_j)$ whenever $i\leq j$, then  $X(x_i,x_j)\geq r_j\rightarrow r_i$ whenever  $i\leq j$. Since the net $\{r_i\}_{i\in D}$ converges to its join and the implication operator of an Archimedean   continuous t-norm is continuous except possibly at $(0,0)$, it follows that $X(x_i,x_j)$ tends to $1$, so $\{x_i\}_{i\in D}$ is forward Cauchy.
	\end{proof}
	
	We say that a real-enriched category $X$ has  \emph{property {\rm(R)}} if, for each pair $(s,t)$ of elements of $[0,1]$ with $0<s\leq t$, there exists $r<1$ such that  \[(x,t\with r')\sqsubseteq (y,s)\iff(x,r')\sqsubseteq (y,t\rightarrow s)\]  for all $x,y\in X$  and   all $r'\geq r$.
	
	\begin{lem}\label{expansion stable} Let $X$ be a real-enriched category  with property {\rm(R)}. Then for all $x,y\in X$, the following are equivalent: 
		\begin{enumerate}[label=\rm(\arabic*)] 
			\item $(x,1)\sqsubseteq(y,1)$. \item $(x,s)\sqsubseteq(y,s)$ for all $s\not=0$. 
			\item $(x,s)\sqsubseteq(y,s)$ for some $s\not=0$.  \end{enumerate}\end{lem}
	
	\begin{proof}It suffices to check $(3)\Rightarrow(1)$. Since $0<s\leq s$, there is some $r<1$ such that  \[(x,s\with r')\sqsubseteq (y,s)\iff(x,r')\sqsubseteq (y,s\rightarrow s)\] whenever $r'\geq r$. Putting $r'=1$ gives that $(x,1)\sqsubseteq (y,1)$. \end{proof}

	\begin{prop} \label{Archi vs reciprocal} {\rm (Yang and Zhang \cite{YangZ2023})} The following are equivalent: \begin{enumerate}[label=\rm(\arabic*)] 
			\item The continuous t-norm $\with$ is Archimedean. \item Every real-enriched category has property {\rm(R)}.  \end{enumerate}    \end{prop}
	
	\begin{proof} $(1)\Rightarrow(2)$ Since $\with$ is   Archimedean, it is either isomorphic to the \L ukasiewicz t-norm or to the product t-norm. In the following we check the conclusion for the case that $\with$ is isomorphic to the  \L ukasiewicz t-norm, leaving the other case  to the reader.
		Without loss of generality, we assume that $\with$ is, not only isomorphic to, the  \L ukasiewicz t-norm; that is, $r\with s=\max\{0,r+s-1\}$.  
		
		Suppose $X$ is a real-enriched category and $0<s\leq t$. If $t=1$,   it is trivial  that \[(x,t\with r')\sqsubseteq (y,s)\iff(x,r')\sqsubseteq (y,t\rightarrow s) \]  for all $r'>0$ and  $x,y\in X$, so each $r>0$ satisfies the requirement. If $t<1$, pick $r\in(1-t,1)$. Then  for all $r'\geq r$ and   $x,y\in X$, \begin{align*}  (x,t\with r')\sqsubseteq (y,s)  &\iff r'+t-1\leq s+X(x,y)-1 \\ &\iff r'\leq s-t+X(x,y)\\ &\iff (x,r')\sqsubseteq(y,t\rightarrow s),\end{align*} which shows that $r$ satisfies the requirement.
		
		$(2)\Rightarrow(1)$ We wish to show that $\with$ has no nontrivial idempotent element. It suffices to show that for all $s\not=0$ and $q\in[0,1]$, if $s\leq s\with q$ then $q=1$. Consider the real-enriched category  $X=\{x,y\}$ with $X(x,x)=X(y,y)=1$ and $X(x,y)=q=X(y,x)$. Since $(x,s)\sqsubseteq (y,s)$ and $X$ has property   {\rm(R)}, then $(x,1)\sqsubseteq(y,1)$, hence $1\leq 1\with X(x,y)=q$. 
	\end{proof}
	
	\begin{cor}\label{characeterizing R} For each real-enriched category $X$, the following are equivalent: 
		\begin{enumerate}[label=\rm(\arabic*)] 
			\item  $X$ has property {\rm (R)}. 
			\item For all $x,y\in X$, if $X(x,y)\geq p$  for some idempotent element $p>0$, then $X(x,y)=1$. 
	\end{enumerate}    \end{cor}
	
	\begin{proof} $(1)\Rightarrow(2)$ Suppose   $X(x,y)\geq p>0$ with $p$  idempotent. Let $s=t=p$. Since $t=p= s\with X(x,y)$, it follows that $(x,p)\sqsubseteq (y,p)$. Then by property {\rm (R)} one gets that $(x,1)\sqsubseteq(y,1)$, hence $X(x,y)=1$.  
		
		$(2)\Rightarrow(1)$  By Proposition \ref{Archi vs reciprocal} we may assume that $\&$ is non-Archimedean with   a nontrivial idempotent $q$. Let $X$ be a real-enriched category such that $X(x,y)=1$ whenever $X(x,y)\geq p$  for some idempotent element $p>0$.  We wish to show that $X$ has property {\rm (R)}.  For this we show that if $0<s\leq t$ and $r>q$, then    \[(x,t\with r)\sqsubseteq (y,s)\iff(x,r)\sqsubseteq (y,t\rightarrow s)\]   
		for all  $x,y\in X$. 
		The direction $\Longleftarrow$ is trivial. For the other direction, 
		suppose $(x,t\with r)\sqsubseteq (y,s)$. We proceed with two cases.
		If  $q\leq t$, then  $q\leq t\with r\leq s\with X(x,y)$, so $X(x,y)=1$ by assumption,  hence $t\with r\leq s$. Therefore $r\leq t\ra s$ and consequently, $(x,r)\sqsubseteq (y,t\rightarrow s)$.
		If $q>t$, then $t=t\with r\leq s\with X(x,y)$, so $t=s$ and $X(x,y)\geq p$, where $p$ is the least idempotent element in $[t,1]$. By assumption we have $X(x,y)=1$,  then $(x,r) \sqsubseteq (y,t\rightarrow s)$.  
	\end{proof} 
	
	\begin{lem}\label{yoneda complete = dcpo} Suppose $X$ is a real-enriched category  with property {\rm(R)},   $\{(x_i,r_i)\}_{i\in D}$ is a directed subset of $\mathrm{B}X $.  If 
		\begin{enumerate}[label=\rm(\roman*)] 
			\item $\bv_{i\in D}r_i=1$, 
			\item $(b,1)$ is a join of   $\{(x_i,r_i)\}_{i\in D}$, and \item for each $t>0$ the directed set $\{(x_i,t\with r_i)\}_{i\in D}$ of $\mathrm{B}X $ has a join, 
		\end{enumerate} then  $b$ is a Yoneda limit  of   $\{x_i\}_{i\in D}$.
	\end{lem}
	
	\begin{proof}   
		We wish to show that for all $y\in X$, \[X(b,y)=\bv_{i\in D}\bw_{j\geq i}X(x_j,y).\]
		
		Fix $i\in D$. For each $j\geq i$, since $(x_j,r_j)\sqsubseteq(b,1)$, then \begin{align*}r_i\with X(b,y)&\leq r_j\with X(b,y) 
			\leq X(x_j,b)\with X(b,y)
			\leq X(x_j,y),\end{align*} hence $$r_i\with X(b,y)\leq \bw_{j\geq i}X(x_j,y).$$  
		Therefore,  \begin{align*}X(b,y)&=\bv_{i\in D}(r_i\with X(b,y))  \leq \bv_{i\in D}\bw_{j\geq i}X(x_j,y).\end{align*}
		
		For the converse inequality, let $$t=\bv_{i\in D}\bw_{j\geq i}X(x_j,y).$$ We  show that $t\leq X(b,y)$. We assume   $t>0$ and finish the proof in two steps. 
		
		{\bf Step 1}. $(b,t)$ is a join of $\{(x_i, t\with r_i)\}_{i\in D}$. 
		
		By assumption $\{(x_i, t\with r_i)\}_{i\in D}$  has a join, say $(z,s)$. We only need to show that   $(b,t)$ is equivalent to $(z,s)$.    Since $(b,t)$ is an upper bound of   $\{(x_i, t\with r_i)\}_{i\in D}$, then $(z,s)\sqsubseteq (b,t)$, hence $s\leq t$. Since $r_i$ tends to $1$ and $t\with r_i\leq s\with X(x_i,z)$ for all $i\in D$, it follows that $t\leq s$.  Therefore $0<s=t$.
		Since $r_i$ tends to $1$, we may assume that  all $r_i$ are large enough. Since $(z,t)$ is a join of $\{(x_i, t\with r_i)\}_{i\in D}$,    $t\with r_i\leq t\with X(x_i,z)$  for all $i\in D$, then by property {\rm(R)} of $X$, $r_i\leq(t\rightarrow t)\with X(x_i,z)=X(x_i,z)$ for all $i\in D$. This shows that $(z,1)$ is an upper bound of $\{(x_i,r_i)\}_{i\in D}$, so   $(b,1)\sqsubseteq (z,1)$,  and  then  $(b,t)\sqsubseteq (z,t)$. Therefore, $(b,t)$ is equivalent to $(z,s)$.
		
		{\bf Step 2}.   $t\leq X(b,y)$. 
		
		For each $i\in D$,    \begin{align*}r_i\with t&= r_i\with\bv_{k\in D}\bw_{j\geq k}X(x_j,y) \\
			&= r_i\with\bv_{k\geq i}\bw_{j\geq k}X(x_j,y)  \\
			&= \bv_{k\geq i}\bw_{j\geq k}r_i\with X(x_j,y) \\
			&\leq \bv_{k\geq i}\bw_{j\geq k}r_j\with X(x_i,x_j)\with X(x_j,y)\\ &\leq X(x_i,y),\end{align*} then $(x_i, t\with r_i)\sqsubseteq (y,1)$. This shows that $(y,1)$ is an upper bound of  $\{(x_i, t\with r_i)\}_{i\in D}$, then $(b,t)\sqsubseteq(y,1)$  and   $t\leq X(b,y)$, as desired. \end{proof}
	
	The condition (iii) in Lemma \ref{yoneda complete = dcpo} cannot be dropped. 
	
	\begin{exmp} \label{second example} (Yang and Zhang \cite{YangZ2023})  Suppose $\&$ is the \L ukasiewicz t-norm or the product t-norm. 
		Consider the real-enriched category $X$, where $X=\{1\}\cup\{1-1/n\mid n\geq2\}$ and  \[X(x,y)=\begin{cases}1/2 &x=1, y\not=1,\\ x\ra y&{\rm otherwise}.\end{cases}\] We assert that $X$ is not Yoneda complete, but every directed subset $\{(x_i,r_i)\}_{i\in D}$   of $\mathrm{B} X$ with $\bv_{i\in D}r_i=1$ has a join. 
		
		The sequence $\{1-1/n\}_{n\geq2}$ is forward Cauchy but has no Yoneda limit, so $X$ is not Yoneda complete. Now we show that every directed subset $\{(x_i,r_i)\}_{i\in D}$   of $\mathrm{B}X$ with $\bv_{i\in D}r_i=1$ has a join. Since   $\{x_i\}_{i\in D}$ is   forward Cauchy, either  $\{x_i\}_{i\in D}$ is  eventually constant or  $\{x_i\}_{i\in D}$ converges to $1$ (in the usual sense).

		Case 1. $\{x_i\}_{i\in D}$ converges to $1$. Then we claim that $(1,1)$ is the only upper bound, hence a join of $\{(x_i,r_i)\}_{i\in D}$. Since $X(z,1)=1$ for all $z\in X$, then $r_i\leq 1= 1\with X(x_i,1)$ for all $i\in D$,  hence $(1,1)$ is an upper bound of $\{(x_i,r_i)\}_{i\in D}$. Let $(y,s)$ be an upper bound of $\{(x_i,r_i)\}_{i\in D}$. Then $s=1$ and  $r_i\leq X(x_i,y)$ for all $i\in D$. Since both  $\{r_i\}_{i\in D}$  and $\{x_i\}_{i\in D}$ converge to $1$, we must have $y=1$. Therefore, $(1,1)$ is the only upper bound  of $\{(x_i,r_i)\}_{i\in D}$.
		
		Case 2. $\{x_i\}_{i\in D}$ is  eventually constant. Then there is some $b\in X$ and some $i\in D$ such that $x_j=b$ whenever $j\geq i$. It is clear that $(b,1)$ is a join of $\{(x_i,r_i)\}_{i\in D}$.  \end{exmp}
	
	Now we are able to prove Theorem \ref{yoneda via formal ball}.
	
	\begin{proof}[Proof of Theorem \ref{yoneda via formal ball}] $(1)\Rightarrow(2)$  Let $\{(x_i,r_i)\}_{i\in D}$ be a directed subset of $\mathrm{B}X $. If $r_i=0$ for all $i\in D$, then  $(x,0)$ is a join of $\{(x_i,r_i)\}_{i\in D}$ for any $x\in X$. If  $r_i>0$ for some $i\in D$, by Lemma \ref{directed set is forward Cauchy}\thinspace(ii) the net $\{x_i\}_{i\in D}$ is forward Cauchy, so it has a Yoneda limit, say $x$. By Lemma \ref{join of directed set in BX}, $(x,r)$ is a join of $\{(x_i,r_i)\}_{i\in D}$, where $r=\bv_{i\in D}r_i$. This shows that $\mathrm{B}X $   is directed complete.
		
		$(2)\Rightarrow(1)$ We show that every ideal $\phi$ of $X$ has a colimit. By Proposition \ref{characterization of  ideal}, there is a directed subset  $\{(x_i,r_i)\}_{i\in D}$  of $\mathrm{B}X$ such that $\bv_{i\in D}r_i=1$ and $$\phi=\bv_{i\in D}\bw_{j\geq i}X(-,x_j).$$ By assumption $\{(x_i,r_i)\}_{i\in D}$ has a join, say $(b,1)$. By Proposition \ref{Archi vs reciprocal} and Lemma \ref {yoneda complete = dcpo}, $b$ is a Yoneda limit of the forward Cauchy net $\{x_i\}_{i\in D}$, hence a colimit of $\phi$ by Lemma \ref{yoneda limit as colimits}. \end{proof}
	
	The argument of  Theorem \ref{yoneda via formal ball} actually proves that for each continuous t-norm and each real-enriched category $X$,  it holds that 
	\begin{enumerate}[label=\rm(\roman*)]\setlength{\itemsep}{0pt} \item If $X$ is  Yoneda complete, then each directed subset $(x_i,r_i)_{i\in D}$ of $\mathbf{B}X $ with $\sup_{i\in D}r_i=1$ has a join. \item If $X$ has  property (R) and  $\mathbf{B} X$ is directed complete, then $X$ is Yoneda complete. \end{enumerate}  
	
	The following examples show that in Theorem \ref{yoneda via formal ball} the requirement that $\with$ is Archimedean  is indispensable.  
	
	\begin{exmp}\label{Example-G} (Yang and Zhang \cite{YangZ2023})  This example shows that if $\with$ is non-Archimedean, then there is a real-enriched category that is Yoneda complete, but its set of formal balls is not directed complete.
		Since $\with$ is  non-Archimedean, there is some $b\in(0,1)$ such that $b\with b=b$. Consider the  real-enriched category $X$, where  $X=(0,b)$ and   \[X(x,y)=\begin{cases}1 & x=y,\\
			\min\{x\rightarrow y,y\rightarrow x\} & x\not=y.\end{cases}\]
		
		Since $X(x,y)\leq b$ whenever $x\not=y$, every forward Cauchy net of $X$ is eventually constant, so $X$ is Yoneda complete. In the following we show that $\mathrm{B}X$ is not directed complete. Pick a strictly increasing sequence $(x_n)_{n\geq1}$  in $(0,b)$ that converges to $b$. For each $n$ let $r_n=x_n$. We claim that the subset $(x_n,r_n)_{n\geq1}$ of $\mathrm{B}X$ is directed and has no join.
		
		By continuity of  $\with$, for all $n\leq m$, we have  \begin{align*}r_n&=x_n =x_m\with(x_m\rightarrow x_n) = r_m\with X(x_n,x_m),\end{align*} so $(x_n,r_n)\sqsubseteq(x_m,r_m)$  and consequently, $(x_n,r_n)_{n\geq1}$ is directed.

		Next  we show that $(x_n,r_n)_{n\geq1}$ does not have a join. Suppose on the contrary that $(x,r)$ is a join of $(x_n,r_n)_{n\geq1}$. Since $x<b$, there is some $n_0$ such that $x<x_m$ whenever $m\geq n_0$.   Since $(x,r)$ is an upper bound  of $(x_n,r_n)_{n\geq1}$, for each $m\geq n_0$  we have \[r_m\leq r\with X(x_m,x)\leq x_m\rightarrow x,\] this is impossible since the left side tends to $b$, while the right side tends to $x$.\end{exmp}
	
	\begin{exmp}\label{exmp3} (Yang and Zhang \cite{YangZ2023}) If $\with$ is the G\"{o}del t-norm, then there is a real-enriched category  that is not Yoneda complete, but its set  of formal balls is   directed complete.  Consider the  real-enriched category $X$, where $X=\{1\}\cup\{1-1/n\mid n\geq2\}$ and  $$X(x,y)=\begin{cases}1 & x=y,\\ 1/3 & x=1, y\not=1,\\ \min\{x,y\} &{\rm otherwise}.  \end{cases}$$  We claim that  $X$ is not Yoneda complete, but $\mathrm{B}X$   is directed complete.
		
		For each $n\geq 2$, let $x_n=1-1/n$. It is readily verified that the sequence $(x_n)_{n\geq2}$ is  forward Cauchy and  has no Yoneda limit, so $X$ is not Yoneda complete.  It remains to check that $\mathrm{B}X$   is directed complete. Let $\{(x_i,r_i)\}_{i\in D}$ be a directed subset  of $\mathrm{B}X$. Then $\{r_i\}_{i\in D}$ is a monotone net of real numbers. Let $r=\bv_{i\in D}r_i$. We distinguish two cases.
		
		Case 1. The net $\{x_i\}_{i\in D}$ is eventually constant; that is, there exist $b\in X$ and   $i\in D$ such that $x_j=b$ whenever $j\geq i$.  In this case   $(b,r)$ is a join of $\{(x_i,r_i)\}_{i\in D}$.
		
		Case 2. The net $\{x_i\}_{i\in D}$ is not eventually constant. Then for each $i$ there is some $j\geq i$ such that $x_i\not=x_j$, hence $$r_i\leq r_j\with X(x_i,x_j)\leq \min\{ r_j,x_i,x_j\}\leq x_i.$$ Let  $b=\min\{x\in X\mid r\leq x\}.$
		Then $(b, r)$ is a join of $\{(x_i,r_i)\}_{i\in D}$. \end{exmp}
	
	
	\section{Flat completeness} \label{flat complete}
	
	\begin{defn}
		Suppose $\phi$ is a weight of a real-enriched category $X$.  \begin{enumerate}[label={\rm(\roman*)}] 
			\item  $\phi$   is  conically flat if the functor $\phi\cpt-\colon (\CPd X)^{\rm op}\lra\sV$ preserves limits of finite conical coweights, where $\sV=([0,1],\alpha_L)$. \item  $\phi$   is  flat if $\phi\cpt-\colon (\CPd X)^{\rm op}\lra\sV$ preserves all finite  limits. \end{enumerate} \end{defn}
	
	The definition of flat weights is based on the notion of \emph{flat distributors} in  B\'{e}nabou \cite[Section 6]{Benabou2000} and the characterization of flat functors in Borceux \cite{Borceux1994a}. When the continuous t-norm is isomorphic to the product t-norm, conically flat weights  appear as \emph{flat left modules of quasi-metric spaces} in  Vickers \cite{Vickers2005}. 
	
	As we shall see below, in contrast to the characterization of Cauchy weights in Proposition \ref{Cauchy weight via sub}, flat weights and ideals are not the same things in general. 
	
	We remind the reader that the underlying order of $\CPd X$ is opposite to the pointwise order of coweights,   the underlying order of $(\CPd X)^{\rm op}$ agrees with the pointwise order,   the cotensor of $r$ with $\psi$ in $(\CPd X)^{\rm op}$ is given by $r\ra\psi$.  
	
	Consider the following conditions: 
	\begin{enumerate}[label={\rm(\alph*)}]  
		\item $1=\phi\cpt1_X= \sup_{x\in X}\phi(x)$; that is, $\phi$ is inhabited. 
		\item For all coweights $\psi_1$ and $\psi_2$ of $X$, $\phi\cpt(\psi_1\wedge\psi_2)= (\phi\cpt\psi_1)  \wedge(\phi\cpt\psi_2) $. 
		\item For all $r\in[0,1]$ and all coweights $\psi$ of $X$, $\phi\cpt(r\ra\psi)=r\ra\phi\cpt\psi$. \end{enumerate} 
	Then, $\phi$ is conically flat if and only if it satisfies (a) and (b); $\phi$ is flat if and only if it satisfies (a)--(c). 
	
	The following characterization of conically flat weights is due to Guti\'{e}rrez Garc\'{i}a,   H\"{o}hle and  Kubiak \cite{GHK2022}.
	
	\begin{prop}  Suppose $\phi$ is a weight of a real-enriched category $X$. Then, $\phi$ is conically flat if and only if it is inhabited and $$(p_1\with \phi(x_1))\wedge(p_2\with \phi(x_2))=\sup_{x\in X}((p_1\with X(x_1,x))\wedge (p_2\with X(x_2,x)))\with\phi(x) $$ for all $x_1,x_2\in X$ and all $p_1,p_2\in[0,1]$. \end{prop}
	
	\begin{proof} Necessity is obvious, sufficiency follows from   that   $$\psi=\sup_{x\in X}\psi(x)\with X(x,-) $$ for each coweight $\psi$ of $X$. \end{proof}
	
	We leave it to the reader to check that when the continuous t-norm $\with$ is the G\"{o}del t-norm,  $\phi$ is conically flat if and only if it is inhabited and satisfies $$\phi(x_1)\wedge\phi(x_2)=\sup_{x\in X}X(x_1,x)\wedge X(x_2,x)\wedge\phi(x).$$

	Proposition \ref{Cauchy weight via sub} shows that every Cauchy weight   is flat. The following conclusion, due to Lai, D\thinspace{\&}\thinspace{G} Zhang \cite{LZZ2020}, says that every ideal is conically flat.
	
	\begin{prop}\label{irr is flat}  Every ideal is conically flat. \end{prop}
	
	\begin{proof} Suppose $\phi$ is an ideal of a real-enriched category $X$.  We wish to show that $\phi$ is conically flat. Since every ideal is inhabited, we only need to check that  for all coweights $\psi_1$ and $\psi_2$ of $X$, $\phi\cpt(\psi_1\wedge\psi_2)= (\phi\cpt\psi_1)  \wedge(\phi\cpt\psi_2) $. 
		
		By Proposition \ref{characterization of  ideal} there is a forward Cauchy net $\{x_i\}_{i\in D}$   such that \[\phi=\bv_{i\in D}\bw_{j\geq i}X(-,x_j).\]   We claim that for each coweight $\psi$ of $X$, \[\phi\cpt\psi=\bv_{i\in D}\bw_{j\geq i}\psi(x_j).\]
		Since $\psi$ is a coweight,   $\{\psi(x_i)\}_{i\in D}$ is a forward Cauchy net of $\sV=([0,1],\alpha_L)$, then 
		\begin{align*}\phi\cpt\psi &= \bw_{p\in[0,1]}(\sub_X(\phi,\psi\ra p)\ra p) &(\text{Lemma \ref{sub vs tensor}})\\ 
			&=  \bw_{p\in[0,1]}\Big(\sub_X\Big(\bv_{i\in D}\bw_{j\geq i}X(-,x_j),\psi\ra p\Big)\ra p\Big) \\ &=  \bw_{p\in[0,1]}\Big(\Big(\bv_{i\in D}\bw_{j\geq i}\sub_X \big(X(-,x_j),\psi\ra p \big)\Big)\ra p\Big)&(\text{Lemma \ref{Yoneda limits in PX}}) \\ &=  \bw_{p\in[0,1]}\Big(\Big(\bv_{i\in D}\bw_{j\geq i}(\psi(x_j)\ra p)\Big)\ra p\Big) &(\text{Yoneda lemma})\\ 
			&=  \bw_{p\in[0,1]}\Big(\Big(\bv_{i\in D}\bw_{j\geq i}\psi(x_j)\ra p\Big)\ra p\Big) &(\text{Lemma \ref{yoneda limit in Q}})\\ 
			&= \bv_{i\in D}\bw_{j\geq i}\psi(x_j). \end{align*}
		
		Now for any coweights $\psi_1$ and $\psi_2$ of $X$, we have \begin{align*}(\phi\cpt \psi_1)\wedge(\phi\cpt\psi_2) 
			&=\bv_{i\in D}\bw_{j\geq i} \psi_1(x_j) \wedge\bv_{i\in D}\bw_{j\geq i} \psi_2(x_j)  \\ 
			&   = \bv_{i\in D}\bw_{j\geq i}(\psi_1(x_j)\wedge\psi_2(x_j))  \\ 
			&  = \phi\cpt(\psi_1\wedge\psi_2),   \end{align*}
		which completes the proof.
	\end{proof}
	
	In 2005, Vickers \cite{Vickers2005} proved that every flat left module of a generalized metric space is generated by a forward Cauchy net. Said differently, when the continuous t-norm is isomorphic to the product t-norm, a weight of a real-enriched category is conically flat if and only if it is an ideal.  The following theorem  says that a necessary and sufficient condition for ideals, flat weights, and conically flat weights to coincide with each other is that the continuous t-norm is Archimedean.
	
	\begin{thm}\label{flat=ideal}  {\rm (Lai, D\thinspace{\&}\thinspace{G} Zhang \cite{LZZ2020})} The following are equivalent: 
		\begin{enumerate}[label=\rm(\arabic*)]  
			\item The continuous t-norm $\&$ is Archimedean.  
			\item  Every ideal is a flat weight.
			\item Every conically flat weight is an ideal.
		\end{enumerate} In this case, ideals, flat weights, and conically flat weights coincide with each other. \end{thm} 
	
	\begin{proof} $(1)\Rightarrow(2)$  It suffices to show that if $\phi$ is an ideal of a real-enriched category $X$, then $\phi\cpt(r\ra\psi)=r\ra\phi\cpt\psi$ for  all $r\in[0,1]$ and all coweights $\psi$ of $X$. First we note that when $\with$ is   continuous and Archimedean,  for each $r\in[0,1]$ the function $r\ra-\colon[0,1]\lra[0,1]$ is continuous. Pick a forward Cauchy net $\{x_i\}_{i\in D}$   such that \[\phi=\bv_{i\in D}\bw_{j\geq i}X(-,x_j).\] Then by the argument of Proposition \ref{irr is flat} we have $$\phi\cpt(r\ra\psi)= \bv_{i\in D}\bw_{j\geq i}(r\ra\psi(x_j))= r\ra \bv_{i\in D}\bw_{j\geq i}\psi(x_j)=r\ra\phi\cpt\psi,$$  the second equality holds due to the continuity of $r\ra-\colon[0,1]\lra[0,1]$.
		
		$(2)\Rightarrow(1)$  Suppose on the contrary that $\&$ is not Archimedean. Take an  idempotent element $b$ of $\with$ other than $0$ and $1$,   consider the ideal $$\phi=\bv_{x<b}\alpha_L(-, x)$$ and the coweight $\psi=\alpha_L(b,-)$ of $\sV=([0,1],\alpha_L)$. Then $$\phi\cpt(b\ra\psi)=\bv_{x<b}(b\ra(b\ra x))= b<1= b\ra \bv_{x<b}x = b\ra\phi\cpt\psi,$$ contradicting that $\phi$ is flat.
		
		$(1)\Rightarrow(3)$ Since a continuous Archimedean t-norm is either isomorphic to the product t-norm or  isomorphic to the \L ukasiewicz t-norm, one readily verifies 
		\begin{itemize}
			\item if $0< r\with s$, then $\with$ is strictly monotone at   $(r,s)$; 
			\item if $0<t\leq r\with s$, then $r\ra t\leq s$. \end{itemize}
		
		Suppose $\phi$ is a flat weight of a real-enriched category $X$. Let  \[\mathrm{B}\phi= \{(x,r)\in X\times (0,1]\mid \phi(x)>r\}.  \]  We assert that $\mathrm{B}\phi$ is a directed subset of the ordered set $\mathrm{B}X$ of formal balls of $X$.  
		
		It is clear that $\mathrm{B}\phi$ is not empty.  For  $(x,r)$ and $(y,s)$ of $\mathrm{B}\phi$, consider the coweights: \[\psi_1 = \max\{r\ra0,s\}\with X(x,-),\quad  \psi_2 = \max\{s\ra0,r\}\with X(y,-).\] Since $\phi(x)>r$ and $\phi(y)>s$, it follows that \[\phi\cpt\psi_1=\max\{r\ra0,s\}\with \phi(x) >r\with  s\] and \[\phi\cpt\psi_2=\max\{s\ra0,r\}\with\phi(y) >r\with s,\] hence \[\phi\cpt(\psi_1\wedge\psi_2)>r\with s.\] So there exists   $z\in X$ such that \[\min\{\phi(z)\with\psi_1(z),\phi(z)\with\psi_2(z)\})>r\with  s.\] This implies   \[\phi(z)\with(\max\{r\ra0,s\}\with X(x,z))>r\with s\] and   \[\phi(z)\with(\max\{s\ra0,r\}\with X(y,z))>r\with s,\] in particular, $X(x,z)> r$ and $X(y,z)> s$. Let $t$ be the larger one of $$ (\max\{r\ra0,s\}\with X(x,z))\ra(r\with s)$$ and $$(\max\{s\ra0,r\}\with X(y,z))\ra(r\with s).$$ Then it is routine to check that \[\phi(z)>t,\quad r\leq t\with X(x,z),\quad s\leq t\with X(y,z),\] which show that $(z,t)$  is an upper bound of $(x,r)$ and $(y,s)$ in $\mathrm{B}\phi$.
		
		
		Index the directed set $\mathrm{B}\phi$ by itself as   $\{(x_i,r_i)\}_{i\in D}$. Then by the argument in Proposition \ref{characterization of  ideal} one sees that the net $\{x_i\}_{i\in D}$ is  forward Cauchy  and    \[ \phi(x)= \sup_{i\in D}\inf_{j\geq i }X(x,x_j) \]  for all $x\in X$,  so $\phi$ is an ideal.

		$(3)\Rightarrow(1)$  Suppose on the contrary that $b$ is a  non-trivial idempotent element of $\&$. Pick $a\in (0,b)$. Consider the weight $\phi$ of $\sV$ given by $\phi(x)=b\vee(x\ra a)$. Since neither $\phi(x)\leq b$ for all $x$ nor $\phi(x)\leq x\ra a$ for all $x$, then $\phi$ is not an ideal. In the following we derive a contradiction by showing that  $\phi$ is conically flat.
		
		It is clear that $\phi$ is inhabited. It remains to check that for all coweights   $\psi_1,\psi_2 $  of $\sV$, $\phi\cpt(\psi_1\wedge\psi_2)= (\phi\cpt\psi_1)\wedge(\phi\cpt\psi_2)$.
		Since $b$ is idempotent and $\psi_i$  ($i=1,2$) is monotone,  then 
		\begin{align*} \phi\cpt \psi_i 
			&= \sup_{x\in [0,1]}((b\with \psi_i(x))\vee((x\ra a)\with \psi_i(x)))\\ 
			&= (b\wedge\psi_i(1))\vee \psi_i(a)\\ 
			& = (b\vee\psi_i(a))\wedge\psi_i(1),   
		\end{align*} therefore 
		\begin{align*}(\phi\cpt\psi_1)\wedge(\phi\cpt\psi_2) 
			&= (b\vee\psi_1(a)) \wedge\psi_1(1) \wedge(b\vee\psi_2(a))\wedge\psi_2(1)\\ 
			&= (b\vee(\psi_1(a)\wedge\psi_2(a))\wedge (\psi_1(1)\wedge\psi_2(1))\\ 
			&= \phi\cpt(\psi_1\wedge\psi_2).    \qedhere \end{align*}    \end{proof}
	
	\begin{defn}A real-enriched category is (conically, resp.) flat complete if it is separated and each of its (conically, resp.) flat weights has a colimit. \end{defn}
	
	For each real-enriched category $X$, let $$\CF X$$ be the subcategory of $\CP X$ composed  of flat weights of $X$. The assignment $X\mapsto \CF X$ defines a functor $\CF\colon \QOrd\lra\QOrd$, which is a subfunctor of the presheaf functor $\CP\colon \QOrd\lra\QOrd$.

	\begin{thm}\label{univeral property of FX}Suppose $X$ is a real-enriched category. Then \begin{enumerate}[label={\rm(\roman*)}] 
			\item   $\CF X$ is flat complete. 
			\item For each functor $f\colon X\lra Y$ with $Y$ flat complete, there is a unique   functor $\overline{f}\colon\CF X\lra Y$ that extends $f$ and preserves colimits of flat weights.  
	\end{enumerate} \end{thm}
	
	\begin{proof} Write $\mathfrak{i}\colon\CF X\lra\CP X$ for the inclusion functor and ${\sf t}\colon X\lra  \mathcal{F}X$ for the functor  obtained by restricting the codomain of the Yoneda embedding $\sy\colon X\lra\CP X$ to $\mathcal{F}X$. It is clear that the composite $ \mathfrak{i}\circ{\sf t}$ is the Yoneda embedding $\sy\colon X\lra\CP X$.  
		
		(i) We  show that $\CF X$ is closed in $\CP X$ under formation of colimits of flat weights. That means, for each flat weight $\Phi$ of $\CF X$, the colimit  of the inclusion functor  $\mathfrak{i}\colon \CF X\lra\CP X$ weighted by $\Phi$ is a flat weight of $X$. Since the composite $\mathfrak{i}\circ{\sf t}$ is the Yoneda embbeding, by Example \ref{sup in PX} it holds that $${\colim}_\Phi\mathfrak{i} = \Phi\circ \mathfrak{i}^*\circ\mathfrak{i}_*\circ {\sf t}_*= \Phi\circ{\sf t}_*.$$   
		We wish to show that $\Phi\circ{\sf t}_*$ is a flat weight  of $X$.   
		
		{\bf Step 1}.   $\Phi\circ{\sf t}_*$ is inhabited. This is easy, since $$\bv_{x\in X}\Phi\circ{\sf t}_*(x)=\bv_{x\in X}\bv_{\phi\in\CF X}\Phi(\phi)\with\phi(x) = \bv_{\phi\in\CF X}\Big(\Phi(\phi)\with\bv_{x\in X}\phi(x)\Big)=1.$$ 
		
		{\bf Step 2}. For any coweights $\psi_1$ and $\psi_2$ of $X$,  $$\Phi\circ{\sf t}_* \cpt(\psi_1\wedge \psi_2)= (\Phi\circ{\sf t}_*\cpt\psi_1)\wedge (\Phi\circ{\sf t}_*\cpt\psi_2).$$  
		
		To see this, for each coweight $\psi$ of $X$ consider the coweight  \[-\cpt\psi \colon \CF X\lra\sV  \]   of $\CF X$.
		Then \begin{align*}\Phi\cpt(-\cpt \psi)
			&= \sup_{\phi\in\CF X}\Big(\Phi(\phi)\with \sup_{x\in X}(\phi(x)\with  \psi(x))\Big)\\ &=\sup_{x\in X}  \sup_{\phi\in\CF X}(\Phi(\phi)\with  \phi(x))\with  \psi(x) \\ &=\sup_{x\in X} (\Phi\circ{\sf t}_*)(x)\with  \psi(x)  \\ 
			& =\Phi\circ{\sf t}_*\cpt\psi. 
		\end{align*} Since $\Phi$ is flat and each element of $\mathcal{F}X$ is flat, it follows that \begin{align*}  
			\Phi\circ{\sf t}_*\cpt(\psi_1\wedge \psi_2)  
			&=\Phi\cpt( -\cpt(\psi_1 \wedge  \psi_2))\\
			&=\Phi\cpt((-\cpt \psi_1)\wedge (-\cpt\psi_2))  \\
			&= (\Phi\cpt(-\cpt \psi_1))\wedge (\Phi\cpt(-\cpt\psi_2)) \\
			&=(\Phi\circ{\sf t}_*\cpt\psi_1)\wedge( \Phi\circ{\sf t}_*\cpt\psi_2).
		\end{align*} 
		
		{\bf Step 3}. Similar to Step 2, one verifies that for all $r$ of $[0,1]$ and all coweight $\psi$  of $X$,  it holds that $$\Phi\circ{\sf t}_* \cpt(r\ra \psi)= r\ra\Phi\circ{\sf t}_*\cpt\psi.$$ 
		
		Therefore, $\Phi\circ{\sf t}_*$ is a flat weight  of $X$.
		
		(ii) The functor \[\overline{f}\colon \CF X\lra Y, \quad \phi\mapsto {\colim}_ \phi  f  
		\] clearly  extends $f$. It remains to check that $\overline{f}$ preserves colimits of flat weights. It is clear that $\overline{f}$ is the composite $\colim\circ  f_\exists\colon\CF X\lra\CF Y\lra Y$. Since $\colim\colon\CF Y\lra Y$ is left adjoint to  ${\sf t}\colon Y\lra\CF Y$,  it preserves colimits. Since $\CF X$ and $\CF Y$ are   closed in $\CP X$ and $\CP Y$  under formation of colimits of flat weights, respectively, and $f_\exists\colon\CP X\lra\CP Y$ preserves all colimits, it follows that the composite $\overline{f}$ preserves colimits of all flat weights.
		Uniqueness of $\overline{f}$ follows from  that  each flat weight $\phi$ of $X$ is the colimit of ${\sf t}\colon X\lra\CF X$ weighted by $\phi$.\end{proof} 
	
	For conically flat weights, there is a similar result. For each real-enriched category $X$, let $$\CF_{\rm c} X$$ be the subcategory of $\CP X$ composed  of conically flat weights of $X$. The assignment $X\mapsto \CF_{\rm c} X$ defines a functor $\CF_{\rm c}\colon \QOrd\lra\QOrd$, which is a subfunctor of the presheaf functor $\CP\colon \QOrd\lra\QOrd$.

	\begin{thm}\label{univeral property of FCX}Suppose $X$ is a real-enriched category. Then \begin{enumerate}[label={\rm(\roman*)}] 
			\item   $\CF_{\rm c} X$ is conically flat complete. \item For each functor $f\colon X\lra Y$ with $Y$ conically flat complete, there is a unique   functor $\overline{f}\colon\CF_{\rm c} X\lra Y$ that extends $f$ and preserves colimits of conically flat weights.  
	\end{enumerate} \end{thm}
	
	\section{Smyth completeness} \label{Smyth complete}

	Smyth completeness originated in the works of Smyth \cite{Smyth88,Smyth94} on quasi-uniform spaces. The following postulation, which extends that of Smyth complete quasi-metric spaces in  Goubault-Larrecq \cite{Goubault},  K\"{u}nzi and Schellekens \cite{KS2002}, is based on the characterization of Smyth complete quasi-uniform spaces in S\"{u}nderhauf \cite{Sunderhauf1995}.
	
	\begin{defn}A real-enriched category  is Smyth complete if each of its forward Cauchy nets  converges uniquely in the open ball topology of its symmetrization. \end{defn}
	
	Proposition \ref{limit in open ball top}  implies that if a net of $X$ converges in the open ball topology of its symmetrization, then it is a Cauchy net of $X$.  Therefore,
	
	\begin{prop}\label{FC in Smyth is C} If a real-enriched category $X$ is Smyth complete, then it is separated and all of its forward Cauchy nets are Cauchy.  \end{prop}	
	
	\begin{exmp}\label{When V is Smyth complete} This example concerns Smyth completeness of the real-enriched categories  ${\sf V}=([0,1],\alpha_L)$ and ${\sf V}^{\rm op}=([0,1],\alpha_R)$.  We distinguish three cases.  
		
		{Case 1}. The continuous t-norm $\&$ is not Archimedean. Pick some $a\in[0,1]$ with $0<a\leq a^+<1$, where $a^+$ is the least idempotent element in $[a,1]$. Since $\bv_{x<a}(a\ra x)\leq a^+$, then $[a,1]$ is open in $\sV$ and $[0,a]$ is open in $\sV^{\rm op}$, hence the singleton set $\{a\}$ is open in the symmetrization of $\sV$.  Since $\{a-1/n\}_n$ is forward Cauchy in $\sV$ and $\{a+1/n\}_n$ is forward Cauchy in $\sV^{\rm op}$, it follows that neither $\sV$ nor $\sV^{\rm op}$ is Smyth complete.
		
		{Case 2}. $\&$ is isomorphic to the {\L}ukasiewicz t-norm. In this case, the open ball topology of $\sV$ is $\{\varnothing,[0,1]\}\cup\{(a,1]\mid a<1\}$, the open ball topology of $\sV^{\rm op}$ is $\{\varnothing,[0,1]\}\cup\{[0,a)\mid a>0\}$.  It is readily verified that a net $\{x_i\}_{i\in D}$ is forward Cauchy in $\sV$ if and only if it is  Cauchy in the usual sense, if and only if it  is forward Cauchy in $\sV^{\rm op}$. Thus, both $\sV$ and $\sV^{\rm op}$ are Smyth complete. 
		
		{Case 3}. $\&$ is isomorphic to the product t-norm. In this case, the open ball topology of $\sV$ is $\{\varnothing,[0,1]\}\cup\{(a,1]\mid a<1\}$, the open ball topology of $\sV^{\rm op}$ is $\{\varnothing,\{0\},[0,1]\}\cup\{[0,a)\mid a>0\}$. Since the sequence  $\{1/n\}_{n\geq1}$ is forward Cauchy  in $\sV^{\rm op}$, but not convergent in its symmetrization, $\sV^{\rm op}$ is not Smyth complete. Since a net $\{x_i\}_{i\in D}$ is forward Cauchy in $\sV$ if and only if it is eventually constant or converges (in the usual sense) to some element other than $0$, it follows that $\sV$ is Smyth complete. \end{exmp} 
	
	

	The following theorem says that Smyth completeness of a real-enriched category can be characterized purely in terms of its categorical structure, without resort to its symmetrization and   open ball topology.
	
	\begin{thm}\label{Smyth via bilimit} {\rm (Yu and Zhang \cite{YuZ2024})} For each real-enriched category $X$, the following are equivalent: \begin{enumerate}[label=\rm(\arabic*)] 
			\item $X$ is Smyth complete. \item Every forward Cauchy net $\{x_i\}_{i\in D}$ of $X$ has a unique bilimit.  \item $X$ is separated and every ideal of $X$ is representable.
			\item $X$ is Cauchy complete and every ideal of $X$ is a Cauchy weight.  
		\end{enumerate} In this case, $X$ is Yoneda complete. \end{thm}
	
	\begin{proof} $(1)\Rightarrow(2)$ By Proposition \ref{FC in Smyth is C}, every forward Cauchy net of a Smyth complete real-enriched category is a Cauchy net, then the conclusion follows from Lemma \ref{bilimit of Cauchy net}.

		$(2)\Rightarrow(3)$ It suffices to show that every ideal $\phi$ of $X$ is a Cauchy weight and has a colimit. By Proposition \ref{characterization of ideal}  there is a forward Cauchy net $\{x_i\}_{i\in D}$ of $X$ such that $$\phi=\sup_{i\in D}\inf_{j\geq i} X(-,x_j).$$ By assumption, there is some $a\in X$ such that for all $x\in X$, \[\sup_{i\in D}\inf_{j\geq i}X(x,x_j)=X(x,a), \quad   \sup_{i\in D}\inf_{j\geq i}X(x_j,x)=X(a,x).  \]   
		Putting $x=a$ one sees that $\{x_i\}_{i\in D}$ is a Cauchy net  of $X$, then $\phi$ is a Cauchy weight by Lemma \ref{Cauchy net implies Cauchy weight}, and consequently,  $a$ is a colimit of  $\phi$ by the argument of Proposition \ref{colimit =bilimit for Cauchy net}.

		$(3)\Rightarrow(4)$  Obvious.
		
		$(4)\Rightarrow(1)$ Suppose $\{x_i\}_{i\in D}$ is a forward Cauchy net  of $X$. Since every ideal of $X$ is a Cauchy weight, by Lemma \ref{Cauchy net implies Cauchy weight}   the net $\{x_i\}_{i\in D}$ is Cauchy, so it has a unique bilimit by Cauchy completeness of $X$ and the argument of Proposition \ref{colimit =bilimit for Cauchy net}, then by  Lemma \ref{bilimit of Cauchy net}, it converges uniquely in the open ball topology of the symmetrization of $X$. 
	\end{proof} 
	
	\begin{cor}A Yoneda complete real-enriched category $X$ is Smyth complete if and only if $\sy\colon X\lra\CI X$ is Yoneda continuous. \end{cor}
	
	\begin{cor} For each real-enriched category $X$, the following are equivalent: \begin{enumerate}[label=\rm(\arabic*)] 
			\item $X$ is Cauchy complete. \item The symmetrization of $X$ is Smyth complete.  \item The symmetrization of $X$ is Cauchy complete.  
	\end{enumerate} \end{cor}

	A separated real-enriched category $X$ is \emph{Smyth completable} if there is a fully faithful functor $f\colon X\lra Y$ with $Y$ Smyth complete.

	\begin{prop}\label{characterizing Smyth completable} {\rm (Yu and Zhang \cite{YuZ2024})} For a separated real-enriched category $X$, the following are equivalent:  \begin{enumerate}[label={\rm(\arabic*)}] 
			\item $X$ is Smyth completable.  \item Every ideal of $X$ is a Cauchy weight. \item Every forward Cauchy net of $X$ is a Cauchy net.  \item Every ideal of $\CI X$ is representable, hence $\CI(\CI X)=\CI X$.  \end{enumerate} \end{prop} 
	
	\begin{proof}The equivalence $(2)\Leftrightarrow(3)$ follows from Proposition \ref{characterization of ideal} and Lemma \ref{Cauchy net implies Cauchy weight}.
		
		$(1)\Rightarrow(2)$ Suppose  $f\colon X\lra Y$ is fully faithful with $Y$ being Smyth complete. Let $\phi$ be an ideal of $X$. Then $\phi\circ f^*$ is an ideal, hence a Cauchy weight of $Y$. Let $\psi\colon\star\oto Y$ be a left adjoint of $\phi\circ f^*\colon Y\oto\star$. We  show that $f^*\circ\psi\colon \star\oto X$ is a left adjoint of $\phi\colon X\oto\star$. First, since $\psi\dashv \phi\circ f^*$, then $\phi\circ(f^*\circ\psi) =(\phi\circ f^*)\circ\psi\geq1$. Next, \begin{align*}\psi\dashv \phi\circ f^* &\implies \psi\circ\phi\circ f^*\leq Y \\ &\implies \psi\circ\phi\leq f_* &(f^*\circ f_*=X) \\ &\implies (f^*\circ\psi)\circ\phi\leq X. &(f^*\circ f_*=X)\end{align*} Therefore, $f^*\circ\psi$ is a left adjoint of $\phi$.
		
		$(2)\Rightarrow(4)$ Since $\CI X=\CC X$ by assumption, it suffices to check that every ideal of $\CC X$ is representable. Suppose $\phi\colon\CC X\oto\star$ is an ideal of $\CC X$. By Lemma \ref{X is isomorphic to CX}, ${\sf c}_*\colon X\oto\CC X$ is an isomorphism in the category of distributors, it follows that $\phi\circ {\sf c}_*\colon X\oto\CC X\oto\star$ is an ideal, hence a Cauchy weight of $X$. Again by that ${\sf c}_*$ is an isomorphism, $\phi$ is a Cauchy weight of $\CC X$. Therefore, $\phi$ is representable since $\CC X$ is Cauchy complete.

		$(4)\Rightarrow(1)$ Since every ideal of $\CI X$ is representable, then $\CI X$ is Smyth complete and consequently, $X$ is Smyth completable. \end{proof}
	
	\begin{cor}A separated real-enriched category  is Smyth completable if and only if its Cauchy completion is Smyth complete. \end{cor}
	\begin{proof} Sufficiency is clear, for necessity assume that $f\colon X\lra Y$ is a fully faithful functor with $Y$ Smyth complete. Since $Y$ is Cauchy complete, there is a  unique functor  $\overline{f}\colon\CC X\lra Y$ that extends $f$. It is not hard to check that $\overline{f}$ is also fully faithful, hence $\CC X$ is Smyth completable. By Theorem \ref{characterizing Smyth completable} every ideal of $\CC X$ is Cauchy, hence representable by Cauchy completeness of $\CC X$, and consequently, $\CC X$ is Smyth complete. \end{proof}

	\section{The presheaf monad}
	
	In this section we show that the presheaf functor $\CP$   is a monad in $\QOrd$,  the subfunctors $\CC$ and $\CI$ are its submonads.
	
	For each real-enriched category $X$, let ${\sf m}_X$ be the functor $${\colim}\colon \CP\CP X\lra \CP X$$ that maps each weight $\Phi$ of $\CP X$ to its colimit. Then, ${\sf m}$ is a natural transformation $\CP^2\lra\CP$ and the triple  \[\mathbb{P}=(\CP,{\sf m},\sy)\] is a monad, where $\sy$ denotes the natural transformation from the identity functor to $\CP$ given by Yoneda embeddings. We call $\mathbb{P}$ the presheaf monad in $\QOrd$. Instead of checking directly that   $(\CP,{\sf m},\sy)$ satisfies the requirements for a monad, we show that it arises from an adjunction.

	Let \[\QSup \] denote the category with separated  and cocomplete real-enriched categories as objects and left adjoints as morphisms.

	Theorem \ref{free cocompletion} implies that the assignment \[f\colon X\lra Y ~ \mapsto ~ f_\exists\colon\CP  X\lra\CP Y\] defines a functor \[\CP\colon  \QOrd \lra \QSup \] that is left adjoint to the forgetful functor $$U\colon\QSup \lra \QOrd .$$    The monad in $\QOrd$ defined by the adjunction  $\CP\dashv U$ is precisely the presheaf monad $\mathbb{P}=(\CP,{\sf m},\sy)$. 
	
	The monad  $\mathbb{P}$ is a typical example of KZ-monads in a locally ordered category. A category $\mathscr{C}$ is \emph{locally ordered} if (i) the home-set $\mathscr{C}(A,B)$  is an ordered set for any objects $A$ and $B$; and (ii) the composition preserves order.  A 2-functor between locally ordered categories is a functor that preserves order on hom-sets.  The functor $\CP\colon\QOrd\lra\QOrd$ is clearly a 2-functor. Locally ordered categories are special cases of 2-categories,  we refer to Borceux \cite{Borceux1994a} or Lack \cite{Lack} for 2-categories and 2-functors.
	
	In a locally ordered category, we say that a morphism $f\colon A\lra B$ is left adjoint to a morphism $g\colon B\lra A$ if $g\circ f\geq 1_A$ and $f\circ g\leq 1_B$. In this case, we write $f\dashv g$. The following
	proposition, extracted from Kock \cite{Kock} and Z\"{o}berlein \cite{Zo76}, is taken from Hofmann \cite{Hof2013}.

	\begin{prop}\label{kz}
		Let $ (T,m,e)$ be a monad in a locally ordered category $\mathscr{C}$ with $T$  a $2$-functor.
		Then the following are equivalent:
		\begin{enumerate}[label={\rm(\arabic*)}] 
			\item $Te_X\leq e_{TX}$ for all objects $X$.
			\item  $Te_X\dashv m_X$ for all objects $X$.
			\item  $m_X\dashv e_{TX}$ for all objects $X$.
		\end{enumerate} 
	\end{prop}
	
	\begin{proof}$(1)\Rightarrow(2)$ Since $m_X\circ Te_X=1_{TX}$, it remains to check that $Te_X\circ m_X\leq 1_{T^2X}$.
		This is easy since
		$$Te_X\circ m_X=m_{TX}\circ T^2e_X\leq m_{TX}\circ Te_{TX}=1_{T^2X}.$$
		
		$(2)  \Rightarrow  (1)$ Since $Te_X\dashv m_X$, then
		$$Te_X=Te_X\circ m_X\circ e_{TX}\leq 1_{T^2X}\circ e_{TX}=e_{TX}.$$
		
		$(1)  \Leftrightarrow  (3)$  Similar to   $(1)  \Leftrightarrow  (2)$.
	\end{proof}
	
	A monad in a locally-ordered category is   of \emph{Kock-Z\"{o}berlein type}, or a KZ-monad, if it satisfies one (hence all) of the equivalent conditions in Proposition \ref{kz}. A KZ-monad is also called a \emph{KZ-doctrine} in the literature.  
	
	\begin{prop} Suppose $ (T,m,e)$ is a KZ-monad in a locally ordered category. Then,  for each object $X$ and each morphism $h\colon  TX\lra X$, the pair $(X,h)$ is a $T$-algebra if and only if $h\circ e_X=1_X$; in which case  $h\dashv e_X$.\end{prop}
	
	\begin{proof}   If $(X,h)$ is a $T$-algebra, then $h\circ e_X=1_X$ by definition of $T$-algebra. Conversely, suppose that $h\circ e_X=1_X$. Then
		$$h\circ Th\leq h\circ Th\circ e_{TX}\circ m_X=h\circ e_X\circ h\circ m_X=h\circ m_X.$$
		Since $T$ is a KZ-monad, it follows that $Te_X\circ m_X\leq 1_{TX}$, hence
		$$h\circ Th\geq h\circ Th\circ Te_X\circ m_X=h\circ T(h\circ e_X)\circ m_X=h\circ m_X.$$ Therefore, $h\circ Th=h\circ m_X$ and consequently, $(X,h)$ is a
		$T$-algebra.
		
		It remains to check that  $h\dashv e_X$ for each $T$-algebra $(X,h)$. Since $T$ is of Kock-Z\"{o}berlein type,
		$$e_X\circ h=Th\circ e_{TX}\geq Th\circ Te_X=T(h\circ e_X)=1_{TX}.$$
		Therefore  $h\dashv e_X$, as desired.
	\end{proof}
	
	\begin{prop}
		\label{pkz}
		The presheaf monad  $\mathbb{P}=(\CP,{\sf m},\sy)$ is  a KZ-monad in the locally ordered category $\QOrd$.
	\end{prop}
	\begin{proof} By Example \ref{sup in PX}, for each real-enriched category $X$   we have
		$$\CP \sy_X\dashv {\sf m}_X  \dashv \sy_{\CP X},$$  the conclusion thus follows.
	\end{proof}
	
	Since $(\CP,{\sf m},\sy)$ is  a KZ-monad, for each real-enriched category $A$, the  following   are equivalent: \begin{enumerate}[label={\rm(\arabic*)}] 
		\item  There is a functor $h\colon \CP A\lra A$ such that   $(A,h)$ is a $\mathbb{P}$-algebra. \item The Yoneda embedding $\sy_A\colon A\lra\CP A$ has a left inverse. \item $A$ is separated and the Yoneda embedding $\sy_A$ has a left adjoint.\end{enumerate} In this case, the $\mathbb{P}$-algebra structure map $h\colon \CP A\lra A$ is a left adjoint of $\sy_A$, hence $h$ maps each weight of $A$ to its colimit. Therefore, a $\mathbb{P}$-algebra is   a separated and cocomplete real-enriched category;   a $\mathbb{P}$-homomorphism  is   a left adjoint between separated cocomplete real-enriched categories. All told,
	\[\mathbb{P}\text{-}{\sf Alg}=\QSup.\]
	From this fact one immediately infers:
	
	\begin{prop}
		\label{PAlg is monadic over QOrd}
		The forgetful functor $\PAlg\lra\QOrd$ is monadic.\end{prop}
	
	Dually, we have a copresheaf monad $\bbP^\dag=(\CPd,{\sf m}^\dag,\syd)$ in $\QOrd$. For each real-enriched category $X$, let \[{\sf m}^\dag_X\colon \CPd\CPd X\lra \CPd X \] be the functor that sends each coweight   of $\CPd X$ to its limit. Then $${\sf m}^\dag\colon (\CPd)^2\lra \CPd$$ is a natural transformation and $$\bbP^\dag=(\CPd,{\sf m}^\dag,\syd)$$ is a monad  in $\QOrd$, called the copresheaf monad. The copresheaf monad $\bbP^\dag$ is   co-KZ  in the sense that for each real-enriched category $X$, \[\CPd\syd_X\vdash{\sf m}^\dag_X\vdash\syd_{\CPd X}\colon \CPd X\lra\CPd\CPd X.\] 
	
	Since $\bbP^\dag$ is  co-KZ, a $\bbP^\dag$-algebra is a separated real-enriched category $A$ such that the coYoneda embedding $\syd_A\colon A\lra\CPd A$ has a right adjoint.  Therefore,  the category \[\PdAlg\] of $\bbP^\dag$-algebras and $\bbP^\dag$-homomorphisms is precisely  the category of separated and complete real-enriched categories   and right adjoint functors.
	Mapping each  $f\colon A\lra B$ in $\PAlg$ to its right adjoint $\overline{f}\colon B\lra A$ defines an isomorphism of categories  $$(\PAlg)^{\rm op}\cong\PdAlg.$$ 
	
	Theorem \ref{P-algebras as modules} below characterizes $\bbP$-algebras as (left) $[0,1]$-modules.
	
	\begin{defn} (Joyal and Tierney \cite{Joyal-Tierney}) A (left) $[0,1]$-module is a  pair $(X,\otimes )$, where $X$ is a complete lattice and $\otimes$ is a $[0,1]$-action  on $X$. A  $[0,1]$-module homomorphism  $f\colon (X,\otimes )\lra(Y,\otimes )$  is a join-preserving map $f\colon X\lra Y$ such that $r\otimes  f(x)= f(r\otimes  x)$ for all $r\in [0,1]$ and $x\in X$.\end{defn}
	
	With $[0,1]$-modules and $[0,1]$-module homomorphisms we have a category $$\QMod.$$
	
	\begin{thm} \label{P-algebras as modules} {\rm (Stubbe \cite{Stubbe2006})}
		The category $\QMod$   is isomorphic to   $\PAlg$. \end{thm}
	
	\begin{proof}If $X$ is a $\bbP$-algebra, then the $[0,1]$-action $(X_0,\otimes )$ is a $[0,1]$-module. If $f\colon X\lra Y$ is a homomorphism between $\bbP$-algebras,   then $f\colon (X_0,\otimes )\lra (Y_0,\otimes )$ is  a $[0,1]$-module homomorphism since  any left adjoint   preserves  tensors.
		Conversely, if $(X,\otimes)$ is a $[0,1]$-module, then the   real-enriched category $(X,\alpha_{\otimes})$ is separated,  order-complete, tensored and cotensored, hence a $\bbP$-algebra. If $f\colon (X,\otimes )\lra(Y,\otimes )$ is a $[0,1]$-module homomorphism, then the map $f\colon (X,\alpha_{\otimes}) \lra(Y,\alpha_{\otimes})$ preserves tensors and   joins (with respect to the underlying order), hence  a left adjoint by Corollary \ref{left adjoint via tensor and join}.
		The conclusion thus follows. \end{proof}
	
	Every $\bbP$-algebra $A$ is clearly a quotient of the free algebra $\CP A$ in $\PAlg$. The following result shows that whether every $\bbP$-algebra can be written as a subalgebra of a free one depends on the structure of the truth-value set $([0,1],\&,1)$.
	
	\begin{prop} {\rm (Lai and Zhang \cite{LaiZ09})} \label{subalgebra of a free one} The following   are equivalent: \begin{enumerate}[label={\rm(\arabic*)}] \setlength{\itemsep}{0pt}
			\item The continuous t-norm $\&$ is isomorphic to the \L ukasiewicz t-norm. \item Each $\mathbb{P}$-algebra $A$ is a subalgebra  of   $\CP X$ in $\PAlg$ for some real-enriched category $X$.    \end{enumerate} \end{prop}
	\begin{proof} $(1)\Rightarrow(2)$ Let $A$ be a $\mathbb{P}$-algebra. Since the coYoneda embedding $\syd\colon A\lra\CPd  A$ is a fully faithful   left adjoint, it suffices to show that $\CPd  A$ is isomorphic to $\CP A$. Since  $\&$ is isomorphic to the \L ukasiewicz t-norm, it is readily verified that \[\CPd  A\lra\CP A, \quad  \psi\mapsto  \psi\ra0\] is an isomorphism.
		
		$(2)\Rightarrow(1)$ It suffices to show that $ (r\ra0)\ra0 \leq r$ for all $r\in [0,1]$. Consider the $\mathbb{P}$-algebra $\sV^{\rm op}=([0,1],\alpha_R)$. By assumption, there is a real-enriched category $X$ and a fully faithful left adjoint $f\colon \sV^{\rm op}\lra\CP X$.     Since $f$ is a left adjoint and $1$ is the least element  of $\sV^{\rm op}_0$,   $f(1)$ must be the least element  of $(\CP X)_0$, namely the weight $0_X$ of $X$ with constant value $0$. Since $f$ preserves tensor and the tensor of $r$ with $x$ in $\sV^{\rm op}$ is $r\ra x$, it follows that for each $r\in [0,1]$,   \[r\otimes f(r)= f(r\ra r)= f(1)=0_X,\] where $\otimes $ denotes  tensor in $\CP X$. Thus, $f(r)\leq r\ra0_X$  and consequently, \begin{align*}f(((r\ra0)\ra0)\ra r)    &=((r\ra0)\ra0)\otimes f(r)\\ & \leq((r\ra0)\ra0)\otimes (r\ra0_X)\\ &\leq 0_X, \end{align*} where the last inequality holds because $r\ra0_X$ is the constant function with value $r\ra0$. Since $f$ is injective, one has $((r\ra0)\ra0)\ra r=1$, hence $(r\ra0)\ra0\leq r$, as desired.   \end{proof}

	A \emph{class of weights} is defined to be a subfunctor $\CT$ of the presheaf functor $\CP$ through which the natural transformation $\sy\colon\id\lra\CP$ factors. In other words, a class of weights is  a functor $$\CT\colon \QOrd \lra\QOrd $$ subject to the following conditions: \begin{enumerate}[label={\rm(\roman*)}]  \item for each real-enriched category $X$, $\CT X $ is a subset of $\CP X$ together with  category structure inherited from $\CP X$;
		\item for each real-enriched category $X$, $\CT X $ contains all representable weights of $X$;
		\item for each functor $f\colon  X\lra Y$ and each  $\phi\in\CT X $, $\CT(f)(\phi)= f_\exists(\phi)$. \end{enumerate}
	
	Condition (ii) implies that for each real-enriched category $X$, the Yoneda embedding $\sy_X\colon X\lra\CP X $ factors through $\CT X$.  
	
	Suppose $\CT$ is a class of weights on $\QOrd$ and     $\mathfrak{i}\colon \CT \lra\CP $ is the inclusion natural transformation.  If the natural transformation ${\sf m}\circ(\mathfrak{i} *\mathfrak{i} )$ factors (uniquely) through $\mathfrak{i}$, where $\mathfrak{i} *\mathfrak{i}$ is the horizontal composite of $\mathfrak{i}$ with itself, then we say that $\CT$ is \emph{closed under multiplication}. \[\bfig \Square [\CT^2`\CT `\CP^2 `\CP ; `\mathfrak{i} *\mathfrak{i} `\mathfrak{i}` {\sf m}] \efig\]  In this case, we also use the symbol ${\sf m}$ to denote the unique natural transformation $\CT^2\lra\CT$ that makes the above square commutative.
	
	If $\CT$ is closed under multiplication, then $(\CT,{\sf m},\sy )$ is a submonad of $(\CP,{\sf m},\sy)$, with the inclusion transformation $\mathfrak{i}$ being a monad morphism, see e.g. Manes \cite{Manes2003}. In this case,   we also say  simply that the functor $\CT$ is a submonad of   $(\CP,{\sf m},\sy)$.

	\begin{prop} \label{equivalents for submonad} Suppose $\CT$ is a class of weights  on $\QOrd$. The  following  are equivalent: \begin{enumerate}[label=\rm(\arabic*)]   \item  $ \CT$ is closed under multiplication, hence a submonad $(\CP,{\sf m},\sy)$.  
			\item For each real-enriched category $X$, $\CT X$ is closed in $\CP X$ under formation of $\CT$-colimits; that is, for each $\Phi\in\CT\CT X$, the colimit of the inclusion functor $\mathfrak{i}_X\colon \CT X\lra \CP X$ weighted by $\Phi$ belongs to $\CT X$.  
			\item For each real-enriched category $X$ and each $\Phi\in\CT\CT X$, it holds that $\Phi\circ\sy_*\in\CT X$, where $\sy\colon X\lra\CT X$ is the Yoneda embedding with codomain restricted to $\CT X$.  
		\end{enumerate}
	\end{prop}
	\begin{proof}This follows directly from   that for each $\Phi\in\CT\CT X$,
		\[(\mathfrak{i} *\mathfrak{i} )_X(\Phi)=\Phi\circ \mathfrak{i}_X^* \] and that \begin{align*}{\colim}_\Phi\mathfrak{i}_X &= {\colim} (\Phi\circ \mathfrak{i}_X^*)=\Phi\circ\mathfrak{i}_X^*\circ(\mathfrak{i}_X)_*\circ  \sy_* = \Phi\circ\sy_*.  \qedhere \end{align*} \end{proof}
	
	\begin{defn}A class of weights $\CT\colon \QOrd \lra\QOrd $ is  saturated  
		if it satisfies one, hence all, of the equivalent conditions in Proposition \ref{equivalents for submonad}. \end{defn}  
	
	In other words, a saturated class of weights on $\QOrd$ is a submonad of the presheaf monad $\mathbb{P}=(\CP,{\sf m},\sy)$. 
	For a saturated class of weights $\CT$, we write $\mathbb{T}=(\CT,{\sf m},\sy)$ for the corresponding monad, which is a submonad of the preheaf monad. The monad $\bbT$ is of Kock-Z\"{o}berlein type since  so is $\mathbb{P}$. Therefore, for the Eilenberg-Moore category $\TAlg$: \begin{enumerate}[label={\rm(\roman*)}] 
		\item  an object  is  a separated real-enriched category $A$ such that every $\phi\in\CT A$ has a colimit; \item a morphism is a functor  that preserves $\CT$-colimits.  \end{enumerate}   
	
	\begin{prop} \label{retract of T-alg} Suppose $\CT$ is a saturated class of weights on $\QOrd$.  Then, every retract of a $\bbT$-algebra in  $\QOrd$ is a $\bbT$-algebra, where $\bbT$ denotes the monad $(\CT,{\sf m},\sy)$. \end{prop}
	
	\begin{proof}Suppose   $B$ is a $\bbT$-algebra;  $s\colon A\lra B$ and $r\colon B\lra A$ are  functors with $r\circ s=1_A$. We wish to show that $A$ is a $\bbT$-algebra. Let $\sy_A\colon A\lra\CT A$ be   the Yoneda embedding with codomain restricted to $\CT A$; and let $\colim_A$ be the composite \[\CT A\to^{\CT s}\CT B\to^{{\colim}_B}B\to^r A.\] Then   \[{\colim}_A\circ\sy_A= r\circ{\colim}_B\circ\CT s\circ \sy_A  = r\circ{\colim}_B\circ\sy_B\circ s =r\circ s =1_A,\]   hence $\colim_A$ is a left inverse of  $\sy_A$ and consequently, $A$ is a $\bbT$-algebra. \end{proof}
	
	For each real-enriched category $X$, since $\CT X$ is closed in $\CP X$ with respect to $\CT$-colimits, it follows that $\CT X$ is a $\bbT$-algebra and the the assignment \[f\colon  X\lra Y\quad\mapsto\quad f_\exists\colon\CT X\lra\CT Y\] defines a functor   \[\CT\colon \QOrd\lra\TAlg .\]

	\begin{prop} \label{free T-algebra} Suppose $\CT$ is a saturated class of weights on $\QOrd$. Then, the functor \[\CT\colon \QOrd \lra\TAlg \] is left adjoint to the forgetful functor \[U\colon \TAlg \lra \QOrd .\] Hence  for each real-enriched category $X$, $\CT X$ is  the free $\CT$-cocompletion of $X$.
	\end{prop}
	
	\begin{proof}
		Let $f\colon X\lra Y$ be a functor  with $Y$   a  $\bbT$-algebra. It is readily verified that  \[\overline{f}\colon \CT X\lra Y, \quad \phi\mapsto {\colim}_\phi  f  \] is the unique $\bbT$-homomorphism  that makes the   diagram \[\bfig\qtriangle[X`\CT X`Y;\sy `f`\overline{f}]\efig\] commutative, the conclusion thus follows.
	\end{proof}  
	
	The monad defined by the adjunction $\CT\dashv U$ is precisely  $\bbT=(\CT,{\sf m},\sy)$. Furthermore,   the forgetful functor $U\colon\TAlg\lra\QOrd$ is monadic.  
	
	\begin{exmp}\label{exmp of submonads} \begin{enumerate}[label={\rm(\roman*)}] 
			\item By Theorem \ref{univeral property of CX} the class of Cauchy weights is saturated. 
			The resulting submonad $\bbC=(\CC,{\sf m},\sy)$   of $\bbP=(\CP,{\sf m},\sy)$ is called the \emph{Cauchy weight monad} in $\QOrd$. A $\bbC$-algebra is   a Cauchy complete real-enriched category; a $\bbC$-homomorphism is just a functor between Cauchy complete real-enriched categories (colimits of Cauchy weights are preserved by all functors). 
			\item   By Theorem \ref{univeral property of IX}, ideals of real-enriched categories constitute a saturated   class of weights. The   monad $\bbI=(\CI,{\sf m},\sy)$ is called  the \emph{ideal monad} in $\QOrd$. An  algebra of   $\bbI$ is   a Yoneda complete real-enriched category; an $\bbI$-homomorphism is a Yoneda continuous functor between Yoneda complete real-enriched categories. Corollary  \ref{underlying order of Yoneda} shows that restricting  domains and codomains of the functors   $\omega$ and $(\text{-})_0$ gives us an adjunction between the category of dcpos and that of $\bbI$-algebras.
			\item By Theorem \ref{univeral property of FX} the class of flat weights is saturated. The resulting submonad $\bbF=(\CF,{\sf m},\sy)$  is called the \emph{flat weight monad} in $\QOrd$. An $\bbF$-algebra is a flat complete real-enriched category; an $\bbF$-homomorphism is   a functor between flat complete real-enriched categories that preserves colimits of flat weights. Likewise, by Theorem \ref{univeral property of FCX} the class of conically flat weights is saturated. The resulting submonad $\bbF_{\rm c}=(\CF_{\rm c},{\sf m},\sy)$  is called the \emph{conically flat weight monad} in $\QOrd$.
			\item For each real-enriched category $X$, let $\CB X=\{r\with\sy(x)\mid x\in X,r\in[0,1]\}$. Then $\CB$ is a saturated class of weights on $\QOrd$. To see this, it suffices to observe that for all $x\in X$ and all $r,s\in[0,1]$, the colimit of the inclusion functor $\CB X\lra\CP X$  weighted by $s\with \CB X(-, r\with\sy(x))$  is  the weight $(s\with r)\with\sy(x)$ of $X$. Write   $\bbB=(\CB,{\sf m},\sy)$ for the resulting monad. 
			A $\bbB$-algebra is precisely a separated and tensored real-enriched category; a $\bbB$-homomorphism is a functor that preserves tensors.
			\item (Stubbe \cite{Stubbe2010}) The class of conical weights is saturated, algebras for the resulting monad are the separated and conically cocomplete real-enriched categories.
	\end{enumerate}  \end{exmp} 
	
	\begin{thm}[Morita equivalence] {\rm(Kelly \cite{Kelly})} \label{Morita equiv} For all real-enriched categories $X$ and $Y$, the following  are equivalent:
		\begin{enumerate}[label={\rm(\arabic*)}] 
			\item $X$ is isomorphic to $Y$   in  $[0,1]$-{\bf Dist}.  \item $\CC X$   is isomorphic to $\CC Y$ in  $\QOrd$, hence in $\bbC$-{\bf Alg}.  \item
			$\CP X$ is isomorphic to $\CP Y$ in  $\bbP$-{\bf Alg}.\end{enumerate} \end{thm}
	
	Suppose $\phi\colon X\oto Y$ is a distributor. Since the functor \[\phi_\exists\colon\CP  Y\lra \CP X,\quad \xi\mapsto \xi\circ\phi\] is left adjoint to \[\phi_\forall\colon\CP  X\lra\CP Y, \quad \lam\mapsto \lam\swarrow\phi,\] it follows that the assignment \[\phi\colon X\oto Y~ \mapsto ~ \phi_\exists  \colon\CP  Y\lra \CP X \] defines a functor \[\mathfrak{P}\colon  ([0,1]\text{-}{\bf Dist})^{\rm op}\lra \PAlg .\]
	
	\begin{lem}\label{P is full and faithful} The functor $\mathfrak{P}\colon  ([0,1]\text{-}{\bf Dist})^{\rm op}\lra \PAlg $ is full and faithful. \end{lem}
	
	\begin{proof} Since for each distributor $\phi\colon X\oto Y$, each $x\in X$ and each $y\in Y$, \[\phi_\exists(\sy_Y(y))(x)=\sy_Y(y)\circ\phi(x)=\phi(x,y),\] it follows that $\mathfrak{P}$ is faithful.
		
		To see that $\mathfrak{P}$ is full, assume that $f\colon \CP Y\lra \CP X$ is a left adjoint. Define a distributor  $\phi\colon X\oto Y$ by $\phi=\sy_Y^*\circ f^*\circ(\sy_X)_* $; that is, \[\phi(x,y)= f(\sy_Y(y))(x).\] Then, for all $y\in Y$, \[\phi_\exists(\sy_Y(y))=\sy_Y(y)\circ\phi=f(\sy_Y(y)).\] Since the Yoneda embedding $\sy_Y\colon Y\lra \CP Y$ is colimit-dense and both $\phi_\exists$ and $f$ preserve colimits, it follows that $f=\phi_\exists$. \end{proof}

	\begin{proof}[Proof of Theorem \ref{Morita equiv}] Firsr  we show that (1) is equivalent to (2). If $\phi\colon X\oto Y$ is an isomorphism in $[0,1]$-{\bf Dist}, then \[-\circ\phi\colon \CC Y\lra \CC X\] is an isomorphism in  $\QOrd$. Conversely, if $f\colon \CC X\lra\CC Y$ is an isomorphism in  $\QOrd$, then the composite \[X\to^{(\sy  _X)_*}\CC X\to^{f_*} \CC Y\to^{\sy  _Y^*}Y\] is an isomorphism in $[0,1]$-{\bf Dist}, where $\sy_X$ is the Yoneda embedding for $X$ with codomain resricted to $\CC X$; likewise for $\sy_Y$.
		
		Next  we show that   (1) is equivalent to (3). If $\phi\colon X\oto Y$ is an isomorphism in $[0,1]$-{\bf Dist}, then \[\phi_\exists\colon\CP  Y\lra\CP X\] is an isomorphism in  $\PAlg$.

		Conversely,  assume that $f\colon \CP X\lra\CP Y$ is an isomorphism  in $\PAlg$. Since the functor $\mathfrak{P}$ is full,  there exist  distributors $\phi\colon Y\oto X$ and  $\psi\colon X\oto Y$ such that $f= \phi_\exists$ and $f^{-1}= \psi_\exists$. Since the functor $\mathfrak{P}$ is  faithful, it follows that $\psi\circ\phi=Y$ and $\phi\circ\psi=X$. Therefore $X$ and $Y$ are isomorphic in  $[0,1]$-{\bf Dist}. \end{proof}
	
	The Morita equivalence can be strengthened to the following:
	
	\begin{thm}{\rm (Kelly and Schmitt \cite{KS2005})}
		\label{generalized morita} Suppose $\CT$  is a saturated class of weights  that contains all Cauchy weights. Then,  for all real-enriched categories $X$ and $Y$, the following  are equivalent:
		\begin{enumerate}[label={\rm(\arabic*)}] 
			\item $X$ is isomorphic to $Y$   in  $[0,1]$-{\bf Dist}.     \item
			$\CT X$ is isomorphic to $\CT Y$ in  $\TAlg$. \end{enumerate} \end{thm}
	
	\begin{proof} $(1)\Rightarrow(2)$ 
		Consider the isomorphism  \[\psi_\exists\colon\CP X\lra \CP Y, \quad \xi\mapsto\xi\circ\psi\]   in the proof of Theorem \ref{Morita equiv}. If we can prove that $\xi\circ\psi\in \CT Y$ for all $\xi\in \CT X$, then restricting the domain and the codomain of $\psi_\exists$ would yield  an isomorphism $\CT X\lra \CT Y$.
		
		For each $x\in X$, let $f(x)$ be the weight $\psi(-,x)$ of $Y$. Since $\psi$ is an isomorphism in $[0,1]$-{\bf Dist}, $f(x)$ is a Cauchy weight. Hence we obtain a functor $f\colon X\lra\CT Y$. Since $\CT$ is saturated, the colimit of the composite functor $ \mathfrak{i}_Y\circ f\colon X\lra \CT Y\lra\CP Y$ weighted by $\xi$ belongs to $\CT Y$,  
		then  \begin{align*}\xi\circ\psi= \sup_{x\in X}\xi(x)\with \psi(-,x) =\sup_{x\in X}\xi(x)\with f(x) =\colim_\xi (\mathfrak{i}_Y\circ f)\in\CT Y, \end{align*}  as desired.
		
		$(2)\Rightarrow(1)$   We say that a distributor $\phi\colon X\oto Y$ is  of type $\CT$  if $\phi(-,y)\in \CT X$ for all $y\in Y$. In particular, a weight $\xi\colon X\oto*$ is  of type $\CT$ if and only if $\xi\in \CT X$. Since $\CT$  contains Cauchy weights, every isomorphism in $[0,1]$-{\bf Dist} is  of type $\CT$. Since $\CT$ is saturated,   $\CT$-distributors  of type $\CT$ are closed under composition. Therefore, they form a subcategory of $[0,1]$-{\bf Dist}, denoted by $\CT\text{-}{\bf Dist}$. 
		A slight improvement of the argument in Lemma \ref{P is full and faithful} shows that the correspondence \[\phi\colon X\oto Y~ \mapsto ~ -\mathrel{\circ} \phi \colon\CT Y \lra \CT X\] defines a full and faithful functor \[ (\CT\text{-}{\bf Dist})^{\rm op}\lra \TAlg ,\] which implies that if $\CT X$ is isomorphic to $\CT Y$ in $\TAlg$, then $X$ is isomorphic to $Y$   in  $[0,1]$-{\bf Dist}.
	\end{proof}
	
	The characterization of Cauchy completion in Proposition \ref{CC as equaliser} can be strengthened as follows. 
	Suppose $\CT$  is a saturated class of weights  that contains all Cauchy weights. For each real-enriched category $X$, let ${\sf t}_X\colon X\lra\CT X$ be the Yoneda embedding with codomain restricted to $\CT X$. Then, by commutativity of the diagram \begin{equation*}\bfig
		\morphism(0,0)|a|/@{->}@<3pt>/<600,0>[\CT X`\CT\CT X;\CT{\sf t}_X]
		\morphism(0,0)|b|/@{->}@<-3pt>/<600,0>[\CT X`\CT \CT X; {\sf t}_{\CT X} ] 
		\morphism(0,-500)|a|/@{->}@<3pt>/<600,0>[\CP X`\CP\CP X;\CP\sy_X]
		\morphism(0,-500)|b|/@{->}@<-3pt>/<600,0>[\CP X`\CP \CP X; \sy_{\CP X} ] 
		\morphism(0,0)|l|/@{->}/<0,-500>[\CT X`\CP X;\mathfrak{i}_X]
		\morphism(600,0)|r|/@{->}/<0,-500>[\CT\CT X`\CP\CP X;(\mathfrak{i}*\mathfrak{i})_X]
		\efig\end{equation*}  
	one sees that the Cauchy completion of $X$ is the equalizer  of   $\CT{\sf t}_X$ and ${\sf t}_{\CT X}$ in the category $\QOrd$. 
	
	\section{The fuzzy powerset monad} \label{fuzzy powerset monad}
	
	Proposition \ref{PAlg is monadic over QOrd} says that the forgetful functor $\PAlg\lra\QOrd$ is monadic. This section shows that the forgetful functor $\PAlg\lra{\bf Set}$ is also monadic. We show that $\bbP$-algebras are precisely the Eilenberg-Moore algebras of the  fuzzy powerset monad in the category of sets.
	
	The (covariant)  fuzzy powerset  monad (see e.g.  H\"{o}hle \cite{Hoehle2001} and Manes \cite{Manes1982}) \[ (\mathscr{P},{\sf m},{\sf e})\] in the category of sets is defined   as follows: 
	\begin{itemize} 
		\item the (covariant) fuzzy powerset functor $\mathscr{P}$ sends each set $X$ to $[0,1]^X$ and sends each map  $f\colon X\lra Y$ to   $\mathscr{P}(f)\colon \sQ^X\lra \sQ^Y$, where for all $\gamma\in[0,1]^X$ and $y\in Y$,  $$\mathscr{P}(f)(\gamma)(y)=f_\exists(\gamma)(y)=\sup\{\gamma(x)\mid f(x)=y\};$$ 
		
		\item   for each set $X$, the unit ${\sf e}_X\colon X\lra \sQ^X$  sends each $x\in X$ to $ 1_x$;
		
		\item  for each set $X$,  the multiplication ${\sf m}_X\colon  \sQ^{\sQ^X} \lra  \sQ^X $ sends each $\Lambda\colon \sQ^X\lra\sQ $ to $$\bv_{\gamma\in\sQ^X}\Lambda(\gamma)\with\gamma.$$   \end{itemize}

	Instead of verifying directly that the triple $(\mathscr{P},{\sf m},{\sf e})$ is a monad, we show that it is induced by an adjunction between the categories of sets and $[0,1]$-modules. 
	
	For each set $X$, consider the action $$\otimes\colon [0,1]\times [0,1]^X \lra [0,1]^X, \quad (r,\lam)\mapsto r\with\lam.$$ Then $([0,1]^X,\otimes)$ is a $[0,1]$-module and the assignment $X\mapsto ([0,1]^X,\otimes)$ defines a functor $$F\colon{\bf Set}\lra\QMod$$ that is left adjoint to the forgetful functor  $$U\colon \QMod\lra{\bf Set}.$$   It is readily seen that $(\mathscr{P},{\sf m},{\sf e})$ is  the monad induced by the adjunction $F\dashv U$.   
	
	Suppose $(X,h)$ is an algebra of the monad $(\mathscr{P},{\sf m},{\sf e})$. Since $(\mathscr{P},{\sf m},{\sf e})$  contains the usual (covariant) powerset monad in {\bf Set} as a submonad, then $X$ together with the restriction of $h$ to $2^X\subseteq [0,1]^X$ is an algebra of the usual powerset monad, hence $X$ is a complete lattice and  $h$  maps each subset of $X$ to its join (see e.g. Mac{\thinspace}Lane \cite[page 142]{MacLane1998}).  Define $\otimes\colon \sQ\times X\lra X$ by \(r\otimes x=h(r_x).\)   Then $(X,\otimes)$ is a $\sQ$-module. Therefore, an algebra of   the fuzzy powerset monad $(\mathscr{P},{\sf m},{\sf e})$ is essentially a $\sQ$-module, hence a $\bbP$-algebra. This proves:
	
	\begin{thm}\label{PAlg is monadic over Set} {\rm (Pedicchio and Tholen \cite{PT1989})}
		The forgetful functor $\PAlg\lra{\bf Set}$ is monadic. \end{thm}
	
	In the following we consider the double contravriant fuzzy powerset monad in the category of sets. The \emph{contravriant fuzzy powerset functor} $$\exp\colon {\bf Set}^{\rm op}\lra{\bf Set}$$   maps each $f\colon X\lra Y$ to   $$\exp f \colon [0,1]^Y\lra [0,1]^X, \quad \mu\mapsto \mu\circ f.$$ It is readily verified that $\exp\colon {\bf Set}^{\rm op}\lra{\bf Set}$ is right adjoint to its opposite $\exp^{\rm op}\colon {\bf Set}\lra{\bf Set}^{\rm op}$. The unit and the counit of the adjunction $\exp^{\rm op}\dashv\exp$ are given respectively by  $$\eta_X\colon X\lra \exp\exp^{\rm op}X, \quad \eta_X(x)(\lam)=\lam(x) $$ and $$\epsilon_X\colon X\lra \exp^{\rm op}\exp X, \quad \epsilon_X(x)(\mu)=\mu(x). $$
	The monad  induced by  the adjunction $\exp^{\rm op}\dashv\exp$ is called \emph{the double contravariant fuzzy powerset monad} in the category of sets. As we shall see, this monad contains the monad $(\mathscr{P},{\sf m},{\sf e})$ as a submonad, so we write   $ (\mathscr{E}, \sfm,{\sf e})$ for the monad defined by   $\exp^{\rm op}\dashv\exp$. Explicitly, for each set $X$ and each function $f\colon X\lra Y$, 
	\begin{itemize} 
		\item $\mathscr{E}X$ is the set of all functions $ [0,1]^X\to[0,1]$; 
		\item   $ \mathscr{E}(f)\colon \mathscr{E}X\lra \mathscr{E}Y $ is given by $\mathscr{E}(f)(\Lambda)(\mu)= \Lambda(\mu\circ f)$ for all $\Lambda\in  \mathscr{E}X$ and $\mu\in[0,1]^Y$; 
		\item the unit ${\sf e}_X\colon X\lra \mathscr{E}(X)$ is given by ${\sf e}_X(x)(\lam)=\lam(x)$ for all $x\in X$  and $\lam\in [0,1]^X$; 
		\item  the multiplication $\sfm_X\colon \mathscr{E}^2X\lra \mathscr{E}X$ is given by $\sfm_X=\exp(\epsilon_{\exp^{\rm op}  X})$; explicitly, $\sfm_X(\mathcal{H})(\lam)= \mathcal{H}(\widehat{\lam})$ for all $\lam\in [0,1]^X$ and   $\mathcal{H}\colon [0,1]^{\mathscr{E}X}\lra [0,1]$,   where  $\widehat{\lam}\colon \mathscr{E}X\lra [0,1]$ sends  $\Lambda\in  \mathscr{E}X$ to $\Lambda(\lam)$. 
	\end{itemize}
	
	The (covariant) fuzzy powerset functor $\mathscr{P}\colon{\sf Set}\to{\sf Set}$ is closely related to the contravariant fuzzy powerset functor  $\exp\colon {\sf Set}^{\rm op}\to{\sf Set}$. Of particular interest is the following 
	
	\begin{lem} \label{relation between P and exp} \begin{enumerate}[label={\rm(\roman*)}] 
			\item A map $f\colon X\lra Y$ between sets is injective if and only if $\exp(f)\circ \mathscr{P}(f) =1_{\exp X}$. \item For any pullback square in the category of sets, as displayed on the left,
			$$\bfig
			\square[A`C`B`D;f`h`j`g] \square(1100,0)<650,500>[\sQ^C`\sQ^A`\sQ^D`\sQ^B;\exp  f`\mathscr{P} (j)`\mathscr{P} (h)
			`\exp  g]
			\efig$$
			the right square is commutative. \end{enumerate} \end{lem}

	It is known that the covariant powerset monad in the category of sets is a submonad of the double contravariant powerset monad, see e.g. Manes \cite[page 79]{Manes2003}, the following proposition says this is also true for fuzzy powersets. For each set $X$ consider the map $$j_X\colon [0,1]^X\to[0,1]^{[0,1]^X}, \quad  j_X(\lambda)(\gamma)=\sub_X(\lambda,\gamma).$$  Then the assignment $X\mapsto j_X$ is a natural transformation from the functor $\mathscr{P}$ to the functor $\mathscr{E}$, since  for each map $f\colon X\to Y$ we have  
	\begin{align*}j_Y\circ\mathscr{P} (f)(\lambda)(\gamma)&=\sub_Y( f_\exists(\lambda),\gamma)) =\sub_X(\lambda,\gamma\circ f)   =  \mathscr{E}(f)(j_X(\lambda))(\gamma).  \end{align*} 
	
	\begin{prop}{\rm(Lu and Zhang \cite{LuZ2023})} The natural transformation $j$ exhibits   $(\mathscr{P},{\sf m},{\sf e})$ as a submonad of $(\mathscr{E}, \sfm,{\sf e})$. \end{prop}
	
	\begin{proof} 
		It is clear that $j_X$ is an injection for each set $X$,   so it suffices to show that for each set $X$  the following square  is commutative: \[\bfig\square<600,500>[\mathscr{P}^2X`\mathscr{E}^2 X` \mathscr{P} X`\mathscr{E} X;(j*j)_X`\sfm_X` \sfm_X`j_X]\efig\] For this we compute:   for all $\Lambda\in\sQ^{\sQ^X}$ and   $\lambda\in\sQ^X$,
		\begin{align*} &~ \hskip 10pt \sfm_X\circ(j*j)_X(\Lambda)(\lambda)\\ &= \sfm_X\circ j_{\mathscr{E}X}\circ \mathscr{P} (j_X) (\Lambda)(\lambda)  & (\text{Definition of $j*j$})
			\\ &=j_{\mathscr{E}X}((j_X)_\exists (\Lambda))(\widehat{\lambda}) 
			\\ &= \bw_{\Xi\in \mathscr{E}X} ((j_X)_\exists (\Lambda)(\Xi)\ra \Xi(\lambda)) & (\text{Definition of $j_{\mathscr{E}X}$})
			\\ &= \bw_{\gamma\in\sQ^X}(\Lambda(\gamma)\ra \sub_X(\gamma,\lambda)) &  ((j_X)_\exists (\Lambda)(\Xi)=\bv_{j_X(\gamma)=\Xi} \Lambda(\gamma))
			\\ &= \sub_X\Big(\bv_{\gamma\in\sQ^X}\Lambda(\gamma)\with\gamma, \lambda\Big) \\ &= j_X\circ\sfm_X(\Lambda)(\lambda).   \end{align*}  The proof is completed. \end{proof}
	
	For each   set $B$, $(\exp B,\exp(\epsilon_{B}))$ is an algebra of the monad $(\mathscr{E}, \sfm,{\sf e})$. The following theorem says that all  algebras of $(\mathscr{E}, \sfm,{\sf e})$ are of this form.
	
	\begin{prop}\label{algebra of the double fuzzy powerset monad} {\rm(Pu and Zhang \cite{PZ2015})}
		Every algebra of the double contravariant fuzzy powerset monad $(\mathscr{E}, \sfm,{\sf e})$ is of the form $(\exp B,\exp(\epsilon_{B}))$ for some set $B$. \end{prop}
	
	\begin{proof} Suppose  $(A,h)$ is an  algebra of $(\mathscr{E}, \sfm,{\sf e})$; that means, $A$ is a  set, $h\colon  \mathscr{E}A \lra A$ is a   map  such that $h\circ\eta_A=1_A$  and that $h\circ\sfm_ A=h\circ \mathscr{E}(h)$. We show that there is some  set $B$ such that $(A,h)$ is isomorphic to $(\exp B,\exp(\epsilon_{B}))$.

		Consider the pullback $$\bfig \Square [B`\exp^{\rm op} A`\exp^{\rm op} A`\exp^{\rm op}\exp\exp^{\rm op} A; i`i'`\exp^{\rm op}(h)`\epsilon_{\exp^{\rm op}A}] \efig$$ in the category of sets. We prove in three steps that $B$ satisfies the requirement.  
		
		\textbf{Step 1}. $i=i'$.  This follows immediately from the  triangular identity $$\exp^{\rm op} (\eta_ A)\circ \epsilon_{\exp^{\rm op} A}=1_{\exp^{\rm op} A}$$ and the equality $$\exp^{\rm op}(\eta_A)\circ \exp^{\rm op}(h) = \exp^{\rm op}(h\circ\eta_A) = 1_{\exp^{\rm op}A}.$$  So $i$ is an equalizer of  $\exp^{\rm op}(h)$ and $\epsilon_{\exp^{\rm op}A}$.

		\textbf{Step 2}. $k_A\coloneqq \exp(i)\circ \eta_A\colon ( A,h)\to (\exp B,\exp(\epsilon_B))$ is a homomorphism of algebras; that means, $k_A\circ h=\exp(\epsilon_{B})\circ \mathscr{E}(k_A)$. To see this, we compute:
		\begin{align*}
			k_A\circ h&=\exp(i) \circ \eta_ A\circ h\\
			&=\exp(i) \circ  \mathscr{E}(h)\circ \eta_{\mathscr{E} A} & (\text{naturailty of $\eta$})\\
			&=\exp(\exp^{\rm op}(h))\circ i)\circ\eta_{\mathscr{E} A} & (\mathscr{E}=\exp\circ\exp^{\rm op})\\
			&=\exp(\epsilon_{\exp^{\rm op}  A}\circ i)\circ \eta_{\mathscr{E} A} & (\text{$i$ equalizes $\exp^{\rm op}(h)$ and $\epsilon_{\exp^{\rm op}  A}$})\\
			&=\exp(i) \circ \exp(\epsilon_{\exp^{\rm op}  A})\circ \eta_{\mathscr{E} A}\\
			&=\exp(i) \circ \sfm_ A\circ \eta_{\mathscr{E} A} &(\sfm_A=\exp(\epsilon_{\exp^{\rm op}  A}))\\
			&=\exp(i) &(\sfm_A\circ\eta_{\mathscr{E} A}=1_{\mathscr{E} A}) \\
			&=\exp(i) \circ\mathscr{E}(h)\circ \mathscr{E}(\eta_A) & (h\circ\eta_A=1_A)\\
			&=\exp(\epsilon_{\exp^{\rm op}  A}\circ i)\circ \mathscr{E}(\eta_A) & (\text{$i$ equalizes $\exp^{\rm op}(h)$ and $\epsilon_{\exp^{\rm op}  A}$})\\
			&=\exp(\exp^{\rm op}\exp(i)\circ\epsilon_B)\circ \mathscr{E}(\eta_A) & (\text{naturailty  of $\epsilon$}) \\
			&=\exp(\epsilon_B)\circ \mathscr{E}(\exp(i))\circ \mathscr{E}(\eta_A) \\
			&=\exp(\epsilon_B)\circ \mathscr{E}(k_A).
		\end{align*}
		
		\textbf{Step 3}. $k_A \colon (A,h)\lra (\exp B,\exp(\epsilon_B))$ is an isomorphism of algebras. It suffices to check that $k_A\colon A\to\exp B$ is a bijection.
		
		Let $l_A=h\circ\mathscr{P}(i)$. On the one hand, by the computation in Step 2 we have $k_ A\circ h=\exp(i) $, then \begin{align*}
			k_A\circ l_A =k_ A\circ h\circ \mathscr{P}(i)  =\exp(i)\circ \mathscr{P}(i) =1_{\exp B},
		\end{align*}
		where the last equality holds because $i\colon B\lra \exp^{\rm op}  A$ is injective.
		
		On the other hand, by virtue of Step 1 and Lemma \ref{relation between P and exp}\thinspace(ii)  one has that
		$$\mathscr{P}(i)\circ \exp(i)=\exp(\epsilon_{\exp^{\rm op}}  A)\circ \mathscr{P}(\exp^{\rm op}(h))
		=\sfm_ A\circ \mathscr{P}(\exp^{\rm op}(h)).$$
		Since  $h\circ\eta_A=1_A$, 
		it follows that  $\exp^{\rm op}(h)$ is an injection,
		hence $$\mathscr{E}(h)\circ \mathscr{P}(\exp^{\rm op}(h))= \exp(\exp^{\rm op}(h))\circ \mathscr{P}(\exp^{\rm op}(h))=1_{\mathscr{E} A}$$ by Lemma \ref{relation between P and exp}\thinspace(i). Consequently, \begin{align*}
			l_A\circ k_A &= h\circ \mathscr{P}(i)\circ\exp(i)\circ\eta_ A\\
			&= h\circ\sfm_ A\circ \mathscr{P}(\exp^{\rm op}(h))\circ \eta_ A\\
			&= h\circ\mathscr{E}(h)\circ \mathscr{P}(\exp^{\rm op}(h))\circ \eta_ A & \text{($(A,h)$ is an algebra)}\\
			&= h\circ\eta_ A & (\mathscr{E}(h)\circ \mathscr{P}(\exp^{\rm op}(h))=1_{\mathscr{E} A})\\
			&= 1_ A.
		\end{align*}
		
		Therefore, $k_A\colon A\to\exp B$ is a bijection, as desired.
	\end{proof}
	
	\section{Categories of \texorpdfstring{$\bbT$}{}-algebras}
	
	In this section, unless otherwise specified, $\CT$ always denotes a saturated class of weights, $\bbT=(\CT,{\sf m},\sy)$ denotes the corresponding submonad of the presheaf monad $\bbP=(\CP,{\sf m},\sy)$.
	This section concerns function spaces in the category $\TAlg$ of $\bbT$-algebras. 
	
	For any real-enriched categories $X$ and $Y$, let $[X,Y]$ be the set of all functors from $X$ to $Y$. For $f,g\in[X,Y]$,  let \[[X,Y](f,g)=\inf_{x\in X}Y(f(x),g(x)).\] Then $[X,Y]$ becomes a real-enriched category, called the functor category.
	
	For any functors $f\colon X_2\lra X_1$ and $g\colon Y_1\lra Y_2$,  the assignment  $h\mapsto g\circ h\circ f$ defines a functor $[X_1,Y_1]\lra[X_2,Y_2]$. In this way we obtain a bifunctor \[[-,-]\colon (\QOrd)^{\rm op}\times\QOrd \lra \QOrd .\]

	For any real-enriched categories $X$ and $Y$, it is readily verified that $(X\times Y,\alpha)$ is a real-enriched category, where $$ \alpha\colon X\times Y\rto X\times Y, \quad \alpha((x_1,y_1),(x_2,y_2))= X(x_1,x_2)\with Y(y_1,y_2).$$   The category $(X\times Y,\alpha)$ is called the tensor product of $X$ and $Y$ and is denoted  by $X\otimes  Y$.\footnote{The symbol $\otimes$ is used to denote both the  tensor in a real-enriched category (i.e. $r\otimes x$)  and the tensor product of  real-enriched categories. Fortunately, it is easy to detect from the context which one is meant.} The assignment $(X,Y)\mapsto X\otimes Y$ defines a bifunctor \[\otimes\colon \QOrd \times\QOrd \lra \QOrd .\]
	
	It is easily seen that the tensor product is symmetric and has the terminal real-enriched category $\star$ as unit, and that for each real-enriched category $X$, the functor $$-\otimes X\colon \QOrd \lra \QOrd $$  is left adjoint to  $$[X,-]\colon \QOrd \lra \QOrd .$$  Therefore, the triple $(\QOrd,\otimes,\star)$ is a symmetric monoidal closed category. 
	
	Given  real-enriched categories $X$ and $Y$, $\phi\in\CP X$ and $\psi\in\CP Y$, the function $$\phi\otimes\psi\colon X\times Y\lra[0,1],\quad (x,y)\mapsto\phi(x)\with\psi(y)$$ is  a weight of $X\otimes Y$. By continuity of the t-norm $\&$ one verifies that the assignment $(\phi,\psi)\mapsto\phi\otimes\psi$ defines a functor $\CP X\otimes\CP Y\lra \CP(X\otimes Y) .$ So, the presheaf functor $\CP$ is  a lax functor on the symmetric monoidal closed category $(\QOrd,\otimes,\star)$.

	\begin{thm}\label{function space} {\rm(Lai and Zhang \cite{LaiZ2007})} Assigning to each pair $(X,Y)$ of $\bbT$-algebras the subcategory $\hom(X,Y)$ of  $[X,Y]$ composed of $\bbT$-homomorphisms  gives rise to a bifunctor \[\hom\colon (\TAlg)^{\rm op}\times\TAlg\lra \TAlg.\]    \end{thm}
	
	\begin{lem}\label{T-cocts} Suppose $X,Y$ are real-enriched categories and $Y$ is  $\CT$-cocomplete.   \begin{enumerate}[label={\rm(\roman*)}] 
			\item  $[X,Y]$ is $\CT$-cocomplete and  $\CT$-colimits in $[X,Y]$ are computed pointwise;   that means, for each $\Phi\in\CT[X,Y]$,  the map \[h\colon X\lra Y, \quad h(x)={\colim}_\Phi[x]\] is a colimit of $\Phi$, where $[x]$ is the projection $[X,Y]\lra Y,~ f\mapsto f(x)$. 
			\item The subcategory $\hom(X,Y)$  of $[X,Y]$ consisting of functors that preserve $\CT$-colimits  is closed under  $\CT$-colimits, hence  $\CT$-cocomplete.
		\end{enumerate}
	\end{lem}
	
	\begin{proof}(i) We prove the conclusion in two steps.

		{\bf Step 1}. $h$ is a functor. Let $x_1,x_2\in X$. For each $y\in Y$, since \[\Phi\circ[x_1]^*(y)= \sup_{f\in[X,Y]}\Phi(f)\with\thinspace[x_1]^*(y,f)= \sup_{f\in[X,Y]}\Phi(f)\with Y(y,f(x_1)), \] it follows that 
		\begin{align*}&\quad\quad \Phi\circ[x_1]^*(y)\ra\Phi\circ[x_2]^*(y) \\ 
			&\geq\inf_{f\in[X,Y]}(\Phi(f)\with Y(y,f(x_1))\ra \Phi(f)\with Y(y,f(x_2)))\\  
			&\geq\inf_{f\in[X,Y]}(Y(y,f(x_1))\ra Y(y,f(x_2)))\\ &\geq X(x_1,x_2). 
		\end{align*} Therefore, \begin{align*}Y(h(x_1),h(x_2)) 
			& = Y({\colim}_\Phi[x_1], {\colim}_\Phi[x_2])\\ 
			&\geq \CP Y(\Phi\circ[x_1]^*, \Phi\circ[x_2]^*)  \\
			&= \inf_{y\in Y} (\Phi\circ[x_1]^*(y)\ra\Phi\circ[x_2]^*(y)) \\ 
			&\geq X(x_1,x_2),\end{align*}   so $h$  is a functor.
		
		{\bf Step 2}. $h$ is a colimit of $\Phi$. For each $g\in[X,Y]$, 
		\begin{align*} [X,Y](h,g)&= \inf_{x\in X}Y(h(x),g(x))\\ 
			&= \inf_{x\in X}Y({\colim}_\Phi[x],g(x))\\ 
			&= \inf_{x\in X}\CP Y(\Phi\circ[x]^*,Y(-,g(x)))\\ 
			&= \inf_{x\in X}\CP[X,Y](\Phi, Y(-,g(x))\circ[x]_*) &(-\circ[x]^*\dashv-\circ[x]_*)\\ 
			&= \inf_{x\in X}\inf_{f\in[X,Y]}(\Phi(f)\ra Y(f(x),g(x)))\\ 
			&= \inf_{f\in[X,Y]}(\Phi(f)\ra[X,Y](f,g))\\ 
			&= [X,Y](-,g)\swarrow\Phi,\end{align*} so $h$ is a colimit of $\Phi$.
		
		(ii) We prove that for each $\Phi\in\CT\hom(X,Y)$, the functor  \[h\colon X\lra Y, \quad h(x)={\colim}_\Phi[x]\] preserves $\CT$-colimits, where $[x]$ denotes the projection \[\hom(X,Y)\lra Y,\quad f\mapsto f(x).\]
		
		Let $\phi\in\CT X $ and $a$ be a colimit of $\phi$. We show that $h(a)$ is a colimit of $h_\exists(\phi)$. On the one hand, \begin{align*}Y(h(a),y)&= Y({\colim}_\Phi[a],y)\\ 
			&= \CP Y(\Phi\circ[a]^*, Y(-,y)) \\ 
			&= \CP(\hom(X,Y))(\Phi,Y(-,y)\circ[a]_*) \\ 
			&= \inf_{f\in\hom(X,Y)}(\Phi(f)\ra Y(f(a),y)) \\ 
			&= \inf_{f\in\hom(X,Y)}(\Phi(f)\ra Y({\colim}_\phi f,y))\\ 
			&= \inf_{f\in\hom(X,Y)}(\Phi(f)\ra \CP Y(\phi\circ f^*,Y(-,y)))\\ &= \inf_{f\in\hom(X,Y)}(\Phi(f)\ra \CP X(\phi,Y(-,y)\circ f_*)) \\ 
			&= \inf_{f\in\hom(X,Y)}\inf_{x\in X}(\Phi(f)\ra (\phi(x)\ra Y(f(x),y) ))\\ 
			&= \inf_{f\in\hom(X,Y)}\inf_{x\in X}(\Phi(f)\with\phi(x)\ra Y(f(x),y)).\end{align*} On the other hand,
		\begin{align*}\CP Y(h_\exists(\phi),Y(-,y))
			&= \CP X(\phi,h^{-1}(Y(-,y)))\\ 
			&=  \inf_{x\in X} (\phi(x)\ra  Y(h(x), y)) \\ 
			&=  \inf_{x\in X} (\phi(x)\ra \CP Y(\Phi\circ[x]^*,Y(-,y))) \\ 
			&=  \inf_{x\in X} (\phi(x)\ra \CP(\hom(X,Y))(\Phi,Y(-,y)\circ[x]_*))\\ 
			&=\inf_{x\in X}\inf_{f\in\hom(X,Y)}(\Phi(f)\with\phi(x)\ra Y(f(x),y)). \end{align*} Therefore  $Y(h(a),y)=\CP Y(h_\exists(\phi),Y(-,y))$, so  $h(a)$ is a colimit of $h_\exists(\phi)$.\end{proof}
	
	\begin{proof}[Proof of Theorem \ref{function space}]  By lemma \ref{T-cocts}\thinspace(ii), it suffices to show that \begin{enumerate}[label={\rm(\roman*)}] 
			\item for any $\bbT$-algebras $X,Y,Z$ and any  $\bbT$-homomorphism   $f\colon Y\lra Z$, the functor \[\hom(X,f)\colon \hom(X,Y)\lra\hom(X,Z), \quad h\mapsto f\circ h\] preserves $\CT$-colimits;   \item for any $\bbT$-algebras $X,Y,Z$ and any $\bbT$-homomorphism $g\colon X\lra Y$,  the  functor \[\hom(g,Z)\colon \hom(Y,Z)\lra\hom(X,Z), \quad h\mapsto h\circ g\] preserves $\CT$-colimits. \end{enumerate}
		
		For (i), assume that $\Phi\in\CT\hom(X,Y)$. For each $x\in X$ and each $h\in\hom(X,Y)$,  since  \[[x]\circ \hom(X,f)(h)= f\circ h(x) =f\circ[x](h)\] and \begin{align*} \hom(X,f)(\colim\Phi)(x)&= f\circ\colim\Phi(x)\\ &= f({\colim}_\Phi[x]) &\text{(definition of $\colim\Phi$)}\\ &= {\colim}_\Phi(f\circ[x]), & \text{($f$ preserves $\CT$-colimits)}
		\end{align*}   it follows that \begin{align*} ({\colim}_\Phi\hom(X,f))(x)&= {\colim}_\Phi([x]\circ\hom(X,f))\\  
			&=  {\colim}_\Phi(f\circ[x]) \\ &=\hom(X,f)(\colim\Phi)(x), \end{align*} hence ${\colim}_\Phi\hom(X,f) =\hom(X,f)(\colim\Phi)$, 
		as desired. 
		
		The assertion (ii) can be verified in a similar way. \end{proof}
	
	Theorem \ref{function space} leads to an interesting question in the theory of real-enriched categories:
	
	\begin{ques}\label{closed TAlg} Does there exist a bifunctor \[\Box\colon \TAlg \times\TAlg \lra \TAlg\] and a bijection \[\hom(Z\Box X, Y)\cong \hom (Z,\hom(X,Y))\]  natural in all variables $X,Y$ and $Z$?  \end{ques}
	If the answer is positive and $\Box$ has a unit, say $I$, then the triple $(\TAlg,\Box,I)$ is a monoidal closed category.
	A  general problem   is to find  subcategories $\mathscr{A}$ of $\TAlg$ that satisfy the following conditions: \begin{enumerate}[label={\rm(\roman*)}]  \item Restricting the domain of the bifunctor \[\hom\colon (\TAlg)^{\rm op}\times\TAlg\lra \TAlg \]  yields a bifunctor \[\hom\colon \mathscr{A}^{\rm op}\times\mathscr{A}\lra \mathscr{A}.\]
		That means, for any objects $X,Y,Z$  and any morphisms $f\colon  Y\lra Z$, $g\colon X\lra Y$ of $\mathscr{A}$, \begin{itemize} \item the real-enriched category $\hom(X,Y)$ composed of $\mathscr{A}$-morphisms  belongs to  $\mathscr{A}$; \item both  \[\hom(X,f)\colon \hom(X,Y)\lra\hom(X,Z), \quad h\mapsto f\circ h\] and \[\hom(g,Z)\colon \hom(Y,Z)\lra\hom(X,Z), \quad h\mapsto h\circ g\] are morphisms of $\mathscr{A}$.\end{itemize}
		\item There exists  a bifunctor \[\Box\colon \mathscr{A} \times\mathscr{A} \lra \mathscr{A}\] and a bijection \[\hom(Z\Box X, Y)\cong \hom (Z,\hom(X,Y))\]  natural in all variables $X,Y$ and $Z$. \end{enumerate}
	
	By a result in category theory, see e.g. Mac\thinspace Lane \cite[page 102, Theorem 3]{MacLane1998}, in the presence of (i),  condition (ii) is equivalent to that  for each object $X$ of $\mathscr{A}$, the functor \[\hom(X,-)\colon \mathscr{A}\lra \mathscr{A}\] has a left adjoint.

	For the class  of Cauchy weights, the class  of ideals, and the largest class  of  weights, the answer to Question \ref{closed TAlg} is   positive, as we see below.
	
	\begin{prop}\label{CAlg is a monoidal closed} 
		The triple $(\CAlg,\otimes,\star)$ is a symmetric monoidal closed category. \end{prop}
	
	\begin{proof} Since $\CAlg$ is a full subcategory of the closed category $(\QOrd,\otimes, \star)$ and the terminal real-enriched category $\star$ is Cauchy complete, we only need to check  that for all Cauchy complete real-enriched categories $X$ and $Y$, both the tensor product  $X\otimes Y$ and the functor category $[X,Y]$  are Cauchy complete. Cauchy completeness of $[X,Y]$ follows from Theorem \ref{function space}. In the following we use the characterizations in Proposition \ref{colimit =bilimit for Cauchy net}  to verify that $X\otimes Y$ is Cauchy complete. Suppose $\{(x_n,y_n)\}_{n\in\bbN}$ is a Cauchy sequence of $X\otimes Y$. It is clear that $\{x_n\}_{n\in\bbN}$ is a Cauchy sequence of $X$ and $\{y_n\}_{n\in\bbN}$ is a Cauchy sequence of $Y$. Suppose   $\{x_n\}_{n\in\bbN}$ converges to $a$ and $\{y_n\}_{n\in\bbN}$ converges to $b$. It is readily verified that $\{(x_n,y_n)\}_{n\in\bbN}$ converges to $(a,b)$, so $X\otimes Y$ is Cauchy complete. \end{proof}
	
	\begin{thm}\label{mono clos of lim} {\rm (Lai and Zhang \cite{LaiZ2016})} The triple  $(\IAlg, \otimes, \star)$ is a symmetric monoidal closed category. \end{thm} 
	
	\begin{proof} 
		Since $(\QOrd,\otimes, \star)$ is a monoidal closed category and $\star$ is Yoneda complete, it suffices to show that  for all Yoneda complete real-enriched categories $X,Y,Z$, it holds that
		
		\begin{enumerate}[label={\rm(\roman*)}] 
			\item $X\otimes Y$ is Yoneda complete;
			\item the subset $\hom(X,Y)$ of $[X,Y]$ consisting of Yoneda continuous maps is Yoneda complete;
			\item $f\colon X\otimes Y\lra Z$ is  Yoneda continuous if and only if $f\colon X\times Y\lra Z$ is  Yoneda continuous separately.
		\end{enumerate}
		
		Assertion (ii) is contained in Theorem \ref{function space}. Assertions (i) and (iii) are contained, respectively, in Lemma  \ref{tens is limi comp} and Lemma \ref{bimor=YC} below.
	\end{proof}
	
	The following fact will be used in Lemma  \ref{tens is limi comp} and Lemma \ref{bimor=YC}. Suppose $X,Y$ are real-enriched categories and $\{(x_i,y_i)\}_{i\in D}$ is a net of $X\otimes Y$. Then $\{(x_i,y_i)\}_{i\in D}$  is a forward Cauchy net of $X\otimes Y$ if and only if  $\{x_i\}_{i\in D}$ is a forward Cauchy net of $X$ and $\{y_i\}_{i\in D}$ is a forward Cauchy net of $Y$, in which case, \begin{align*}  
		\sup_{i\in D}\inf_{i\leq j,k\leq l}X(x_j,x_l)\with Y(y_k,y_l)&=1, \end{align*}  where ``$i\leq j,k\leq l$'' means ``$i\leq j\leq l$ and $i\leq k\leq l$''
	
	\begin{lem}\label{tens is limi comp} If both of the real-enriched categories $X$ and $Y$ are  Yoneda complete,  then so is the tensor product  $X\otimes Y$.  \end{lem}
	
	\begin{proof} We prove that for any real-enriched categories $X$ and $Y$ and any  forward Cauchy net $\{(x_i,y_i)\}_{i\in D}$  of $X\otimes Y$, if $a$  is a Yoneda limit of $\{x_i\}_{i\in D}$ and $b$ is a Yoneda limit of $\{y_i\}_{i\in D}$, then $(a,b)$ is a Yoneda limit of $\{(x_i,y_i)\}_{i\in D}$. For this we check that for all $x\in X$ and $y\in Y$,
		\[ X\otimes Y((a,b),(x,y))=\sup_{i\in D}\inf_{j\geq i}X(x_j,x)\with Y(y_j,y). \]

		On the one hand,
		\begin{align*}
			X\otimes Y((a,b),(x,y))&= X(a,x)\with  Y(b,y)\\
			&= \Big(\sup_{i\in D}\inf_{j\geq i}X(x_j,x)\Big)\with
			\Big(\sup_{i\in D}\inf_{j\geq i}Y(y_j,y)\Big)\\
			&= \sup_{i\in D}\Big[\Big(\inf_{j\geq i}X(x_j,x)\Big)\with
			\Big(\inf_{j\geq i}Y(y_j,y)\Big)\Big]\\
			&\leq \sup_{i\in D}\inf_{j\geq i}X(x_j,x)\with Y(y_j,y).  
		\end{align*} 
		On the other hand, 
		since   both $\{x_i\}_{i\in D}$ and $\{y_i\}_{i\in D}$ are forward Cauchy,  then  \begin{align*}
			& \quad\quad ~ \sup_{i\in D}\inf_{j\geq i}X(x_j,x)\with Y(y_j,y) \\
			&=\Big(\sup_{i\in D}\inf_{j\geq i}X(x_j,x)\with Y(y_j,y)\Big)\with
			\Big(\sup_{i\in D}\inf_{i\leq j,k\leq l}X(x_j,x_l)\with Y(y_k,y_l)\Big)\\ 
			&\leq  \sup_{i\in D}
			\inf_{i\leq j,k\leq l} X(x_l,x)\with Y(y_l,y) \with X(x_j,x_l)\with Y(y_k,y_l) \\
			&\leq \sup_{i\in D}\inf_{j\geq i}\inf_{k\geq i} X(x_j,x)\with Y(y_k,y) \\ 
			&= \sup_{i\in D}\Big[\Big(\inf_{j\geq i}X(x_j,x)\Big)\with\Big(\inf_{k\geq i}Y(y_k,y)\Big)\Big] \\ 
			&= \Big(\sup_{i\in D}\inf_{j\geq i}X(x_j,x)\Big)\with\Big(\sup_{i\in D}\inf_{k\geq i}Y(y_k,y)\Big) \\ 
			&= X(a,x)\with Y(b,y)\\
			& = X\otimes Y((a,b),(x,y)). \qedhere
		\end{align*}
	\end{proof}
	
	\begin{lem}\label{bimor=YC} Suppose  $X,Y,Z$ are Yoneda complete real-enriched categories. Then,  a functor $f\colon X\otimes Y\lra Z$  is Yoneda continuous if and only if $f\colon X\times Y\lra Z$ is Yoneda continuous separately.
	\end{lem}
	
	\begin{proof}
		For necessity  we check that $f$ is Yoneda continuous in the first variable for example.  Let $b\in Y$ and let $\{x_i\}_{i\in D}$ be a forward Cauchy net of $X$ with   a Yoneda limit $a$. We  show that $f(a,b)$ is a Yoneda limit of $\{f(x_i,b)\}_{i\in D}$. Because $f\colon X\otimes Y\lra Z$ is Yoneda continuous, it suffices to check that $(a,b)$ is a Yoneda limit of the forward Cauchy net $\{(x_i,b)\}_{i\in D}$ of $X\otimes Y$. This follows from that for all $x\in X$ and $ y\in Y$,
		\begin{align*}
			X\otimes Y((a,b),(x,y))&=  X(a,x)\with Y(b,y)\\
			&= \Big(\sup_{i\in D}\inf_{j\geq i}X(x_j,x)\Big)\with Y(b,y)  \\
			&= \sup_{i\in D} \inf_{j\geq i}(X(x_j,x) \with Y(b,y) ) 
			\\ &= \sup_{i\in D} \inf_{j\geq i}X\otimes Y((x_j, b), (x,y)).
		\end{align*}

		For sufficiency suppose $\{(x_i,y_i)\}_{i\in D}$ is a forward Cauchy net of $X\otimes Y$ and  $(a,b)$ is a Yoneda limit of $\{(x_i,y_i)\}_{i\in D}$.   We   show that $f(a,b)$ is a Yoneda limit of $\{f(x_i,y_i)\}_{i\in D}$.  Since $Z$ is Yoneda complete,  $\{f(x_i,y_i)\}_{i\in D}$  has a Yoneda limit, say $c$. We wish to show that $c$ is isomorphic to $f(a,b)$. 
		
		Since \begin{align*}
			Z(c,f(a,b))
			&= \sup_{i\in D}\inf_{j\geq i}Z(f(x_j,y_j),f(a,b))\\
			&\geq \sup_{i\in D}\inf_{j\geq i} X\otimes Y( (x_j,y_j), (a,b))\\
			&=   X\otimes Y((a,b),(a,b))\\
			&=  1,
		\end{align*} it follows that $c\sqsubseteq f(a,b)$.
		It remains to check that $f(a,b)\sqsubseteq c$. On the one hand,  since   $\{x_i\}_{i\in D}$ is a forward Cauchy net of $X$ and $\{y_i\}_{i\in D}$ is a forward Cauchy net of $Y$,  by Yoneda completeness of $X,Y$ and   Lemma \ref{tens is limi comp} one sees that $a$ is a Yoneda limit of $\{x_i\}_{i\in D}$,  $b$ is a Yoneda limit of $\{y_i\}_{i\in D}$,  then \begin{align*}
			Z(f(a,b),c) 
			&= \sup_{i\in D}\inf_{j\geq i}Z(f(x_j,b),c)  
			\\
			&= \sup_{i\in D}\inf_{j\geq i}\sup_{h\in D}\inf_{k\geq h}Z(f(x_j,y_k),c)  
			\\
			&\geq \sup_{i\in D}\inf_{j\geq i}\inf_{k\geq i}
			Z(f(x_j,y_k),c).
		\end{align*} On the other hand, since $\{(x_i,y_i)\}_{i\in D}$ is a forward Cauchy net of $X\otimes Y$ and $c$ is a Yoneda limit of the net $\{f(x_i,y_i)\}_{i\in D}$, then  
		\begin{align*}
			1&= Z(c,c) \\ &= \sup_{i\in D}\inf_{l\geq i}Z(f(x_l,y_l),c)\\
			&= \Big(\sup_{i\in D}\inf_{l\geq i}Z(f(x_l,y_l),c)\Big)\with \Big(\sup_{i\in D}\inf_{i\leq j,k\leq l}X\otimes Y((x_j,y_k),(x_l,y_l))\Big)\\
			&\leq \Big(\sup_{i\in D}\inf_{l\geq i}Z(f(x_l,y_l),c)\Big)\with \Big(\sup_{i\in D}\inf_{i\leq j,k\leq l} Z(f(x_j,y_k),f(x_l,y_l))\Big)\\
			&\leq  \sup_{i\in D}\inf_{i\leq j,k\leq l}Z(f(x_l,y_l),c)\with
			Z(f(x_j,y_k),f(x_l,y_l)) \\
			&\leq \sup_{i\in D}\inf_{j\geq i}\inf_{k\geq i}Z(f(x_j, y_k),c).
		\end{align*}
		Therefore, $Z(f(a,b),c)=1$ and consequently, $f(a,b)\sqsubseteq c$.
	\end{proof}
	
	Proposition \ref{CAlg is a monoidal closed} and Theorem \ref{mono clos of lim} show that the category $\CAlg$ of $\bbC$-algebras and the category $\IAlg$ of $\bbI$-algebras are closed in the monoidal closed category $(\QOrd,\otimes, \star)$ with respect to various operations, so they are themselves monoidal closed categories. But, the situation for the largest class of weights, i.e. the class $\CP$ of all weights, is quite different. 
	
	\begin{exmp} Let $\&$ be the product t-norm on $[0,1]$. Then the tensor product $\sV\otimes \sV$ of the complete real enriched category $\sV=([0,1],\alpha_L)$ with itself is not tensored. Otherwise, there is some $(a,b)\in[0,1]\times[0,1]$, the tensor of $1/2$ with $(1/2,1/2)$, such that $$ \sV\otimes\sV((a,b),(x,y))  = 1/2\ra\sV(1/2,x)\with\sV(1/2,y) $$
		for all $(x,y)\in[0,1]\times[0,1]$. 
		Putting $x=1/4$ and $y=1/2$ gives that $a\leq 1/4$; putting $x=1/2$ and $y=1/4$ gives that $b\leq 1/4$. Then $$\sV\otimes\sV((a,b),(1/4,1/4))=(a\ra 1/4)\with (b\ra 1/4)=1,$$ but $$1/2\ra\sV(1/2,1/4)\with\sV(1/2,1/4)
		=1/2,$$ a contradiction. \end{exmp} 
	
	\begin{thm}\label{PAlg is a monoidal closed} {\rm(Joyal and Tierney \cite[page 9]{Joyal-Tierney})}
		For each $\bbP$-algebra $X$, the functor $\hom(X,-) \colon\PAlg\lra\PAlg$ has a left adjoint. \end{thm}
	
	\begin{lem} For all $\bbP$-algebras $X$ and $Y$,   the real-enriched categories $\hom(X,Y)$ and   $\hom(Y^{\rm op},X^{\rm op})$ are isomorphic.\end{lem}
	
	\begin{proof} For each $f\in\hom(X,Y)$, let $G(f)\colon Y^{\rm op}\lra X^{\rm op}$ be the opposite of the right adjoint of $f$. We claim that $G\colon\hom(X,Y)\lra\hom(Y^{\rm op},X^{\rm op})$ is an isomorphism. For this we   check that $G$ preserves tensors and  joins (w.r.t. the underlying order).
		
		We check that $G$ preserves tensors for example.  Let  $f\in\hom(X,Y)$ and $r\in[0,1]$. It is readily verified that if $g\colon Y\lra X$ is a right adjoint of $f$,   the correspondence $y\mapsto r\multimap g(y)$ defines a right adjoint of $r\otimes f$, which shows that $G$ preserves tensors. \end{proof}

	\begin{proof}[Proof of Theorem \ref{PAlg is a monoidal closed}] First of all, we note that  \begin{equation*}\label{swith variables} \hom(Z,\hom(X,Y)) \cong \hom(X,\hom(Z,Y)) \end{equation*} for  all $\bbP$-algebras $X,Y$ and $Z$.

		For each $\bbP$-algebra $Y$, let \[Y\Box X = \hom(Y,X^{\rm op})^{\rm op}. \] For each morphism   $f\colon Y\lra Z$ in $\PAlg$, let \[f\Box\id_X\colon Y\Box X\lra Z\Box X\]  be the opposite of the right adjoint of \[\hom(f,X^{\rm op})\colon  \hom(Z,X^{\rm op}) \lra \hom(Y,X^{\rm op}) .\]
		The correspondence   $f\mapsto f\Box\id_X$ gives rise to a functor \[-\Box X\colon  \PAlg\lra\PAlg.\]
		
		We claim that this functor is   left adjoint to \[\hom(X,-)\colon  \PAlg\lra\PAlg.\]
		For this, we show that for any $Y,Z$ of $\PAlg$, there is a  bijection \[\hom(Z\Box X, Y)\cong \hom (Z,\hom(X,Y)) \] natural in both $Y$ and $Z$. The correspondence $f\mapsto \overline{f}$ displayed below exhibits one such:  
		\begin{align*}f\colon Z\Box X\lra Y & \\
			f\colon \hom(Z,X^{\rm op})^{\rm op}\lra Y & \quad\quad \text{definition of $Z\Box X$} \\ g\colon  Y\lra\hom(Z,X^{\rm op})^{\rm op}  & \quad\quad \text{the right adjoint of $f$} \\ g^{\rm op}\colon  Y^{\rm op}\lra\hom(Z,X^{\rm op})  & \quad\quad \text{the opposite of $g$} \\ h\colon  Z\lra\hom(Y^{\rm op},X^{\rm op}) & \quad\quad \text{switch $Y^{\rm op}$ and $Z$}
			\\ \overline{f}\colon  Z\lra\hom(X,Y) & \quad\quad \hom(Y^{\rm op},X^{\rm op})\cong\hom(X,Y) \qedhere \end{align*}
	\end{proof}
	
	The bifunctor $\Box$ is symmetric, associative, and has $\sV=([0,1],\alpha_L)$ as  unit. All told,
	
	\begin{prop} The triple $(\PAlg,\Box,\sV)$ is a symmetric monoidal closed category. \end{prop}

	Let $\mathbb{F}$ be the monad $(\CF,{\sf m},\sy)$  in $\QOrd$. As a special case of
	Question \ref{closed TAlg} we ask the following:
	
	\begin{ques}Is the category $\mathbb{F}\text{-}{\bf Alg}$ of $\mathbb{F}$-algebras   monoidal closed?  \end{ques}
	
	\section{Continuous \texorpdfstring{$\bbT$}{}-algebras}
	
	As in the previous section,  unless otherwise specified, in this section $\CT$ always denotes a saturated class of weights on $\QOrd$, and $\bbT$ denotes the corresponding submonad of the presheaf monad in $\QOrd$.
	
	\begin{defn}\label{T-continuous} Suppose $A$ is a $\bbT$-algebra. We say that $A$ is continuous if the left adjoint $\colim\colon\CT A\lra A$ of $\sy\colon A\lra\CT A$  has a left adjoint. In other words, $A$ is continuous if there exists a string of adjunctions
		\[\wayb \dashv{\colim} \dashv\sy\colon A\lra\CT A.\]
	\end{defn}
	
	Suppose $A$ is a continuous $\bbT$-algebra. We claim that $\colim\wayb x=x$ for each   $x\in A$. On the one hand, since $$A(x,\colim\wayb x)=\CT A(\wayb x,\wayb x)=1,$$ then $x\sqsubseteq\colim\wayb x$, where $\sqsubseteq$ is the underlying order of $A$. On the other hand, since $$A(\colim\wayb x,x)=\CT A(\wayb x,\sy(x))=A(x,\colim\sy(x))=A(x,x)=1,$$ then $\colim\wayb x\sqsubseteq x$.  Therefore, $\colim\wayb x=x$ since $A$ is separated.
	
	\begin{exmp}Consider the saturated class $\CC$ of Cauchy weights and the monad $\bbC=(\CC,{\sf m},\sy)$. Then every $\bbC$-algebra $A$ is continuous, because every Cauchy weight of $A$ is representable,  so $\sy\colon A\lra\CC A$ is an isomorphism.  \end{exmp}
	
	\begin{prop}\label{TA is T-continuous}  For every real-enriched category $A$, the  $\bbT$-algebra $\CT A$ is continuous.
	\end{prop}
	
	\begin{proof} Since $\bbT=(\CT,{\sf m},\sy)$ is a KZ-monad,  for every real-enriched category $A$, there is a string of adjunctions \[\CT\sy_A\dashv{\sf m}_A\dashv\sy_{\CT A}\colon \CT A\lra\CT\CT A,\] which shows that $\CT A$ is   continuous. \end{proof}
	
	For a $\bbT$-algebra $A$,   the distributor \[\mathfrak{t}\coloneqq{\colim}^*\searrow\sy_*\colon A\oto A\] is called the \emph{$\CT$-below distributor} on $A$.
	$$\bfig \qtriangle(0,0)/>`>`<-/<520,500>[A`\CT A`A;\sy_*`\mathfrak{t}`{\colim}^*]
	\place(230,500)[\circ] \place(260,250)[\circ] \place(520,275)[\circ]   \efig$$
	Explicitly, for all $x,y\in A$, \[\mathfrak{t}(y,x)=\inf_{\phi\in\CT A}(A(x,\colim\phi)\ra\phi(y)).\]
	Putting $\phi=\sy(x)$  gives that  $\mathfrak{t}(y,x)\leq A(y,x)$.
	
	Suppose $X,Y$ are real-enriched categories. For each functor $f\colon Y\lra \CP X$, the distributor  $$X\oto Y, \quad (x,y)\mapsto f(y)(x)$$ is called the  \emph{transpose}  of $f$. In particular,  the transpose of the  Yoneda embedding $\sy\colon X\lra\CP X$ is the identity distributor $X\oto X$. Conversely, for each distributor $\phi\colon X\oto Y$, the assignment $y\mapsto\phi(-,y)$ defines a functor $Y\lra\CP X$, called the \emph{transpose} of $\phi$. Taking transpose defines a natural bijection between   functors $Y\lra\CP X$  and   distributors $X\oto Y$.
	
	\begin{prop}\label{way below as left adjoint} Suppose $A$ is a $\bbT$-algebra. Then $A$ is continuous   if and only if for all $x\in A$, the weight $\mathfrak{t}(-,x)$ belongs to $\CT A$ and has $x$ as  colimit. In this case,  the left adjoint $\wayb\colon A\lra\CT A$ and the $\CT$-below distributor $\mathfrak{t}\colon A\oto A$ are transpose of each other.  
	\end{prop}
	
	\begin{proof} We prove the necessity first. Given a continuous $\bbT$-algebra $A$, we show that  for each $x\in A$, $\mathfrak{t}(-,x)=\wayb x$,  so $\mathfrak{t}(-,x)$ belongs to $\CT A$ and has $x$ as  colimit.   On the one hand, since $\colim\wayb x=x$, then $$\mathfrak{t}(-,x)\leq A(x,\colim\wayb x)\ra\wayb x= \wayb x.$$ On the other hand, since  for all   $\phi\in\CT A$ we have $$A(x,\colim\phi)\with\wayb x = \CT A(\wayb x,\phi)\with\wayb x\leq \phi,$$   it follows that $\wayb x\leq A(x,\colim\phi)\ra\phi$ for all $\phi\in\CT A$, hence $\wayb x\leq \mathfrak{t}(-,x)$.  
		
		As for the sufficiency, we show that the assignment $x\mapsto \wayb x\coloneqq \mathfrak{t}(-,x)$ defines a  left  adjoint of $\colim\colon\CT A\lra A$. Let $x\in A$ and $\phi\in\CT A$. By definition of  $\mathfrak{t}(-,x)$, we have $\wayb x\leq   A(x,\colim\phi)\ra\phi$, so $A(x,\colim\phi)\leq \CT A(\wayb x,\phi)$. Since $\colim$ is a functor and $\colim\wayb x=x$, then $\CT A(\wayb x,\phi)\leq A(\colim\wayb x,\colim\phi)=A(x,\colim\phi)$. Therefore, $\wayb$ is a left adjoint of $\colim$ and $A$ is continuous. \end{proof} 
	
	\begin{prop}Suppose $A$ is a continuous $\bbT$-algebra. Then the $\CT$-below distributor $\mathfrak{t}\colon A\oto A$   interpolates in the sense that $\mathfrak{t}\circ \mathfrak{t}=\mathfrak{t}$. \end{prop}
	
	\begin{proof} 
		Before proving the conclusion, we would like to note that the conclusion (or for some specific class of weights) has been proved in several places, e.g. Hofmann and Waskiewicz \cite{HW2011,HW2012}, Stubbe \cite{Stubbe2007}, and Waskiewicz \cite{Waskiewicz09}. 
		
		Since $\mathfrak{t}$ is the transpose of the left adjoint $\wayb\colon A\lra\CT A$,
		it suffices to show that for all $x\in A$, \[\wayb x=\sup_{z\in A}\wayb x(z)\with \wayb  z.\] 
		
		Since $A$ is a continuous $\bbT$-algebra,   $\colim\wayb x =x$ for all $x\in A$. Since $\wayb \colon A\lra\CT A$ is a left adjoint, it preserves colimits, so $\wayb x$ is the colimit of the functor $\wayb \colon A\lra\CT A$ weighted by $\wayb x$, i.e. \[\wayb x ={\colim}_{\wayb x}\wayb.\]  Since  $\CT A$ is closed under $\CT$-colimits in $\CP A$, then ${\colim}_{\wayb x}\wayb={\colim}_{\wayb x}(\mathfrak{i}\circ\wayb)$, where $\mathfrak{i}\colon\CT X\lra\CP X$ is the inclusion functor. Thus, by Corollary \ref{calculation of colimit in PX}, \begin{align*}\wayb x&={\colim}_{\wayb x}(\mathfrak{i}\circ\wayb) = \sup_{z\in A}\wayb x(z)\with \wayb z.\qedhere \end{align*}  
	\end{proof}
	
	\begin{prop} \label{retract of continuous T-algebra} 
		Every retract  of a continuous $\bbT$-algebra in    $\TAlg$  is a continuous $\bbT$-algebra. Furthermore, a $\bbT$-algebra $A$ is continuous if and only if it is a retract of   $\CT A$ in $\TAlg$.
	\end{prop}
	\begin{proof} Suppose  $B$ is a continuous $\bbT$-algebra;  $s\colon A\lra B$ and $r\colon B\lra A$ are $\mathbb{T}$-homomorphisms with $r\circ s=1_A$. We show that   $$\wayb_A\coloneqq \CT r\circ \wayb_B\circ\thinspace s\colon A\lra B\lra\CT B\lra \CT A$$ is left adjoint to $\colim_A\colon\CT  A\lra A$, so $A$ is a continuous $\bbT$-algebra. 
		
		On the one hand,
		\begin{align*}{\colim}_A\circ \wayb_A&={\colim}_A\circ\CT r\circ \wayb_B\circ s\\
			&=r\circ{\colim}_B\circ \wayb_B\circ\thinspace s &(\mbox{$r$ is a $\mathbb{T}$-homomorphism})\\
			&=r\circ\thinspace s &({\colim}_B\circ \wayb_B=\id_B)\\
			&=1_A.
		\end{align*}
		
		On the other hand,
		\begin{align*}\wayb_A\circ\thinspace{\colim}_A&=\CT r\circ \wayb_B\circ\thinspace s\circ{\colim}_A\\
			&=\CT r\circ \wayb_B \circ\thinspace  {\colim}_B \circ\CT s&(\mbox{$s$ is a $\mathbb{T}$-homomorphism})\\
			&\leq\CT r\circ\CT s &(\wayb_B\dashv{\colim}_B)\\
			&=1_{\CT A}.
		\end{align*}
		
		Thus, $\wayb_A$ is left adjoint to $\colim_A$, as desired. 
		
		The second conclusion follows from that $\colim_A\circ\wayb_A=\id_A$ for every continuous $\bbT$-algebra $A$.  \end{proof}
	
	\begin{cor}\label{cc}
		Suppose  $A,B$ are separated real-enriched categories, $g\colon A\lra B$  is a right adjoint  preserving $\CT$-colimits.
		\begin{enumerate}[label=\rm(\roman*)] 
			\item If $g$ is surjective and  $A$ is a
			continuous $\bbT$-algebra, then so is $B$.
			\item If $g$ is injective  and  $B$ is
			continuous $\bbT$-algebra, then so is $A$.
		\end{enumerate}
	\end{cor}
	\begin{proof}
		(i) Let $f\colon B\lra A$ be the left adjoint of $g$. Then $g\circ f=1_B$, so $B$ is a $\mathbb{T}$-algebra by Proposition \ref{retract of T-alg}, hence a retract of $A$ in $\TAlg$, and consequently, a continuous $\bbT$-algebra by the above proposition.
		
		(ii) Let $f\colon B\lra A$ be the left adjoint of $g$. Then $f\circ g=1_A$, so $ A$ is a retract of $B$ in $\TAlg$, hence a continuous $\bbT$-algebra. \end{proof}
	

	Let $\mathscr{A}$ be a subcategory (not necessarily full) of $\QOrd$. If  \begin{itemize}\item for each morphism $f\colon A\lra B$ of $\mathscr{A}$, $\CT f\colon \CT A\lra\CT B$ is a morphism of $\mathscr{A}$; and \item for each object $A$ of $\mathscr{A}$,  the Yoneda embedding $\sy_A$   with codomain restricted to $\CT A$ is a morphism of $\mathscr{A}$,  \end{itemize}
	then we say that (the restriction of)  $\CT$ is \emph{a class of weights on $\mathscr{A}$}.

	Let  $$\TCL$$ be the category   of complete and continuous   $\bbT$-algebras and functors preserving $\CT$-colimits and  arbitrary enriched limits. 
	
	\begin{prop}\label{restriction to QInf}  
		The following  are equivalent:
		\begin{enumerate}[label=\rm(\arabic*)] 
			\item The restriction of $\CT$ is a class of weights on the category of $\bbP^\dag$-algebras.
			
			\item For each $\bbP^\dag$-algebra $A$, $\CT A$ is a $\bbP^\dag$-algebra.
			\item  For each $\bbP^\dag$-algebra $A$,
			the inclusion $\CT A\lra\CP A$ is a right adjoint.
		\end{enumerate}
		
		In this case, $\mathbb{T}=(\CT,{\sf m},\sy)$ is a monad in the category $\PdAlg$ of $\bbP^\dag$-algebras and the forgetful functor $\TCL\lra \PdAlg$ is monadic.
	\end{prop}
	\begin{proof}
		$(1)\Rightarrow(2)$ Obvious.
		
		$(2)\Rightarrow(3)$ By the dual of Corollary \ref{left adjoint via tensor and join}, it suffices to check that $\CT A$ is closed in $\CP A$ with respect to   cotensors and   meets.
		For   $r\in[0,1]$ and $\phi\in\CT A$, let $r\multimap\phi$ be the cotensor of $r$ and $\phi$ in $\CT A$. Then  for all $x\in A$,  \[(r\multimap\phi)(x)=\CT A(\sy (x),r\multimap\phi)=r\ra\CT A(\sy (x),\phi)=r\ra\phi(x),\] thus, $\CT A$ is closed in $\CP A$ with respect to  cotensors.
		Let $\phi$ be the meet of a subset $\{\phi_i\}_{i\in I}$ of $\CT A$. Since $\CT A$ is complete,  by Proposition \ref{property of tensor} for all $x\in A$, \[\phi(x)=\CT A(\sy_A (x),\phi)=\inf_{i\in I}\CT A(\sy_A (x),\phi_i)=\inf_{i\in I}\phi_i(x),\] so $\CT A$ is closed in $\CP A$ with respect to meets.
		
		$(3)\Rightarrow(1)$ First, we show that if $f\colon A\lra B$ is a morphism of $\PdAlg$, then so is $\CT f\colon \CT A\lra\CT B$.    Since the inclusion $\CT A\lra\CP A$ is a right adjoint,  $\CT A$ is a retract of   $\CP A$ in $\QOrd$,  it is a $\bbP$-algebra by Proposition \ref{retract of T-alg}, then a $\bbP^\dag$-algebra. Since $\CT$, being a subfunctor of $\CP$,  preserves adjunctions,  it follows that $\CT f\colon \CT A\lra\CT B$ is a right adjoint, hence a  morphism  of $\PdAlg$.  
		Second, we show that for each  $\bbP^\dag$-algebra $A$, the Yoneda embedding $\sy\colon A\lra\CT A$ is a morphism in $\PdAlg$. This follows from that it is right adjoint to the functor $\colim\colon\CT A\lra   A$. Therefore, the restriction of $\CT$ is a class of weights on  $\PdAlg$.
		
		Finally, we show that   $\bbT=(\CT,{\sf m},\sy)$ is a monad in $\PdAlg$. It suffices to check that for each object $A$ of $\PdAlg$, the component ${\sf m}_A\colon\CT\CT A\lra\CT A$ of the multiplication ${\sf m}\colon\CT^2\lra\CT$ at $A$ is a morphism of  $\PdAlg$.
		Since $\bbT$ is a KZ-monad in $\QOrd$, for each real-enriched category $X$, \[{\sf m}_X={\colim} \colon \CT\CT X\lra\CT X\] is both a left   and a right adjoint.
		In particular,  for each $\bbP^\dag$-algebra $A$, the functor ${\sf m}_A\colon\CT\CT A\lra\CT A$ is a morphism of  $\PdAlg$. 
	\end{proof}
	
	By a \emph{lifting}  of the monad $\bbT=(\CT,{\sf m},\sy)$ along the forgetful functor $$U\colon \PdAlg \lra\QOrd$$ we mean a monad $\widetilde{\bbT}=(\widetilde{\CT}, \widetilde{{\sf m}},\widetilde{\sy})$ in $\PdAlg$ such that \begin{equation*}\label{lifting to QInfII} \bfig
		\Square[\PdAlg`\PdAlg`\QOrd`\QOrd;
		(\widetilde{\CT}, \widetilde{{\sf m}},\widetilde{\sy})`U`U`(\CT,{\sf m},\sy)]
		\efig \end{equation*}  \[U\circ \widetilde{\CT} = \CT\circ U,~~ U\circ\widetilde{{\sf m}}={\sf m}\circ U,~~ U\circ \widetilde{\sy}=\sy\circ U.\]
	
	It is clear that such a lifting exists if and only if \begin{itemize}\item for each morphism $f\colon A\lra B$ of $\PdAlg$, $\CT f\colon \CT A\lra\CT B$ is a morphism of $\PdAlg$; \item for each object $A$ of $\PdAlg$,     $\sy_A\colon A\lra \CT A$ is a morphism of $\PdAlg$.  \end{itemize}   
	Therefore, such a lifting exists if and only if  the restriction of $\CT$ is a class of weights on $\PdAlg$. In this case the lifting is exactly the restriction of $\bbT$ to the subcategory $\PdAlg$. This implies that the lifting of $\bbT$ along the forgetful functor, when exists, is unique.
	
	If $\widetilde{\bbT}=(\widetilde{\CT}, \widetilde{{\sf m}},\widetilde{\sy})$ is a lifting of $\bbT=(\CT,{\sf m},\sy)$ along   $U\colon \PdAlg \lra\QOrd$, the monad in $\QOrd$ defined by the composite   adjunction 
	\begin{equation*}\bfig
		\morphism(700,0)|b|/@{->}@<4pt>/<-700,0>[\PdAlg`(\PdAlg)^{\widetilde{\bbT}};\CT]
		\morphism(0,0)|a|/@{->}@<5.5pt>/<700,0>[(\PdAlg)^{\widetilde{\bbT}}`\PdAlg; U ] \place(360,0)[\rotatebox{90}{$\dashv$}] 
		\morphism(1400,0)|b|/@{->}@<4pt>/<-700,0>[\QOrd`\PdAlg;\CPd]
		\morphism(700,0)|a|/@{->}@<5.5pt>/<700,0>[ \PdAlg`\QOrd; U ] \place(1025,0)[\rotatebox{90}{$\dashv$}]
		\efig\end{equation*}
	is called the composite of $\bbP^\dag$ and $\bbT$, denoted by  $\bbT\circ\bbP^\dag$, where $(\PdAlg)^{\widetilde{\bbT}}$ denotes the Eilenberg-Moore  category of the monad $\widetilde{\bbT}$.  
	
	It is known in category theory, see e.g. \cite[II.3.8]{Monoidal top}, that a lifting of the monad $\bbT=(\CT,{\sf m},\sy)$  along the forgetful functor $U\colon \PdAlg \lra\QOrd$ determines, and is determined by,  a unique distributive law  of the monad $\bbP^\dag=(\CPd,{\sf m}^\dag,\syd)$  over the monad $\mathbb{T}=(\CT,{\sf m},\sy)$. A \emph{distributive law} of  $\bbP^\dag $ over   $\mathbb{T} $ is a natural transformation $\delta\colon\CPd\CT\lra\CT\CPd$   satisfying certain conditions,   details are omitted here.  Since the lifting  of $\bbT $ along the forgetful functor, when exists, is unique, there is at most one distributive law of $\bbP^\dag$ over $\bbT$, it is safe for us to  say  that ``$\bbP^\dag$ distributes over $\bbT$''.  
	
	From Proposition \ref{restriction to QInf} the following conclusion follows immediately.  
	
	\begin{prop}\label{comp mond}  
		The following are equivalent:
		\begin{enumerate}[label=\rm(\arabic*)]  \item The copresheaf monad $\bbP^\dag$ distributes over $\bbT$.  
			\item The monad  $\bbT$ in $\QOrd$ has a  lifting to $\PdAlg$. 
			\item For every $\bbP^\dag$-algebra $A$, $\CT A$ is a complete real-enriched category.
			\item  For every $\bbP^\dag$-algebra $A$,     the inclusion $\CT A\lra\CP A$ is a right adjoint.
			\item The restriction of $\CT$ is a class of weights on  $\PdAlg$.
		\end{enumerate}
	\end{prop}

	\begin{exmp}\label{Pd distributes over P} {\rm (Stubbe \cite{Stubbe2017})} The copresheaf monad $\bbP^\dag=(\CPd,{\sf m}^\dag,\syd)$  distributes over the presheaf monad $\bbP=(\CP,{\sf m},\sy)$.  This follows directly from the above proposition and the fact that for each real-enriched category $X$, the real-enriched category $\CP X$ is  complete.  \end{exmp}
	
	\begin{prop} 
		If the copresheaf monad $\bbP^\dag$ distributes over $\bbT$, then for each separated real-enriched category $A$, the following  are equivalent:\begin{enumerate}[label=\rm(\arabic*)]  \item $A$ is a   complete and continuous  $\bbT$-algebra. \item  $A$ is an algebra of  $\bbT$ when $\bbT$ is viewed as a monad in $\PdAlg$. \item   $A$  is an algebra of the composite monad  $\bbT\circ\bbP^\dag$; that means, there is map $h\colon\CT\CPd A\lra A$ such that $(A,h)$ is an algebra of the monad $\bbT\circ \bbP^\dag$. \end{enumerate}
	\end{prop}
	
	\begin{proof}  We check   $(1)\Leftrightarrow(2)$ first. Let $A$ be a complete and continuous   $\bbT$-algebra.  Then $\CT A$ is an object of $\PdAlg$ by Proposition \ref{comp mond}. So, the string of adjunctions
		\[\wayb\dashv\colim\dashv\sy \colon A\lra\CT A \]
		ensures that $\colim\colon\CT A\lra A$ is a morphism in $\PdAlg$ and  consequently, $A$ is   an algebra  for $\bbT$,  viewed as a monad in $\PdAlg$. Conversely, suppose   $A$ is an algebra for  $\bbT$,  viewed as a monad in $\PdAlg$. Then  $\sy \colon A\lra\CT A$ has a left adjoint $\colim\colon\CT A\lra A$ which is also a morphism in $\PdAlg$. That means, $\colim$ has a left adjoint. Hence $A$ is a continuous $\bbT$-algebra.
		
		The equivalence $(2)\Leftrightarrow(3)$ is a special case of a general result in category theory, see e.g.  \cite[II.3.8.4]{Monoidal top}, which says that the   Eilenberg-Moore category of the lifting $\widetilde{\bbT}$ is isomorphic to that of the composite monad $\bbT\circ\bbP^\dag$.
	\end{proof}

	\begin{cor}\label{Talg monadic over QCat} If the copresheaf monad $\bbP^\dag$ distributes over $\bbT$, then the category $\TCL$ is monadic over $\QOrd$. \end{cor}
	
	\begin{ques}Is the full subcategory of $\TAlg$ composed of complete and continuous $\bbT$-algebras   monoidal closed? \end{ques}
	
	\section{Completely distributive real-enriched categories} \label{cdl}

	\begin{defn}A real-enriched category is said to be completely distributive if it is a continuous  $\mathbb{P}$-algebra. \end{defn}
	
	Completely distributive real-enriched categories are clearly analog of completely distributive lattices in the enriched context.    In Stubbe \cite{Stubbe2007} they are said to be \emph{totally continuous}, and they are a special case of completely (totally) distributive categories in the sense of Marmolejo,   Rosebrugh and  Wood \cite{MRW2012}.
	As a special case of Proposition \ref{TA is T-continuous}, we obtain that for every real-enriched category $X$, the real-enriched category $\CP X$ is  completely distributive. In particular, $\sV=([0,1],\alpha_L)$ is completely distributive.
	
	\begin{prop}
		A   $\mathbb{P}$-algebra  is  completely distributive if and only if it is a retract of some power of $\sV=([0,1],\alpha_L)$ in   $\PAlg$.
	\end{prop}
	
	\begin{proof} For each set $X$, the power  $\sV^X$   in $\PAlg$ is isomorphic to  $\CP X$ when $X$ is viewed as a discrete real-enriched category. Then  sufficiency follows from propositions   \ref{TA is T-continuous} and \ref{retract of continuous T-algebra}. Necessity follows from  that each completely distributive real-enriched  $A$ is a retract of $\CP A$  and that $\CP A$ is a retract of $\CP|A|$, where $|A|$ is the discrete real-enriched category with the same objects as those of $A$. \end{proof}
	
	Proposition \ref{CD is self dual} says that for complete lattices, the notion of complete distributivity is self-dual. But, this is  in general   not true in the enriched context.
	
	\begin{thm}\label{Vop is CD} {\rm (Lai and Shen \cite{LaiS2018})} The following are equivalent: \begin{enumerate}[label=\rm(\arabic*)] 
			\item The continuous t-norm $\&$ is isomorphic to the \L ukasiewicz t-norm. 
			\item  If a real-enriched category $X$ is  completely distributive, then so is $X^{\rm op}$.
			\item The real-enriched category $\sV^{\rm op}$ is  completely distributive.  
	\end{enumerate} \end{thm}
	
	\begin{proof} $(1)\Rightarrow(2)$ First, since $\with$ is isomorphic to the \L ukasiewicz t-norm, it is readily verified that for each real-enriched category $X$, the map $$\neg\colon (\CP X)^{\rm op}\lra\CP (X^{\rm op}), \quad \phi\mapsto \phi\ra0$$ is an isomorphism of real-enriched categories. Next, let $X$ be a completely distributive  real-enriched category. Then the Yoneda embedding $\sy\colon X\lra \CP X$ is a right adjoint  and its left adjoint $\colim\colon \CP X\lra X$ is also a right adjoint. The left adjoints  $$\colim\circ\neg\colon \CP (X^{\rm op})\lra (\CP X)^{\rm op}\lra X^{\rm op}$$ and $$\neg\circ\sy\colon X^{\rm op}\lra(\CP X)^{\rm op}\lra  \CP (X^{\rm op}) $$ exhibit $X^{\rm op}$ as a retract of $\CP(X^{\rm op})$ in $\PAlg$, so by Proposition \ref{TA is T-continuous} and Proposition \ref{retract of continuous T-algebra}, $X^{\rm op}$ is completely distributive.

		$(2)\Rightarrow(3)$ Trivial, since  $\sV=([0,1],\alpha_L)$ is completely distributive.
		
		$(3)\Rightarrow(1)$ Since $\sV^{\rm op}$ is completely distributive, the functor $${\colim}_{\sV^{\rm op}}\colon \CP(\sV^{\rm op})\lra\sV^{\rm op} $$ is a right adjoint. Since $\CP(\sV^{\rm op})=(\CPd\sV)^{\rm op}$, it follows that $${\lim}_\sV\colon \CPd\sV\lra\sV $$ is a left adjoint. So the functor  ${\lim}_\sV\colon \CPd\sV\lra\sV $ preserves tensors. Since the tensor of $r$ with  (the constant coweight) $0_\sV$ in $\CPd \sV$ is given by  $$(r\multimap0_\sV)(x) =r\ra0$$ for all $x\in[0,1]$, then  for all $r\in[0,1]$, 
		\begin{align*}r &=r\with{\lim}_\sV\thinspace0_\sV &({\lim}_\sV\thinspace0_\sV=1)\\ 
			&={\lim}_\sV(r\multimap0_\sV) &(\text{${\lim}_\sV$ preserves tensor})\\ 
			&=   \bw_{x\in [0,1]}((r\ra0)\ra x)\\ 
			&=(r\ra0)\ra0\end{align*} which shows that $\with$ is isomorphic to the \L ukasiewicz t-norm. \end{proof} 
	
	Since  every $\bbP$-algebra is a $\bbT$-algebra, it is natural to ask whether every completely distributive real-enriched category (i.e. every continuous $\bbP$-algebra) is a continuous $\bbT$-algebra. The answer depends on whether the copresheaf monad $\bbP^\dag$  distributes over the monad $\bbT$. 
	
	
	
	\begin{thm}\label{CCD implies bbT} {\rm (Lai and Zhang \cite{LaiZ2020})} The following are equivalent: 
		\begin{enumerate}[label=\rm(\arabic*)] 
			\item Every  completely distributive real-enriched category is a continuous $\bbT$-algebra. 
			\item For every real-enriched category $A$,  $\CP A$ is  a continuous $\bbT$-algebra. 
			\item  The copresheaf monad $\bbP^\dag$ distributes over $\bbT$.  \end{enumerate} 
	\end{thm}

	\begin{proof} 
		$(1)\Rightarrow(2)$ Trivial, since $\CP A$ is completely distributive. 
		
		$(2)\Rightarrow(3)$ By Proposition \ref{comp mond}, it suffices to show that for each $\bbP^\dag$-algebra $A$, the inclusion functor $\mathfrak{i}_A\colon \CT A\lra\CP A$ is a right adjoint. 
		
		Since $\CP A$ is   a continuous $\bbT$-algebra, there is a string of adjunctions \[\wayb\dashv {\colim}  \dashv\sy_{\CP A}\colon \CP A\lra\CT\CP A. \]
		Since   the Yoneda embedding $\sy_A\colon A\lra\CP A$ has a left adjoint ${\colim}_A\colon\CP  A\lra A$ and any 2-functor preserves adjunctions,    $\CT {\colim}_A\colon\CT\CP A\lra\CT A$ is left adjoint to $\CT \sy_A \colon\CT A\lra\CT\CP A$. Thus, \[\CT {\colim}_A\circ \wayb\colon \CP A\lra \CT\CP A\lra \CT A\] is left adjoint to \[\colim\circ \CT \sy_A\colon\CT A\lra \CT\CP A\lra\CP A.\]
		
		Since $\sy_A\colon A\lra\CP A$ is fully faithful, by Proposition \ref{left and right kan extension} and Example \ref{sup in PX} we have $1_{\CP A}= \sy_A^{-1}\circ(\sy_A)_\exists= \colim_{\CP A}\circ\CP\sy_A$.  Since $\CT$ is a subfunctor of $\CP$, the  diagram   \[\bfig\square<600,400>[\CT A`\CT\CP A`\CP A`\CP\CP A;\CT\sy_A`\mathfrak{i}_A`\mathfrak{i}_{\CP A}`\CP\sy_A] \morphism(600,0)|m|<0,-525>[\CP\CP A`\CP A;{\colim}_{\CP A}] \morphism(0,0)|l|<600,-525>[\CP A`\CP A;1_{\CP A}] 
		\morphism(600,400)|r|/{@{->}@/^3em/}/<0,-925>[\CT\CP A`\CP A;\colim]\efig\]
		is commutative.  
		Therefore,   $\mathfrak{i}_A\colon\CT A\lra\CP A$, being equal to $\colim\circ \CT \sy_A$, is a right adjoint of $\CT {\colim}_A \circ \wayb\colon\CP A\lra\CT A$.
		
		$(3)\Rightarrow(1)$ Suppose $A$ is a completely distributive real-enriched category. Since $\bbP^\dag$  distributes over $\CT$, by  Proposition \ref{comp mond} the inclusion $\CT A\lra\CP A$ has a left adjoint. The  composite of the left adjoint of ${\colim}_A\colon\CP  A\lra A$ with the left adjoint of the inclusion $\CT A\lra\CP A$ is then  a left adjoint of $\colim\colon\CT A\lra A$, so $A$ is a continuous $\bbT$-algebra.  \end{proof}
	
	We write $$\QCD$$  for the category  of  completely distributive real-enriched categories  and    functors preserving arbitrary enriched limits and enriched colimits.  The category $\QCD$ is clearly an analog in the enriched context of that of completely distributive lattices and maps  preserving arbitrary meets and joins.

	\begin{prop} {\rm (Stubbe \cite{Stubbe2017})} The category $\QCD$  is monadic over the category $\QOrd$ of real-enriched categories.  
	\end{prop}
	
	\begin{proof} Since the copresheaf monad $\bbP^\dag$  distributes over the presheaf monad $\bbP$ (Example \ref{Pd distributes over P}),   we have a composite monad $\bbP\circ\bbP^\dag$. Since the Eilenberg-Moore category of $\bbP\circ\bbP^\dag$ is the Eilenberg-Moore category of the lifting of $\bbP$ along the forgetful functor $\PdAlg\lra\QOrd$, which is $\QCD$, it follows that $\QCD$  is monadic over $\QOrd$. \end{proof}   
	
	Let $U\colon\QCD\lra{\bf Set}$ be the forgetful funcor, which is equal to the composite of forgetful funcors $\QCD\lra \PdAlg \lra {\bf Set}$. Then $U$ has a left adjoint, say $F\colon{\bf Set}\lra\QCD$. 
	With a slight improvement the argument of Theorem \ref{CL is monadic over Set} in next section  can be applied to show that the forgetful functor $U\colon\QCD\lra{\bf Set}$ is monadic. 
	
	\begin{prop}{\rm (Pu and Zhang \cite{PZ2015})} The category $\QCD$  is monadic over the category of sets. \end{prop}
	
	The monad induced by the adjunction $F\dashv U\colon\QCD\lra{\bf Set}$ turns out to be a submonad of the double  contravariant fuzzy powerset monad $(\mathscr{E},\sfm,{\sf e})$ in the category of sets (see Section \ref{fuzzy powerset monad}). 
	To see this, we present a lemma first.
	
	\begin{lem}\label{submonad of E} Let $$\mathscr{G}\colon {\bf Set} \lra{\bf Set}$$ be a subfunctor of the double  contravariant fuzzy powerset functor $\mathscr{E}$. Then $\mathscr{G}$ is a submonad of of the double  contravariant fuzzy powerset monad $(\mathscr{E},\sfm,{\sf e})$ if, and only if, it satisfies the following conditions: \begin{enumerate}[label=\rm(\roman*)] 
			\item For each set $X$ and each $x\in X$, the function $$ [0,1]^X\lra[0,1],\quad \lam\mapsto \lam(x)$$ belongs to $\mathscr{G}X$. 
			\item For each $\mathcal{H}\colon [0,1]^{\mathscr{G}X}\lra [0,1]$   of $\mathscr{G}^2X$, the function $$  [0,1]^{X}\lra[0,1],\quad \lam\mapsto   \mathcal{H}(\widehat{\lam}|\mathscr{G}X)$$ belongs to $\mathscr{G}X$, where   $\widehat{\lam}\colon\mathscr{E}X\lra [0,1]$  sends  each $\Lambda\in  \mathscr{E}X$ to $ \Lambda(\lam)$.
	\end{enumerate}  \end{lem} 
	
	\begin{proof} Let $\mathfrak{i}\colon \mathscr{G}\lra\mathscr{E}$ be the inclusion natural transformation. Then, condition (i) is equivalent to that the unit ${\sf e}$ of the monad  $(\mathscr{E},\sfm,{\sf e})$ factors through $\mathfrak{i}$; condition (ii) is equivalent to that $\mathscr{G}$ is closed under the multiplication, since for each $\mathcal{H}\colon [0,1]^{\mathscr{G}X}\lra [0,1]$   of $\mathscr{G}^2X$ and each $\lam\in[0,1]^X$,  $$\sfm_X\circ(\mathfrak{i} *\mathfrak{i})_X(\mathcal{H})(\lam) = \mathcal{H}(\widehat{\lam}|\mathscr{G}X).$$ This proves the conclusion.  
	\end{proof}
	
	For each set $X$  let $$\mathscr{C}X$$ be the set of functions $\Lambda\colon [0,1]^X\lra [0,1]$ such that $ \Lambda(\mu)\with\sub_X(\lam,\mu) \leq\Lambda(\lam) $ for all $\lam,\mu\in[0,1]^X$. In other words, $\mathscr{C}X=\CP\CPd X$ with $X$ viewed as a discrete real-enriched category. The assignment $X\mapsto \mathscr{C}X$ defines a subfunctor $$\mathscr{C}\colon {\bf Set} \lra{\bf Set}$$ of the double  contravariant fuzzy powerset functor which satisfies the conditions in Lemma \ref{submonad of E}, so $\mathscr{C}X$ is a submonad of $(\mathscr{E},\sfm,{\sf e})$.  It is readily verified that   $(\mathscr{C},\sfm,{\sf e})$ is the monad defined by the adjunction $F\dashv U\colon\QCD\lra{\bf Set}$.

	
	
	\section{Real-enriched domains}
	
	For each real-enriched category $X$, let $$\CJ X=\{\phi\in\CI X\mid \phi ~\text{has a colimit}\}.$$ Then, every forward Cauchy net of $X$ has a Yoneda limit if and only if $\CJ X=\CI X$.
	
	The Yoneda embedding $\sy\colon X\lra\CP X$ factors through $\CJ X$, so we'll write $\sy\colon X\lra\CJ X$ for  the Yoneda embedding with codomain restricted to $\CJ X$. It is clear that $\sy\colon X\lra\CJ X$ has a left adjoint $\colim\colon \CJ X\lra X$ that sends each $\phi\in\CJ X$ to its colimit.
	
	\begin{defn} (Waszkiewicz \cite{Waskiewicz09}) \label{defn of continuous cats} Suppose $X$ is a real-enriched category. If   $\colim\colon \CJ X\lra X$ has a left adjoint, then we say that $X$  is continuous. In other words, $X$  is continuous if there exists a string of adjunctions  \[\wayb\dashv\colim \dashv \sy\colon X\lra\CJ X.\]    A real-enriched domain is a Yoneda complete and continuous real-enriched category. \end{defn}
	
	A real-enriched category $A$ is a real-enriched domain if and only if it is a continuous $\bbI$-algebra, where   $\bbI=(\CI,{\sf m},\sy)$ is the ideal monad in $\QOrd$; that means, $A$ is a separated real-enriched category  such that there is a string of adjunctions: \[\wayb\dashv\colim \dashv \sy\colon A\lra\CI A.\]   
	
	For each ordered set $(X,\leq)$, it is clear that $\omega(X,\leq)$ is a real-enriched domain if and only if $(X,\leq)$ is a domain. But, the underlying ordered set of a real-enriched domain may fail to be a domain.
	
	\begin{exmp}\label{underlying order of q-domain} (Yu and Zhang \cite{YuZ2022}) 
		In this example we assume that the implication  operator $\ra\colon[ 0,1]^2\lra[0,1]$ is  continuous  at every point off the diagonal. Consider the real-enriched category $X$, where $X=[0,1]\cup\{\infty\}$ and   $$ X(x,y)= \begin{cases}x\ra y & x,y\in [0,1], \\ 1 & x=y=\infty,\\  0 & x\in [0,1], y=\infty, \\  y & x=\infty,y\in[0,1]. \end{cases} $$
		
		Since forward Cauchy nets in $X$ are  forward Cauchy nets in $([0,1],\alpha_L)$ plus constant nets with value $\infty$,   $X$ is Yoneda complete. Since  $\ra\colon[ 0,1]^2\lra[0,1]$ is  continuous  at every point off the diagonal, a routine calculation shows that  \[{\sf d}\colon X\lra \CI X, \quad {\sf d}(x)=\begin{cases}X(-,\infty) & x=\infty,\\    X(-,0) & x=0,\\ \sup\limits_{y<x}X(-,y) & x\in(0,1]\end{cases}\] is left adjoint to $\colim\colon\CI X\lra X$, so $X$ is a real-enriched domain. But  $X_0$, as visualized below, is not a domain. 
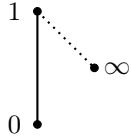
\begin{center}\begin{tikzpicture} \draw[thick] (0,0)--(0,1.5); \draw[thick,dotted] (0,1.5)--(0.75, 0.75); \draw [fill] (0.75, 0.75) circle [radius=0.05]; \draw [fill] (0,0) circle [radius=0.05];\draw [fill] (0,1.5) circle [radius=0.05]; \node at (-0.3,0) {0}; \node at (-0.3,1.5) {1}; \node at (1.05,0.75) {$\infty$};\end{tikzpicture} \end{center}
	\end{exmp}
	
	Let $X$ be  a real-enriched category. The \emph{way below distributor}  of $X$ refers to the distributor $\mathfrak{w}\colon X\oto X$ given by \[\mathfrak{w}(y,x)=\inf_{\phi\in\CJ X}(X(x,\colim\phi)\ra\phi(y)).\]    It is clear that $\mathfrak{w}(y,x)\leq X(y,x)$.
	
	\begin{prop}\label{way below relation} {\rm(Waszkiewicz \cite{Waskiewicz09})} A real-enriched category  $X$ is continuous   if and only if for all $x\in X$, the weight $\mathfrak{w} (-,x)$ belongs to $\CJ X$ and has $x$ as a  colimit. In this case,    $\wayb x=\mathfrak{w} (-,x)$. 
	\end{prop}
	
	\begin{proof} We prove the necessity first. Suppose $X$ is continuous. We show that  for each $x\in X$, $\mathfrak{w} (-,x)=\wayb x$,   hence $\mathfrak{w} (-,x)$ belongs to $\CJ X$ and has $x$ as  colimit.   Since $\colim\wayb x=x$, then $$\mathfrak{w} (-,x)\leq X(x,\colim\wayb x)\ra\wayb x= \wayb x.$$ On the other hand, since  for all   $\phi\in\CJ X$, we have $$X(x,\colim\phi)\with\wayb x = \CJ X(\wayb x,\phi)\with\wayb x\leq \phi,$$  then $\wayb x\leq X(x,\colim\phi)\ra\phi$ for all $\phi\in\CJ X$, hence $\wayb x\leq \mathfrak{w} (-,x)$.  
		
		For sufficiency we show that the assignment $x\mapsto \wayb x\coloneqq \mathfrak{w} (-,x)$ defines a  left  adjoint of $\colim\colon\CJ X\lra X$. Let $x\in X$ and $\phi\in\CJ X$. By definition   we have $\wayb x\leq   X(x,\colim\phi)\ra\phi$, hence $$X(x,\colim\phi)\leq \CJ X(\wayb x,\phi).$$  Since taking colimit is functorial $\CJ X\lra X$ and $\colim\wayb x=x$, then $$\CJ X(\wayb x,\phi)\leq X(\colim\wayb x,\colim\phi)=X(x,\colim\phi).$$ Therefore, $\wayb$ is a  left adjoint  of $\colim\colon\CJ X\lra X$, as desired. \end{proof} 
	
	The following conclusion, contained in Waszkiewicz \cite[Section 3]{Waskiewicz09},  extends the fact that the way below relation of a continuous partially ordered set is interpolative (see Corollary \ref{way below interpolates}).
	
	\begin{prop} The way-below distributor $\mathfrak{w}$ of a continuous real-enriched category   interpolates in the sense that $\mathfrak{w} \circ \mathfrak{w} =\mathfrak{w} $. \end{prop}
	
	\begin{proof}  By Proposition \ref{way below relation}, $\mathfrak{w}(-,x)$ is the least element of $\CJ X$ having $x$ as colimit. Since $\mathfrak{w}(y,x)\leq X(y,x)$ for all $y,x\in X$, it follows that $\mathfrak{w} \circ \mathfrak{w}(-,x)\leq\mathfrak{w}(-,x)$, so  we  only need to show that $\mathfrak{w}\circ\mathfrak{w}(-,x)$ belongs to $\CI X$ and has $x$ as colimit. 
		
		{\bf Step 1}. $\mathfrak{w}\circ\mathfrak{w}(-,x)\in \CI X$. 
		
		Let $\lam,\mu\in\CP X$. Since $$\sub_X(\mathfrak{w}(-,y),\lam)\with X(z,y)\with \mathfrak{w}(-,z) 
		\leq \lam, $$ it follows that $$\sub_X(\mathfrak{w}(-,y),\lam)\with X(z,y)\leq \sub_X(\mathfrak{w}(-,z),\lam), $$ so the assignment $y\mapsto \sub_X(\mathfrak{w}(-,y),\lam)$ defines a weight of $X$. Then,  
		\begin{align*} &\quad\quad ~  \sub_X(\mathfrak{w}\circ\mathfrak{w}(-,x),\lam\vee\mu)\\ 
			& =\bw_{y\in X}\sub_X(\mathfrak{w}(y,x)\with\mathfrak{w}(-,y),\lam\vee \mu)\\ 
			&= \bw_{y\in X}\Big(\mathfrak{w}(y,x)\ra\sub_X(\mathfrak{w}(-,y),\lam\vee \mu)\Big)\\ 
			&=\bw_{y\in X}\Big(\mathfrak{w}(y,x)\ra(\sub_X(\mathfrak{w}(-,y),\lam)\vee\sub_X(\mathfrak{w}(-,y),\mu)\Big)\\ &=\bw_{y\in X}\Big(\mathfrak{w}(y,x)\ra\sub_X(\mathfrak{w}(-,y),\lam)\Big)\vee\bw_{y\in X}\Big(\mathfrak{w}(y,x)\ra\sub_X(\mathfrak{w}(-,y),\mu)\Big)\\ &=\sub_X(\mathfrak{w}\circ\mathfrak{w}(-,x),\lam)\vee\sub_X(\mathfrak{w}\circ\mathfrak{w}(-,x),\mu).
		\end{align*} Likewise, one verifies that for all  $\lam\in\CP X$ and  $r\in[0,1]$, $$\sub_X(\mathfrak{w}\circ\mathfrak{w}(-,x),r\with\lam)=r\with\sub_X(\mathfrak{w}\circ\mathfrak{w}(-,x),\lam).$$  This proves that $\mathfrak{w}\circ\mathfrak{w}(-,x)\in \CI X$. 
		
		{\bf Step 2}. $\colim\mathfrak{w}\circ\mathfrak{w}(-,x)=x$.
		
		For all $b\in X$, \begin{align*} X(x,b) &=  \bw_{y\in X}(\mathfrak{w}(y,x)\ra X(y,b)) &(\colim\mathfrak{w}(-,x)=x) \\ &= \bw_{y\in X}(\mathfrak{w}(y,x)\ra\CP X(\mathfrak{w}(-,y),X(-,b))) &(\colim\mathfrak{w}(-,y)=y)\\ &= \CP X(\mathfrak{w}\circ\mathfrak{w}(-,x),X(-,b)) , \end{align*} hence $\colim\mathfrak{w}\circ\mathfrak{w}(-,x)=x$.
	\end{proof}

	\begin{cor}\label{w(x,-)}  
		Let $X$ be a continuous real-enriched category. Then for each $x\in X$, the functor $$\mathfrak{w}(x,-)\colon X\lra ([0,1],\alpha_L)$$ is Yoneda continuous.
	\end{cor}
	
	\begin{proof} We need to show that $\mathfrak{w}(x,\colim\phi) =\colim_\phi\mathfrak{w}(x,-)$ for each ideal $\phi$ of $X$ that has a colimit.  Since $\mathfrak{w}(x,-)$ is a functor, it suffices to check that $\mathfrak{w}(x,\colim\phi) \leq \colim_\phi\mathfrak{w}(x,-)$. This is easy since \begin{align*} {\colim}_\phi\mathfrak{w}(x,-)
			&= \sup_{y\in X}\phi(y)\with  \mathfrak{w}(x,y)\\ 
			&\geq \sup_{y\in X}\mathfrak{w}(y, \colim\phi)\with  \mathfrak{w}(x,y) \\ 
			&= \mathfrak{w}(x,\colim\phi). \qedhere \end{align*} 
	\end{proof}    
	
	\begin{defn}
		An element $a$ of a real-enriched category $X$  is compact if the functor $$X(a,-)\colon X\lra\sV$$  is Yoneda continuous, where $\sV=([0,1],\alpha_L)$.  \end{defn}

	This is a direct extension of compact elements (also called finite elements) in generalized metric spaces, see e.g. Bonsangue,   van Breugel and Rutten \cite{BBR98}, and Goubaul-Larrecq \cite{Goubault}.
	
	It follows from Proposition \ref{Cauchy weight via sub} that a Cauchy weight  of a real-enriched category $X$ is precisely  an ideal of $X$ that is a compact element of $\CP X$.
	
	\begin{lem}\label{characterizing compact element} For each real-enriched category $X$ and each element $a$ of $X$, the following  are equivalent: \begin{enumerate}[label=\rm(\arabic*)] 
			\item    $a$   is  compact. 
			\item $X(a,\colim\phi) =\phi(a)$ for all  $\phi\in\CJ X$.  
			\item $\mathfrak{w}(-,a)=X(-,a)$. \end{enumerate} \end{lem}
	
	\begin{proof} For all $\phi\in \CJ X$, by Example \ref{cpt as colimit} the colimit of   $X(a,-)\colon X\lra\sV$ weighted by $\phi$ is equal to the composite of $X(a,-)$ and $\phi$; that is,   $\colim_\phi  X(a,-) = \phi\circ X(a,-)=\phi(a)$.  This proves $(1)\Leftrightarrow(2)$.  
		
		$(2)\Rightarrow(3)$ For each $y\in X$, \begin{align*}\mathfrak{w}(y,a)&=\inf_{\phi\in\CJ X}(X(a,\colim\phi)\ra\phi(y))\\ & = \inf_{\phi\in\CJ X}(\phi(a)\ra\phi(y))\\ & = X(y,a).\end{align*}
		
		$(3)\Rightarrow(2)$ Since taking colimit is functorial, then $$\phi(a)=\CP X(X(-,a),\phi)\leq X(a,\colim\phi) $$ for all $\phi\in\CJ X$. On the other hand, since by definition $$\mathfrak{w}(y,a)=\inf_{\phi\in\CJ X}(X(a,\colim\phi)\ra\phi(y)),$$  then $X(a,\colim\phi)\leq \mathfrak{w}(y,a)\ra \phi(y)$ for all $y\in X$ and all $\phi\in\CJ X$,  hence $$X(a,\colim\phi)\leq \CP X(\mathfrak{w}(-,a),\phi)=\CP X(X(-,a),\phi)=\phi(a).$$ Therefore, $X(a,\colim\phi)=\phi(a)$. \end{proof}

	A real-enriched category $X$ is  \emph{algebraic} if each of its elements is a Yoneda limit of a forward Cauchy net consisting of compact elements. Said differently, $X$  is algebraic if   for each $x\in X$, there is an ideal $\phi$ of the real-enriched category $K(X)$ composed of compact elements of $X$ such that $x$ is a colimit of the inclusion functor $i\colon K(X)\lra X$ weighted by $\phi$.  
	\begin{exmp}\label{compact element of V}  The real-enriched category $\sV=([0,1],\alpha_L)$ is algebraic if and only if so is its opposite $\sV^{\rm op}$, if and only if the continuous t-norm $\with$ is Archimedean.   We distinguish three cases. 
		
		Case 1. $\&$ is isomorphic to the \L ukasiewicz t-norm. Then all elements of $\sV$ and $\sV^{\rm op}$ are compact, so both $\sV$ and $\sV^{\rm op}$ are algebraic. 
		
		Case 2. $\&$ is isomorphic to the product t-norm. Then all elements of $\sV$ are compact and $0$ is the only non-compact element of $\sV^{\rm op}$, so, both $\sV$ and $\sV^{\rm op}$ are algebraic. 
		
		Case 3. $\&$ is not Archimedean. Let $b$ be an idempotent element other than $0$ and $1$.  For each $a\in[0,b]$  consider the ideal $\phi=\bv_{x<a}\sV(-,x)$ of $\sV$.  Since $\colim\phi =a$ and  $$\phi(a)=\bv_{x<a}(a\ra x)\leq b<1=\sV(a,\colim\phi),$$ it follows that $a$ is not compact. By help of this fact one readily verifies that $\sV$ is not algebraic. Likewise, one sees that each $a\in[0,b]$ is not compact in $\sV^{\rm op}$, so $\sV^{\rm op}$ is not algebraic either.  \end{exmp}
	
	\begin{prop}\label{algebraic is continuous}  Every  algebraic real-enriched category is continuous.   \end{prop}
	
	\begin{proof}  
		We need to check that the functor $\colim\colon \CJ X  \lra X$   has a left adjoint. For each  $a\in X$, pick a forward Cauchy net $\{a_i\}_{i\in D}$ of compact elements of $X$ with $a$ as a Yoneda limit  and let $$\wayb a=\inf_{i\in D}\sup_{j\geq i}X(-,a_j).$$ Then  $\wayb a$ is an ideal of $X$ with $a$ as a colimit (Lemma \ref{yoneda limit as colimits}), hence it is an element of $\CJ X$. Since $\wayb a$ is a Yoneda limit of the forward Cauchy net $\{X(-,a_i)\}_{i\in D}$ in $\CP X$ (Lemma \ref{Yoneda limits in PX}), it is then a Yoneda limit of   $\{X(-,a_i)\}_{i\in D}$ in $\CJ X$. Thus, for every   $\phi\in \CJ X$, \begin{align*} \CJ X(\wayb a,\phi) &=\inf_{i\in D}\sup_{j\geq i}\CJ X(X(-,a_j), \phi)&(\text{$\wayb  a$ is a Yoneda limit})\\ &=\inf_{i\in D}\sup_{j\geq i}\phi(a_j)&({\rm Yoneda~lemma}) \\ &=\inf_{i\in D}\sup_{j\geq i}X(a_j, \colim\phi)&( a_j{\rm~is~compact}) \\ &=X(a,\colim\phi),&(a~ \text{is a Yoneda limit}) \end{align*} which shows $\wayb$ is a left adjoint of $\colim$. \end{proof}
	
	\begin{prop}A real-enriched category is a real-enriched domain if and only if it is a retract  of an algebraic real-enriched category in the category of $\bbI$-algebras.  \end{prop}
	\begin{proof} This follows from Proposition \ref{retract of continuous T-algebra} and the fact that  for each real-enriched category $X$, the category $\CI X$ of ideals of $X$ is algebraic. \end{proof}
	\begin{prop}For a Smyth completable real-enriched category $X$, the functor $\sy\colon X\lra\CJ X$ is both left and right adjoint to  $\colim\colon \CJ X\lra X$. In particular, $X$ is continuous. \end{prop}
	
	\begin{proof} First we show that every element of $X$ is compact; that is, $X(a,\colim\phi)=\phi(a)$ for all $a\in X$ and all $\phi\in \CJ X$. Since $X$ is Smyth completable, $\phi$ is   Cauchy, hence representable because it has a colimit. This means $\phi=X(-,b)$ for some $b\in X$. Therefore, \begin{align*} X(a,\colim\phi)   &= X(\colim X(-,a),b) \\ &=\CP X(X(-,a), X(-,b))\\ &=\CP X(X(-,a), \phi)\\ &=\phi(a).  \end{align*}
		
		Next we show that   $\sy\colon X\lra\CJ X$   is left adjoint to $\colim\colon \CJ X\lra X$. This is easy, since $ \CJ X(\sy(a),\phi)=\phi(a)= X(a,\colim\phi)$ for all $a\in X$ and $\phi\in \CJ X$. \end{proof}
	
	\begin{cor}Every Smyth completeable real-enriched category is algebraic, hence continuous. \end{cor}
	
	The equivalence between (1) and (3) in the next corollary is due to	Ali-Akbari, Honarii, Pourmahdian and Rezaii \cite{Ali}.
	\begin{cor}For each separated real-enriched category $X$, the following are equivalent:  \begin{enumerate}[label={\rm(\arabic*)}]  \item $X$ is Smyth complete.  \item  The functor $\sy\colon X\to\CI X$  is both  left   and  right adjoint to  the functor $\colim\colon \CI X\to X$. \item  Every forward Cauchy net of $X$ has a Yoneda limit and every element of $X$ is compact.  \end{enumerate} \end{cor}
	
	Theorem \ref{yoneda via formal ball} shows that when the continuous t-norm $\&$ is   Archimedean, Yoneda complete  real-enriched categories can be characterized purely in terms of their set of formal balls. Theorem \ref{continuity via formal ball} below shows that in this case, under a mild assumption, there is a parallel result for continuous real-enriched categories.

	The following notion is a direct extension of that of \emph{standard quasi-metric space} in  Goubault-Larrecq and Ng \cite[Definition 2.1]{Goubault-Ng}.
	\begin{defn}Suppose $\&$ is a continuous Archimedean t-norm. A real-enriched category $X$ is standard provided that for each directed subset $\{(x_i,r_i)\}_{i\in D}$ of $\mathrm{B}X$ with $\bv_{i\in D}r_i>0$,    if   $\{(x_i,r_i)\}_{i\in D}$  has a join, then so do the following subsets: \begin{enumerate}[label={\rm(\roman*)}] 
			\item $\{(x_i,t\with r_i)\}_{i\in D}$ for each $t>0$; and \item $\{(x_i,t\ra r_i)\}_{i\in D}$ for $t=\bv_{i\in D}r_i$.\end{enumerate}
	\end{defn} 
	
	It follows from Lemma \ref{directed set is forward Cauchy} and Lemma \ref{join of directed set in BX} that, when $\&$ is a continuous Archimedean t-norm, every Yoneda complete real-enriched category is standard.   The real-enriched category in Example \ref{second example} is not standard, otherwise, it would be Yoneda complete by Lemma \ref {yoneda complete = dcpo}.
	
	\begin{lem}
		Suppose $\&$ is a continuous Archimedean t-norm; suppose $X$ is a standard  real-enriched category and  $\{(x_i,r_i)\}_{i\in D}$ is a directed subset   of $\mathrm{B}X$ with  $\bv_{i\in D}r_i>0$. If $(x,r)$ is a join of $\{(x_i,r_i)\}_{i\in D}$,  then \begin{enumerate}[label={\rm(\roman*)}] 
			\item $x$ is a Yoneda limit of the forward Cauchy net $\{x_i\}_{i\in D}$ and $r=\bv_{i\in D}r_i$; \item   $(x,1)$ is a join of   $\{(x_i,r\ra r_i)\}_{i\in D}$; \item for each $t\in[0,1]$, $(x,t\with r)$ is a join of   $\{(x_i,t\with r_i)\}_{i\in D}$. 
	\end{enumerate}   \end{lem}
	
	\begin{proof}It suffices to show that $x$ is a Yoneda limit of the forward Cauchy net $\{x_i\}_{i\in D}$, the other conclusions will follow by Lemma \ref{join of directed set in BX}.
		Let $t=\bv_{i\in D}r_i$. Since $\&$ is Archimedean, then $\bv_{i\in D}(t\ra r_i)=1$. Since $X$ is standard, the directed set $\{(x_i,t\ra r_i)\}_{i\in D}$ has a join, say $(b,1)$. By Lemma \ref {yoneda complete = dcpo}, $b$ is a Yoneda limit of the forward Cauchy net $\{x_i\}_{i\in D}$, then $(b, t)$ is a join of $\{(x_i,r_i)\}_{i\in D}$ by Lemma \ref{join of directed set in BX}. Therefore, $r=t$ and $x$ is isomorphic to $b$, so $x$ is a Yoneda limit of  $\{x_i\}_{i\in D}$. \end{proof} 
	
	The following theorem   is a combination and a slight improvement of related results in  Goubault-Larrecq and Ng \cite[Section 3]{Goubault-Ng} and Kostanek and Waszkiewicz \cite[Section 9]{KW2011}. It implies in particular that for standard quasi-metric spaces, continuity in the sense of Definition \ref{defn of continuous cats} coincides with that in the sense of Goubault-Larrecq and Ng \cite[Definition 3.10]{Goubault-Ng}. 
	
	\begin{thm}\label{continuity via formal ball} Suppose $\&$ is a continuous Archimedean t-norm and $X$ is a standard real-enriched category. The following are equivalent:  \begin{enumerate}[label={\rm(\arabic*)}] 
			\item $X$ is continuous. \item  The ordered set $\mathrm{B} X$ is continuous. 
		\end{enumerate} In this case, the way below distributor $\mathfrak{w}$ of $X$ and the way below relation $\ll$ of $\mathrm{B}X$ determine each other via $$(x,r)\ll(y,s)\iff  r<s\with\mathfrak{w}(x,y)$$ for all $x,y\in X$ and $r,s>0$. \end{thm}
	
	\begin{proof} 
		$(1)\Rightarrow(2)$ First we present a fact that will be used  in Step 1 below.
		Suppose $X$ is a real-enriched category and $\{(x_i,r_i)\}_{i\in D}$ is  a directed subset of $\mathrm{B}X$ with $\bv_{i\in D}r_i=1$.
		Let   $\phi_D=\bv_{i\in D}\bw_{j\geq i} X(-,x_j)$ and  let $\mathrm{B}\phi_D =\{(x,r)\mid \phi_D(x)>r\}$. By  Proposition \ref{characterization of  ideal}, $\mathrm{B}\phi_D$ is a directed subset of $\mathrm{B}X$. We assert that for each $(x,r)\in \mathrm{B}\phi_D$, there exists $j\in D$ such that $(x,r)\sqsubseteq(x_j,r_j)$.   To see this, suppose $(x,r)\in \mathrm{B}\phi_D$. Since $\phi_D(x)=\bv_{i\in D}\bw_{j\geq i} X(x,x_j)$, there is $i\in D$ such that $r<\bw_{j\geq i} X(x,x_j)$. Since $\bv_{j\geq i}r_j=1$, there is   $j\geq i$ such that $r<r_j\with X(x,x_j)$, so $(x,r)\sqsubseteq(x_j,r_j)$. 
		
		Now we turn to the proof of the conclusion. Since $X$ is continuous, the functor $\colim\colon \CJ X\lra X$ has a left adjoint $\wayb\colon X\lra\CJ X$. For each $y\in X$, let $$d(y)=\{(x,r)\in\mathrm{B}X\mid r<\mathfrak{w}(x,y)\}.$$ Then $d(y)$ is a directed subset of $\mathrm{B}X$ with $(y,1)$ as a join.  In the following we prove in two steps that  every element of $\mathrm{B}X$ is the join of the elements way below it, so $\mathrm{B}X$ is continuous.
		
		{\bf Step 1}. We show that  $(x,r)\ll(y,1)$ for all $(x,r)\in d(y)$, hence $(y,1)$ is the join of the elements way below it. 
		
		Let $(x,r)\in d(y)$. Suppose $D=\{(z_i,t_i)\}_{i\in D}$ is a directed subset of $\mathrm{B}X$ with a join, say $(z,t)$, larger than $(y,1)$. It is clear that $t=1$ and $X(y,z)=1$. Since $X$ is standard,  $\bv_{i\in D}t_i=1$ and $z$ is a Yoneda limit of the forward Cauchy net $\{z_i\}_{i\in D}$.    So, $z$ is a colimit of the ideal $\phi_D$ generated by   $\{z_i\}_{i\in D}$. Since $\wayb$ is left adjoint to $\colim$ we have $$\CJ X(\wayb z,\phi_D)=X(z,\colim\phi_D)=X(z,z)=1, $$ then $\wayb z\leq \phi_D$ and   $d(z)\subseteq \mathrm{B}\phi_D$. Since $X(y,z)=1$ and $\wayb$ is a functor, then $\wayb y\leq\wayb z$. Therefore, $d(y)\subseteq d(z)\subseteq \mathrm{B}\phi_D$ and consequently,   $(x,r)\sqsubseteq(z_j,t_j)$ for some $j\in D$. This proves that $(x,r)\ll(y,1)$.
		
		{\bf Step 2}. We show that $(x,r)\ll(y,s)$ whenever $r<s\with\mathfrak{w}(x,y)$, in particular $(x,s\with r)\ll (y,s)$ for all $(x,r)\in d(y)$,   hence  by Step 1, $(y,s)$ is the join of the elements way below it.  
		
		Suppose $D=\{(z_i,t_i)\}_{i\in D}$ is a directed subset of $\mathrm{B}X$ with a join, say $(z,t)$, larger than $(y,s)$. We wish to show that $(x,r)\sqsubseteq(z_j,t_j)$ for some $j\in D$.  Since $X$ is standard, the monotone net $\{t_i\}_{i\in D}$ converges to $t$. Since $s\leq t\with X(y,z)$ and $r<s\with\mathfrak{w}(x,y)$, it follows that $$  r< t\with X(y,z) \with\mathfrak{w}(x,y)\leq t\with \mathfrak{w}(x,z). $$ Let $\phi= \bv_{i\in D}\bw_{j\geq i} X(-,z_j)$. Since $\phi$ is an ideal of $X$ with $z$ as a colimit,   it follows that $\mathfrak{w}(-,z)\leq \phi$, then   \begin{align*}  r &< t\with \mathfrak{w}(x,z)\\ 
			& = t\with \bv_{w\in X}(\mathfrak{w}(w,z)\with \mathfrak{w}(x,w))\\ 
			& \leq t\with \bv_{w\in X}(\phi(w)\with \mathfrak{w}(x,w)) \\ 
			&= t\with \bv_{w\in X}\Big[\Big(\bv_{i\in D}\bw_{j\geq i} X(w,z_j)\Big)\with \mathfrak{w}(x,w)\Big] \\ 
			&\leq t\with \bv_{i\in D}\Big(\bv_{w\in X}\bw_{j\geq i} X(w,z_j) \with \mathfrak{w}(x,w)\Big) \\ 
			&\leq \Big(\bv_{i\in D}t_i\Big)\with \Big(\bv_{i\in D}\bw_{j\geq i}   \mathfrak{w}(x,z_j)\Big).\end{align*} Therefore, there is some $j\in D$ such that $$r< t_j\with\mathfrak{w}(x,z_j)\leq t_j\with X(x,z_j),$$  hence   $(x,r)\sqsubseteq(z_j,t_j)$.
		
		$(2)\Rightarrow(1)$  First we present a fact that will be used  in Step 3 below.
		Suppose   $X$ is a real-enriched category, $\phi\in\CJ X$,     $(x,r)\in\mathrm{B}X$. We assert that if $(x,r)\ll(\colim\phi,1)$, then $ \phi(x)>r$.  To see this, index the directed subset $\{(y,s)\mid \phi(y)>s\}$ of $\mathrm{B}X$ as $\{(y_i,s_i)\}_{i\in D}$. Then $\bv_{i\in D}s_i=1$ and $\{y_i\}_{i\in D}$ is a forward Cauchy net with $\colim\phi$ as a Yoneda limit. By Lemma \ref{join of directed set in BX} $(\colim\phi,1)$ is a join of $\{(y_i,s_i)\}_{i\in D}$, then $(x,r)\sqsubseteq(y_i,s_i)$ for some $i\in D$, hence $$r\leq s_i\with X(x,y_i)<\phi(y_i)\with X(x,y_i)\leq \phi(x),$$ showing that $\phi(x)>r$.  
		
		Now we turn to the proof of the conclusion.
		For each $y\in X$, let $$D=\{(x,r)\mid (x,r)\ll(y,1)\}.$$ Then $D$ is a directed set of $\mathrm{B}X$ with $(y,1)$ as a join. Index $D$ by itself, i.e. $D=\{(x_i,r_i)\}_{i\in D}$. Then $\bv_{i\in D}r_i=1$ and $\{x_i\}_{i\in D}$ is a forward Cauchy net. Let $$k(y)=\bv_{i\in D}\bw_{j\geq i}X(-,x_j).$$ Then $k(y)$ is an ideal of $X$ with $y$ as a colimit. We show that the assignment $y\mapsto k(y)$ defines a left adjoint of $\colim\colon \CJ X\lra X$, so $X$ is continuous. We do this in three steps.
		
		{\bf Step 1}. $k(y)=\bv_{i\in D}r_i\with X(-,x_i)$. This implies that, if $k(y)(x)>r$ then $(x,r)\ll(y,1)$. 
		
		For any $(x_i,r_i)\sqsubseteq(x_j,r_j)$ we have $r_i \leq X(x_i,x_j)$, then $r_i\with X(-,x_i)\leq X(-,x_j)$. Therefore, $$r_i\with X(-,x_i)\leq \bw_{j\geq i}X(-, x_j),$$ hence  $$k(y)\geq \bv_{i\in D}r_i\with X(-,x_i).$$ To see the converse inequality, let $t<1$. Since $r_i$ tends to $1$, there exists  $i_0\in D$ such that $r_j\geq t$ whenever $j\geq i_0$. Then $$t\with k(y)\leq \bv_{i\geq i_0}\bw_{j\geq i}r_j\with X(-,x_j) \leq \bv_{i\in D}r_i\with X(-,x_i),$$ hence $k(y)\leq \bv_{i\in D}r_i\with X(-,x_i)$ by arbitrariness of $t$.
		
		{\bf Step 2}.  If $(x,r)\ll(y,1)$, then $(x,t\with r)\ll(y,t)$ for all $t\in[0,1]$. 
		
		If $t\with r=0$ there is nothing to prove, so in the following we assume that $t\with r>0$. Index the directed set $\{(z,s)\mid (z,s)\ll(y,t)\}$ as $\{(z_i,s_i)\}_{i\in D}$. Since $\mathrm{B}X$ is continuous,  $(y,t)$ is a join of $\{(z_i,s_i)\}_{i\in D}$, then $(y,1)$ is a join of the directed set $\{(z_i,t\ra s_i)\}_{i\in D}$, so there exists   $i\in D$ such that $(x,r)\leq (z_i,t\ra s_i)$, hence  $$(x,t\with r)\leq(z_i,s_i)\ll (y,t).$$ 
		
		{\bf Step 3}. $k\colon X\lra \CJ X$ is left adjoint to $\colim\colon\CJ X\lra X$.

		First we check that $k$ is a functor. Suppose $y_1,y_2\in X$ and $t=X(y_1,y_2)$. If $(x,r)\ll (y_1,1)$, then $(x,t\with r)\ll(y_1,t)$ by Step 2. Since $(y_1,t)\sqsubseteq(y_2,1)$, it follows that $(x,t\with r)\ll(y_2,1)$,   then  
		\begin{align*} t\with k(y_1)
			&= \bv\{t\with r\with X(-,x)\mid (x,r)\ll(y_1,1)\}\\ &\leq \bv\{  r\with X(-,x)\mid (x,r)\ll(y_2,1)\} \\ &\leq k(y_2),\end{align*} so $k$ is a functor. By Theorem \ref{Characterization of adjoints}
		it remains to check that for all $y\in X$ and $\phi\in\CJ X$ it holds that $$k(y)\leq\phi\iff y\sqsubseteq\colim\phi.$$ The direction $\implies$ is obvious since $\colim$ is functorial and $\colim k(y)=y$. For the other direction it suffices to check that $k(\colim\phi)\leq\phi$. This follows from that $\phi(x)>r$ whenever $(x,r)\ll(\colim\phi,1)$. \end{proof} 
	
	\begin{cor}{\rm(Goubault-Larrecq and Ng \cite[Theorem 3.7]{Goubault-Ng})} \label{domain via formal ball} Suppose $\&$ is a continuous Archimedean t-norm. Then a Yoneda complete real-enriched category $X$   is a real-enriched domain if and only if $\mathrm{B} X$ is directed complete and continuous.   \end{cor}
	
	The following corollary extends  the characterization of Smyth complete quasi-metric spaces in Romaguera and Valero \cite[Theorem 3.1]{RV2010}.
	
	\begin{cor}\label{Smyth via formal ball}  Suppose $\&$ is a continuous Archimedean t-norm. Then for each separated and standard real-enriched category $X$, the following are equivalent:  \begin{enumerate}[label={\rm(\arabic*)}] 
			\item $X$ is Smyth complete.  \item  $\mathrm{B} X$ is directed complete, continuous, and    $$(x,r)\ll(y,s)\iff   r< s\with X(x,y)$$  for all $x,y\in X$ and $r,s>0$.
			\item  $\mathrm{B} X$ is directed complete, continuous, and   $$(x,r)\ll(y,1)\iff   r< X(x,y)$$ for all $x,y\in X$ and $r>0$.
	\end{enumerate}    \end{cor}

	\begin{proof} $(1)\Rightarrow(2)$ Since $X$ is Smyth comple,  it is a real-enriched domain with way below distributor given by $\mathfrak{w}(x,y)=X(x,y)$. Then, from Theorem \ref{yoneda via formal ball} and Theorem \ref{continuity via formal ball} it follows that $\mathrm{B}X$ is directed complete and continuous, and that  for all $x,y\in X$ and $r,s>0$, $$(x,r)\ll(y,s)\iff   r< s\with X(x,y).$$ 
		
		$(2)\Rightarrow(3)$ Trivial. 
		
		$(3)\Rightarrow(1)$ By Theorem \ref{yoneda via formal ball} and Theorem \ref{continuity via formal ball}, $X$ is a real-enriched domain with way below distributor given by $\mathfrak{w}(x,y)=X(x,y)$, then every element of $X$ is compact and $X$ is Smyth complete. \end{proof} 
	
	Consider the real-enriched category $X$ in Example \ref{Example-G}. Since every forward Cauchy net of $X$ is eventually constant, it follows that $X$ is Smyth complete. But, the ordered set of formal balls of $X$ fails to be directed complete. This shows that in corollaries \ref{domain via formal ball} and \ref{Smyth via formal ball}  the assumption  that $\&$ is Archimedean is indispensable. 
	
	\vskip 5pt \noindent{\bf Real-enriched continuous lattices}
	
	A complete (hence cocomplete)  real-enriched domain is called a \emph{real-enriched continuous lattice}.
	
	Before proceeding with our investigation of real-enriched continuous lattices, we present a property of ideals of separated and complete real-enriched categories.  
	
	Let $X$ be a real-enriched category. For  each ideal $I$ of the ordered set $X_0$, the conical weight  \[\Lambda(I)\coloneqq\sup_{x\in I}X(-,x)\]   is clearly an ideal of $X$. Conversely, if $X$ is finitely complete and finitely cocomplete,   with help of Proposition \ref{chara of cotensored}   one verifies that for each ideal $\phi$ of $X$, the set \[\chi(\phi)\coloneqq\{x\in  X\mid\phi(x)=1\}\] is an ideal of the ordered set $X_0$.

	\begin{lem}\label{ideals and directed sets}
		If $X$ is a  finitely complete and finitely cocomplete real-enriched category and  $\idl X_0$ is the set of  ideals  of the ordered set $X_0$, then $\Lambda\colon \idl X_0 \lra(\CI X)_0$ is  left adjoint to $\chi\colon (\CI X)_0\lra \idl X_0 $.   \end{lem}
	\begin{proof}   Suppose   $I$ is an ideal of the ordered set $X_0$ and 
		$\phi$ is an ideal of $X$. Then
		\[I\subseteq\chi(\phi)\iff\forall x\in I,\phi(x)= 1 \iff \Lambda(I)\leq\phi,\]
		so $\Lambda\dashv \chi$. \end{proof}

	\begin{prop}\label{ideals and directed sets II} Let $X$ be a separated and complete real-enriched category. Then for each weight $\phi$ of $X$, the following are equivalent: \begin{enumerate}[label=\rm(\arabic*)]  \item $\phi$ is an ideal of $X$. \item  $\phi=\Lambda(I)$ for some ideal $I$ of the complete lattice $X_0$.   
		\end{enumerate} In this case, the colimit of $\phi$ is a join of $\chi(\phi)$ in the complete lattice $X_0$. 
	\end{prop}
	\begin{proof}
		$(1)\Rightarrow(2)$  We show that the ideal $\chi(\phi)$ satisfies the requirement; that is, $\Lambda\circ\chi(\phi)=\phi$.  Since $\Lambda$ is left adjoint to $\chi$, then $\Lambda\circ\chi(\phi)\leq\phi$. For the converse inequality, pick a forward Cauchy net $\{x_i\}_{i\in D}$ of $X$ such that \[\phi(x)=\sup_{i\in D}\inf_{j\geq i} X(x,x_j).\]  Let \[I  =\Big\{x\in X\mid   x\sqsubseteq \inf_{j\geq i}x_j~\text{for some}~ i\in D\Big\},\] where $\inf_{j\geq i}x_j$ denotes the meet in   $X_0$.
		Then, $I$ is an ideal of the ordered set $X_0$ and \[\phi =\sup_{x\in I} X(-,x),\] so $I\subseteq \chi(\phi)$  and consequently,   $\phi\leq\Lambda\circ\chi(\phi)$.
		
		$(2)\Rightarrow(1)$ Trivial, since every ideal  of the ordered set $X_0$ may be viewed as a forward Cauchy net of $X$.

		Finally, we check that the colimit of $ \phi$ is the join   $\sup\chi(\phi)$   of $\chi(\phi)$ in the complete lattice $X_0$. Let $I$ be the ideal  of $X_0$ given in   $(1)\Rightarrow(2)$. Since $I$ is contained in $\chi(\phi)$ and $\colim \phi=\sup_{z\in X}\phi(z)\otimes z$,  it follows that $\colim \phi\sqsupseteq\sup\chi(\phi)$. The converse inequality follows from that \begin{align*}\colim\phi &= \colim \sup_{x\in I} X(-,x)= \sup I\sqsubseteq\sup\chi(\phi). \qedhere \end{align*}
	\end{proof}
	
	\begin{cor}\label{Yoneda continuity = Scott continuity} A functor $f\colon  X\lra Y$  between complete real-enriched categories  is Yoneda continuous if and only if \(f\colon  X_0\lra Y_0\)   is Scott continuous. \end{cor}
	
	As another consequence of  Proposition \ref{ideals and directed sets II} we present a  useful characterization of ideals of a separated and complete real-enriched category.  
	
	\begin{prop}\label{ideal of complete A} Suppose $X$ is a separated and complete real-enriched category. Then, a weight $\phi$ of  $X$ is an ideal if and only if  \begin{enumerate}[label=\rm(\roman*)]  \item $\phi(x)>0$ for some $x\in X$. \item $\phi(x\vee y)=\phi(x)\wedge\phi(y)$ for all $x,y\in X$, where $x\vee y$ denotes the join in $X_0$. \item   $\phi(r\otimes x)=1$ whenever $r<\phi(x)$. \end{enumerate} \end{prop}
	
	\begin{proof} We prove  the necessity first. Since $\phi$ is an ideal, it is inhabited, hence there is some $x$ for which $\phi(x)>0$.   In the following we use Proposition \ref{ideals and directed sets II} to check that $\phi$ satisfies both (ii) and (iii). Pick an ideal $I$ of $X_0$ such that $\phi=\Lambda(I)$. We index the elements of $I$ by itself; that means, we write an element of $I$ as $x_i$, with $i\leq j$ if   $x_i\sqsubseteq x_j$ in $X_0$. Then  \begin{align*}
			\phi(x\vee y) & =\sup_{i\in I}\inf_{j\geq i}X(x\vee y,x_j)\\
			& = \sup_{i\in I}\inf_{j\geq i} X(x,x_j)\wedge X(y,x_j) \\
			& = \sup_{i\in I} \Big(\inf_{j\geq i} X(x,x_j)\wedge \inf_{j\geq i}X(y,x_j)\Big)\\
			& = \Big(\sup_{i\in I}\inf_{j\geq i}X(x,x_j)\Big)\wedge \Big(\sup_{i\in I}\inf_{j\geq i}X(y,x_j)\Big)  \\
			& =  \phi(x)\wedge\phi(y),
		\end{align*} which shows that $\phi$ satisfies (ii). For (iii), assume that $\phi(x)>r$. Then there exists  $i_0\in I$ such that $\inf_{j\geq i_0}X(x,x_j)\geq r$, hence \begin{align*}
			\phi(r\otimes x) & =\sup_{i\in I}\inf_{j\geq i}X(r\otimes x,x_j)\\ & = \sup_{i\in I}\inf_{j\geq i}(r\ra X(x,x_j)) \\ &\geq r\ra \inf_{j\geq i_0}X(x,x_j) \\ &=1.\end{align*} 
		
		Now we prove the sufficiency. Let
		$ I=\{x\in X\mid\phi(x)=1\}.$ 
		By (i) and (iii) one sees that $I$ is not empty. By (ii) one sees that $I$ is directed. So $I$ is an ideal of $X_0$. It is clear that $\Lambda(I)\leq\phi$. To see that $\phi$ is an idea, it suffices to check that $\phi(x)\leq\Lambda (I)(x)$ for all $x\in X$. If $\phi(x)>r$,  then $\phi(r\otimes x)=1$ by (iii), hence  $r\otimes x\in I$  and consequently, $\Lambda(I)(x)\geq X(x,r\otimes x)\geq r$. Then $\phi(x)\leq\Lambda (I)(x)$ by arbitrariness of $r$. \end{proof}
	
	\begin{cor}\label{CA closed under meets}
		For each separated and complete real-enriched category $X$, the set   $\CI X$ of ideals of $X$  is closed in $[0,1]^X$ (ordered pointwise)  under   meets and directed joins.  \end{cor}
	
	\begin{proof} Let $\{\phi_i\}_{i\in J}$ be a subset of $\CI X$. The meet $\bw_{i\in J}\phi_i$ clearly satisfies (i) and (ii) in Proposition \ref{ideal of complete A}. It remains to check it also satisfies the condition (iii).  If $r<\bw_{i\in J}\phi_i(x)$, then for each $i\in J$, $r<\phi_i(x)$, hence $\phi_i(r\otimes x)=1$ and consequently, $\bw_{i\in J}\phi_i(r\otimes x)=1$. Likewise,  $\CI X$ is closed in $[0,1]^X$  under directed joins. \end{proof}
	
	\begin{thm}  \label{chacl} {\rm (Lai and Zhang \cite{LaiZ2020})}
		For each separated and complete real-enriched category $X$, the following  are equivalent: \begin{enumerate}[label=\rm(\arabic*)] 
			\item $X$ is a real-enriched continuous lattice.
			\item    $X_0$ is a continuous  lattice and for each $x\in X$ and each ideal $\phi$ of $X$,
			$$ X(x,\colim\phi)=\inf_{y\ll x}\phi(y),$$
			where $\ll$ denotes the way below relation in $X_0$.
			\item $X_0$ is a continuous  lattice and the map \[{\sf d}\colon X\lra \mathcal{I}X,\quad {\sf d}(x)=\sup_{y\ll x}X(-,y)\] is a functor, 	where $\ll$ denotes the way below relation in $X_0$. \end{enumerate}	
	\end{thm}
	
	\begin{proof} 
		$(1)\Rightarrow(2)$  Since $\CI X$ is closed in $[0,1]^X$  under  meets and directed joins, then $(\CI X)_0$ is a continuous lattice since so is $[0,1]$. The adjunction $\wayb \dashv \colim \colon\CI X\lra X$ ensures that   $X_0$ is a retract of $(\CI X)_0$ in the category of dcpos,  so $X_0$ is a continuous lattice. 
		To see that 
		$$ X(x,\colim\phi)=\inf_{y\ll x}\phi(y)$$ for all $x\in X$ and all $\phi\in\CI X$, it suffices to show that the left adjoint $\wayb\colon A\lra\CI A$ of $\colim \colon\CI A\lra A$ is given by $\wayb x=\sup_{y\ll x}\sy(y)$, because $$\CI X\Big(\sup_{y\ll x}\sy(y), \phi\Big)=\sub_X\Big(\sup_{y\ll x}\sy(y), \phi\Big)= \inf_{y\ll x}\phi(y). $$
		
		Since $x$ is a colimit of the ideal $\sup_{y\ll x}\sy(y)$, then $\wayb x\leq \sup_{y\ll x}\sy(y)$.  Since $\sup\chi(\wayb x)=\colim \wayb x=x$ by Proposition \ref{ideals and directed sets II}, then every element   way below $x$ in $X_0$ belongs to $\chi(\wayb x)$, so $\sup_{y\ll x}\sy(y)\leq \wayb x$.
		
		$(2)\Rightarrow(3)$ Obvious.
		
		$(3)\Rightarrow(1)$    
		By virtue of Theorem \ref{Characterization of adjoints} we only need to show that ${\sf d}\colon X_0\lra (\mathcal{I}X)_0$ is left adjoint to $\colim\colon  (\mathcal{I}X)_0\lra X_0$. This follows directly from Proposition \ref{ideals and directed sets II}, which guarantees that for each $x\in X$, $\sup_{y\ll x}X(-,y)$ is the smallest ideal  of $X$ (w.r.t. underlying order) that has $x$ as a colimit. 
	\end{proof}
	\begin{prop}\label{way-product}
		Suppose $X$ is a real-enriched continuous  lattice. Then for all $x,y\in X$,
		$$\mathfrak{w}(y,x)=\sup_{z\ll x}X(y,z)=\sup_{z\ll x}\mathfrak{w}(y,z),$$ where $\ll$ refers to the way below relation in the complete lattice $X_0$.
	\end{prop}
	
	\begin{proof} The first equality follows from the argument of Theorem \ref{chacl}. The second equality follows from  the first one and the interpolative property of the way below relation in $X_0$.
	\end{proof}
	
	\begin{cor}\label{cond impl cont} Suppose $X$ is a  separated and complete real-enriched category. If  $X_0$ is a continuous lattice  and  for all $p\in [0,1]$, the map \[X_0\lra X_0, \quad x\mapsto p\multimap x\] is Scott continuous, then $X$ is a real-enriched continuous lattice.
	\end{cor}
	
	\begin{proof} We show that for each $x\in X$ and each ideal $\phi$ of $X$, \[X(x, \colim \phi)=\inf_{y\ll x}\phi(y),\] where $\ll$ denotes the way below relation in $X_0$.
		
		On the one hand, since $\{y\in X\mid y\ll x\}$ is a directed set of $X_0$ with join $x$,  the conical weight $\sup_{y\ll x}X(-,y)$ is an ideal of $X$ with $x$ as colimit, then \[X(x, \colim \phi)\geq \CI X\Big(\sup_{y\ll x}\sy(y),\phi\Big)=\inf_{y\ll x}\phi(y).\]
		
		On the other hand,   for all $p\in[0,1]$,
		\begin{align*}
			p\leq X(x, \colim \phi)\implies& x\leq p\multimap \colim  \phi\\
			\implies & x\leq\sup_{d\in \chi(\phi)}(p\multimap d)  
			\\
			\implies & \forall y\ll x, ~ y\leq p\multimap d ~\text{for some}~  d\in \chi(\phi) 
			\\
			\implies & \forall y\ll x, ~p\leq X(y,d)~\text{for some}~  d\in \chi(\phi), \\
			\implies & \forall y\ll x, p\leq\sup_{d\in \chi(\phi)} X(y,d)\\
			\implies  & \forall {y\ll x}, p\leq \phi(y)\\
			\implies & p\leq\inf_{y\ll x}\phi(y).
		\end{align*}
		Therefore,  $X(x,\colim \phi)\leq\inf_{y\ll x}\phi(y)$.
	\end{proof}

	\begin{exmp}  \label{dr is continuous} For each continuous t-norm $\with$,  $\sV^{\rm op}=([0,1],\alpha_R)$ is a real-enriched continuous  lattice. We use Corollary \ref{cond impl cont} to verify the claim. First, since the underlying order of $\sV^{\rm op}$ is the opposite of the usual order  between real numbers,   $(\sV^{\rm op})_0$ is a continuous lattice. Second, since the cotensor $p\multimap x$ of $p$ with $x$ in $\sV^{\rm op}$ is given by $p\with x$, the map $p\multimap-\colon (\sV^{\rm op})_0\lra (\sV^{\rm op})_0$ is Scott continuous.   
	\end{exmp}
	
	\begin{thm}\label{CD implies CL}   {\rm (Lai and Zhang \cite{LaiZ2020})}
		The following are equivalent:
		\begin{enumerate}[label=\rm(\arabic*)] 
			\item Every real-enriched completely distributive lattice is a real-enriched domain.
			\item The real-enriched category $\sV=([0,1],\alpha_L)$ is a real-enriched domain.
			
			\item The implication operator $\ra\colon[ 0,1]^2\lra[0,1]$ is  continuous  at every point off the diagonal.  
			\item For each $\bbP^\dag$-algebra $X$, the inclusion $\CI X\lra\CP X$ is a right adjoint.
			\item  The copresheaf monad $\bbP^\dag=(\CPd,{\sf m}^\dag,\syd)$ distributes over the ideal monad $\bbI=(\CI,{\sf m},\sy)$.
		\end{enumerate}
	\end{thm}
	
	\begin{proof}
		$(1)\Rightarrow(2)$  Obvious.
		
		$(2)\Rightarrow(3)$ By Theorem \ref{chacl}, it suffices to show that if   \[{\sf d}\colon \sV\lra \CI\sV,\quad   {\sf d}(x)(t) =\begin{cases}t\ra 0 &x=0,\\ \sup\limits_{y<x}(t\ra y) &x>0 \end{cases} \]
		is a functor, then the implication operator $\ra\colon[ 0,1]^2\lra[0,1]$ is  continuous at every point off the diagonal. 
		
		Suppose on the contrary that $\ra\colon[ 0,1]^2\lra[0,1]$ is not continuous at some point off the diagonal. Then there are idempotent elements $p,q>0$  such that the restriction of $\&$ on $[p,q]$   is isomorphic to the \L ukasiewicz t-norm. Pick $x\in(p,q)$. Then  
		\begin{align*}\CI\sV({\sf d}(x), {\sf d}(p))&= \inf_{t\in[0,1]}\Big(\sup_{y<x}(t\ra y)\rightarrow\sup_{z<p}(t\rightarrow z)\Big)\\ &\leq \sup_{y<x}(x\rightarrow y)\rightarrow\sup_{z<p}(x\rightarrow z)\\ &=q\rightarrow p   \\ & <x\rightarrow p,\end{align*} which shows that ${\sf d}\colon \sV\lra \CI\sV$ is not a functor, a contradiction. 
		
		$(3)\Rightarrow(4)$  It suffices to show that for each $\bbP^\dag$-algebra $X$,  $\CI X$ is closed in $\CP X$ under meets and cotensors. That $\CI X$ is closed in $\CP X$ under meets is ensured by Lemma \ref{CA closed under meets}, so we only need to show that $\CI X$ is closed in $\CP X$ under cotensors.
		
		For  $p\in[0,1]$ and $\phi\in\CI X$, set \[D=\{d\in X\mid p\leq\phi(d)\}.\] Then $D$ is    a directed subset of  $X_0$. First we show that  \[\{p\multimap y\mid y\in \chi(\phi)\}\subseteq D.\] Since $\phi=\sup_{z\in\chi(\phi)}X(-,z)$, then for all $y\in \chi(\phi)$,
		\begin{align*}p\ra\phi(p\multimap y)&=p\ra\sup_{z\in \chi(\phi)} X(p\multimap y,z)\\ &\geq p\ra X(p\multimap y,y)\\ &= X(p\multimap y,p\multimap y)\\ &=1,\end{align*}
		which implies that $p\multimap y\in D$.
		
		Next, let $\rho=\sup_{d\in D} X(-,d).$ We wish to show that $p\ra\phi=\rho$, from which the conclusion follows since $\rho$ is an ideal of $X$ and $p\ra\phi$ is the cotensor of $p$ with $\phi$ in $\CP X$.
		That $\rho\leq p\ra\phi$ is clear.   It remains to check that $p\ra\phi(x)\leq\rho(x)$ for all $x\in X$. If $p\leq\phi(x)$, then  $x\in D$ and \[  \rho(x)=\sup_{d\in D} X(x,d)\geq X(x,  x) =1= p\ra\phi(x).\]
		If  $p>\phi(x)$, then the implication operator $\ra$ is continuous at $(p,\phi(x))$, so, \begin{align*}p\ra\phi(x)&=p\ra\sup_{y\in \chi(\phi)} X(x,y)\\ &
			=\sup_{y\in \chi(\phi)}(p\ra X(x,y))   
			\\ &
			=\sup_{y\in \chi(\phi)} X(x,p\multimap y) \\ &
			\leq\sup_{d\in D} X(x,d) &(p\multimap y\in D)\\ &
			=\rho(x).
		\end{align*}
		
		$(4)\Rightarrow(5)$  Proposition \ref{comp mond}.
		
		$(5)\Rightarrow(1)$ Theorem \ref{CCD implies bbT}.
	\end{proof}
	
	\begin{rem} \begin{enumerate}[label=\rm(\roman*)] 
			\item Continuity of the implication operator off the diagonal is related to the equivalence of the logic formulas $$\exists x(p\ra q(x))\quad \text{and}\quad p\ra\exists xq(x).$$ So, in many-valued logic whether complete distributivity implies continuity depends on the structure of the truth-values. \item Ideals of real-enriched categories may be thought of as analogue of ind-objects of (set-valued) categories, but in contrast to the fact that the category of ind-objects of a small category with finite colimits is complete and cocomplete (see e.g. Johnstone \cite[Chapter VI, Section 1]{Johnstone82}), the real-enriched category $\mathcal{I}X$ of ideals of a complete real-enriched category $X$ may fail to be complete. Actually, by Proposition \ref{restriction to QInf} and the above theorem, in order that $\mathcal{I}X$  be complete  for every complete real-enriched category $X$,  it is necessary and sufficient that the implication operator of the continuous t-norm $\with$ is continuous at every point off the diagonal.  \end{enumerate}\end{rem}
	
	Let \[\QCL\] be the category of real-enriched continuous lattices and Yoneda continuous functors preserving enriched limits.  
	When the implication operator of the t-norm $\with$ is continuous at each point off the diagonal, $\QCL$ is the Eilenberg-Moore category of the composite monad $\bbI\circ\bbP^\dag$ of the copresheaf monad and the ideal monad, hence monadic over $\QOrd$. In this case, $\QCL$ is also monadic over the category of sets, as we see below.
	
	\begin{thm}\label{CL is monadic over Set} If the implication operator of  $\with$ is continuous at every point off the diagonal, then the forgetful functor \(U\colon \QCL\lra{\bf Set}\) is    monadic. \end{thm}

	We say that a functor $k\colon X\lra X$ is a \emph{kernel operator} if $k^2=k$ and $k(x)\sqsubseteq x$ in $X_0$ for all $x\in X$. If $k\colon X\lra X$ is a kernel operator, then $k\colon X\lra k(X)$ is right adjoint to the inclusion $k(X)\lra X$, so $k(X)$ is a retract of $X$ in $\QOrd$. In particular, if $X$ is a real-enriched continuous  lattice and $k\colon X\lra X$ is a Yoneda continuous kernel operator, then  $k\colon X\lra k(X)$ is a Yoneda continuous right adjoint, hence $k(X)$ is a real-enriched continuous  lattice by Corollary \ref{cc}.
	The verification of the following lemma is routine.
	
	\begin{lem}\label{cong on CL}
		Suppose that $X$ is a separated and complete real-enriched category. If $R$ is an equivalence relation on $X$ subject to the following conditions:
		\begin{enumerate}[label=\rm(\roman*)]  \item     $R$ is closed w.r.t. directed joins in $X_0 \times  X_0$,
			\item   $R$ is closed w.r.t. meets in $X_0\times  X_0$,
			\item If $(x,y)\in R$, then  $(r\multimap x, r\multimap  y)\in R$ for all $r\in[0,1]$,
		\end{enumerate} then  the map $k\colon   X\lra X$ that sends each $x\in X$ to the meet of $\{y\mid (x,y)\in R\}$ in $X_0$  is a Yoneda continuous kernel operator.
	\end{lem}
	
	\begin{proof}[Proof of Theorem \ref{CL is monadic over Set}] Since   the forgetful functor $U\colon \QCL\lra {\bf Set}$ is the composite of the forgetful functors $\QCL\lra \PdAlg$ and $\PdAlg\lra {\bf Set}$, it is a right adjoint. So it suffices to check that it creates split coequalizers. 
		
		Let $f,g\colon  X\lra Y$ be a parallel pair of morphisms in $\QCL$  and let  $h\colon  Y\lra Z$ be a  split coequalizer of $f,g$ in  ${\bf Set}$. By definition there exist morphisms  $Z\to^i Y\to^j X $ in ${\bf Set}$ such that
		$$h\circ f=h\circ g,~ f\circ j=\id,~ h\circ i=\id,~ g\circ j=i\circ h.$$
		Let \[R=\{(y_1,y_2)\in Y\times Y\mid  h(y_1)=h(y_2)\}.\] It is not hard to check that $(y_1,y_2)\in R$ if and only if there exists $(x_1,x_2)\in X\times X$ such that $g(x_1)=g(x_2)$, $y_1=f(x_1)$, $y_2=f(x_2)$. With help of this fact, one readily verifies that $R$ satisfies the conditions (i)-(iii) in Lemma \ref{cong on CL}, hence  determines a Yoneda continuous kernel operator $k\colon  Y\lra  Y$. Since $k(Y)$ is a real-enriched continuous  lattice with an underlying   set   equipotent to $Z$,   $Z$ can be made into a real-enriched continuous  lattice  so that $h\colon  Y\lra Z$ is a Yoneda continuous right adjoint. This proves that  the forgetful functor $\QCL\lra {\bf Set}$ creates split coequalizers. \end{proof}
	
	Suppose the implication operator of  $\with$ is continuous at each point off the diagonal and $F\colon{\bf Set}\lra\QCL$ is the left adjoint of the forgetful functor $U\colon \QCL\lra {\bf Set}$. In the following we show that the monad in the category of sets induced by the adjunction $F\dashv U$ is a submonad of the double  contravariant fuzzy powerset monad $(\mathscr{E},\sfm,{\sf e})$ in Section \ref{fuzzy powerset monad}.  
	
	Let $X$ be a real-enriched category. Following Antoniuk and Waszkiewicz \cite{AW2011} we define a \emph{filter} of $X$ to be an ideal of its opposite $X^{\rm op}$. In other words, a filter of $X$ is a coweight $\psi$ of $X$ such that $$\psi =\sup_{i\in D}\inf_{j\geq i} X(x_j,-) $$ for some net  $\{x_i\}_{i\in D}$  that is \emph{backward Cauchy} in the sense that $$\sup_{i\in D}\inf_{k\geq j\geq i}X(x_k,x_j)=1.$$
	
	\begin{defn} \label{Q-filter}   Let $X$ be a set. Then a filter of the  real-enriched category $([0,1]^X,\sub_X)$ is called a  conical filter  of    $X$. \end{defn} 
	
	\begin{prop}\label{conical filter} {\rm (Morsi \cite{Morsi1995a})} Let $X$ be a set. Then, $\mathfrak{F}\colon [0,1]^X\lra \sQ$ is a conical filter  if and only if  for all    $\lam,\mu\in [0,1]^X$ and all $r\in[0,1]$, \begin{enumerate}[label=\rm (CF\arabic*)] \item \label{FF3} $\sub_X(\lam,\mu)\leq\mathfrak{F}(\lam)\ra \mathfrak{F}(\mu)$;  
			\item \label{FF1} $\mathfrak{F}(1_X)=1$; \item \label{FF2} $\mathfrak{F}(\lam)\wedge\mathfrak{F}(\mu)=\mathfrak{F}(\lam\wedge\mu)$;  \item  $\mathfrak{F}(r\ra\lam)=1$ whenever $\mathfrak{F}(\lam)>r$.\end{enumerate} \end{prop}
	
	\begin{proof} Follows from Corollary \ref{ideal of complete A} directly. \end{proof} 
	
	\begin{prop}\label{CSF is closed under meets} Let $X$ be a set.    \begin{enumerate}[label=\rm(\roman*)]  \item   For each family     $\{\mathfrak{F}_i\}_{i\in J}$ of conical  filters  of $X$,  the meet $\bw_{i\in J}\mathfrak{F}_i$ is a conical filter of $X$.
			\item  For each directed family  $\{\mathfrak{F}_i\}_{i\in J}$ of  conical  filters  of $X$,  the join $\bv_{i\in J}\mathfrak{F}_i$ is a conical filter of $X$. 
			\item If  the implication operator $\ra\colon[0,1]\times[0,1]\lra[0,1]$ is continuous at each point off the diagonal, then for each $p\in[0,1]$ and each conical filter $\mathfrak{F}$ of $X$, $p\ra\mathfrak{F}$ is a conical filter of $X$.
	\end{enumerate}   \end{prop}
	
	\begin{proof} (i) and (ii) are a special case of Corollary \ref{CA closed under meets}; (iii) is a special case of the implication $(3)\Rightarrow(4)$ in Theorem \ref{CD implies CL}. \end{proof}
	
	For each set $X$ let $\mathscr{F}X$ be the set of all conical filters of $X$; or equivalently, $\mathscr{F}X=\CI \CPd X$ with $X$ viewed as a discrete real-enriched category. For each function $f\colon X\lra Y$ and each conical filter $\mathfrak{F}$ of $X$, the function $$f(\mathfrak{F})\colon [0,1]^Y\lra [0,1], \quad \mu\mapsto \mathfrak{F}(\mu\circ f)$$ is  a conical filter of $Y$.  In this way we obtain a functor $ \mathscr{F}\colon {\bf Set}\lra{\bf Set},$  called the \emph{conical filter functor}. The conical filter functor $\mathscr{F}$ is a subfunctor of the double  contravariant fuzzy powerset functor $\mathscr{E}\colon {\bf Set}\lra{\bf Set}$, also a subfunctor of the functor $\mathscr{C}\colon {\bf Set}\lra{\bf Set}$ in Section \ref{cdl}.

	\begin{lem}\label{conical filter form a monad}   {\rm (Lai, D\thinspace{\&}\thinspace G Zhang \cite{LZZ2021})}
		The following are equivalent:
		\begin{enumerate}[label=\rm(\arabic*)] 
			\item The implication operator $\ra\colon[0,1]\times[0,1]\lra[0,1]$ is continuous at each point off the diagonal $\{(x,x)\mid x\in[0,1]\}$. 
			
			\item The conical filter functor $\mathscr{F}\colon {\bf Set}\lra{\bf Set}$ is a submonad of the double  contravariant fuzzy powerset monad $(\mathscr{E},\sfm,{\sf e})$. \end{enumerate} \end{lem} 
	
	\begin{proof} $(1)\Rightarrow(2)$ By Lemma \ref{submonad of E} we only need to check  that for each set $X$ and each conical filter $\mathcal{F}$ of $\mathscr{F}X$, the function \[ k(\mathcal{F})\colon [0,1]^X\lra [0,1],\quad  \lam\mapsto \mathcal{F}(\widehat{\lam})\] is a conical filter of $X$, where the function $\widehat{\lam}\colon \mathscr{F}X\lra[0,1]$ maps each conical filter $\mathfrak{F}$ of $X$ to $\mathfrak{F}(\lam)$. 
		
		Since $\mathcal{F}$ is a conical filter of $\mathscr{F}X$, there is a directed subset $D$ of $[0,1]^{\mathscr{F}X}$ (ordered pointwise) such that $$\mathcal{F}=\sup_{\xi\in D}\sub_{\mathscr{F}X}(\xi,-).$$ Then for each $\lam\in [0,1]^X$ it holds that \begin{align*} \mathcal{F}(\widehat{\lam})&= \bv_{\xi\in D}\sub_{\mathscr{F}X}(\xi,\widehat{\lam})  = \bv_{\xi\in D} \bw_{\mathfrak{F}\in \mathscr{F}X}(\xi(\mathfrak{F})\ra \mathfrak{F})(\lam),\end{align*} which shows that   \[ k(\mathcal{F}) =  \bv_{\xi\in D} \bw_{\mathfrak{F}\in \mathscr{F}X}(\xi(\mathfrak{F})\ra \mathfrak{F}),\] hence $k(\mathcal{F})$ is a conical filter of $X$ by Proposition \ref{CSF is closed under meets}. The conical filter $k(\mathcal{F})$ is called the \emph{diagonal filter}  or the \emph{Kowalsky sum} of $\mathcal{F}$. 
		
		$(2) \Rightarrow(1)$ First we show that if the conical filter functor $\mathscr{F}\colon {\bf Set}\lra{\bf Set}$ is a submonad of   $(\mathscr{E},\sfm,{\sf e})$, then for each set $X$, each conical filter $\mathfrak{F}$ of $X$ and each $r\in [0,1]$, the function $r\ra\mathfrak{F}$ is a conical filter of $X$. Let $\mathfrak{i}\colon \mathscr{F}\lra\mathscr{E}$ be the inclusion  transformation and   $\mathcal{F}$ be the conical filter of $\mathscr{F}X$ generated  by the subset $$\{\xi\colon\mathscr{F}X\lra[0,1]\mid \xi(\mathfrak{F})\geq r\}$$ of $[0,1]^{\mathscr{F}X}$. Since $\mathscr{F}$ is a submonad of $(\mathscr{E},\sfm,{\sf e})$, then $\sfm_X\circ(\mathfrak{i}*\mathfrak{i})_X(\mathcal{F})$ is a conical filter of $X$. Since for each $\lam\in[0,1]^X$, \[\sfm_X\circ(\mathfrak{i}*\mathfrak{i})_X(\mathcal{F})(\lam)= \mathcal{F}(\widehat{\lam})=\bv_{\xi(\mathfrak{F})\geq r}\sub_{\mathscr{F}X}(\xi,\widehat{\lam})=r\ra \mathfrak{F}(\lam),\] it follows that $r\ra \mathfrak{F}$ is a conical filter of $X$, as desired. 
		
		Now we prove the conclusion.
		Suppose on the contrary that the implication operator $\ra\colon[0,1]\times[0,1]\lra[0,1]$ is not continuous at some point not on the diagonal. Then there exist idempotent elements $p,q$   of $\&$ such that $0<p<q$  and that the restriction of $\& $ on $[p,q]$ is isomorphic to the {\L}ukasiewicz t-norm. Pick  $t,s\in(p,q)$ such that $p<s<t\ra p <q$.
		
		Let $X=[0,1]$ and pick a strictly increasing sequence $\{b_n\}_n$  of $[0,1]$ that converges to $1$. Consider the conical  filter $\mathfrak{F}$ of $X$ generated by the subset $$\{1_{B_n}\mid   n\geq 1\}$$ of $[0,1]^X$,   where $B_n=\{b_m \mid m\geq n\}$; that means for all $\mu\in[0,1]^X$,   $$ \mathfrak{F}(\mu)  = \bv_{n\geq 1}\bw_{m\geq n}\mu(b_m).$$ 
		
		In the following we derive a contradiction by showing that $ t\ra\mathfrak{F} $ is not a conical filter of $X$.  Consider the function $$ \lam \colon X\lra[0,1], \quad \lam(x)=p x.$$  Since \[\mathfrak{F}(\lam)=p\bv_{n\geq 1}\bw_{m\geq n}b_m=p,\] it follows that \[s<t\ra p=(t\ra\mathfrak{F})(\lam).\] But, \begin{align*} (t\ra\mathfrak{F})(s\ra\lam)  &=t\ra\bv_{n\geq 1}\bw_{m\geq n}(s\ra\lam(b_m))\\ 
			&=t\ra\bv_{n\geq 1}\bw_{m\geq n} \lam(b_m)&(\lam(b_m)< p<s)) \\ 
			&=t\ra p\\ &<1. \end{align*} This shows, by Proposition \ref{conical filter}, that $t\ra\mathfrak{F} $ is not a conical filter of $X$. \end{proof}
	
	Suppose the implication operator $\ra\colon[0,1]\times[0,1]\lra[0,1]$ is continuous at each point off the diagonal and $F\colon{\bf Set}\lra\QCL$ is the left adjoint of the forgetful functor $U\colon \QCL\lra {\bf Set}$. Then, it is readily verified that the monad induced by   $F\dashv U$ is the conical filter momad $(\mathscr{F},\sfm,{\sf e})$.   This proves:
	
	\begin{prop} The conical filter functor $\mathscr{F}\colon {\bf Set}\lra{\bf Set}$ is a submonad of the double contravariant fuzzy powerset monad $ (\mathscr{E}, \sfm,{\sf e})$ if and only if the implication operator of the continuous t-norm   is continuous at each point off the diagonal. In this case, the Eilenberg-Moore category of $ (\mathscr{F}, \sfm,{\sf e})$ is the category of real-enriched continuous lattices and Yoneda continuous functors preserving  enriched limits. \end{prop}
	  
\begin{ques}
	A conical filter $\mathfrak{F}$ of a set $X$ is  \emph{proper} if $\mathfrak{F}(p_X)=p$ for all $p\in[0,1]$. Assigning to each set $X$ the set of all  conical proper filters of $X$ gives rise to a subfunctor of conical filter functor $\mathscr{F}$. 
	It is not hard to check that if the implication operator $\ra\colon[0,1] \times[0,1] \lra[0,1]$ is continuous at each point off the diagonal, then the conical proper filter functor is a submonad of   $(\mathscr{F},\sfm,{\sf e})$.   What are the Eilenberg-Moore algebras of this monad?
\end{ques}

\begin{ques}Suppose   $\bbT$ is a submonad of the presheaf monad   over which the copresheaf monad $\bbP^\dag$ distributes. By Corollary \ref{Talg monadic over QCat}, the category $\TCL$ of complete and continuous $\bbT$-algebras and right adjoint $\bbT$-homomorphisms is monadic over $\QOrd$.  Is $\TCL$  monadic over the category of sets too? \end{ques}

\begin{ques} When does the copresheaf monad $\bbP^\dag$ distributes over the the flat weight monad  or over the conically flat weight monad? \end{ques}

\section{Saturation}

For each class of weights $\CT$, no matter   saturated or not, we write \[\CT\text{-}{\sf Alg}\] for the category of separated  $\CT$-cocomplete real-enriched categories and $\CT$-colimits preserving functors. If $\CT$ happens to be saturated, then ${\CT}\text{-}{\sf Alg}$ is  the category of Eilenberg-Moore algebras of the monad $\bbT=(\CT,{\sf m},\sy)$. In this section, we show that for each class of weights $\CT$, there is a unique saturated class of weights $\widetilde{\CT}$ such that ${\CT}\text{-}{\sf Alg}=\widetilde{\CT}\text{-}{\sf Alg}$, which implies, in particular, that ${\CT}\text{-}{\sf Alg}$ is category of Eilenberg-Moore algebras of a unique submonad of the presheaf monad.

The main result of section, namely   Theorem \ref{saturation of class of weights}, is a special case of a result of Albert and Kelly \cite{AK98} on saturation of class of weights in enriched category theory.

\begin{thm}\label{saturation of class of weights} For each class of weights $\CT$, there exists a unique saturated class   of weights $\widetilde{\CT}$, called the saturation of  $\CT$, such that  \begin{enumerate}[label={\rm(\roman*)}] 
		\item  a   real-enriched category $X$  is $\CT$-cocomplete if and only if it is $\widetilde{\CT}$-cocomplete;
		\item  a  functor   $f\colon X\lra Y$ between  $\CT$-cocomplete real-enriched categories preserves $\CT$-colimits if and only if it preserves $\widetilde{\CT}$-colimits.
	\end{enumerate} In particular, ${\CT}\text{-}{\sf Alg}=\widetilde{\CT}\text{-}{\sf Alg}$.
\end{thm} 

To prove the theorem, we need to introduce two classes of weights: $\widehat{\CT} $ and $\overline{\CT} $. 

For each real-enriched category $X$, define a subcategory $$\widehat{\CT}X$$ of $\CP X$ by putting a weight $\phi$ of $X$ in  $\widehat{\CT}X$ if, for any morphism $g\colon Y\lra Z$ of ${\CT}\text{-}{\sf Alg}$ and any functor $f\colon X\lra Y$, the colimit of $f$ weighted by $\phi$ always exists and is preserved by $g$, i.e.  $g(\colim_\phi f)=\colim_\phi(g\circ f) .$

It is readily verified that \begin{enumerate}[label={\rm(\roman*)}] 
	\item $\widehat{\CT}$ is a class of weights and contains the class $\CT$;  
	\item $\widehat{\widehat{\CT}} =\widehat{\CT}$; 
	\item  $\widehat{\CT}$ is the largest class of weights for which $\widehat{\CT}\text{-}{\sf Alg}={\CT}\text{-}{\sf Alg}.$ \end{enumerate}

Now we define the the class $\overline{\CT} $.
Suppose $X$ is a real-enriched category, $A$ is a subset of $\CP X$, viewed as a subcategory.   We say that $A$  is  \emph{closed under  $\CT$-colimits} if, for each $\phi\in \CT A$, the colimit of the inclusion functor $i\colon A\lra\CP X$ weighted by $\phi$  belongs to $A$. 
It is readily seen that if  $A\subseteq \CP X$  is closed under $\CT$-colimits, then for each functor $g\colon K\lra A$ and each $\phi\in\CT K$, the colimit of $i\circ g$ weighted by $\phi$ belongs to $A$.  

For each real-enriched category $X$, let $$\overline{\CT}X$$ be the intersection of all subsets of $\CP X$ that contain all representable weights of $X$ and are closed under  $\CT$-colimits. Then $\overline{\CT}X$ is the least such subset of $\CP X$. Let \begin{itemize} \item $\mathfrak{i}_X\colon \overline{\CT}X\lra\CP X$ be the inclusion functor; and \item  ${\sf t}_X\colon  X\lra \overline{\CT}X$ be the Yoneda embedding with  codomain restricted to $\overline{\CT}X$. \end{itemize} The composite $\mathfrak{i}_X\circ {\sf t}_X$ is the Yoneda embedding $\sy_X\colon X\lra\CP X$.

\begin{lem} For each real-enriched category $X$, $\CT X\subseteq \overline{\CT}X$.\end{lem}

\begin{proof} 
	Because for each weight $\phi$ of $X$, $\phi$ is the colimit of $\mathfrak{i}_X\circ{\sf t}_X $ weighted by $\phi$. \end{proof}

\begin{lem}\label{Tbar is functorial} For each functor $f\colon X\lra Y$   and each $\phi\in \overline{\CT}X$,   $f_\exists(\phi)\in \overline{\CT}Y$. \end{lem}

\begin{proof} Let \[Z  =
	\{\phi\in\overline{\CT}X\mid f_\exists(\phi)\in\overline{\CT}Y\}.\] It is clear that $Z$ contains all representable weights of $X$. If we can show that $Z$ is closed under  $\CT$-colimits, then $Z= \overline{\CT}X$ and the conclusion follows.
	
	Let $j\colon  Z \lra\CP X$ be the inclusion functor. We show that for each
	$\Phi\in \CT Z$, the colimit of $j$ weighted by $\Phi$ belongs to $ Z$. 
	
	Since $f_\exists\colon\CP  X\lra \CP Y$ preserves colimits, then $ f_\exists({\colim}_\Phi j)  = {\colim}_\Phi(f_\exists \circ j)$.  Since $f_\exists\circ j$ factors through $\overline{\CT}Y$ by definition of $Z$, $\overline{\CT}Y$ is  closed under $\CT$-colimits, it follows that $${\colim}_\Phi(f_\exists \circ j)\in \overline{\CT}Y,$$ which implies that ${\colim}_\Phi j \in Z$, as desired.
\end{proof}

The above two lemmas imply that $\overline{\CT}$ is a class of weights on $\QOrd$.  The assignment $X\mapsto \overline{\CT}X$ defines a functor $$\overline{\CT}\colon\QOrd \lra{\CT}\text{-}{\sf Alg},$$ which is indeed  left adjoint to  the forgetful functor $U\colon {\CT}\text{-}{\sf Alg} \lra \QOrd $, as we see below. This is a  special case of a general result of Kelly \cite{Kelly} in the theory of enriched categories.

\begin{prop} \label{saturation} The functor $\overline{\CT}\colon \QOrd \lra{\CT}\text{-}{\sf Alg} $ is left adjoint to the forgetful functor $U\colon {\CT}\text{-}{\sf Alg} \lra \QOrd$.  
\end{prop}

To prove Proposition \ref{saturation} we need a lemma.

\begin{lem} \label{phi in S(Y) has a sup} For each  $\CT$-cocomplete real-enriched category $Y$,  \[\overline{\CT}Y\subseteq \{\phi\in\CP Y\mid \phi~\text{has a colimit}\}.\] \end{lem}

\begin{proof}It suffices to show that
	\[W \coloneqq\{\phi\in\CP Y\mid \phi~\text{has a colimit}\}\]
	is closed under   $\CT$-colimits.
	
	Let $j\colon W \lra\CP Y$ be the inclusion map. We wish to show that for each $\Phi\in\CT W$,  the colimit $\colim_\Phi j$, as a weight  of $Y$, has a colimit. First of all, by Corollary \ref{calculation of colimit in PX} we have \[{\colim}_\Phi j=\sup_{\phi\in W}\Phi(\phi)\with \phi.\]
	Since every $\phi\in W$ has a colimit, assigning to each $\phi\in W$  its colimit defines a functor $h\colon W \lra Y$. Since $Y$ is $\CT$-cocomplete, the colimit of   $h\colon W \lra Y$ weighted by $\Phi$ exists, say $b$. We claim that $b$ is a colimit of the weight $\colim_\Phi j$ of $Y$.
	In fact,  for all $y\in Y$,  
	\begin{align*}Y(b,y)&=\CP Y(\Phi\circ h^*,\sy(y)) \\
		&=\CP W(\Phi,h^{-1}(\sy(y))) \\
		&= \inf_{\phi\in W} (\Phi(\phi)\ra Y(h(\phi),y) )\\
		&= \inf_{\phi\in W}\Big[\Phi(\phi)\ra\Big( \inf_{z\in Y}(\phi(z)\ra Y(z,y) )\Big) \Big] \\
		&= \inf_{z\in Y}\Big[\Big(\sup_{\phi\in W}\Phi(\phi)\with \phi(z)\Big) \ra Y(z,y) \Big]\\
		&= \CP Y({\colim}_\Phi\thinspace j,\sy(y)),
	\end{align*} which shows that $b$ is a colimit of the weight $\colim_\Phi j$ of $Y$. \end{proof}

\begin{proof}[Proof of Proposition \ref{saturation}]
	It suffices to show that for each functor $f\colon X\lra Y$ with $Y$ in ${\CT}\text{-}{\sf Alg}$, there is a unique $\CT$-colimits preserving functor  $\overline{f}\colon \overline{\CT}X\lra Y$ such that $f=\overline{f}\circ  {\sf t}_X$.  
	
	Since  the real-enriched category $W$ given in Lemma \ref{phi in S(Y) has a sup} is  $\CT$-cocomplete, it is readily  verified that $\overline{\CT}X$ is contained in   \[\{\phi\in\CP X\mid f_\exists(\phi)\in W\}.\] Let $h\colon W\lra Y$ be the functor that sends each $\phi$ of $W$ to its colimit. Then   \[\overline{f}\coloneqq h\circ f_\exists\colon  \overline{\CT}X\lra W \lra Y\]  satisfies the requirement.   This proves the existence. To see the uniqueness, suppose  
	$g\colon \overline{\CT}X\lra Y$ satisfies the requirement. Then  \[\{\phi\in\overline{\CT}X\mid g(\phi)=  \overline{f}(\phi)\}\] contains all representable weights of $X$ and is closed under  $\CT$-colimits. Hence,   $g=\overline{f}$ by definition of  $\overline{\CT}X$.
\end{proof}

For each   $Y$ of ${\CT}\text{-}{\sf Alg}$, the functor  $\colim\colon \overline{\CT}Y\lra Y$   sending each $\phi$ to its colimit is the component  of the counit of the adjunction $\overline{\CT}\dashv U$ in Proposition \ref{saturation}. Hence for each  morphism  $g\colon Y\lra Z$ of ${\CT}\text{-}{\sf Alg}$, the following square is commutative: \begin{equation*}\label{counit of S dashv U}\bfig\Square[\overline{\CT}Y`\overline{\CT}Z`Y`Z; g_\exists`\colim`\colim`g]\efig\end{equation*} This fact will be used in the proof of Theorem \ref{saturation of class of weights}.

\begin{proof}[Proof of Theorem \ref{saturation of class of weights}] For existence we show that $\widetilde{\CT}\coloneqq\widehat{\CT}$ satisfies the requirement.   Since  ${\CT}\text{-}{\sf Alg} =\widehat{\CT}\text{-}{\sf Alg}$, it remains to show that $\widehat{\CT}$ is saturated.   
	
	Since $\overline{\CT} X$ is closed under $\CT$-colimits in $\CP X$, if we can prove that $\overline{\CT} X=\widehat{\CT}X$  for each real-enriched category $X$,  then $\widehat{\CT}X$ is closed under $\CT$-colimits, hence under $\widehat{\CT}$-colimits, and consequently, $\widehat{\CT}$ is saturated.
	
	First  we prove that $\overline{\CT}X\subseteq\widehat{\CT}X$. That means, if $\phi\in \overline{\CT}X$ and $g\colon Y\lra Z$ is a morphism in ${\CT}\text{-}{\sf Alg}$, then for any functor $f\colon X\lra Y$, the colimit of $f$ weighted by $\phi$   exists and is preserved by $g$. Since $f_\exists(\phi)\in \overline{\CT}Y$,   the colimit of   $f$ weighted by $\phi$ exists. By commutativity of the above  square we have $\colim(g_\exists\circ f_\exists(\phi))= g(\colim f_\exists(\phi))$,  hence the colimit of $f$ weighted by $\phi$    is preserved by $g$.  
	
	Next we prove that $\widehat{\CT}X\subseteq\overline{\CT}X$.   
	Since ${\CT}\text{-}{\sf Alg} =\widehat{\CT}\text{-}{\sf Alg}$, each $A$ of ${\CT}\text{-}{\sf Alg}$ is $\widehat{\CT}$-cocomplete.
	Since $\overline{\CT}X$ belongs to ${\CT}\text{-}{\sf Alg}$, it is  $\widehat{\CT}$-cocomplete.  Since the inclusion   $\mathfrak{i}_X\colon \overline{\CT}X\lra\CP X$ preserves $\CT$-colimits,  it  preserves $\widehat{\CT}$-colimits too. Then, for each $\phi\in \widehat{\CT}X$,  $$ {\colim}_\phi\thinspace{\sf t}_X = {\colim}_\phi(\mathfrak{i}_X\circ{\sf t}_X)={\colim}_\phi\thinspace\sy_X=\phi,$$  hence $\phi\in \overline{\CT}X$.
	
	To see uniqueness, suppose $\mathcal{S} $ is a saturated class of weights with $\mathcal{S}\text{-}{\sf Alg}= {\CT}\text{-}{\sf Alg}$.  Then we prove that $\mathcal{S} =\widetilde{\CT}$ by showing that $\mathcal{S}$ is a subclass of $\widehat{\CT}$ and contains the class $\overline{\CT}$.
	
	Since $\widehat{\CT}$  is the largest class of weights for which ${\CT}\text{-}{\sf Alg} =\widehat{\CT}\text{-}{\sf Alg}$,   $\mathcal{S}$ is   a subclass of $\widehat{\CT}$.
	To see that $\mathcal{S}$   contains the class $\overline{\CT}$, first  we show that $\mathcal{S}$   contains the class $\CT$; that is, $\CT X\subseteq \mathcal{S}X$ for each real-enriched category $X$. Since $\mathcal{S}$ is saturated, then \begin{itemize} \item  $\mathcal{S}X$ is $\mathcal{S}$-cocomplete, hence $\CT$-cocomplete; \item  the inclusion  functor $j_X\colon \mathcal{S} X \lra\CP X$ preserves $\mathcal{S} $-colimits, hence $\CT$-colimits. \end{itemize} Let ${\sf s}_X\colon X\lra \mathcal{S}X $ be  the Yoneda embedding with codomain restricted  to $\mathcal{S}X $.  Then for each $\phi\in \CT X$ we have   \[{\colim}_\phi\thinspace  {\sf s}_X = {\colim}_\phi(j_X\circ {\sf s}_X)  = {\colim}_\phi\thinspace \sy_X =\phi,\] so $\phi\in\mathcal{S}X$ and $\CT X\subseteq\mathcal{S}X $.
	
	Now we show that $\mathcal{S}$   contains  $\overline{\CT}$. Since $\mathcal{S}$ is saturated and contains the class $\CT$, it follows that for each real-enriched category $X$, $\mathcal{S}X $ is closed in $\CP X$ under $\CT$-colimits, then  $\overline{\CT}X\subseteq \mathcal{S}X$.
\end{proof}

\bibliographystyle{plain}

\end{document}